\documentclass[10pt]{amsart}
\usepackage{amssymb}

\newtheorem{thm}{Theorem}[section]
\newtheorem{lemma}[thm]{Lemma}
\newtheorem{proposition}[thm]{Proposition}

\newtheorem{definition}[thm]{Definition}
\newtheorem{corollary}[thm]{Corollary}

\newtheorem{notation}[thm]{Notation}

\newcommand{\p}{\mathbb{P}}
\newcommand{\q}{\mathbb{Q}}

\newcommand{\ot}{\mathrm{ot}}
\newcommand{\cf}{\mathrm{cf}}
\newcommand{\cof}{\mathrm{cof}}

\newcommand{\ran}{\mathrm{ran}}

\newcommand{\cl}{\mathrm{cl}}

\begin{document}

\title{Mitchell's Theorem Revisited}

\author{Thomas Gilton and John Krueger}

\address{Thomas Gilton \\ Department of Mathematics \\
University of California, Los Angeles\\
Box 951555\\
Los Angeles, CA 90095-1555}
\email{tdgilton@math.ucla.edu}

\address{John Krueger \\ Department of Mathematics \\ 
University of North Texas \\
1155 Union Circle \#311430 \\
Denton, TX 76203}
\email{jkrueger@unt.edu}

\date{June 2015; revised November 2016}

\thanks{2010 \emph{Mathematics Subject Classification:} 
Primary 03E35, 03E40; Secondary 03E05.
}

\thanks{\emph{Key words and phrases.} Forcing, approachability ideal, side conditions, adequate sets.}

\begin{abstract}
Mitchell's theorem on the approachability ideal states that it is consistent 
relative to a greatly Mahlo cardinal that there is no stationary subset 
of $\omega_2 \cap \cof(\omega_1)$ in the approachability ideal $I[\omega_2]$. 
In this paper we give a new proof of Mitchell's theorem, deriving it from 
an abstract framework of side condition methods.
\end{abstract}

\maketitle

\tableofcontents

\newpage

\begin{center}
\textbf{Introduction} 
\end{center}

\bigskip

The approachability ideal $I[\lambda^+]$, for an uncountable cardinal $\lambda$, 
is defined as follows. 
For a given sequence $\vec a = \langle a_i : i < \lambda^+ \rangle$ 
of bounded subsets of 
$\lambda^+$, let $S_{\vec a}$ denote the set of limit ordinals 
$\alpha < \lambda^+$ for which there 
exists a set $c \subseteq \alpha$, which is club in $\alpha$ with order type $\cf(\alpha)$, such that for all 
$\beta < \alpha$, there is $i < \alpha$ with $c \cap \beta = a_i$. 
Intuitively speaking, the set $S_{\vec a}$ carries a kind of weak square sequence, namely a sequence of clubs 
such that for each $\alpha$ in $S_{\vec a}$, 
the club attached to $\alpha$ has its initial segments 
enumerated at stages prior to $\alpha$. 
Define $I[\lambda^+]$ as the collection of sets $S \subseteq \lambda^+$ for which there exists a sequence 
$\vec a$ as above and a club $C \subseteq \lambda^+$ such that 
$S \cap C \subseteq S_{\vec a}$.
In other words, $I[\lambda^+]$ is the ideal of subsets of $\lambda^+$ which is generated modulo the 
club filter by sets of the form $S_{\vec a}$.

Let $\lambda$ be a regular uncountable cardinal. 
Shelah \cite{shelah} proved that the set $\lambda^+ \cap \cof(< \! \lambda)$ is in $I[\lambda^+]$. 
Therefore the structure of $I[\lambda^+]$ is determined by which subsets of 
$\lambda^+ \cap \cof(\lambda)$ belong to it. 
At one extreme, the weak square principle $\Box_\lambda^*$ implies that 
$\lambda^+ \cap \cof(\lambda)$ is in $I[\lambda^+]$; 
therefore $I[\lambda^+]$ is just the power set of $\lambda^+$. 
The opposite extreme would be that no stationary subset of 
$\lambda^+ \cap \cof(\lambda)$ belongs to $I[\lambda^+]$, in other words, 
that $I[\lambda^+]$ is the nonstationary ideal when restricted to cofinality $\lambda$. 
Whether the second extreme is consistent was open for several decades, and was eventually 
solved by Mitchell \cite{mitchell}. 
Mitchell proved that it is consistent, relative to the consistency of a greatly Mahlo cardinal, 
that there does not exist a stationary subset of $\omega_2 \cap \cof(\omega_1)$ in $I[\omega_2]$. 
We will refer to this result as \emph{Mitchell's theorem}.

Mitchell's theorem is important not only for solving a deep and long-standing 
open problem in combinatorial set theory, but also for introducing powerful new 
techniques in forcing. 
A basic tool in the proof is a forcing poset for adding a club subset 
of $\omega_2$ with finite conditions, using finite sets of countable models as side conditions. 
A similar forcing poset was introduced by Friedman \cite{friedman} 
around the same time. 
The use of countable models in Friedman's and Mitchell's forcing posets for adding a club 
expanded the original side condition method 
of {Todor\v cevi\' c} \cite{todor}, which was designed to add a generic object of size $\omega_1$, 
to adding a generic object of size $\omega_2$. 
In addition, Mitchell's proof introduced the new concepts of strongly generic conditions 
and strongly proper forcing posets, which are closely related to the approximation property.

Several years later, Neeman \cite{neeman} developed a general framework of side 
conditions, which he called \emph{sequences of models of two types}. 
An important distinction between Neeman's side conditions and those of Friedman and 
Mitchell is that the two-type side conditions include both countable and uncountable models. 
A couple of years later, Krueger \cite{jk21} developed an alternative framework of 
side conditions called \emph{adequate sets}. 
This approach bases the analysis of side conditions on the ideas of 
the \emph{comparison point} and \emph{remainder points} of two countable models. 
Notably, this approach has led to the solution of an open problem of Friedman \cite{friedman}, 
by showing how to add a club subset of $\omega_2$ with finite conditions while 
preserving the continuum hypothesis (\cite{jk25}). 
Other applications are given in \cite{jk22}, \cite{jk23}, \cite{jk24}, and \cite{jk26}.

Notwithstanding the merits of the frameworks of Neeman \cite{neeman} and 
Krueger \cite{jk21}, these frameworks are limited in the sense that they are intended to add 
a single subset of $\omega_2$ (or of a cardinal $\kappa$ which is collapsed to become $\omega_2$). 
The proof of Mitchell's theorem, on the other hand, involves adding $\kappa^+$ many club subsets 
of a cardinal $\kappa$. 
Many consistency proofs in set theory about a cardinal $\kappa$ involve adding $\kappa^+$ 
many subsets of $\kappa$ by forcing, so that each of the potential 
counterexamples to the statement being forced 
is captured in some intermediate generic extension and dealt with by the rest of the forcing extension.

The goal of this paper is to extend the framework of adequate sets to allow for 
adding many subsets of $\omega_2$, or of a cardinal $\kappa$ which is collapsed to 
become $\omega_2$. 
The purpose of this extension is to provide general tools which will be useful for 
proving new consistency results on $\omega_2$. 
In Part III we give an example by deriving Mitchell's 
theorem from the abstract framework developed in Parts I and II. 
The paper includes a very detailed treatment of adequate sets and remainder points in Sections 1 and 2, 
and of Mitchell's application of the square principle to side conditions in Sections 7 and 8.  
We also develop some new ideas, including canonical models in 
Sections 9 and 10, and the main proxy lemma in Section 11.

We will analyze finite sets of countable elementary substructures of $H(\kappa^+)$. 
The method of adequate sets handles the interaction of the models below $\kappa$. 
Following Mitchell, we employ the square principle $\Box_{\kappa}$ to describe and control the 
interaction of countable models between $\kappa$ and $\kappa^+$. 
We introduce a new kind of side condition, which we call an 
\emph{$\vec S$-obedient side condition}. 
We show that the forcing poset consisting of $\vec S$-obedient side conditions on 
$H(\kappa^+)$, where $\kappa$ is a greatly Mahlo cardinal, ordered by 
component-wise inclusion, forces that $\kappa = \omega_2$ and 
there is no stationary subset of $\omega_2 \cap \cof(\omega_1)$ 
in the approachability ideal $I[\omega_2]$.

\bigskip

This project began with the M.S.\ thesis of Gilton at the University of North Texas, 
in which he reconstructed the original proof of 
Mitchell's theorem in the context of adequate sets. 
Krueger is indebted to Gilton for explaining to him many of the details of Mitchell's proof, 
especially the use of $\Box_\kappa$. 
Gilton isolated a workable requirement on remainder points which later evolved into the idea of 
$\vec S$-obedient side conditions.

After Gilton's thesis was complete, Krueger returned to the problem and made a number of advances. 
Krueger developed the new idea of canonical models, which is dealt with in Sections 9 and 10. 
Canonical models are models which appear in a given model $N$, reflect information about 
models lying outside of $N$, and are determined by canonical parameters which arise in the 
comparison of models. 
He isolated the main proxy lemma, Lemma 11.5, which significantly simplifies the method of proxies 
used by Mitchell. 
And he introduced the idea of $\vec S$-obedient side conditions, and showed that forcing with 
pure side conditions on a greatly Mahlo cardinal produces a generic extension in which the 
approachability ideal on $\omega_2$ restricted to cofinality $\omega_1$ is the nonstationary ideal.

\bigskip

This paper was written for an audience with a minimum 
background of one year of graduate studies in set theory, with a working knowledge of 
forcing and proper forcing, and with some familiarity with generalized stationarity.

For a regular uncountable cardinal $\mu$ and a set $X$ with $\mu \subseteq X$, we let 
$P_\mu(X)$ denote the set $\{ a \subseteq X : |a| < \mu \}$. 
A set $S \subseteq P_\mu(X)$ being stationary is equivalent to the statement that for 
any function $F : X^{<\omega} \to X$, there exists $a \in S$ such that $a \cap \mu \in \mu$ and 
$a$ is closed under $F$. 

If $a$ is a set of ordinals, then $\lim(a)$ denotes the set of ordinals $\beta$ such that 
for all $\gamma < \beta$, $a \cap (\gamma,\beta) \ne \emptyset$. 
We let $\cl(a) = a \cup \lim(a)$. 
If $M$ is a set, we write $\sup(M)$ to denote $\sup(M \cap On)$.

If $\mathcal A$ is a structure in a first order language, and $X_1,\ldots,X_k$ are subsets of 
the underlying set of $\mathcal A$, then we write 
$(\mathcal A,X_1,\ldots,X_k)$ to denote the expansion of the structure $\mathcal A$ 
obtained by adding $X_1,\ldots,X_k$ as predicates.

\bigskip

\part{Basic side condition methods}

\bigskip

\addcontentsline{toc}{section}{1. Adequate sets}

\textbf{\S 1. Adequate sets}

\stepcounter{section}

\bigskip

We begin the paper by working out the basic framework of adequate sets. 
Roughly speaking, this framework provides methods for describing and handling 
the interaction of countable elementary substructures below $\omega_2$, 
or below $\kappa$ for some regular uncountable cardinal $\kappa$ which 
is intended to become $\omega_2$ in a forcing extension.

Adequate sets were introduced by Krueger \cite{jk21}; 
many of the results of this section 
appear in \cite{jk21}, although in a slightly different form.

\bigskip

We fix objects 
$\kappa$, $\lambda$, $T^*$, $\pi^*$, $C^*$, $\Lambda$, 
$\mathcal X_0$, and $\mathcal Y_0$ as follows.

\begin{notation}
For the remainder of the paper, $\kappa$ is a regular cardinal 
with $\omega_2 \le \kappa$.
\end{notation}

In \cite{jk21} we only considered the case when 
$\kappa = \omega_2$. 
In the proof of Mitchell's theorem given in Part III, $\kappa$ is a greatly Mahlo 
cardinal.

\begin{notation}
Fix a cardinal $\lambda$ such that $\kappa \le \lambda$. 
In Parts II and III we will let $\lambda = \kappa^+$.
\end{notation}

\begin{definition}
A set $T \subseteq P_{\omega_1}(\kappa)$ is \emph{thin} 
if for all $\beta < \kappa$, 
$$
|\{ a \cap \beta : a \in T \}| < \kappa.
$$
\end{definition}

The idea of a thin stationary set was introduced by Friedman \cite{friedman}, 
who used a thin stationary set to develop a forcing poset 
for adding a club subset of a fat stationary subset of 
$\omega_2$ with finite conditions.

Observe that if $|\beta^{\omega}| < \kappa$ for all $\beta < \kappa$, 
then $P_{\omega_1}(\kappa)$ itself is thin. 
Krueger proved that the existence of a thin stationary subset of 
$P_{\omega_1}(\omega_2)$ is independent of ZFC; see \cite{kruegerthin}.

\begin{notation}
Fix a thin stationary set $T^* \subseteq P_{\omega_1}(\kappa)$ which 
satisfies the property that for all $\beta < \kappa$ and 
$a \in T^*$, $a \cap \beta \in T^*$. 
In Part III, we will let $T^* = P_{\omega_1}(\kappa)$.
\end{notation}

Note that if $T$ is a thin stationary set, then the set 
$\{ a \cap \beta : a \in T, \ \beta < \kappa \}$ is a thin stationary set which 
satisfies the property of being 
closed under initial segments which is described in Notation 1.4.

Observe that if $T$ is a thin stationary set, then $|T| = \kappa$.

\begin{notation}
Fix a bijection $\pi^* : T^* \to \kappa$. 
\end{notation}

\begin{notation}
Let $C^*$ denote the set of $\beta < \kappa$ such that 
whenever $a$ is a bounded subset of $\beta$ in $T^*$, 
then $\pi^*(a) < \beta$.
\end{notation}

The fact that $T^*$ is thin easily implies that $C^*$ is a club subset of $\kappa$.

\begin{notation}
Let $\Lambda$ denote the set $C^* \cap \cof(>\! \omega)$.
\end{notation}

\begin{notation}
For the remainder of the paper, let $\unlhd$ denote a 
well-ordering of $H(\lambda)$.
\end{notation}

\begin{notation}
Let $\mathcal X_0$ denote the set of $M$ in $P_{\omega_1}(H(\lambda))$ 
such that $M \cap \kappa \in T^*$ and $M$ is an elementary substructure 
of $(H(\lambda),\in,\unlhd,\kappa,T^*,\pi^*,C^*,\Lambda)$.
\end{notation}

\begin{notation}
Let $\mathcal Y_0$ denote the set of $P$ in $P_{\kappa}(H(\lambda))$ 
such that $P \cap \kappa \in \kappa$ and 
$P$ is an elementary substructure of 
$(H(\lambda),\in,\unlhd,\kappa,T^*,\pi^*,C^*,\Lambda)$.
\end{notation}

Note that if $P$ and $Q$ are in $\mathcal Y_0$, then 
$P \cap Q$ is in $\mathcal Y_0$. 
And if $M \in \mathcal X_0$ and $P \in \mathcal Y_0$, 
then $M \cap P$ is in $\mathcal X_0$. 
For the presence of the well-ordering $\unlhd$ implies that 
$P \cap Q$ and $M \cap P$ are elementary substructures, 
and $M \cap P \cap \kappa$ is an initial segment of $M \cap \kappa$ 
and hence is in $T^*$. 
For the intersection of models in $\mathcal X_0$, 
see Lemma 1.23.

This completes the introduction of the basic objects. 

\bigskip

Next we will define comparison points and a way to compare two models 
in $\mathcal X_0$.

\begin{definition}
For $M \in \mathcal X_0$, let $\Lambda_M$ denote the set of 
$\beta \in \Lambda$ such that 
$$
\beta = \min(\Lambda \setminus \sup(M \cap \beta)).
$$
\end{definition}

Observe that since any member of $\Lambda_M$ is determined by an 
ordinal in $\cl(M)$, and $\cl(M)$ is countable, 
it follows that $\Lambda_M$ is countable.

\begin{lemma}
Let $M \in \mathcal X_0$. 
If $\beta \in \Lambda_M$ and 
$\beta_0 \in \Lambda \cap \beta$, then 
$M \cap [\beta_0,\beta) \ne \emptyset$.
\end{lemma}

\begin{proof}
If $M \cap [\beta_0,\beta) = \emptyset$, then 
$\sup(M \cap \beta) \le \beta_0$. 
So 
$$
\beta = \min(\Lambda \setminus \sup(M \cap \beta)) \le \beta_0 < \beta,
$$
which is a contradiction.
\end{proof}

\begin{lemma}
Let $M$ and $N$ be in $\mathcal X_0$. 
Then $\Lambda_M \cap \Lambda_N$ has a maximum element.
\end{lemma}

\begin{proof}
Note that the first member of $\Lambda$ is in both $\Lambda_M$ and 
$\Lambda_N$, and therefore $\Lambda_M \cap \Lambda_N$ is nonempty. 
Suppose for a contradiction that 
$\gamma := \sup(\Lambda_M \cap \Lambda_N)$ 
is not in $\Lambda_M \cap \Lambda_N$. 
Fix an increasing sequence 
$\langle \gamma_n : n < \omega \rangle$ in 
$\Lambda_M \cap \Lambda_N$ which is cofinal in $\gamma$. 
Then for each $n < \omega$, 
$M \cap [\gamma_n,\gamma_{n+1})$ is nonempty by Lemma 1.12. 
So $\gamma$ is a limit point of $M$. 
Similarly, $\gamma$ is a limit point of $N$. 
Let $\beta = \min(\Lambda \setminus \gamma)$. 
Since $\gamma$ has cofinality $\omega$, $\gamma < \beta$, and since 
$\gamma$ is a limit point of $M$ and a limit point of $N$, 
easily $\beta \in \Lambda_M \cap \Lambda_N$. 
This contradicts that $\gamma = \sup(\Lambda_M \cap \Lambda_N)$ 
and $\gamma < \beta$. 
\end{proof}

\begin{definition}
For $M$ and $N$ in $\mathcal X_0$, let $\beta_{M,N}$ be 
the maximum element of $\Lambda_M \cap \Lambda_N$. 
The ordinal $\beta_{M,N}$ is called the 
\emph{comparison point of $M$ and $N$}.
\end{definition}

The most important property of $\beta_{M,N}$ is described in the next lemma.

\begin{lemma}
Let $M$ and $N$ be in $\mathcal X_0$. 
Then 
$$
\cl(M \cap \kappa) \cap \cl(N \cap \kappa) \subseteq \beta_{M,N}.
$$
\end{lemma}

\begin{proof}
Suppose for a contradiction that $\xi$ is in 
$\cl(M \cap \kappa) \cap \cl(N \cap \kappa)$ but $\beta_{M,N} \le \xi$. 
Let $\beta = \min(\Lambda \setminus (\xi + 1))$. 
Since $\beta$ is a limit ordinal, 
$\beta_{M,N} \le \xi < \xi + 1 < \beta$. 
We claim that $\beta \in \Lambda_M \cap \Lambda_N$. 
Then by the maximality of $\beta_{M,N}$, 
$\beta \le \beta_{M,N}$, which is a contradiction.

First, assume that $\xi \in \Lambda$. 
Then $\xi$ has uncountable cofinality. 
So $\xi$ cannot be a limit point of $M$ or of $N$. 
Hence $\xi \in M \cap N$. 
By elementarity, $\xi + 1 \in M \cap N$. 
Since $\xi+1 \in M \cap \beta$, 
$\xi + 1 \le \sup(M \cap \beta) < \beta$. 
As $\beta = \min(\Lambda \setminus (\xi+1))$, clearly 
$\beta = \min(\Lambda \setminus \sup(M \cap \beta))$. 
So $\beta \in \Lambda_M$. 
The same argument shows that $\beta \in \Lambda_N$, and we are done.

Secondly, assume that $\xi \notin \Lambda$. 
Then $\min(\Lambda \setminus \xi) = \min(\Lambda \setminus (\xi+1)) = \beta$. 
Since $\xi < \beta$ and $\xi$ is either in $M \cap \kappa$ or is a limit point of $M \cap \kappa$, 
clearly $\xi \le \sup(M \cap \beta) < \beta$. 
Hence $\beta = \min(\Lambda \setminus \sup(M \cap \beta))$, and therefore 
$\beta \in \Lambda_M$. 
The same argument shows that $\beta \in \Lambda_N$, finishing the proof.
\end{proof}

The next lemma provides some useful technical facts about 
comparison points. 
Statement (4) is not very intuitive; however it turns out that this observation 
simplifies some of the material in the original development of 
adequate sets in \cite{jk21}.

\begin{lemma}
Let $L$, $M$, and $N$ be in $\mathcal X_0$. 
\begin{enumerate}
\item If $L \cap \kappa \subseteq M \cap \kappa$ 
then $\Lambda_L \subseteq \Lambda_M$. 
Hence $\beta_{L,N} \le \beta_{M,N}$.
\item If $L \cap \kappa \subseteq \beta$ where $\beta \in \Lambda$, 
then $\Lambda_L \subseteq \beta+1$. 
Hence $\beta_{L,M} \le \beta$.
\item If $\beta < \beta_{M,N}$ and $\beta \in \Lambda$, then 
$M \cap [\beta,\beta_{M,N}) \ne \emptyset$.
\item Suppose that $M \cap \beta_{L,M} \subseteq N$. 
Then $\beta_{L,M} \le \beta_{L,N}$.
\end{enumerate}
\end{lemma}

\begin{proof}
Statements (1) and (2) can be proven in a straightforward way 
from the definitions, and (3) follows immediately from Lemma 1.12. 
(4) By definition, $\beta_{L,M} \in \Lambda_L$. 
Since $M \cap \beta_{L,M} \subseteq N$, 
$\sup(M \cap \beta_{L,M}) \le \sup(N \cap \beta_{L,M})$. 
As $\beta_{L,M} \in \Lambda_M$, by definition 
$\beta_{L,M} = \min(\Lambda \setminus \sup(M \cap \beta_{L,M}))$. 
So clearly $\beta_{L,M} = \min(\Lambda \setminus \sup(N \cap \beta_{L,M}))$. 
Hence $\beta_{L,M} \in \Lambda_N$. 
So $\beta_{L,M} \in \Lambda_L \cap \Lambda_N$. 
Therefore $\beta_{L,M} \le \max(\Lambda_L \cap \Lambda_N) = \beta_{L,N}$.
\end{proof}

\bigskip

Now we introduce our way of comparing models.

\begin{definition}
Let $M$ and $N$ be in $\mathcal X_0$.
\begin{enumerate}
\item Let $M < N$ if $M \cap \beta_{M,N} \in N$.
\item Let $M \sim N$ if 
$M \cap \beta_{M,N} = N \cap \beta_{M,N}$.
\item Let $M \le N$ if either $M < N$ or $M \sim N$.
\end{enumerate}
\end{definition}

\begin{definition}
A finite set $A \subseteq \mathcal X_0$ is said to be \emph{adequate} 
if for all $M$ and $N$ in $A$, either $M < N$, $M \sim N$, or $N < M$.
\end{definition}

If $M < N$, then by elementarity $\cl(M \cap \beta_{M,N})$ is a member of $N$. 
Since $\cl(M \cap \beta_{M,N})$ is countable, 
$\cl(M \cap \beta_{M,N}) \subseteq N$. 
Also every initial segment of $M \cap \beta_{M,N}$ is in $N$. 
For any proper initial segment has the form 
$M \cap \gamma = M \cap \beta_{M,N} \cap \gamma$ 
for some $\gamma \in M \cap \beta_{M,N}$, and 
since $M \cap \beta_{M,N}$ and $\gamma$ are in $N$, so is 
$M \cap \gamma$.

The next lemma provides some useful technical facts about the relation on models 
just introduced.

\begin{lemma}
Let $\{ M, N \}$ be adequate. 
\begin{enumerate}
\item If $(N \cap \beta_{M,N}) \setminus M$ is nonempty, then $M < N$.
\item If $M \le N$ then $M \cap \beta_{M,N} = M \cap N \cap \kappa = 
M \cap N \cap \beta_{M,N}$.
\item $\beta_{M,N} = \min(\Lambda \setminus \sup(M \cap N \cap \kappa))$. 
\item If $M < N$ then $\beta_{M,N} \in N$.
\item If $\beta < \beta_{M,N}$ and $\beta \in \Lambda$, then 
$(M \cap N) \cap [\beta,\beta_{M,N}) \ne \emptyset$.
\end{enumerate}
\end{lemma}

\begin{proof}
The assumption of (1) implies that $M \sim N$ and $N < M$ are impossible. 
(2) Both $M \cap \beta_{M,N} \in N$ and $M \cap \beta_{M,N} = 
N \cap \beta_{M,N}$ imply that $M \cap \beta_{M,N} \subseteq N$. 
So $M \cap \beta_{M,N} \subseteq M \cap N \cap \kappa$. 
Conversely by Lemma 1.15, $M \cap N \cap \kappa \subseteq \beta_{M,N}$, 
so $M \cap N \cap \kappa \subseteq M \cap \beta_{M,N}$. 
This proves that $M \cap N \cap \kappa = M \cap \beta_{M,N}$. 
Since $M \cap N \cap \kappa \subseteq \beta_{M,N}$ by Lemma 1.15, 
$M \cap N \cap \kappa = M \cap N \cap \beta_{M,N}$.

(3) Without loss of generality assume that 
$M \le N$. 
Then $M \cap N \cap \kappa = M \cap \beta_{M,N}$ by (2). 
Since $\beta_{M,N} \in \Lambda_M$, by definition 
$$
\beta_{M,N} = \min(\Lambda \setminus (\sup(M \cap \beta_{M,N}))) = 
\min(\Lambda \setminus (\sup(M \cap N \cap \kappa))).
$$

(4) If $M < N$ then $M \cap \beta_{M,N} \in N$. 
By (2), $M \cap \beta_{M,N} = M \cap N \cap \kappa$. 
So $M \cap N \cap \kappa \in N$. 
By (3), $\beta_{M,N} = \min(\Lambda \setminus \sup(M \cap N \cap \kappa))$. 
So $\beta_{M,N} \in N$ by elementarity.

(5) Without loss of generality assume that $M \le N$. 
Then by (2), $M \cap N \cap \beta_{M,N} = M \cap \beta_{M,N}$. 
Since $\beta_{M,N} \in \Lambda_M$, Lemma 1.12 implies that 
$M \cap [\beta,\beta_{M,N})$ is nonempty. 
Fix $\xi \in M \cap [\beta,\beta_{M,N})$. 
Then $\xi \in M \cap \beta_{M,N} = M \cap N \cap \beta_{M,N}$. 
So $(M \cap N) \cap [\beta,\beta_{M,N})$ is nonempty.
\end{proof}

\begin{lemma}
Let $M$ and $N$ be in $\mathcal X_0$, and assume that $\{ M, N \}$ is adequate. 
Then 
$$
\cl(M \cap N \cap \kappa) = \cl(M \cap \kappa) \cap \cl(N \cap \kappa).
$$
\end{lemma}

\begin{proof}
The forward inclusion is immediate. 
Suppose that $\alpha$ is in $\cl(M \cap \kappa) \cap \cl(N \cap \kappa)$. 
Then by Lemma 1.15, $\alpha < \beta_{M,N}$. 
Without loss of generality, assume that $M \le N$. 
Then 
$$
\alpha \in \cl(M \cap \kappa) \cap \beta_{M,N} = 
\cl(M \cap \beta_{M,N}) = \cl(M \cap N \cap \kappa)
$$
by Lemma 1.19(2). 
\end{proof}

If $\{ M, N \}$ is adequate, then the relation which holds between $M$ and $N$ 
is determined by the intersection of $M$ and $N$ with $\omega_1$.

\begin{lemma}
Let $\{ M, N \}$ be adequate. Then:
\begin{enumerate}
\item $M < N$ iff $M \cap \omega_1 < N \cap \omega_1$;
\item $M \sim N$ iff $M \cap \omega_1 = N \cap \omega_1$.
\end{enumerate}
\end{lemma}

\begin{proof}
Suppose that $M < N$. Then $M \cap \beta_{M,N} \in N$. 
Since $\beta_{M,N}$ has uncountable cofinality, 
$\omega_1 \le \beta_{M,N}$. 
So $M \cap \omega_1$ is an initial segment of $M \cap \beta_{M,N}$, 
and hence $M \cap \omega_1 \in N$. 
So $M \cap \omega_1 < N \cap \omega_1$.

Suppose that $M \sim N$. 
Then $M \cap \beta_{M,N} = N \cap \beta_{M,N}$. 
Since $\omega_1 \le \beta_{M,N}$, 
$M \cap \omega_1 = N \cap \omega_1$. 

Conversely if $M \cap \omega_1 < N \cap \omega_1$, then the facts just 
proved imply that $M < N$ is the only possibility of how $M$ and $N$ relate. 
Similarly $M \cap \omega_1 = N \cap \omega_1$ implies that $M \sim N$.
\end{proof}

\begin{lemma}
Let $A$ be an adequate set. 
Then the relation $<$ is irreflexive and transitive on $A$, $\sim$ is an equivalence 
relation on $A$, and the relations $<$ and $\le$ respect $\sim$.
\end{lemma}

\begin{proof}
Immediate from Lemma 1.21.
\end{proof}

\bigskip

In proving amalgamation results over countable models, 
we will need to be able to enlarge an adequate set $A$ 
by adding $M \cap N$ to $A$, where $M < N$ are in $A$. 
Let us show that we can do this while preserving adequacy.

First we note that $M \cap N$ is in $\mathcal X_0$.

\begin{lemma}
Let $\{ M, N \}$ be adequate. 
Then $M \cap N$ is in $\mathcal X_0$.
\end{lemma}

\begin{proof}
Without loss of generality, assume that $M \le N$. 
Then by Lemma 1.19(2), $M \cap N \cap \kappa = M \cap \beta_{M,N}$. 
Since $T^*$ is closed under initial segments and $M \cap \kappa \in T^*$, 
it follows that $M \cap \beta_{M,N} \in T^*$. 
Hence $M \cap N \cap \kappa \in T^*$. 
Also clearly $M \cap N$ is an elementary substructure.
\end{proof}

\begin{lemma}
Let $K$, $M$, and $N$ be in $\mathcal X_0$. 
Suppose that $M < N$ and $\{ K, M \}$ is adequate. 
Then:
\begin{enumerate}
\item $\beta_{K,M \cap N} \le \beta_{K,M}$ and 
$\beta_{K,M \cap N} \le \beta_{M,N}$;
\item $M < K$ iff $M \cap N < K$;
\item $K \sim M$ iff $K \sim M \cap N$;
\item $K < M$ iff $K < M \cap N$.
\end{enumerate}
In particular, $\{ K, M \cap N \}$ is adequate.
\end{lemma}

\begin{proof}
(1) Since $M \cap N \subseteq M$, 
$\beta_{K,M \cap N} \le \beta_{K,M}$ by Lemma 1.16(1). 
Also $M \cap N \cap \kappa \subseteq \beta_{M,N}$ by Lemma 1.15, which 
implies that 
$\beta_{K,M \cap N} \le \beta_{M,N}$ by Lemma 1.16(2). 
This proves (1). 

Since $M \cap N \cap \beta_{M,N} = M \cap \beta_{M,N}$ 
by Lemma 1.19(2) and 
$\beta_{K,M \cap N} \le \beta_{M,N}$, it follows that 
$$
M \cap N \cap \beta_{K,M \cap N} = M \cap \beta_{K,M \cap N}.
$$

(2,3,4) First we will prove the forward implications of (2), (3), and (4). 
If $M < K$ then $M \cap \beta_{K,M}$ is in $K$. 
But since $\beta_{K,M \cap N} \le \beta_{K,M}$, 
$M \cap \beta_{K,M \cap N}$ is an initial segment 
of $M \cap \beta_{K,M}$, and hence is in $K$. 
So $M \cap N \cap \beta_{K,M \cap N} = M \cap \beta_{K,M \cap N}$ 
is in $K$, and therefore 
$M \cap N < K$. 

If $K \sim M$, then $K \cap \beta_{K,M} = M \cap \beta_{K,M}$. 
Since $\beta_{K,M \cap N} \le \beta_{K,M}$, 
$$
K \cap \beta_{K,M \cap N} = 
M \cap \beta_{K,M \cap N} = 
M \cap N \cap \beta_{K,M \cap N}.
$$
Therefore $K \sim M \cap N$.

Suppose that $K < M$. 
Then $K \cap \beta_{K,M} \in M$. 
Since $\beta_{K,M \cap N} \le \beta_{K,M}$, 
$K \cap \beta_{K,M \cap N} \in M$. 
So to show that $K < M \cap N$, it suffices to show that 
$K \cap \beta_{K,M \cap N} \in N$.

Since $K \cap \kappa \in T^*$ by the definition of $\mathcal X_0$, 
$K \cap \beta_{K,M \cap N} \in T^*$ as $T^*$ is 
closed under initial segments. 
Recall from Notation 1.5 that $\pi^* : T^* \to \kappa$ is a bijection. 
As $M$ is closed under $\pi^*$ by elementarity, 
$\pi^*(K \cap \beta_{K,M \cap N}) \in M \cap \kappa$. 
Since $K \cap \beta_{K,M \cap N}$ is a bounded subset of 
$\beta_{K,M\cap N}$ and $\beta_{K,M \cap N} \le \beta_{M,N}$, 
we have that $K \cap \beta_{K,M \cap N}$ is a bounded subset 
of $\beta_{M,N}$. 
Since $\beta_{M,N} \in \Lambda$, it follows that 
$\pi^*(K \cap \beta_{K,M \cap N}) < \beta_{M,N}$ 
by the definitions of $C^*$ and $\Lambda$ from 
Notations 1.6 and 1.7. 
Hence $\pi^*(K \cap \beta_{K,M \cap N}) \in M \cap \beta_{M,N} \subseteq N$. 
Since $N$ is closed under the inverse of $\pi^*$ by elementarity, 
$K \cap \beta_{K,M \cap N} \in N$.

Now we consider the reverse implications of (2), (3), and (4). 
Suppose that $M \cap N < K$. 
Since $\{ K, M \}$ is adequate, either $K < M$, $K \sim M$, or $M < K$. 
But $K \sim M$ and $K < M$ are ruled out by the forward implications 
of (3) and (4). 
So $M < K$. 
The other converses are proved similarly.
\end{proof}

\begin{proposition}
Let $A$ be an adequate set and $N \in \mathcal X_0$. 
Let $M$ be in $A$, and suppose that $M < N$. 
Then $A \cup \{ M \cap N \}$ is adequate.
\end{proposition}

\begin{proof}
Immediate from Lemma 1.24.
\end{proof}

\bigskip

Our next goal is to prove the first amalgamation result over countable models, 
which is stated in Proposition 1.29 below. 
See Proposition 13.1 for a much deeper result.

\begin{lemma}
Let $L$, $M$, and $N$ be in $\mathcal X_0$. 
Suppose that $N \le M$ and $L \in N$. 
Then $L < M$.
\end{lemma}

\begin{proof}
Since $L \in N$, $\beta_{L,M} \le \beta_{M,N}$ by Lemma 1.16(1). 
Also $L \cap \beta_{L,M}$ is in $N \cap T^*$, since it is an initial segment 
of $L \cap \kappa$. 
As $N$ is closed under $\pi^*$ by elementarity, the ordinal 
$\pi^*(L \cap \beta_{L,M})$ is in $N \cap \kappa$. 
And as $\beta_{M,N} \in \Lambda$ and 
$L \cap \beta_{L,M}$ is a bounded subset of $\beta_{M,N}$ in $T^*$, 
it follows that $\pi^*(L \cap \beta_{L,M}) < \beta_{M,N}$ 
by the definition of $\Lambda$. 
Hence $\pi^*(L \cap \beta_{L,M}) \in N \cap \beta_{M,N} \subseteq M$. 
By elementarity, $M$ is closed under the inverse of $\pi^*$, 
so $L \cap \beta_{L,M} \in M$.
\end{proof}

\begin{lemma}
Let $L$, $M$, and $N$ be in $\mathcal X_0$. 
Suppose that $M < N$ and $L \in N$. 
Then:
\begin{enumerate}
\item $\beta_{L,M} = \beta_{L,M \cap N}$;
\item $L \sim M \cap N$ iff $L \sim M$;
\item $L < M \cap N$ iff $L < M$;
\item $M \cap N < L$ iff $M < L$.
\end{enumerate}
\end{lemma}

\begin{proof}
(1) Since $M \cap N \cap \kappa \subseteq M \cap \kappa$, 
$\beta_{L,M \cap N} \le \beta_{L,M}$ by Lemma 1.16(1), 
which proves one direction of the equality. 
Since $L \cap \kappa \subseteq N \cap \kappa$, $\beta_{L,M} \le \beta_{M,N}$ 
by Lemma 1.16(1). 
So 
$$
M \cap \beta_{L,M} \subseteq M \cap \beta_{M,N} \subseteq M \cap N.
$$
By Lemma 1.16(4), $\beta_{L,M} \le \beta_{L,M \cap N}$.

(2,3,4) First we will prove the forward implications 
of (2), (3), and (4). 
As $\beta_{L,M} \le \beta_{M,N}$ and 
$M \cap \beta_{M,N} = M \cap N \cap \beta_{M,N}$, it follows that 
$$
M \cap \beta_{L,M} = M \cap N \cap \beta_{L,M}.
$$ 
If $L \sim M \cap N$, then 
$$
L \cap \beta_{L,M} = L \cap \beta_{L,M \cap N} = 
M \cap N \cap \beta_{L,M \cap N} = M \cap N \cap \beta_{L,M} = 
M \cap \beta_{L,M}.
$$
So $L \cap \beta_{L,M} = M \cap \beta_{L,M}$, and hence $L \sim M$. 
And if $L < M \cap N$, then 
$$
L \cap \beta_{L,M} = L \cap \beta_{L,M \cap N} \in M \cap N \subseteq M.
$$
So $L \cap \beta_{L,M} \in M$, and hence $L < M$. 
If $M \cap N < L$, then 
$$
M \cap \beta_{L,M} = M \cap N \cap \beta_{L,M} = 
M \cap N \cap \beta_{L,M \cap N} \in L.
$$
So $M \cap \beta_{L,M} \in L$, and therefore $M < L$.

For the reverse implications, each of the assumptions 
$L \sim M$, $L < M$, and $M < L$ 
implies that $\{ L, M \}$ is adequate. 
Hence these assumptions imply that $L \sim M \cap N$, 
$L < M \cap N$, and $M \cap N < L$ respectively by Lemma 1.24.
\end{proof}

\begin{lemma}
Let $L$, $M$, and $N$ be in $\mathcal X_0$. 
Suppose that $M < N$ and $L \in N$. 
If $\{ L, M \cap N \}$ is adequate, then $\{ L, M \}$ is adequate.
\end{lemma}

\begin{proof}
Immediate from Lemma 1.27.
\end{proof}

\begin{proposition}
Let $A$ be adequate, $N \in A$, and suppose that for all $M \in A$, 
if $M < N$ then $M \cap N \in A \cap N$. 
Suppose that $B$ is adequate and 
$A \cap N \subseteq B \subseteq N$. 
Then $A \cup B$ is adequate.
\end{proposition}

\begin{proof}
Let $L \in B$ and $M \in A$, and we will show that $\{ L, M \}$ is adequate. 
Since $L \in B$ and $B \subseteq N$, $L \in N$. 
If $N \le M$, then $L < M$ by Lemma 1.26. 
Suppose that $M < N$. 
Then $M \cap N \in A \cap N$ by assumption. 
Since $A \cap N \subseteq B$, $M \cap N \in B$. 
As $B$ is adequate, $\{ L, M \cap N \}$ is adequate. 
Since $L \in N$ and $\{ L, M \cap N \}$ is adequate, 
$\{ L, M \}$ is adequate by Lemma 1.28.
\end{proof}

In the last proposition, we assumed that $M < N$ implies that $M \cap N \in N$, 
for $M \in A$.  
At this point we do not have any reason to believe this implication is true 
in general. 
In Section 7, we will define a subclass of $\mathcal X_0$ 
on which this implication holds. 
See Notation 7.7 and Lemma 8.2.

\bigskip

So far we have discussed the interaction of countable 
models in $\mathcal X_0$. 
We now turn our attention to how models in $\mathcal X_0$ relate to 
models in $\mathcal Y_0$.

\begin{lemma}
Let $M$ and $N$ be in $\mathcal X_0 \cup \mathcal Y_0$.
Suppose that:
\begin{enumerate}
\item $M$ and $N$ are in $\mathcal X_0$ and $M < N$, or
\item $M$ and $N$ are in $\mathcal Y_0$ and $M \cap \kappa < N \cap \kappa$, or 
\item $M \in \mathcal X_0$, $N \in \mathcal Y_0$, and 
$\sup(M \cap N \cap \kappa) < N \cap \kappa$.
\end{enumerate}
Then $M \cap N \cap \kappa \in N$.
\end{lemma}

\begin{proof}
(1) If $M$ and $N$ are in $\mathcal X_0$, then $M < N$ implies that 
$M \cap \beta_{M,N} \in N$. 
By Lemma 1.19(2), $M \cap \beta_{M,N} = M \cap N \cap \kappa$, so 
$M \cap N \cap \kappa \in N$. 
(2) If $M$ and $N$ are in $\mathcal Y_0$, then since 
$M \cap \kappa < N \cap \kappa$, 
$M \cap N \cap \kappa = M \cap \kappa \in N$.

(3) Suppose that $M \in \mathcal X_0$, $N \in \mathcal Y_0$, and 
$\sup(M \cap N \cap \kappa) < N \cap \kappa$. 
Let $\beta := N \cap \kappa$. 
By the elementarity of $N$, $\beta$ is a limit point of $\Lambda$. 
So fix $\gamma \in N \cap \Lambda$ such that 
$\sup(M \cap \beta) < \gamma$. 
Then $M \cap N \cap \kappa = M \cap \gamma$. 
Since $\gamma$ has uncountable cofinality, $M \cap \gamma$ is a bounded 
subset of $\gamma$, and as $M \in \mathcal X_0$, 
$M \cap \gamma \in T^*$. 
By the definition of $C^*$ and $\Lambda$, 
$\pi^*(M \cap \gamma) < \gamma < N \cap \kappa$. 
Since $N$ is closed under the inverse of $\pi^*$ by elementarity, 
$M \cap \gamma = M \cap N \cap \kappa \in N$.
\end{proof}

Note that (3) holds if $\cf(N \cap \kappa) > \omega$, which is the 
typical situation that we will consider.

\begin{lemma}
Let $M \in \mathcal X_0$ and $N \in \mathcal Y_0$, and assume that 
$\sup(M \cap N \cap \kappa) < N \cap \kappa$. 
Then 
$$
\cl(M \cap N \cap \kappa) = \cl(M \cap \kappa) \cap \cl(N \cap \kappa) \cap (N \cap \kappa).
$$
\end{lemma}

\begin{proof}
The forward inclusion is immediate. 
Let 
$$
\alpha \in \cl(M \cap \kappa) \cap \cl(N \cap \kappa) \cap (N \cap \kappa).
$$
Since $\cl(N \cap \kappa) = (N \cap \kappa) \cup \{ N \cap \kappa \}$, 
$$
\alpha \in \cl(M \cap \kappa) \cap (N \cap \kappa) = \cl(M \cap N \cap \kappa).
$$
\end{proof}

\bigskip

Recall that if $M \in \mathcal X_0$ and $P \in \mathcal Y_0$, 
then $M \cap P \in \mathcal X_0$. 
We show next that we can add $M \cap P$ to an adequate set 
and preserve adequacy.

\begin{lemma}
Let $K$ and $M$ be in $\mathcal X_0$ and $P$ in $\mathcal Y_0$. 
Assume that $\{ K, M \}$ is adequate and 
$\sup(M \cap P \cap \kappa) < P \cap \kappa$. 
Then: 
\begin{enumerate}
\item $\beta_{K,M \cap P} \le \beta_{K,M}$ and 
$\beta_{K,M \cap P} < P \cap \kappa$;
\item $M < K$ iff $M \cap P < K$;
\item $K \sim M$ iff $K \sim M \cap P$;
\item $K < M$ iff $K < M \cap P$.
\end{enumerate}
In particular, $\{ K, M \cap P \}$ is adequate.
\end{lemma}

\begin{proof}
(1) Since $M \cap P \subseteq M$, 
$\beta_{K,M \cap P} \le \beta_{K,M}$ by Lemma 1.16(1). 
As $\sup(M \cap P \cap \kappa) < P \cap \kappa$ and $\Lambda$ is 
unbounded in $P \cap \kappa$ by elementarity, 
we can fix $\beta \in \Lambda$ with 
$\sup(M \cap P \cap \kappa) < \beta < P \cap \kappa$. 
By Lemma 1.16(2), 
$$
\beta_{K,M \cap P} \le \beta < P \cap \kappa.
$$
This proves (1). 
It follows that 
$$
M \cap \beta_{K,M \cap P} = M \cap P \cap \beta_{K,M \cap P}.
$$

(2,3,4) First we will prove the forward implications 
of (2), (3), and (4). 
Assume that $M < K$. 
Then $M \cap \beta_{K,M} \in K$. 
Since $\beta_{K,M \cap P} \le \beta_{K,M}$, 
$M \cap \beta_{K,M \cap P} \in K$. 
So 
$$
M \cap P \cap \beta_{K,M \cap P} = M \cap \beta_{K,M \cap P} \in K.
$$
Hence $M \cap P < K$.

Suppose that $K \sim M$. 
Then $K \cap \beta_{K,M} = M \cap \beta_{K,M}$. 
Since $\beta_{K,M \cap P} \le \beta_{K,M}$, it follows that 
$$
K \cap \beta_{K,M \cap P} = 
M \cap \beta_{K,M \cap P} = M \cap P \cap \beta_{K,M \cap P}.
$$ 
Therefore $K \sim M \cap P$.

Finally, assume that $K < M$. 
Then $K \cap \beta_{K,M} \in M$. 
Since $\beta_{K,M \cap P} \le \beta_{K,M}$, 
$K \cap \beta_{K,M \cap P} \in M$. 
As $\beta_{K,M \cap P} < P \cap \kappa$, by elementarity there is 
$\gamma \in P \cap \Lambda$ with $\beta_{K,M \cap P} < \gamma$. 
Then $K \cap \beta_{K,M \cap P}$ is a bounded subset of $\gamma$ in $T^*$. 
Hence $\pi^*(K \cap \beta_{K,M \cap P}) < \gamma$. 
In particular, $\pi^*(K \cap \beta_{K,M \cap P}) \in P \cap \kappa$. 
By the elementarity of $P$, $P$ is closed under the inverse of $\pi^*$. 
So $K \cap \beta_{K,M \cap P} \in P$. 
Thus $K \cap \beta_{K,M \cap P} \in M \cap P$, and therefore 
$K < M \cap P$.

Conversely, assume that $M \cap P < K$. 
Since $\{ K, M \}$ is adequate, either $M < K$, 
$M \sim K$, or $K < M$. 
But the forward implications of (3) and (4) rule out 
$M \sim K$ and $K < M$. 
Hence $M < K$. 
The other converses are proved similarly.
\end{proof}

\begin{proposition}
Let $A$ be an adequate set. 
Let $M$ be in $A$ and $P$ in $\mathcal Y_0$, and assume that 
$\sup(M \cap P \cap \kappa) < P \cap \kappa$. 
Then $A \cup \{ M \cap P \}$ is adequate.
\end{proposition}

\begin{proof}
Immediate from Lemma 1.32.
\end{proof}

\bigskip

Next we will prove an amalgamation result for uncountable models. 
See Proposition 13.2 for a deeper result.

\begin{lemma}
Let $L$ and $M$ be in $\mathcal X_0$ and $P \in \mathcal Y_0$. 
Assume that $L \in P$ and $\sup(M \cap P \cap \kappa) < P \cap \kappa$. 
Then: 
\begin{enumerate}
\item $\beta_{L,M} = \beta_{L,M \cap P}$ and 
$\beta_{L,M} < P \cap \kappa$;
\item $L \sim M \cap P$ iff $L \sim M$;
\item $M \cap P < L$ iff $M < L$;
\item $L < M \cap P$ iff $L < M$.
\end{enumerate}
In particular, $\{ L, M \cap P \}$ is adequate iff $\{ L, M \}$ is adequate. 
\end{lemma}

\begin{proof}
(1) Since $M \cap P \subseteq M$, $\beta_{L,M \cap P} \le \beta_{L,M}$ 
by Lemma 1.16(1), 
which proves one direction of the equality. 
As $L \in P$, by elementarity, $\Lambda_{L} \in P$. 
Since $\Lambda_{L}$ is countable, $\Lambda_{L} \subseteq P$. 
As $\beta_{L,M} \in \Lambda_L$, $\beta_{L,M} \in P \cap \kappa$. 
So $M \cap \beta_{L,M} \subseteq M \cap P$. 
By Lemma 1.16(4), it follows that 
$\beta_{L,M} \le \beta_{L,M \cap P}$.

(2,3,4) First we will prove the forward implications 
of (2), (3), and (4). 
Since $\beta_{L,M} \in P$ as noted above, 
$$
M \cap \beta_{L,M} = M \cap P \cap \beta_{L,M}.
$$
If $L \sim M \cap P$, then 
$$
L \cap \beta_{L,M} = L \cap \beta_{L,M \cap P} = 
M \cap P \cap \beta_{L,M \cap P} = 
M \cap P \cap \beta_{L,M} = 
M \cap \beta_{L,M}.
$$
So $L \cap \beta_{L,M} = M \cap \beta_{L,M}$, and hence 
$L \sim M$.

If $M \cap P < L$, then 
$$
M \cap \beta_{L,M} = 
M \cap P \cap \beta_{L,M} = M \cap P \cap \beta_{L,M \cap P} \in L.
$$
So $M \cap \beta_{L,M} \in L$, and therefore $M < L$. 
And if $L < M \cap P$, then 
$$
L \cap \beta_{L,M} = L \cap \beta_{L,M \cap P} \in M \cap P \subseteq M.
$$
So $L \cap \beta_{L,M} \in M$, and therefore $L < M$.

Conversely, the assumptions $M < L$, $L \sim M$, and $L < M$ imply 
that $\{ L, M \}$ is adequate. 
Hence each of these assumptions imply that $M \cap P < L$, $L \sim M \cap P$, 
$L < M \cap P$ respectively by Lemma 1.32.
\end{proof}

\begin{proposition}
Let $A$ be adequate, $P \in \mathcal Y_0$, and assume that 
for all $M \in A$, $M \cap P \in A \cap P$. 
Suppose that $B$ is adequate and $A \cap P \subseteq B \subseteq P$. 
Then $A \cup B$ is adequate.
\end{proposition}

\begin{proof}
Let $L \in B$ and $M \in A$. 
Then $M \cap P \in A \cap P \subseteq B$. 
Since $B$ is adequate, $\{ L , M \cap P \}$ is adequate. 
As $M \cap P \in P$, $\sup(M \cap P \cap \kappa) < P \cap \kappa$. 
By Lemma 1.34, $\{ L, M \}$ is adequate.
\end{proof}

\bigskip

We conclude the discussion about models in $\mathcal X_0$ and 
$\mathcal Y_0$ with the following useful lemma.

\begin{lemma}
Let $M$ and $N$ be in $\mathcal X_0$, and assume that 
$\{ M, N \}$ is adequate. 
Let $P \in \mathcal Y_0$. 
Then either $\beta_{M,N} = \beta_{M \cap P,N}$, or 
$P \cap \kappa < \beta_{M,N}$. 
\end{lemma}

\begin{proof}
Since $M \cap P \subseteq M$, $\beta_{M \cap P,N} \le \beta_{M,N}$. 
If $\beta_{M,N} = \beta_{M \cap P,N}$, then we are done. 
So assume that $\beta_{M \cap P,N} < \beta_{M,N}$. 
We claim that $P \cap \kappa < \beta_{M,N}$. 
Suppose for a contradiction that $\beta_{M,N} \le P \cap \kappa$. 
Since $\beta_{M \cap P,N} < \beta_{M,N}$, by Lemma 1.19(5), 
we can fix 
$$
\xi \in (M \cap N) \cap [\beta_{M \cap P,N},\beta_{M,N}).
$$
As $\beta_{M,N} \le P \cap \kappa$, 
$$
\xi \in (M \cap N) \cap P \cap \kappa = 
(M \cap P) \cap N \cap \kappa.
$$
Therefore $\xi < \beta_{M \cap P,N}$ by Lemma 1.15, which contradicts the 
choice of $\xi$.
\end{proof}

\bigskip

Finally, we prove an amalgamation result over transitive models.

\begin{lemma}
Let $M$, $M'$, $N$, and $N'$ be in $\mathcal X_0$. 
Assume that $M \cap \kappa = M' \cap \kappa$ and $N \cap \kappa = 
N' \cap \kappa$. 
Then:
\begin{enumerate}
\item $\beta_{M,N} = \beta_{M',N'}$;
\item $M \sim N$ iff $M' \sim N'$;
\item $M < N$ iff $M' < N'$;
\item $N < M$ iff $N' < M'$.
\end{enumerate}
In particular, $\{ M, N \}$ is adequate iff $\{ M', N' \}$ is adequate.
\end{lemma}

\begin{proof}
(1) Since $M \cap \kappa \subseteq M' \cap \kappa$ and 
$N \cap \kappa \subseteq N' \cap \kappa$, it follows that 
$\beta_{M,N} \le \beta_{M',N} \le \beta_{M',N'}$ by Lemma 1.16(1). 
Similarly, the reverse inclusions imply that $\beta_{M',N'} \le \beta_{M,N}$. 
So $\beta_{M,N} = \beta_{M',N'}$.

(2,3,4) It suffices to prove the forward direction of the iff's of (2), (3), and (4), 
since the converses hold by symmetry. 
If $M \sim N$, then 
$$
M' \cap \beta_{M',N'} = 
M \cap \beta_{M,N} = N \cap \beta_{M,N} = N' \cap \beta_{M',N'},
$$
which proves (2). 
Suppose that $M < N$. 
Then 
$$
M' \cap \beta_{M',N'} = M \cap \beta_{M,N} \in N.
$$
By elementarity, 
$$
\pi^*(M' \cap \beta_{M',N'}) \in N \cap \kappa = N' \cap \kappa.
$$
Since $N'$ is closed under the inverse of $\pi^*$ by elementarity, 
$M' \cap \beta_{M',N'} \in N'$. 
(4) is similar.
\end{proof}

\begin{proposition}
Let $A$ be an adequate set. 
Assume that $X \prec (H(\kappa^+),\in)$, $|X| = \kappa$, and 
$X \cap \kappa^+ \in \kappa^+$. 
Let $B$ be an adequate set such that $A \cap X \subseteq B \subseteq X$. 
Suppose that for all $M \in A$, there is 
$M' \in B$ such that $M \cap \kappa = M' \cap \kappa$. 
Then $A \cup B$ is adequate. 
\end{proposition}

\begin{proof}
Let $M \in A$ and $K \in B$ be given. 
Fix $M' \in B$ such that $M \cap \kappa = M' \cap \kappa$. 
As $\{ M', K \} \subseteq B$, $\{ M', K \}$ is adequate. 
Therefore $\{ M, K \}$ is adequate by Lemma 1.37.
\end{proof}

\bigskip

\addcontentsline{toc}{section}{2. Analysis of remainder points}

\textbf{\S 2. Analysis of remainder points}

\stepcounter{section}

\bigskip

In this section we will provide a detailed analysis of remainder points; some 
of these arguments appeared previously in \cite{jk22} and \cite{jk24}, although in a 
less complete form. 
This analysis will be the foundation from which we derive the 
amalgamation results of Section 13.

\begin{definition}
Let $\{ M, N \}$ be adequate. 
Let $R_M(N)$, the \emph{set of remainder points of $N$ over $M$}, 
be defined as the set of $\zeta$ satisfying either:
\begin{enumerate}
\item $\zeta = \min((N \cap \kappa) \setminus \beta_{M,N})$, 
provided that $M \sim N$, or
\item there is $\gamma \in (M \cap \kappa) \setminus \beta_{M,N}$ such that 
$\zeta = \min((N \cap \kappa) \setminus \gamma)$.
\end{enumerate}
\end{definition}

Note that if $N < M$, then $\beta_{M,N} \in M$ by Lemma 1.19(4). 
It follows that $\min((N \cap \kappa) \setminus \beta_{M,N}) \in R_M(N)$ by 
Definition 2.1(2).

The next lemma describes some basic properties of remainder points.

\begin{lemma}
Let $\{ M, N \}$ be adequate. 
Then:
\begin{enumerate}
\item $R_M(N) \cap \cl(M \cap \kappa) = \emptyset$;
\item $R_M(N)$ is finite;
\item suppose that $\zeta \in R_M(N)$ and $\zeta > \min(R_M(N) \cup R_N(M))$; 
then $\sigma := \min((M \cap \kappa) \setminus \sup(N \cap \zeta)) \in R_N(M)$ and 
$\zeta = \min((N \cap \kappa) \setminus \sigma)$.
\end{enumerate}
\end{lemma}

\begin{proof}
(1) If $\zeta \in R_M(N)$, then by definition, $\zeta \in N$ and $\beta_{M,N} \le \zeta$. 
Hence $\zeta \notin \cl(M \cap \kappa)$ by Lemma 1.15.

(2) Suppose for a contradiction that $\langle \zeta_n : n < \omega \rangle$ 
is a strictly increasing sequence from $R_M(N)$. 
Then by definition, for each $n > 0$ there is $\gamma_n \in M$ such that 
$\zeta_n = \min((N \cap \kappa) \setminus \gamma_n)$. 
Let $\zeta := \sup \{ \zeta_n : n < \omega \}$. 
Then $\zeta = \sup \{ \gamma_n : n < \omega \}$. 
Therefore 
$$
\zeta \in \cl(M \cap \kappa) \cap \cl(N \cap \kappa).
$$
Hence $\zeta < \beta_{M,N}$ by Lemma 1.15. 
But 
$$
\beta_{M,N} \le \zeta_0 < \zeta,
$$
which is a contradiction.

(3) Since $\zeta > \min(R_M(N) \cup R_N(M))$ and $R_M(N)$ and $R_N(M)$ are finite, 
let $\sigma_0$ be the largest member of $R_M(N) \cup R_N(M)$ less than $\zeta$. 
We claim that $\sigma_0 \in R_N(M)$. 
If not, then $\sigma_0 \in R_M(N)$, and in particular, $\sigma_0 \in (N \cap \zeta) \setminus \beta_{M,N}$. 
Since $\zeta \in R_M(N)$, by the definition of $R_M(N)$ we have that 
$M \cap (\sigma_0,\zeta) \ne \emptyset$. 
But then $\min((M \cap \kappa) \setminus \sigma_0)$ is in $R_N(M)$ and is between 
$\sigma_0$ and $\zeta$, which contradicts the maximality of $\sigma_0$.

We claim that $\zeta = \min((N \cap \kappa) \setminus \sigma_0)$. 
Otherwise $\min((N \cap \kappa) \setminus \sigma_0)$ is in $R_M(N)$ and is 
between $\sigma_0$ and $\zeta$, which contradicts the maximality of $\sigma_0$. 
It follows that $\sup(N \cap \zeta) \le \sigma_0$. 
Finally, we show that $\sigma_0 = \min((M \cap \kappa) \setminus \sup(N \cap \zeta))$. 
Therefore $\sigma = \sigma_0$, and we are done. 
Suppose for a contradiction that $\sigma < \sigma_0$. 
As $\sup(N \cap \zeta) \le \sigma$, we have that $N \cap (\sigma,\sigma_0) = \emptyset$.

Observe that $\beta_{M,N} \le \sigma$. 
For if $\sigma < \beta_{M,N}$, then $\sigma \in (M \cap \beta_{M,N}) \setminus N$, which implies 
that $N < M$. 
And since $\sup(N \cap \zeta) \le \sigma$, it follows that 
$\zeta = \min((N \cap \kappa) \setminus \beta_{M,N})$. 
So $\zeta = \min(R_M(N) \cup R_N(M))$, which is a contradiction. 
Hence $\beta_{M,N} \le \sigma < \sigma_0$. 
Since $\sigma_0 \in R_N(M)$, there is $\gamma \in N$ such that 
$\sigma_0 = \min((M \cap \kappa) \setminus \gamma)$. 
But then $\sigma < \gamma < \sigma_0$, which contradicts that 
$N \cap (\sigma,\sigma_0) = \emptyset$.
\end{proof}

The rest of the section follows roughly the same sequence of topics covered in the 
previous section. 
Lemma 2.3 describes the remainder points which appear 
when adding $M \cap N$ to an adequate set, where $M < N$, as in 
Lemma 1.24 and Proposition 1.25. 
Then Lemmas 2.4--2.6 analyze remainder points which appear in the process
of amalgamating over countable models, as in Proposition 1.29.

\begin{lemma}
Let $K$, $M$, and $N$ be in $\mathcal X_0$. 
Suppose that $M < N$ and $\{ K, M, N \}$ is adequate. Then:
\begin{enumerate}
\item $R_K(M \cap N) \subseteq R_K(M)$;
\item $R_{M \cap N}(K) \subseteq R_M(K) \cup R_N(K)$.
\end{enumerate}
\end{lemma}

\begin{proof}
Note that by Lemma 1.24, 
$\{ K, M \cap N \}$ is adequate, 
$\beta_{K,M \cap N} \le \beta_{K,M}$, and 
$\beta_{K,M \cap N} \le \beta_{M,N}$.

(1) Let $\zeta \in R_K(M \cap N)$, and we will show that $\zeta \in R_K(M)$. 
Then either (a) $K \sim M \cap N$ and 
$\zeta = \min((M \cap N \cap \kappa) \setminus \beta_{K,M \cap N})$, or 
(b) there is $\gamma \in (K \cap \kappa) \setminus \beta_{K,M \cap N}$ such that 
$\zeta = \min((M \cap N \cap \kappa) \setminus \gamma)$.

\bigskip

\emph{Case a:} $K \sim M \cap N$ and 
$\zeta = \min((M \cap N \cap \kappa) \setminus \beta_{K,M \cap N})$. 
Then by Lemma 1.24, $K \sim M$. 
We claim that $\beta_{K,M} \le \zeta$. 
Suppose for a contradiction that $\zeta < \beta_{K,M}$. 
Then since $K \sim M$ and $\zeta \in M \cap \beta_{K,M}$, 
it follows that $\zeta \in K$. 
But this contradicts that $\zeta \in R_K(M \cap N)$. 

Since $\beta_{K,M \cap N} \le \beta_{K,M} \le \zeta$, 
it follows that $\zeta = \min((M \cap N \cap \kappa) \setminus \beta_{K,M})$. 
As $M < N$, $M \cap N \cap \kappa = M \cap \beta_{M,N}$, which is an initial 
segment of $M \cap \kappa$. 
So $\zeta = \min((M \cap \kappa) \setminus \beta_{K,M})$, and hence 
$\zeta \in R_K(M)$.

\bigskip

\emph{Case b:} There is $\gamma \in (K \cap \kappa) \setminus \beta_{K,M \cap N}$ such that 
$\zeta = \min((M \cap N \cap \kappa) \setminus \gamma)$. 
Since $M \cap N \cap \kappa = M \cap \beta_{M,N}$ is an initial segment of $M \cap \kappa$, 
it follows that 
$\zeta = \min((M \cap \kappa) \setminus \gamma)$. 
If $\beta_{K,M} \le \gamma$, then since $\gamma \in K$, 
$\zeta \in R_K(M)$. 
So assume that $\gamma < \beta_{K,M}$. 

Now $\zeta \in M \cap N \cap \kappa$ implies that $\zeta < \beta_{M,N}$. 
So $\gamma < \beta_{M,N}$. 
Since $\gamma \in (K \cap \kappa) \setminus \beta_{K,M \cap N}$, $\gamma \notin M \cap N$. 
But $M \cap N \cap \kappa = M \cap \beta_{M,N}$, so 
$\gamma \notin M \cap \kappa$. 
Since $\gamma < \beta_{K,M}$ and $\gamma \in K \setminus M$, 
we have that $M < K$. 
So $M \cap \beta_{K,M} \subseteq K$. 
As $\zeta \in R_K(M \cap N)$, $\zeta \notin K$. 
Since $M \cap \beta_{K,M} \subseteq K$ and $\zeta \in M \setminus K$, 
it follows that $\beta_{K,M} \le \zeta$. 
In conclusion, $\gamma < \beta_{K,M} \le \zeta$. 
Hence $\zeta = \min((M \cap \kappa) \setminus \beta_{K,M})$. 
Since $M < K$, this implies that $\zeta \in R_K(M)$.

\bigskip

(2) Let $\zeta \in R_{M \cap N}(K)$. 
Then either (a) $K \sim M \cap N$ and 
$\zeta = \min((K \cap \kappa) \setminus \beta_{K,M \cap N})$, or 
(b) there is $\gamma \in (M \cap N) \setminus \beta_{K,M \cap N}$ such that 
$\zeta = \min((K \cap \kappa) \setminus \gamma)$. 
We will show that either $\zeta \in R_M(K)$ or $\zeta \in R_N(K)$.

\bigskip

\emph{Case a:} $K \sim M \cap N$ and 
$\zeta = \min((K \cap \kappa) \setminus \beta_{K,M \cap N})$. 
Then $K \sim M$ by Lemma 1.24. 
Assume first that $\beta_{K,M} \le \zeta$. 
Then $\beta_{K,M \cap N} \le \beta_{K,M} \le \zeta$. 
So $\zeta = \min((K \cap \kappa) \setminus \beta_{K,M})$, which implies 
that $\zeta \in R_M(K)$. 

Now assume that $\zeta < \beta_{K,M}$. 
Since $K \sim M$ and $\zeta \in K \cap \beta_{K,M}$, $\zeta \in M$. 
As $\zeta \in R_{M \cap N}(K)$, $\zeta \notin M \cap N$, 
so $\zeta \notin N$. 
Since $K \sim M < N$, $K < N$. 
As $\zeta \in K \setminus N$ and $K < N$, $\beta_{K,N} \le \zeta$. 
Since $M \cap N \subseteq N$, $\beta_{K,M \cap N} \le \beta_{K,N}$. 
Hence 
$$
\beta_{K,M \cap N} \le \beta_{K,N} \le \zeta.
$$
So $\zeta = \min((K \cap \kappa) \setminus \beta_{K,N})$, and therefore 
$\zeta \in R_N(K)$.

\bigskip

\emph{Case b:} $\zeta = \min((K \cap \kappa) \setminus \gamma)$, for some 
$\gamma \in (M \cap N) \setminus \beta_{K,M \cap N}$. 
If $\beta_{K,M} \le \gamma$, then $\gamma \in (M \cap \kappa) \setminus \beta_{K,M}$, 
and hence $\zeta \in R_M(K)$. 
Suppose that $\gamma < \beta_{K,M} \le \zeta$. 
Then $\zeta = \min((K \cap \kappa) \setminus \beta_{K,M})$. 
Since $\gamma \in (M \cap \beta_{K,M}) \setminus K$, $K < M$. 
Therefore $\zeta \in R_M(K)$.

The remaining case is that $\gamma < \zeta < \beta_{K,M}$. 
Since $\beta_{K,M \cap N} \le \gamma$ and $\gamma \in M \cap N$, 
$\gamma \notin K$. 
So $\gamma \in (M \cap \beta_{K,M}) \setminus K$. 
It follows that $K < M$. 
But $\zeta \in K \cap \beta_{K,M}$, so $\zeta \in M$. 
As $\zeta \in R_{M \cap N}(K)$ and $\zeta \in M$, 
$\zeta \notin N$. 
But $K < M < N$, so $K < N$. 
As $\zeta \in K \setminus N$, $\beta_{K,N} \le \zeta$.

If $\beta_{K,N} \le \gamma$, then $\gamma \in (N \cap \kappa) \setminus \beta_{K,N}$, 
and therefore $\zeta \in R_N(K)$. 
Suppose that $\gamma < \beta_{K,N} \le \zeta$. 
Then $\zeta = \min((K \cap \kappa) \setminus \beta_{K,N})$. 
Since $K < N$, $\zeta \in R_N(K)$.
\end{proof}

Lemmas 2.4 and 2.5 describe the same situation we considered in 
Lemmas 1.26 and 1.27.

\begin{lemma}
Let $N \le M$ and $L \in N$, where $L$, $M$, and $N$ are in $\mathcal X_0$. 
Then:
\begin{enumerate}
\item for all $\zeta \in R_L(M)$, $\beta_{M,N} \le \zeta$ and 
$\zeta \in R_N(M)$;
\item for all $\zeta \in R_M(L)$, there is $\xi \in R_M(N)$ such that 
$\zeta = \min((L \cap \kappa) \setminus \xi)$.
\end{enumerate}
\end{lemma}

\begin{proof}
Note that by Lemma 1.26, $L < M$.

(1) Let $\zeta \in R_L(M)$. 
Since $L < M$, 
there is $\gamma \in (L \cap \kappa) \setminus \beta_{L,M}$ such that 
$\zeta = \min((M \cap \kappa) \setminus \gamma)$. 
Since $\gamma \in L$ and $L \in N$, $\gamma \in N$. 
So $\gamma \in N \setminus M$. 
Since $N \le M$, $\beta_{M,N} \le \gamma$. 
Hence $\beta_{M,N} \le \zeta$. 
As $\zeta = \min((M \cap \kappa) \setminus \gamma)$, $\zeta \in R_N(M)$.

(2) Let $\zeta \in R_M(L)$. 
Since $L < M$, 
there is $\gamma \in (M \cap \kappa) \setminus \beta_{L,M}$ such that 
$\zeta = \min((L \cap \kappa) \setminus \gamma)$. 
Now $\zeta \in L \setminus M$ and $L \in N$. 
So $\zeta \in N \setminus M$. 
Since $N \le M$, this implies that $\beta_{M,N} \le \zeta$. 

If $\gamma < \beta_{M,N}$, then let $\xi := \min((N \cap \kappa) \setminus \beta_{M,N})$. 
Since $N \le M$, $\xi \in R_M(N)$. 
As $L \subseteq N$, clearly 
$\zeta = \min((L \cap \kappa) \setminus \xi)$. 
If $\beta_{M,N} \le \gamma$, then let 
$\xi := \min((N \cap \kappa) \setminus \gamma)$, which exists since 
$\zeta \in N$. 
Then $\xi \in R_M(N)$, and since $L \subseteq N$, 
$\zeta = \min((L \cap \kappa) \setminus \xi)$. 
\end{proof}

\begin{lemma}
Let $M < N$ and $L \in N$, where $L$, $M$, and $N$ are in $\mathcal X_0$. 
Then:
\begin{enumerate}
\item for all $\zeta \in R_L(M)$, either $\zeta < \beta_{M,N}$ and 
$\zeta \in R_L(M \cap N)$, or $\beta_{M,N} \le \zeta$ and 
$\zeta \in R_N(M)$;
\item for all $\zeta \in R_M(L)$, either 
$\zeta \in R_{M \cap N}(L)$ or there is $\xi \in R_M(N)$ 
such that $\zeta = \min((L \cap \kappa) \setminus \xi)$.
\end{enumerate}
\end{lemma}

\begin{proof}
Note that by Lemma 1.27, $\beta_{L,M} = \beta_{L,M \cap N}$. 
And since $M < N$, $M \cap \beta_{M,N} = M \cap N \cap \kappa$.

(1) Let $\zeta \in R_L(M)$. 
Then either (a) $L \sim M$ and $\zeta = \min((M \cap \kappa) \setminus \beta_{L,M})$, 
or (b) there is $\gamma \in (L \cap \kappa) \setminus \beta_{L,M}$ such that 
$\zeta = \min((M \cap \kappa) \setminus \gamma)$. 
Assume first that $\zeta < \beta_{M,N}$. 
In case (a), $L \sim M \cap N$ by Lemma 1.27. 
Since $\zeta < \beta_{M,N}$, 
$\zeta = \min((M \cap N \cap \kappa) \setminus \beta_{L,M \cap N})$. 
In case (b), $\gamma \in (L \cap \kappa) \setminus \beta_{L,M \cap N}$ and 
$\zeta = \min((M \cap N \cap \kappa) \setminus \gamma)$. 
In either case, $\zeta \in R_L(M \cap N)$.

Now assume that $\beta_{M,N} \le \zeta$. 
In case (a), since 
$$
\beta_{L,M} \le \beta_{M,N} \le \zeta,
$$
$\zeta = \min((M \cap \kappa) \setminus \beta_{M,N})$. 
Since $M < N$, this implies that $\zeta \in R_N(M)$. 
In case (b), if $\gamma < \beta_{M,N}$, then again 
$\zeta = \min((M \cap \kappa) \setminus \beta_{M,N})$, and so $\zeta \in R_N(M)$. 
Otherwise $\gamma \in (N \cap \kappa) \setminus \beta_{M,N}$ and 
$\zeta = \min((M \cap \kappa) \setminus \gamma)$, so $\zeta \in R_N(M)$.

\bigskip

(2) Let $\zeta \in R_M(L)$. 
Then either (a) $L \sim M$ and $\zeta = \min((L \cap \kappa) \setminus \beta_{L,M})$, 
or (b) there is $\gamma \in (M \cap \kappa) \setminus \beta_{L,M}$ such that 
$\zeta = \min((L \cap \kappa) \setminus \gamma)$. 
In case (a), $L \sim M \cap N$ by Lemma 1.27 and 
$\zeta = \min((L \cap \kappa) \setminus \beta_{L,M \cap N})$.
Hence $\zeta \in R_{M \cap N}(L)$.

Assume (b). 
First consider the case that $\gamma < \beta_{M,N}$. 
Then 
$$
\gamma \in M \cap \beta_{M,N} \subseteq M \cap N.
$$ 
So 
$$
\gamma \in (M \cap N \cap \kappa) \setminus \beta_{L,M \cap N}
$$
and 
$\zeta = \min((L \cap \kappa) \setminus \gamma)$. 
Hence $\zeta \in R_{M \cap N}(L)$. 
Now consider the case that $\beta_{M,N} \le \gamma$. 
Then $\gamma \in (M \cap \kappa) \setminus \beta_{M,N}$. 
Let $\xi := \min((N \cap \kappa) \setminus \gamma)$, which exists since 
$\zeta \in N$. 
Then $\xi \in R_M(N)$ and $\zeta = \min((L \cap \kappa) \setminus \xi)$.
\end{proof}

When amalgamating over a countable model $N$, the presence of $M \cap N$ 
prevents certain incompatibilities between $M$ and the object we build in $N$. 
But oftentimes $M \cap N$ does not have enough information about $M$. 
In that case, we will use a model $M'$ in $N$ which is more representative 
of $M$ than $M \cap N$.

\begin{lemma}
Let $L$, $M$, $M'$, and $N$ be in $\mathcal X_0$. 
Assume that $M < N$ and $L \in N$. 
Also suppose that $M' \in N$, $\{ L, M \cap N, M' \}$ is adequate, and 
$M \cap \beta_{M,N} = M' \cap \beta_{M,N}$. 
Then:
\begin{enumerate}
\item either $\beta_{L,M} = \beta_{L,M'}$ or 
$\beta_{M,N} < \beta_{L,M'}$; 
\item if $\beta_{L,M} = \beta_{L,M'}$ and $\zeta \in R_{M \cap N}(L)$, 
then $\zeta \in R_{M'}(L)$.
\end{enumerate}
\end{lemma}

\begin{proof}
Note that $\{ L, M \}$ is adequate by Lemma 1.28. 
We claim that $\beta_{L,M} \le \beta_{L,M'}$. 
Otherwise $\beta_{L,M'} < \beta_{L,M}$. 
Since $\{ L, M \}$ is adequate, we can fix 
$\xi \in (L \cap M) \cap [\beta_{L,M'},\beta_{L,M})$ by Lemma 1.19(5). 
Since $L \in N$, $\xi \in N$. 
So 
$$
\xi \in M \cap N \cap \kappa = M \cap \beta_{M,N} \subseteq M'.
$$
Hence $\xi \in (L \cap M' \cap \kappa) \setminus \beta_{L,M'}$, which is 
impossible.

(1) If $\beta_{L,M} = \beta_{L,M'}$, then we are done. 
So assume that $\beta_{L,M} < \beta_{L,M'}$. 
We claim that $\beta_{M,N} < \beta_{L,M'}$. 
Otherwise 
$$
\beta_{L,M} < \beta_{L,M'} \le \beta_{M,N}.
$$ 
Since $\{ L, M' \}$ is adequate, 
we can fix $\xi \in (L \cap M') \cap [\beta_{L,M},\beta_{L,M'})$ by Lemma 1.19(5). 
Then 
$$
\xi \in M' \cap \beta_{M,N} \subseteq M.
$$
So $\xi \in (L \cap M \cap \kappa) \setminus \beta_{L,M}$, which is a contradiction.

(2) Assume that $\beta_{L,M} = \beta_{L,M'}$ and 
$\zeta \in R_{M \cap N}(L)$. 
By Lemma 1.27, 
$$
\beta_{L,M'} = \beta_{L,M} = \beta_{L,M \cap N}.
$$

First, assume that $L \sim M \cap N$ and 
$\zeta = \min((L \cap \kappa) \setminus \beta_{L,M \cap N})$. 
Then $\zeta = \min((L \cap \kappa) \setminus \beta_{L,M'})$. 
Also 
$$
L \cap \omega_1 = (M \cap N) \cap \omega_1 = M \cap \beta_{M,N} \cap \omega_1 = 
M' \cap \beta_{M,N} \cap \omega_1 = M' \cap \omega_1.
$$
Since $\{ L, M' \}$ is adequate and $L \cap \omega_1 = M' \cap \omega_1$, 
$L \sim M'$ by Lemma 1.21. 
Since $L \sim M'$ and $\zeta = \min((L \cap \kappa) \setminus \beta_{L,M'})$, 
$\zeta \in R_{M'}(L)$.

Secondly, suppose that 
$\gamma \in (M \cap N \cap \kappa) \setminus \beta_{L,M \cap N}$ and 
$\zeta = \min((L \cap \kappa) \setminus \gamma)$. 
Then 
$$
\gamma \in (M \cap N \cap \kappa) \setminus \beta_{L,M'}.
$$
Since 
$$
M \cap N \cap \kappa = M \cap \beta_{M,N} \subseteq M',
$$
$\gamma \in (M' \cap \kappa) \setminus \beta_{L,M'}$. 
So $\zeta \in R_{M'}(L)$.
\end{proof}

The statement of the next technical lemma is not very intuitive. 
But its discovery led to substantial simplifications of some of the arguments 
from \cite{jk22}.

\begin{lemma}
Let $K$, $M$, and $N$ be in $\mathcal X_0$ such that 
$\{ K, M, N \}$ is adequate. 
Suppose that 
$$
\zeta \in R_M(N), \ \zeta \notin K, \ 
\theta = \min((K \cap \kappa) \setminus \zeta), \ 
\textrm{and} \ \theta < \beta_{K,N}.
$$
Then $\theta \in R_M(K)$.
\end{lemma}

\begin{proof}
Since $\zeta < \theta < \beta_{K,N}$ and $\zeta \in N \setminus K$, 
it follows that $K < N$. 
In particular, $K \cap (\theta+1) \subseteq N$. 

\bigskip

\emph{Case 1:} $N \le M$. 
Then $K < N \le M$, so $K < M$. 
We claim that $\beta_{K,M} \le \beta_{M,N}$. 
Otherwise $\beta_{M,N} < \beta_{K,M}$, which implies that 
$$
(K \cap M) \cap [\beta_{M,N},\beta_{K,M}) \ne \emptyset
$$
by Lemma 1.19(5). 
Let $\gamma = \min((K \cap \kappa) \setminus \beta_{M,N})$. 
Then since the intersection above is nonempty, $\gamma < \beta_{K,M}$, 
and hence $\gamma \in K \cap M$. 
But $\beta_{M,N} \le \zeta < \theta$ and $\theta \in K$ implies that 
$\gamma \le \theta$. 
Since $K \cap (\theta+1) \subseteq N$, $\gamma \in N$. 
So $\gamma \in (M \cap N) \setminus \beta_{M,N}$, which is impossible. 
This proves that $\beta_{K,M} \le \beta_{M,N}$.

Suppose that $\zeta = \min((N \cap \kappa) \setminus \gamma)$ for some 
$\gamma \in (M \cap \kappa) \setminus \beta_{M,N}$. 
Since $\beta_{K,M} \le \beta_{M,N}$, it follows that 
$\gamma \in (M \cap \kappa) \setminus \beta_{K,M}$. 
As $K \cap (\theta+1) \subseteq N$, 
$\theta = \min((K \cap \kappa) \setminus \gamma)$. 
Hence $\theta \in R_M(K)$.

Suppose that $M \sim N$ and $\zeta = \min((N \cap \kappa) \setminus \beta_{M,N})$. 
Since $K \cap (\theta+1) \subseteq N$, 
it follows that 
$\theta = \min((K \cap \kappa) \setminus \beta_{M,N})$. 
We claim that $\theta = \min((K \cap \kappa) \setminus \beta_{K,M})$, 
which implies that $\theta \in R_M(K)$ as desired. 
If not, then there is $\pi \in K \cap [\beta_{K,M},\beta_{M,N})$. 
But $\beta_{M,N} \le \zeta < \theta < \beta_{K,N}$, so 
$\pi \in K \cap \beta_{K,N} \subseteq N$. 
Hence $\pi \in N \cap \beta_{M,N} \subseteq M$. 
So $\pi \in M$. 
Therefore $\pi \in (K \cap M) \setminus \beta_{K,M}$, which is 
impossible.

\bigskip

\emph{Case 2:} $M < N$. 
Since $\zeta \in R_M(N)$, there is 
$\gamma \in (M \cap \kappa) \setminus \beta_{M,N}$ such that 
$\zeta = \min((N \cap \kappa) \setminus \gamma)$. 
If $\beta_{K,M} \le \gamma$, then 
$\gamma \in (M \cap \kappa) \setminus \beta_{K,M}$, and since 
$K \cap (\theta+1) \subseteq N$, 
$\theta = \min((K \cap \kappa) \setminus \gamma)$. 
Hence $\theta \in R_M(K)$.

Otherwise $\gamma < \beta_{K,M}$. 
Since $\gamma \notin N$, $\gamma < \theta$, and 
$K \cap (\theta+1) \subseteq N$, it follows that $\gamma \notin K$. 
So $\gamma \in (M \cap \beta_{K,M}) \setminus K$, 
which implies that $K < M$. 
Since $K \cap (\theta+1) \subseteq N$, it follows that 
$\theta = \min((K \cap \kappa) \setminus \gamma)$. 
As $\theta \in (N \cap \kappa) \setminus \beta_{M,N}$, 
$\theta \notin M$. 
As $K < M$ and $\theta \in K \cap \kappa$, $\beta_{K,M} \le \theta$. 
So $\gamma < \beta_{K,M} \le \theta$. 
Hence $\theta = \min((K \cap \kappa) \setminus \beta_{K,M})$, which implies 
that $\theta \in R_M(K)$.
\end{proof}

\bigskip

The next three lemmas are analogues of Lemmas 2.3, 2.5, and 2.6, where the 
countable model $N$ in $\mathcal X_0$ 
is replaced by an uncountable model $P$ in $\mathcal Y_0$.

\begin{lemma}
Let $K$ and $M$ be in $\mathcal X_0$ and $P \in \mathcal Y_0$. 
Assume that $\{ K, M \}$ is adequate and 
$\sup(M \cap P \cap \kappa) < P \cap \kappa$. 
Then:
\begin{enumerate}
\item $R_K(M \cap P) \subseteq R_K(M)$;
\item if $\zeta \in R_{M \cap P}(K)$, then either 
$\zeta \in R_M(K)$ or 
$\zeta = \min((K \cap \kappa) \setminus (P \cap \kappa))$.
\end{enumerate}
\end{lemma}

\begin{proof}
Note that by Lemma 1.32, 
$\beta_{K,M \cap P} \le \beta_{K,M}$, 
$\beta_{K,M \cap P} < P \cap \kappa$, and 
$\{ K, M \cap P \}$ is adequate.

\bigskip

(1) Let $\zeta \in R_K(M \cap P)$. 
Then either (a) $K \sim M \cap P$ and 
$\zeta = \min((M \cap P \cap \kappa) \setminus \beta_{K,M \cap P})$, or 
(b) there is $\gamma \in (K \cap \kappa) \setminus \beta_{K,M \cap P}$ 
such that 
$\zeta = \min((M \cap P \cap \kappa) \setminus \gamma)$.

\bigskip

\emph{Case a:} $K \sim M \cap P$ and 
$\zeta = \min((M \cap P \cap \kappa) \setminus \beta_{K,M \cap P})$. 
Then $K \sim M$ by Lemma 1.32. 
By Lemma 1.36, either $\beta_{K,M} = \beta_{K,M \cap P}$, or 
$P \cap \kappa < \beta_{K,M}$.

We claim that $\beta_{K,M} = \beta_{K,M \cap P}$. 
Suppose for a contradiction that $P \cap \kappa < \beta_{K,M}$. 
Since $\zeta \in M \cap P \cap \kappa \subseteq 
P \cap \kappa$, $\zeta < \beta_{K,M}$. 
But since $K \sim M$ and $\zeta \in M \cap \beta_{K,M}$, $\zeta \in K$. 
So $\zeta \in K \cap (M \cap P) \cap \kappa$, which contradicts that 
$\zeta \in R_{K}(M \cap P)$.

So $\beta_{K,M} = \beta_{K,M \cap P}$. 
Since $M \cap P \cap \kappa$ is an initial segment of $M \cap \kappa$, 
$\zeta = \min((M \cap \kappa) \setminus \beta_{K,M})$. 
Hence $\zeta \in R_K(M)$. 

\bigskip

\emph{Case b:} $\zeta = \min((M \cap P \cap \kappa) \setminus \gamma)$, 
for some $\gamma \in (K \cap \kappa) \setminus \beta_{K,M \cap P}$. 
Since $M \cap P \cap \kappa$ is an initial segment of $M \cap \kappa$, 
$\zeta = \min((M \cap \kappa) \setminus \gamma)$. 
By Lemma 1.36, either $\beta_{K,M} = \beta_{K,M \cap P}$ or 
$P \cap \kappa < \beta_{K,M}$. 
In the first case, $\gamma \in (K \cap \kappa) \setminus \beta_{K,M}$, 
so $\zeta \in R_K(M)$.

We prove that the other case is impossible. 
Suppose for a contradiction that $P \cap \kappa < \beta_{K,M}$. 
Since $\gamma < \zeta < P \cap \kappa$, $\gamma \in P$. 
But $\gamma \in (K \cap \kappa) \setminus \beta_{K,M \cap P}$ implies that 
$\gamma \notin M \cap P$. 
So $\gamma \notin M$. 
As $\gamma < P \cap \kappa < \beta_{K,M}$, we have that 
$\gamma \in (K \cap \beta_{K,M}) \setminus M$. 
Hence $M < K$. 
Since $\zeta \in M \cap P \cap \kappa$, $\zeta \in M \cap \beta_{K,M}$. 
As $M < K$, $\zeta \in K$. 
But this is impossible since $\zeta \in R_K(M \cap P)$.

\bigskip

(2) Let $\zeta \in R_{M \cap P}(K)$. 
We will prove that either $\zeta \in R_M(K)$, or 
$\zeta = \min((K \cap \kappa) \setminus (P \cap \kappa))$. 
Either (a) $K \sim M \cap P$ and 
$\zeta = \min((K \cap \kappa) \setminus \beta_{K,M \cap P})$, or 
(b) there is $\gamma \in (M \cap P \cap \kappa) \setminus \beta_{K,M \cap P}$ 
such that $\zeta = \min((K \cap \kappa) \setminus \gamma)$. 

\bigskip

\emph{Case a:} $K \sim M \cap P$ and 
$\zeta = \min((K \cap \kappa) \setminus \beta_{K,M \cap P})$. 
Then $K \sim M$ by Lemma 1.32. 
Also by Lemma 1.36, 
either $\beta_{K,M} = \beta_{K,M \cap P}$ or $P \cap \kappa < \beta_{K,M}$.

First, assume that $\beta_{K,M} = \beta_{K,M \cap P}$. 
Then $\zeta = \min((K \cap \kappa) \setminus \beta_{K,M})$, so 
$\zeta \in R_M(K)$.

Secondly, assume that $P \cap \kappa < \beta_{K,M}$. 
Suppose that $\beta_{K,M} \le \zeta$. 
Since $\beta_{K,M \cap P} \le \beta_{K,M}$, it follows that 
$\zeta = \min((K \cap \kappa) \setminus \beta_{K,M})$. 
Therefore $\zeta \in R_M(K)$.

Otherwise $\zeta < \beta_{K,M}$. 
But then $K \sim M$ and $\zeta \in K \cap \beta_{K,M}$ imply that $\zeta \in M$. 
Since $\zeta \in R_{M \cap P}(K)$ and $\zeta \in M$, 
$\zeta \notin P \cap \kappa$. 
Therefore $\beta_{K,M \cap P} < P \cap \kappa \le \zeta$. 
So $\zeta = \min((K \cap \kappa) \setminus (P \cap \kappa))$.

\bigskip

\emph{Case b:} 
$\zeta = \min((K \cap \kappa) \setminus \gamma)$ for some 
$\gamma \in (M \cap P \cap \kappa) \setminus \beta_{K,M \cap P}$. 
If $P \cap \kappa \le \zeta$, then $\gamma < P \cap \kappa \le \zeta$ 
implies that $\zeta = \min((K \cap \kappa) \setminus (P \cap \kappa))$.

Suppose that $\zeta < P \cap \kappa$. 
If $\beta_{K,M} \le \gamma$, 
then $\gamma \in (M \cap \kappa) \setminus \beta_{K,M}$, 
and therefore $\zeta \in R_M(K)$. 
So assume that $\gamma < \beta_{K,M}$. 
First consider the case that $\beta_{K,M} \le \zeta$. 
Then $\zeta = \min((K \cap \kappa) \setminus \beta_{K,M})$. 
Since $\gamma \in (M \cap \beta_{K,M}) \setminus K$, 
it follows that $K < M$. 
So $\zeta \in R_M(K)$.

In the final case, assume that $\gamma < \zeta < \beta_{K,M}$. 
We will show that this case does not occur. 
Then 
$$
\beta_{K,M \cap P} \le \gamma < \zeta < \beta_{K,M}.
$$
Since $\gamma \in (M \cap \beta_{K,M}) \setminus K$, 
it follows that $K < M$. 
So as $\zeta \in K \cap \beta_{K,M}$, $\zeta \in M$. 
But also $\zeta \in P \cap \kappa$. 
So $\zeta \in M \cap P$, which contradicts that $\zeta \in R_{M \cap P}(K)$.
\end{proof}

\begin{lemma}
Let $L$ and $M$ be in $\mathcal X_0$ and $P$ in $\mathcal Y_0$. 
Assume that $L \in P$, $\{ L, M \cap P \}$ is adequate, and 
$\sup(M \cap P \cap \kappa) < P \cap \kappa$. 
Then:
\begin{enumerate}
\item if $\zeta \in R_L(M)$, then either $\zeta \in R_L(M \cap P)$ or 
$\zeta = \min((M \cap \kappa) \setminus (P \cap \kappa))$;
\item $R_M(L) \subseteq R_{M \cap P}(L)$.
\end{enumerate}
\end{lemma}

\begin{proof}
Note that by Lemma 1.34, 
$\beta_{L,M} = \beta_{L,M \cap P}$, $\beta_{L,M} < P \cap \kappa$, 
and $\{ L, M \}$ is adequate.

\bigskip

(1) Let $\zeta \in R_L(M)$. 
Then either (a) $L \sim M$ and 
$\zeta = \min((M \cap \kappa) \setminus \beta_{L,M})$, 
or (b) there is $\gamma \in (L \cap \kappa) \setminus \beta_{L,M}$ such that 
$\zeta = \min((M \cap \kappa) \setminus \gamma)$. 

\bigskip

\emph{Case a:} $L \sim M$ and 
$\zeta = \min((M \cap \kappa) \setminus \beta_{L,M})$. 
Then $L \sim M \cap P$ by Lemma 1.34. 
If $P \cap \kappa \le \zeta$, then since $\beta_{L,M} < P \cap \kappa$, 
it follows that $\zeta = \min((M \cap \kappa) \setminus (P \cap \kappa))$.
Suppose that $\zeta < P \cap \kappa$. 
Then 
$$
\zeta = \min((M \cap P \cap \kappa) \setminus \beta_{L,M}) = 
\min((M \cap P \cap \kappa) \setminus \beta_{L,M \cap P}).
$$
So $\zeta \in R_L(M \cap P)$.

\bigskip

\emph{Case b:} There is 
$\gamma \in (L \cap \kappa) \setminus \beta_{L,M}$ such that 
$\zeta = \min((M \cap \kappa) \setminus \gamma)$. 
Then $\gamma \in (L \cap \kappa) \setminus \beta_{L,M \cap P}$. 
If $\zeta < P \cap \kappa$, then 
$\zeta = \min((M \cap P \cap \kappa) \setminus \gamma)$, 
so $\zeta \in R_L(M \cap P)$. 
Otherwise $P \cap \kappa \le \zeta$, and since $\gamma \in L$, 
$\gamma < P \cap \kappa$. 
So $\zeta = \min((M \cap \kappa) \setminus (P \cap \kappa))$.

\bigskip

(2) Let $\zeta \in R_M(L)$, and we will show that 
$\zeta \in R_{M \cap P}(L)$. 
Either (a) $L \sim M$ and 
$\zeta = \min((L \cap \kappa) \setminus \beta_{L,M})$,  
or (b) there is $\gamma \in (M \cap \kappa) \setminus \beta_{L,M}$ such that 
$\zeta = \min((L \cap \kappa) \setminus \gamma)$. 

Assume (a). 
Then $L \sim M \cap P$ by Lemma 1.34. 
Also $\zeta = \min((L \cap \kappa) \setminus \beta_{L,M \cap P})$, 
so $\zeta \in R_{M \cap P}(L)$.

Assume (b). 
Since $\zeta \in L$ and $L \in P$, $\zeta \in P$. 
As $\gamma < \zeta$ and 
$\zeta \in P \cap \kappa$, $\gamma < P \cap \kappa$. 
So $\gamma \in M \cap P$. 
Thus $\gamma \in (M \cap P \cap \kappa) \setminus \beta_{L,M \cap P}$ and 
$\zeta = \min((L \cap \kappa) \setminus \gamma)$. 
So $\zeta \in R_{M \cap P}(L)$.
\end{proof}

\begin{lemma}
Let $L$, $M$, and $M'$ be in $\mathcal X_0$, and let 
$P$ and $P'$ be in $\mathcal Y_0$. 
Assume that $\{ L, M, M' \}$ is adequate, and $L$, $M'$, and $P'$ are in $P$. 
Let $\beta := P \cap \kappa$ and $\beta' := P' \cap \kappa$. 
Suppose that $\sup(M \cap \beta) < \beta'$ and 
$M \cap \beta = M' \cap \beta'$. 
Then:
\begin{enumerate}
\item $\beta_{L,M} < \beta'$;

\item either $\beta_{L,M} = \beta_{L,M'}$ or 
$\beta' < \beta_{L,M'}$;

\item if $\beta_{L,M} = \beta_{L,M'}$ and 
$\zeta \in R_{M \cap P}(L)$, then $\zeta \in R_{M'}(L)$.
\end{enumerate}
\end{lemma}

\begin{proof}
(1) Since $M \cap \beta \subseteq \beta'$ and $L \cap \kappa \subseteq \beta$, 
$L \cap M \cap \kappa \subseteq \beta'$. 
As $\sup(M \cap \beta) < \beta'$, $L \cap M \cap \kappa$ is a 
bounded subset of $\beta'$. 
By the elementarity of $P'$, fix $\gamma \in \Lambda$ such that 
$\sup(L \cap M \cap \kappa) < \gamma < \beta'$. 
By Lemma 1.19(3), 
$$
\beta_{L,M} = \min(\Lambda \setminus \sup(L \cap M \cap \kappa)) 
\le \gamma < \beta'.
$$

(2) If $\beta_{L,M} = \beta_{L,M'}$, then we are done. 
So suppose not. 
We claim that $\beta_{L,M} < \beta_{L,M'}$. 
Suppose for a contradiction that $\beta_{L,M'} < \beta_{L,M}$. 
By Lemma 1.19(3), 
$\beta_{L,M'} = \min(\Lambda \setminus \sup(L \cap M' \cap \kappa))$. 
But 
$$
\beta_{L,M'} < \beta_{L,M} < \beta'
$$
by (1) and 
the assumption just made. 
So $\beta_{L,M'} < \beta'$. 
Hence $L \cap M' \cap \kappa = L \cap M' \cap \beta'$. 
Since $M \cap \beta = M' \cap \beta'$, it follows that 
$$
\sup(L \cap M' \cap \kappa) = \sup(L \cap M' \cap \beta') = 
\sup(L \cap M \cap \beta) = \sup(L \cap M \cap \kappa).
$$
So 
$$
\beta_{L,M'} = \min(\Lambda \setminus \sup(L \cap M \cap \kappa)) = 
\beta_{L,M}.
$$
But this contradicts the assumption that $\beta_{L,M'} < \beta_{L,M}$.

This proves that $\beta_{L,M} < \beta_{L,M'}$. 
By Lemma 1.19(5), 
we can fix $\xi \in (L \cap M') \cap [\beta_{L,M},\beta_{L,M'})$. 
Since $\beta_{L,M} \le \xi$ and $\xi \in L$, it follows that $\xi \notin M$. 
But $M \cap \beta = M' \cap \beta'$. 
Since $\xi \in (M' \cap \kappa) \setminus M$, $\beta' \le \xi$. 
As $\xi < \beta_{L,M'}$, it follows that $\beta' < \beta_{L,M'}$.

(3) Suppose that $\beta_{L,M} = \beta_{L,M'}$ and 
$\zeta \in R_{M \cap P}(L)$. 
We will prove that $\zeta \in R_{M'}(L)$. 
Since $\beta_{L,M} = \beta_{L,M \cap P}$ by Lemma 1.34, 
$\beta_{L,M'} = \beta_{L,M \cap P}$. 
First assume that $L \sim M \cap P$ and 
$\zeta = \min((L \cap \kappa) \setminus \beta_{L,M \cap P})$. 
Hence $\zeta = \min((L \cap \kappa) \setminus \beta_{L,M'})$. 
As $L \sim M \cap P$, 
$$
L \cap \omega_1 = M \cap P \cap \omega_1 = 
M' \cap \omega_1.
$$
So $L \sim M'$ by Lemma 1.21. 
Hence $\zeta \in R_{M'}(L)$.

Now assume that $\zeta = \min((L \cap \kappa) \setminus \gamma)$, 
where $\gamma \in (M \cap P \cap \kappa) \setminus \beta_{L,M \cap P}$. 
Since $M \cap P \cap \kappa = M \cap \beta \subseteq M'$, 
$\gamma \in M'$. 
And 
$$
\beta_{L,M \cap P} = \beta_{L,M} = \beta_{L,M'} \le \gamma.
$$
So $\gamma \in (M' \cap \kappa) \setminus \beta_{L,M'}$. 
Therefore $\zeta \in R_{M'}(L)$.
\end{proof}

\bigskip

The final lemma concerning remainder points 
will be used when amalgamating over transitive models.

\begin{lemma}
Let $M$, $M'$, $N$, and $N'$ be in $\mathcal X_0$. 
Assume that $M \cap \kappa = M' \cap \kappa$ and 
$N \cap \kappa = N' \cap \kappa$. 
Then $R_M(N) = R_{M'}(N')$.
\end{lemma}

\begin{proof}
We will show that $R_M(N) \subseteq R_{M'}(N')$. 
The reverse 
inclusion follows by symmetry. 
So let $\zeta \in R_M(N)$.

First, assume that $M \sim N$ and 
$\zeta = \min((N \cap \kappa) \setminus \beta_{M,N})$. 
Then by Lemma 1.37, $\beta_{M,N} = \beta_{M',N'}$ and $M' \sim N'$. 
Since $N' \cap \kappa = N \cap \kappa$, clearly 
$\zeta = \min((N' \cap \kappa) \setminus \beta_{M',N'})$. 
So $\zeta \in R_{M'}(N')$.

Secondly, 
assume that $\zeta = \min((N \cap \kappa) \setminus \gamma)$, for some 
$\gamma \in (M \cap \kappa) \setminus \beta_{M,N}$. 
By Lemma 1.37, $\beta_{M,N} = \beta_{M',N'}$. 
Since $M \cap \kappa = M' \cap \kappa$, 
$\gamma \in (M' \cap \kappa) \setminus \beta_{M',N'}$. 
As $N \cap \kappa = N' \cap \kappa$, 
$\zeta = \min((N' \cap \kappa) \setminus \gamma)$. 
So $\zeta \in R_{M'}(N')$.
\end{proof}

\bigskip

\addcontentsline{toc}{section}{3. Strong genericity and cardinal preservation}

\textbf{\S 3. Strong genericity and cardinal preservation}

\stepcounter{section}

\bigskip

In this section we will discuss the idea of a strongly generic condition, 
which is due to Mitchell \cite{mitchell}. 
Then we will use the existence of strongly generic conditions to 
prove cardinal preservation results. 
All of the results in this section are either due to Mitchell, or are based on 
standard proper forcing arguments.

\begin{definition}
Let $\q$ be a forcing poset, $q \in \q$, and $N$ a set. 
We say that $q$ is a \emph{strongly $N$-generic condition} if 
for any set $D$ which is a dense subset of $N \cap \q$, 
$D$ is predense in $\q$ below $q$.
\end{definition}

Note that if $q$ is strongly $N$-generic and $r \le q$, then 
$r$ is strongly $N$-generic.

\begin{notation}
For a forcing poset $\q$, let $\lambda_\q$ denote the least cardinal 
such that $\q \subseteq H(\lambda_\q)$.
\end{notation}

Note that $q$ is strongly $N$-generic iff 
$q$ is strongly $(N \cap H(\lambda_\q))$-generic.

The following proposition gives a more intuitive description of 
strong genericity.

\begin{lemma}
Let $\q$ be a forcing poset, $q \in \q$, and $N \prec (H(\chi),\in,\q)$, 
where $\lambda_\q \le \chi$ is a cardinal. 
Then $q$ is a strongly $N$-generic condition iff $q$ forces 
that $N \cap \dot G$ is a $V$-generic filter on $N \cap \q$.
\end{lemma}

\begin{proof}
Suppose that $q$ is a strongly $N$-generic condition, and let 
$G$ be a $V$-generic filter on $\q$ containing $q$. 
We will show that $N \cap G$ is a 
$V$-generic filter on $N \cap \q$.

First, we show that $N \cap G$ is a filter on $N \cap \q$. 
If $p \in N \cap G$ and $t \in N \cap \q$ with $p \le t$, then 
$t \in G$ since $G$ is a filter, and hence $t \in N \cap G$.
Suppose that $s$ and $t$ are in $N \cap G$, and we will find 
$p \in N \cap G$ such that $p \le s, t$. 
The set $D$ of $p$ in $N \cap \q$ 
which are either 
incompatible with one of $s$ and $t$, or below both $s$ and $t$, is 
a dense subset of $N \cap \q$ by the elementarity of $N$. 
Since $q$ is strongly $N$-generic, 
$D$ is predense below $q$. 
As $q \in G$ and $G$ is a $V$-generic filter, we can fix 
$p \in G \cap D$. 
Since $s$, $t$, and $p$ are in $G$, $p$ is compatible with $s$ and $t$, 
and therefore $p \le s, t$ by the definition of $D$.
As $D \subseteq N$, $p \in N \cap G$.

Secondly, we prove that $N \cap G$ is $V$-generic on $N \cap \q$. 
So let $D$ be a dense subset of $N \cap \q$. 
Since $q$ is a strongly $N$-generic condition, $D$ is predense below $q$. 
As $q \in G$, it follows that $D \cap G \ne \emptyset$. 
But $D \subseteq N$, so 
$D \cap N \cap G \ne \emptyset$.

Conversely, suppose that $q$ forces that $N \cap \dot G$ 
is a $V$-generic filter on $N \cap \q$, 
and we will show that $q$ is strongly $N$-generic. 
Let $D$ be a dense subset of $N \cap \q$. 
If $D$ is not predense below $q$, then we can fix $r \le q$ which 
is incompatible with every condition in $D$. 
Let $G$ be a $V$-generic filter on $\q$ containing $r$. 
Since $r \le q$, $q \in G$. 
Hence by assumption, $N \cap G$ is a $V$-generic filter on $N \cap \q$. 
Since $D$ is dense in $N \cap \q$, 
we can fix $s \in G \cap D$. 
Then $r$ is incompatible with $s$ by the choice of $r$, and yet $r$ and $s$ 
are compatible since they are both in the filter $G$.
\end{proof}

The following combinatorial characterization of 
strong genericity is very useful in practice.

\begin{lemma}
Let $\q$ be a forcing poset, $q \in \q$, and $N$ a set. 
Then the following are equivalent:
\begin{enumerate}
\item $q$ is strongly $N$-generic;
\item for all $r \le q$, there exists $v \in N \cap \q$ such that for all 
$w \le v$ in $N \cap \q$, $r$ and $w$ are compatible.
\end{enumerate}
\end{lemma}

\begin{proof}
For the forward direction, suppose that there is 
$r \le q$ for which there does not exist a condition 
$v \in N \cap \q$ all of whose extensions in $N \cap \q$ 
are compatible with $r$. 
Let $D$ be the set of $w \in N \cap \q$ which are incompatible with $r$. 
The assumption on $r$ implies that $D$ is dense in $N \cap \q$. 
But $D$ is not predense below $q$ since every condition in $D$ is 
incompatible with $r$. 
So $q$ is not strongly $N$-generic.

Conversely, assume that there is a function 
$r \mapsto v_r$ as described in (2). 
Let $D$ be dense in $N \cap \q$, and let $r \le q$. 
Since $D$ is dense in $N \cap \q$, we can 
fix $w \le v_r$ in $D$. 
Then $r$ and $w$ are compatible by the choice of $v_r$. 
So $D$ is predense below $q$.
\end{proof}

\bigskip

The next idea was introduced by Cox-Krueger \cite{jk26}.

\begin{definition}
Let $\q$ be a forcing poset, $q \in \q$, and $N$ a set. 
We say that $q$ is a \emph{universal strongly $N$-generic condition} if 
$q$ is a strongly $N$-generic condition and for all $p \in N \cap \q$, 
$p$ and $q$ are compatible.
\end{definition}

The strongly generic conditions used in this paper are universal. 
This fact allows us to factor forcing posets over 
elementary substructures in such a way that the quotient forcing has 
nice properties. 
See Section 6 for more details on this topic.

\bigskip

\begin{definition}
Let $\q$ be a forcing poset and $\mu \le \lambda_\q$ a regular 
uncountable cardinal. 
We say that $\q$ is \emph{$\mu$-strongly proper on a stationary 
set} if there are stationarily many $N$ in $P_{\mu}(H(\lambda_\q))$ 
such that for all $p \in N \cap \q$, there is $q \le p$ such that 
$q$ is strongly $N$-generic.
\end{definition}

When we say that $\q$ is strongly proper on a stationary set, we will mean 
that it is $\omega_1$-strongly proper on a stationary set. 

By standard arguments, 
$\q$ is $\mu$-strongly proper on a stationary set iff 
for any cardinal $\lambda_\q \le \chi$, there are stationarily many 
$N$ in $P_{\mu}(H(\chi))$ such that for all $p \in N \cap \q$, 
there is $q \le p$ such that $q$ is strongly $N$-generic.

\begin{lemma}
Let $\q$ be a forcing poset and $\mu \le \lambda_\q$ a regular 
uncountable cardinal. 
If there are stationarily many $N$ in $P_{\mu}(H(\lambda_\q))$ such that 
there exists a universal strongly $N$-generic condition, then $\q$ is 
$\mu$-strongly proper on a stationary set.
\end{lemma}

\begin{proof}
Let $N \in P_{\mu}(H(\lambda_\q))$ be such that there exists a 
universal strongly $N$-generic condition $q_N$. 
Let $p \in N \cap \q$, and we will find $r \le p$ which is 
strongly $N$-generic. 
Since $q_N$ is universal, $p$ and $q_N$ are compatible. 
So fix $r \le p, q_N$. 
Then $r \le p$ and $r$ is strongly $N$-generic.
\end{proof}

\bigskip

\begin{definition}
Let $\mu$ be a regular uncountable cardinal. 
A forcing poset $\q$ is said to satisfy the \emph{$\mu$-covering property} if 
$\q$ forces that for any set $a \subseteq On$ in the generic extension, 
if $a$ has size less than $\mu$ in the generic extension, then 
there is $b$ in the ground model with size less than $\mu$ in the ground 
model such that $a \subseteq b$.
\end{definition}

Note that if $\q$ has the $\mu$-covering property, then $\q$ 
forces that $\mu$ is regular.

\begin{proposition}
Let $\q$ be a forcing poset, and let $\mu \le \lambda_\q$ be a regular 
uncountable cardinal. 
Suppose that $\q$ is $\mu$-strongly proper on a stationary set. 
Then $\q$ satisfies the $\mu$-covering property.\footnote{The proof of this proposition is 
basically the same as a standard proof that proper forcing posets preserve $\omega_1$.}
\end{proposition}

\begin{proof}
Let $p$ be a condition, and suppose that $p$ forces that 
$\dot a$ is a set of ordinals of size less than $\mu$. 
We will find $q \le p$ and a set $x$ of size less than $\mu$ such that 
$q$ forces that $\dot a \subseteq x$. 
Extending $p$ if necessary, we can assume that $p$ forces that $\dot a$ has 
size $\mu_0$, for some cardinal $\mu_0 < \mu$. 
Fix a sequence 
$\langle \dot \alpha_i : i < \mu_0 \rangle$ of $\q$-names 
such that $p$ forces that $\dot a = \{ \dot \alpha_i : i < \mu_0 \}$.

Fix a regular cardinal $\lambda_\q \le \chi$ such that 
$\q$, $\dot a$, and $\langle \dot \alpha_i : i < \mu_0 \rangle$ 
are members of $H(\chi)$. 
Fix $N \in P_{\mu}(H(\chi))$ 
such that $N \prec 
(H(\chi),\in,\q,p,\dot a,\langle \dot \alpha_i : i < \mu_0 \rangle)$, 
$\mu_0 \subseteq N$, 
and for all $p_0 \in N \cap \q$, 
there is $q \le p_0$ which is a strongly $N$-generic condition. 
In particular, since $p \in N \cap \q$, we can fix $q \le p$ such that 
$q$ is strongly $N$-generic.

We claim that $q$ forces that for all $i < \mu_0$, 
$\dot \alpha_i \in N$. 
Let $i < \mu_0$. 
Let $D$ be the set of 
$s \in N \cap \q$ such that $s$ decides the 
value of $\dot \alpha_i$. 
By the elementarity of $N$, it is easy to see that 
$D$ is dense in $N \cap \q$. 
Since $q$ is strongly $N$-generic, $D$ is predense below $q$. 
Therefore $q$ forces that $\dot \alpha_i$ is decided by a condition in $N$. 
By elementarity, the value of the name $\dot \alpha_i$ decided by a 
condition in $N$ lies in $N$. 
Hence $q$ forces that $\dot \alpha_i \in N$. 
It follows that $q$ forces that $\dot a \subseteq N \cap On$. 
Since $N$ has size less than $\mu$, we are done.
\end{proof}

\begin{corollary}
Let $\q$ be a forcing poset, and let $\mu \le \lambda_\q$ be a regular 
uncountable cardinal. 
Suppose that there are stationarily many $N$ in 
$P_{\mu}(H(\lambda_\q))$ for which there exists a universal 
strongly $N$-generic condition. 
Then $\q$ satisfies the $\mu$-covering property. 
In particular, $\q$ forces that $\mu$ is a regular cardinal.
\end{corollary}

\begin{proof}
Immediate from Lemma 3.7 and Proposition 3.9.
\end{proof}

\bigskip

\begin{proposition}
Let $\q$ be a forcing poset, and let $\mu \le \lambda_\q$ be a regular 
uncountable cardinal. 
Suppose that there are stationarily many 
$N \in P_{\mu}(H(\lambda_\q))$ 
such that every condition in $\q$ is a strongly $N$-generic condition. 
Then $\q$ is $\mu$-c.c.
\end{proposition}

Note that if $\q$ has a maximum condition, then every condition in $\q$ 
being strongly $N$-generic is equivalent to the maximum condition 
being strongly $N$-generic.

\begin{proof}
Let $A$ be a maximal antichain in $\q$, and we will show that 
$|A| < \mu$. 
Let $N \in P_{\mu}(H(\lambda_\q))$ be such that 
$N \prec (H(\lambda_\q),\in,\q,A)$ and every condition in $\q$ is 
strongly $N$-generic. 

Note that by the elementarity of $N$ and since $A$ is a maximal antichain, 
$N \cap A$ is predense in $N \cap \q$.  
Namely, if $u \in N \cap \q$, then $u$ is compatible with some member of $A$. 
By elementarity, $u$ is compatible with some member of $N \cap A$. 
Let $D$ be the set of $t \in N \cap \q$ 
such that for some $s \in N \cap A$, $t \le s$. 
Then easily $D$ is dense in $N \cap \q$. 
Since every condition in $\q$ is strongly $N$-generic, $D$ is predense in $\q$.

We claim that $A \subseteq N$, which implies that $|A| \le |N| < \mu$. 
So let $p \in A$ be given, and we will show that $p \in N$. 
Since $D$ is predense in $\q$, fix $t \in D$ which is compatible with $p$. 
By the definition of $D$, we can fix $s \in N \cap A$ such that $t \le s$. 
Then $p$ and $s$ are compatible. 
But $p$ and $s$ are both in $A$ and $A$ is an antichain, 
so $p = s$. 
Since $s \in N$, $p \in N$.
\end{proof}

In general, forcings which include adequate sets as side conditions 
will collapse $\kappa$ to become $\omega_2$. 
In other words, all the cardinals $\mu$ with $\omega_1 < \mu < \kappa$ 
will be collapsed to have size $\omega_1$. 
The next result describes some general properties of a forcing poset 
which imply that such collapsing takes place.

\begin{proposition}
Suppose that $\q$ is a forcing poset which preserves 
$\omega_1$ and satisfies:
\begin{enumerate}
\item there exists an integer $k < \omega$ such that 
the conditions of $\q$ are of the form 
$(x_1,\ldots,x_k,A)$, 
where $x_1,\ldots,x_k$ are finite subsets of $H(\lambda)$, 
and $A$ is an adequate set;
\item if 
$(y_1,\ldots,y_k,B) \le (x_1,\ldots,x_k,A)$, then $A \subseteq B$;
\item there are stationarily many $N \in \mathcal X_0$ 
such that whenever $(x_1,\ldots,x_k,A) \in N \cap \q$, then 
$(x_1,\ldots,x_k,A \cup \{ N \})$ is a condition below 
$(x_1,\ldots,x_k,A)$.
\end{enumerate}
Then for any cardinal $\omega_1 < \mu < \kappa$, $\q$ collapses 
$\mu$ to have size $\omega_1$.
\end{proposition}

\begin{proof}
It suffices to show that $\q$ singularizes all regular cardinals $\mu$ with 
$\omega_1 < \mu < \kappa$. 
For suppose that this is true, but there is a cardinal 
$\mu$ in the interval $(\omega_1,\kappa)$ in some generic extension. 
Assume moreover that $\mu$ is the least such cardinal. 
Then $\mu = \omega_2$ in the generic extension. 
By downwards absoluteness, $\mu$ is regular in the ground model. 
This contradicts our assumption that all regular cardinals in the interval 
$(\omega_1,\kappa)$ are singularized.

Let $G$ be a $V$-generic filter on $\q$. 
Define 
$$
X := \{ N : \exists (x_1,\ldots,x_k,A) \in G, \ N \in A \}.
$$
Let 
$$
X_\mu := \{ N \in X : \mu \in N \}.
$$
Then by (2) and the fact that $G$ is a filter, for any 
$M$ and $N$ in $X_\mu$, there is 
a condition $(x_1,\ldots,x_k,A) \in G$ such that 
$M$ and $N$ are in $A$. 
Since $A$ is adequate, $\{ M, N \}$ is adequate. 
As $\mu \in M \cap N \cap \kappa$, $\mu < \beta_{M,N}$. 
Therefore either $M \cap \mu = N \cap \mu$, $M \cap \mu \in N$, 
or $N \cap \mu \in M$. 
Moreover, which of these three relations holds is determined by how 
$M \cap \omega_1$ and $N \cap \omega_1$ are ordered, by Lemma 1.21. 
It follows that 
$\{ \sup(N \cap \mu) : N \in X_\mu \}$ is a strictly increasing sequence of 
ordinals with order type at most $\omega_1$.

We claim that the set $\{ \sup(N \cap \mu) : N \in X_\mu \}$ is cofinal 
in $\mu$. 
The claim implies that $\mu$ has cofinality less than or equal to $\omega_1$ 
in $V[G]$, finishing the proof. 
Fix a name $\dot X_\mu$ which is forced to be equal to the set $X_\mu$ defined above.

Let $\gamma < \mu$ and $(x_1,\ldots,x_k,A)$ be a condition. 
By (3), there is $N \in \mathcal X_0$ 
such that $(x_1,\ldots,x_k,A)$, $\gamma$, and $\mu$ are in $N$, and 
$(x_1,\ldots,x_k,A \cup \{ N \})$ is a condition below 
$(x_1,\ldots,x_k,A)$. 
Since $\mu \in N$, $(x_1,\ldots,x_k,A \cup \{ N \})$ forces that 
$N \in \dot X_\mu$. 
As $\gamma \in N$, $\gamma < \sup(N \cap \mu)$.
\end{proof}

\bigskip

\addcontentsline{toc}{section}{4. Adding a club}

\textbf{\S 4. Adding a club}

\stepcounter{section}

\bigskip

In this section we give an example 
to illustrate the methods developed so far, by showing how 
to add a club subset of a stationary subset of $\omega_2$ 
using adequate sets of models. 
Adding a club with finite conditions was the original application 
of the side conditions of Friedman \cite{friedman} 
and Mitchell \cite{mitchell}. 
Later Neeman \cite{neeman} defined a forcing for adding a club using 
his method of two-type side conditions. 
The forcing poset we develop in this section is the first example 
of a forcing which adds a club subset of $\omega_2$ using conditions 
which are just finite sets of models ordered by reverse inclusion.

The following general lemma will be used frequently in this section.

\begin{lemma}
Let $A$ be an adequate set. 
Let $K$, $M$, and $N$ be in $A$, and $\zeta \in R_M(N)$. 
Suppose that $\theta = \min((K \cap \kappa) \setminus \zeta)$. 
Then 
$$
\theta \in R_M(N) \cup R_M(K) \cup R_N(K).
$$
\end{lemma}

\begin{proof}
If $\theta = \zeta$, then $\theta \in R_M(N)$ and we are done. 
Assume that $\zeta < \theta$, which means that $\zeta \notin K$. 
If $\theta < \beta_{K,N}$, then $\theta \in R_M(K)$ by Lemma 2.7. 

Suppose that $\beta_{K,N} \le \theta$. 
If $\beta_{K,N} \le \zeta$, then since $\zeta \in N$, we have that 
$$
\theta = \min((K \cap \kappa) \setminus \zeta) \in R_N(K).
$$
Otherwise $\zeta < \beta_{K,N} \le \theta$. 
Then $\theta = \min((K \cap \kappa) \setminus \beta_{K,N})$. 
Since $\zeta \in (N \cap \beta_{K,N}) \setminus K$, 
it follows that $K < N$. 
So $\theta \in R_N(K)$. 
\end{proof}

For the remainder of this section, let $\kappa = \lambda = \omega_2$. 
Recall that $T^*$ is a thin stationary subset of $P_{\omega_1}(\omega_2)$. 
We will also assume that $2^{\omega_1} = \omega_2$, and hence 
that $H(\omega_2)$ has size $\omega_2$. 
Fix a bijection $g^* : \omega_2 \to H(\omega_2)$. 

Let $\mathcal B$ denote the structure 
$$
(H(\omega_2),\in,\unlhd,T^*,\pi^*,C^*,\Lambda,g^*).
$$
Note that if $N$ is a countable elementary substructure of $\mathcal B$ 
and $N \cap \omega_2 \in T^*$, then $N \in \mathcal X_0$. 
Also note that if $M$ and $N$ are countable elementary substructures of 
$\mathcal B$ and $M \cap \omega_2 \in N$, then 
by the elementarity of $M$, $M = g^*[M \cap \omega_2]$, and 
hence by the elementarity of $N$, $M \in N$.

Fix a stationary set $S \subseteq \omega_2 \cap \cof(\omega_1)$. 
We will define a forcing poset which adds a club subset of 
$S \cup \cof(\omega)$.\footnote{More generally, 
it is possible to add a club subset 
to a fat stationary subset of $\omega_2$ using adequate sets as 
side conditions, but the argument is more complicated than the one 
which we give here. 
See \cite{jk24}.}

\begin{definition}
A finite set $A$ of elementary substructures of $\mathcal B$ in 
$\mathcal X_0$ is \emph{$S$-adequate} if 
it is adequate, and for all $M$ and $N$ in $A$, 
$R_M(N) \subseteq S$.
\end{definition}

Recall that $(\mathcal B,S)$ is the structure $\mathcal B$ augmented with the 
additional predicate $S$. 
Note that the property of being $S$-adequate is definable 
in the structure $(\mathcal B,S)$.

\begin{definition}
Let $\p$ be the forcing poset consisting of $S$-adequate sets, ordered 
by reverse inclusion.
\end{definition}

We will show that $\p$ preserves all cardinals, and adds a club 
subset of $S \cup \cof(\omega)$.

Note that since $H(\omega_2)$ has size $\omega_2$ and 
$\p \subseteq H(\omega_2)$, $\p$ has size $\omega_2$ and thus 
preserves all cardinals greater than $\omega_2$.

\begin{proposition}
The forcing poset $\p$ is strongly proper on a stationary set. 
Therefore $\p$ satisfies the $\omega_1$-covering property and 
preserves $\omega_1$.
\end{proposition}

\begin{proof}
Let $N$ be a countable elementary substructure of $(\mathcal B,S)$ 
such that $N \cap \omega_2 \in T^*$. 
Note that $N \in \mathcal X_0$. 
Let $A_0$ be in $N \cap \p$. 
Define $A_1 := A_0 \cup \{ N \}$. 
Observe that $A_1$ is adequate, since for all $M \in A_0$, 
$M \cap \beta_{M,N} = M \cap \omega_2$, which is in $N$. 
Also $A_1$ is $S$-adequate, because for all $M \in A_0$, 
$R_M(N)$ and $R_N(M)$ are empty. 
Thus $A_1$ is in $\p$ and $A_1 \le A_0$.

We claim that $A_1$ is a strongly $N$-generic condition. 
By Lemma 3.4, it suffices to show that for all 
$A_2 \le A_1$, there exists $B \in N \cap \p$ such that for all 
$C \le B$ in $N \cap \p$, $A_2 \cup C$ is $S$-adequate. 
Let $A_2 \le A_1$.

We claim that whenever $A_3 \le A_2$, $K$ and $M$ are in $A_3$, 
and $M < N$, then 
$$
R_{M \cap N}(K) \cup R_K(M \cap N) \subseteq S.
$$
But this follows immediately from Lemma 2.3 and the fact that 
$A_3$ is $S$-adequate. 

By applying Proposition 1.25 and the last claim 
finitely many times, we get that the set 
$$
A := A_2 \cup \{ M \cap N : M \in A_2, \ M < N \}
$$
is $S$-adequate. 
Hence $A \in \p$ and $A \le A_2$.

Let 
$$
x := \bigcup \{ R_M(N) : M \in A \}.
$$
Since $x \subseteq N$ and $x$ is finite, $x \in N$.

The sets $A$ and $N$ 
witness that the following statement holds in $(\mathcal B,S)$:

\bigskip

There are $B$ and $N'$ such that $B$ is $S$-adequate, 
$A \cap N \subseteq B$, $N' \in B$, 
and $x = \bigcup \{ R_M(N') : M \in B \}$.

\bigskip

The parameters which appear in the above statement, namely 
$A \cap N$ and $x$, are members of $N$. 
By the elementarity of $N$, we can find $B$ and $N'$ in $N$ which satisfy 
the same statement.

Suppose that $C \in N \cap \p$ and $C \le B$. 
We claim that $A \cup C$ is $S$-adequate, which finishes the proof. 
Note that if $M \in A$ and $M < N$, then $M \cap N \in A$. 
By Lemma 1.19(2), $M \cap \beta_{M,N} = M \cap N \cap \omega_2$. 
Since $M < N$, it follows that $M \cap N \cap \omega_2 \in N$. 
But $M \cap N = g^*[M \cap N \cap \omega_2]$ by the elementarity of $M \cap N$, 
so $M \cap N \in N$ by the elementarity of $N$. 
Hence the assumptions of Proposition 1.29 hold. 
Therefore $A \cup C$ is adequate.

To show that $A \cup C$ is $S$-adequate, let $M \in A$ and $L \in C$. 
Let $\zeta \in R_L(M)$, and we will show that $\zeta \in S$. 
By Lemmas 2.4(1) and 2.5(1), we have that 
$$
\zeta \in R_N(M) \cup R_L(M \cap N) .
$$
Since $A$ and $C$ are $S$-adequate, 
$M$ and $N$ are in $A$, and $L$ and $M \cap N$ are in $C$, 
it follows that $\zeta \in S$.

Let $\theta \in R_M(L)$. 
By Lemmas 2.4(2) and 2.5(2), either $M < N$ and 
$\theta \in R_{M \cap N}(L)$, or there is $\xi \in R_M(N)$ such that 
$\theta = \min((L \cap \omega_2) \setminus \xi)$. 
In the first case, $\theta \in S$ since $C$ is $S$-adequate 
and $M \cap N$ and $L$ are in $C$. 
In the second case, $\xi \in x$, and hence for some $K \in B$, 
$\xi \in R_{K}(N')$. 
By Lemma 4.1, 
$$
\theta \in R_{K}(N') \cup R_{K}(L) \cup R_{N'}(L).
$$
Since $K$, $L$, and $N'$ are in $C$ and $C$ is $S$-adequate, 
it follows that $\theta \in S$.
\end{proof}

\begin{lemma}
Suppose that $M \in \mathcal X_0$ and $Q$ is in $\mathcal Y_0$. 
Let $\beta := Q \cap \omega_2$, and assume that $\cf(\beta) = \omega_1$ 
and $\beta \in M$.  
Then $M \sim M \cap Q$, $R_M(M \cap Q) = \emptyset$, 
and $R_{M \cap Q}(M) = \{ \beta \}$. 
\end{lemma}

\begin{proof}
By the comments after Notation 1.10, $M \cap Q \in \mathcal X_0$. 
By the elementarity of $Q$, $\beta$ is a limit point of $\Lambda$. 
Hence the ordinal 
$$
\beta_0 := \min(\Lambda \setminus \sup(M \cap \beta))
$$
is less than $\beta$. 
By the definition of $\beta_0$, clearly 
$$
\beta_0 \in \Lambda_M \cap \Lambda_{M \cap Q}.
$$
And since $M \cap Q \cap \omega_2 \subseteq \beta_0$, $\beta_0$ is the 
maximal element of $\Lambda_M \cap \Lambda_{M \cap Q}$. 
Therefore $\beta_0 = \beta_{M,M \cap Q}$. 
As $M \cap \beta_0 = M \cap Q \cap \beta_0$, we have that 
$M \sim M \cap Q$.

Since $M \cap Q \cap \omega_2 \subseteq \beta_0 = \beta_{M,M \cap Q}$, 
$R_M(M \cap Q) = \emptyset$. 
As $M \sim M \cap Q$ and 
$$
\beta = \min((M \cap \omega_2) \setminus \beta_0) = 
\min((M \cap \omega_2) \setminus \beta_{M,M \cap Q}),
$$
we have that $\beta \in R_{M \cap Q}(M)$. 
And the fact that $M \cap Q \cap \omega_2 \subseteq \beta_0$ implies that 
$R_{M \cap Q}(M) = \{ \beta \}$.
\end{proof}

\begin{lemma}
Let $Q$ be an elementary substructure of $(\mathcal B,S)$ such that 
$Q$ has size $\omega_1$ and $Q \cap \omega_2 \in S$. 
Let $\beta := Q \cap \omega_2$. 
Let $A_0 \in Q \cap \p$. 
Suppose that $M \in \mathcal X_0$, and 
$A_0$ and $\beta$ are in $M$. 
Then $\beta \in R_{M \cap Q}(M)$, and 
$$
A_0 \cup \{ M \} \cup \{ M \cap Q \}
$$
is a strongly $Q$-generic condition.
\end{lemma}

\begin{proof}
Define $A_1 := A_0 \cup \{ M \}$. 
Then $A_1$ is $S$-adequate and $A_1 \le A_0$. 
Namely, for all $K \in A_0$, $K \in M$ implies that 
$K \cap \beta_{K,M} = K \cap \omega_2 \in M$. 
So $K < M$ for all $K \in A_0$. 
Also $R_K(M)$ and $R_M(K)$ are both empty.

Define $A := A_1 \cup \{ M \cap Q \}$. 
By Proposition 1.33, $A$ is adequate. 
We claim that $A$ is $S$-adequate. 
So let $K$ be in $A_1$. 
If $K \in A_0$, then $K \in M \cap Q$. 
So $R_K(M \cap Q)$ and $R_{M \cap Q}(K)$ are empty. 
Suppose that $K = M$. 
Then by Lemma 4.5, 
$M \sim M \cap Q$, 
$R_{M}(M \cap Q) = \emptyset$, and $R_{M \cap Q}(M) = \{ \beta \}$. 
Since $\beta \in S$, we are done.

Thus we have established that $A$ is $S$-adequate. 
We will show that $A$ is strongly $Q$-generic. 
So let $A_2 \le A$ be given. 

We claim that for all $A_3 \le A_2$, for all 
$K \in A_3$, $A_3 \cup \{ K \cap Q \}$ is $S$-adequate. 
By Proposition 1.33, $A_3 \cup \{ K \cap Q \}$ is adequate. 
To show that it is $S$-adequate, let $N \in A_3$, and we 
will show that $R_N(K \cap Q)$ and $R_{K \cap Q}(N)$ are subsets of $S$. 

By Lemma 2.8(1), $R_N(K \cap Q) \subseteq R_N(K)$. 
Since $K$ and $N$ are in $A_3$ and $A_3$ is $S$-adequate, 
$R_N(K) \subseteq S$. 
Thus $R_N(K \cap Q) \subseteq S$. 
Now suppose that $\zeta \in R_{K \cap Q}(N)$. 
Then by Lemma 2.8(2), either $\zeta \in R_K(N)$, or 
$\zeta = \min((N \cap \omega_2) \setminus \beta)$. 
In the first case, since $K$ and $N$ are in $A_3$, we have that 
$\zeta \in R_K(N) \subseteq S$. 
Assume the second case. 
Since $\beta \in R_{M \cap Q}(M)$ by Lemma 4.5, and 
$M \cap Q$, $M$, and $N$ are in $A_3$, 
Lemma 4.1 implies that $\zeta$ is in $R_{M \cap Q}(M)$, 
$R_{M \cap Q}(N)$, or $R_M(N)$. 
Since $A_3$ is $S$-adequate, $\zeta \in S$, which proves the claim.

By applying the claim finitely many times, we get that the set 
$$
A^* := A_2 \cup \{ K \cap Q : K \in A_2 \}
$$ 
is $S$-adequate. 

Next we claim that for all $K \in A_2$, $K \cap Q$ is in $Q$. 
By Lemma 1.30, $K \cap Q \cap \omega_2$ is in $Q$. 
Since $K$ and $Q$ are elementary in $\mathcal B$, 
$K \cap Q = g^*[K \cap Q \cap \omega_2]$. 
As $Q$ is elementary in $\mathcal B$, 
$K \cap Q = g^*[K \cap Q \cap \omega_2]$ is in $Q$. 
It follows that for all $K \in A^*$, $K \cap Q \in Q$.

Let $B := A^* \cap Q$. 
We will show that for any $C \le B$ in $\p \cap Q$, 
$A^* \cup C$ is $S$-adequate, which finishes the proof. 
So let $C \le B$ be in $\p \cap Q$. 
By the previous claim, for all $K \in A^*$, $K \cap Q \in A^* \cap Q$. 
And $A^* \cap Q \subseteq C \subseteq Q$. 
By Proposition 1.35, $A^* \cup C$ is adequate.

To prove that $A^* \cup C$ is $S$-adequate, 
let $L \in C$ and $N \in A^*$, and we will show that 
$$
R_L(N) \cup R_N(L) \subseteq S.
$$
By Lemma 2.9(2), $R_N(L) \subseteq R_{N \cap Q}(L)$. 
Since $L$ and $N \cap Q$ are in $C$, $R_{N \cap Q}(L) \subseteq S$. 
Hence $R_N(L) \subseteq S$.

Let $\zeta \in R_L(N)$. 
Then by Lemma 2.9(1), either $\zeta \in R_L(N \cap Q)$, or 
$\zeta = \min((N \cap \omega_2) \setminus \beta)$. 
In the first case, since $L$ and $N \cap Q$ are in $C$ and $C$ is $S$-adequate, 
it follows that $\zeta \in S$. 
Assume the second case. 
Then since $\beta \in R_{M \cap Q}(M)$, and $M$, $M \cap Q$, and $N$ 
are in $A^*$, 
by Lemma 4.1 we have that 
$\zeta$ is in 
$R_{M \cap Q}(M)$, $R_{M \cap Q}(N)$, or $R_M(N)$. 
Since $A^*$ is $S$-adequate, it follows that $\zeta \in S$.
\end{proof}

\begin{corollary}
The forcing poset $\p$ is $\omega_2$-strongly proper on a stationary 
set. 
Therefore $\p$ preserves $\omega_2$ and has the 
$\omega_2$-covering property.
\end{corollary}

\begin{proof}
Immediate from Lemma 4.6.
\end{proof}

\begin{proposition}
The forcing poset $\p$ adds a club subset of $S \cup \cof(\omega)$.
\end{proposition}

\begin{proof}
Let $\dot C$ be a $\p$-name for the set 
$$
\bigcup \{ R_M(N) : \exists A \in \dot G, \ M, N \in A \}.
$$
It follows easily from Lemma 4.6 that $\p$ forces that 
$\dot C$ is cofinal in $\omega_2$.

We claim that $\p$ forces that 
$$
\lim(\dot C) \subseteq S \cup \cof(\omega),
$$
which completes the proof. 
Let $\beta < \omega_2$, and assume that $A$ is a condition which forces 
that $\beta$ is a limit point of $\dot C$. 
We will prove that $\beta$ is in $S \cup \cof(\omega)$. 
If $\beta$ has cofinality $\omega$, then we are done. 
So assume that $\cf(\beta) = \omega_1$. 
We will show that $\beta \in S$.

Fix $N \in \mathcal X_0$ such that $A$ and $\beta$ are in $N$. 
Then $A \cup \{ N \}$ is an $S$-adequate set, and thus is in $\p$. 
Since $A \cup \{ N \} \le A$, $A \cup \{ N \}$ forces that 
$\beta$ is a limit point of $\dot C$ with uncountable cofinality. 
Hence we can fix $B \le A \cup \{ N \}$, $K$ and $M$ in $B$, and 
$\gamma \in R_K(M)$ such that 
$$
\sup(N \cap \beta) < \gamma < \beta.
$$
Since $\beta \in N$, we have that 
$\beta = \min((N \cap \omega_2) \setminus \gamma)$. 
By Lemma 4.1, 
$$
\beta \in R_K(M) \cup R_K(N) \cup R_M(N).
$$
As $B$ is $S$-adequate, it follows that $\beta \in S$.
\end{proof}

\bigskip

We remark that it is not necessary to assume that 
$\kappa$ is $\omega_2$. 
If $\kappa > \omega_2$, 
we can fix any stationary set $S \subseteq \kappa \cap \cof(> \! \omega)$, 
and then the forcing poset 
$\p$ defined above will add a club subset of $S \cup \cof(\omega)$, 
and collapse $\kappa$ to become $\omega_2$ by Proposition 3.12.

\bigskip

\addcontentsline{toc}{section}{5. $\vec S$-obedient 
side conditions}

\textbf{\S 5. $\vec S$-obedient side conditions}

\stepcounter{section}

\bigskip

We now generalize the idea of an $S$-adequate set to the case when we 
have a sequence $\vec S$ of sets, instead of a single set $S$. 
For the remainder of this section fix a sequence 
$\vec S = \langle S_\eta : \eta < \lambda \rangle$, where each $S_\eta$ 
is a subset of $\kappa \cap \cof(>\!\omega)$.

\begin{definition}
A set $P \in \mathcal Y_0$ is \emph{$\vec S$-strong} if for all 
$\tau \in P \cap \kappa^+$, $P \cap \kappa \in S_\tau$.
\end{definition}

Note that if $P$ is $\vec S$-strong, then $\cf(P \cap \kappa) > \omega$, 
since $P \cap \kappa \in S_0 \subseteq \kappa \cap \cof(>\!\omega)$.

For the next two definitions, we fix a class $\mathcal Y \subseteq \mathcal Y_0$. 
The definitions of $\vec S$-adequate and $\vec S$-obedient are made relative 
to the class $\mathcal Y$.

\begin{definition}
Let $A$ be an adequate set. 
We say that $A$ is \emph{$\vec S$-adequate} 
if for all $M$ and $N$ in 
$A$ and $\zeta \in R_M(N)$:
\begin{enumerate}
\item for all $\tau \in M \cap N \cap \kappa^+$, $\zeta \in S_\tau$;
\item if $P \in M \cap \mathcal Y$ is $\vec S$-strong and 
$\sup(N \cap \zeta) < P \cap \kappa < \zeta$, then for all 
$\tau \in N \cap P \cap \kappa^+$, $\zeta \in S_\tau$.
\end{enumerate}
\end{definition}

\begin{definition}
A pair $(A,B)$ is an \emph{$\vec S$-obedient side condition} 
if:
\begin{enumerate}
\item $A$ is an $\vec S$-adequate set;
\item $B$ is a finite set of $\vec S$-strong models in $\mathcal Y_0$;
\item for all $M \in A$ and $P \in B$, if 
$\zeta = \min((M \cap \kappa) \setminus (P \cap \kappa))$, 
then for all $\tau \in P \cap M \cap \kappa^+$, $\zeta \in S_\tau$.
\end{enumerate}
\end{definition}

The next two lemmas show that we can add certain models to an 
$\vec S$-obedient side condition and preserve 
$\vec S$-obedience.

\begin{lemma}
Let $(A,B)$ be an $\vec S$-obedient side condition.
\begin{enumerate}
\item If $N \in \mathcal X_0$ and $(A,B) \in N$, then $(\{N\},\emptyset)$ and 
$(A \cup \{ N \},B)$ are $\vec S$-obedient side conditions.
\item If $P \in \mathcal Y_0$ is $\vec S$-strong 
and $(A,B) \in P$, then $(\emptyset,\{P\})$ and  
$(A,B \cup \{ P \})$ are $\vec S$-obedient side conditions.
\end{enumerate}
\end{lemma}

\begin{proof}
(2) is trivial. 
(1) The fact that $(\{N\},\emptyset)$ is an $\vec S$-obedient side condition 
is easy. 
The set $A \cup \{N\}$ is $\vec S$-adequate because for all $M \in A$, 
$M \cap \beta_{M,N} = M \cap \kappa$ is in $N$, and 
$R_M(N)$ and $R_N(M)$ are empty. 
If $P \in B$, then $\min((N \cap \kappa) \setminus (P \cap \kappa)) = P \cap \kappa$. 
And if $\tau \in P \cap N \cap \kappa^+$, then 
$P \cap \kappa \in S_\tau$ since $P$ is $\vec S$-strong.
\end{proof}

\begin{lemma}
Let $(A,B)$ be an $\vec S$-obedient side condition.
\begin{enumerate}
\item Let $M$ and $N$ be in $A$, and suppose that $M < N$. 
Then $(A \cup \{ M \cap N \},B)$ is an $\vec S$-obedient side condition.
\item Let $M \in A$ and $P \in B$. 
Then $(A \cup \{ M \cap P \},B)$ is an $\vec S$-obedient side condition.
\item Suppose that $P$ and $Q$ are in $B$ and 
$P \cap \kappa < Q \cap \kappa$. 
Then $(A,B \cup \{ P \cap Q \})$ is an $\vec S$-obedient side condition.
\end{enumerate}
\end{lemma}

\begin{proof}
(1) The set $A \cup \{ M \cap N \}$ is adequate by Proposition 1.25. 
To show that $A \cup \{ M \cap N \}$ is $\vec S$-adequate, 
it suffices to show that 
for all $K \in A$, $\{ K, M \cap N \}$ is $\vec S$-adequate. 
So let $K \in A$ be given.

Let $\zeta \in R_K(M \cap N)$. 
Then $\zeta \in R_K(M)$ by Lemma 2.3. 
Let $\tau \in K \cap (M \cap N) \cap \kappa^+$, 
and we will show that $\zeta \in S_\tau$. 
Then $\tau \in K \cap M$, which implies that $\zeta \in S_\tau$ 
since $\zeta \in R_K(M)$.

Suppose that $P \in K \cap \mathcal Y$ is $\vec S$-strong and 
$\sup(M \cap N \cap \zeta) < P \cap \kappa < \zeta$. 
Let $\tau \in (M \cap N) \cap P \cap \kappa^+$, and we will show that 
$\zeta \in S_\tau$. 
Since $\zeta \in M \cap N \cap \kappa$, $\zeta < \beta_{M,N}$. 
And since $M \cap N \cap \kappa = M \cap \beta_{M,N}$ is an initial 
segment of $M \cap \kappa$, 
$\sup(M \cap N \cap \zeta) = \sup(M \cap \zeta)$. 
So $\sup(M \cap \zeta) < P \cap \kappa < \zeta$. 
Since $\zeta \in R_K(M)$, it follows that $\zeta \in S_\tau$.

Let $\zeta \in R_{M \cap N}(K)$. 
Then by Lemma 2.3, either (i) $\zeta \in R_M(K)$ or (ii) $\zeta \in R_N(K)$. 
Consider $\tau \in K \cap (M \cap N) \cap \kappa^+$, and we will show that 
$\zeta \in S_\tau$. 
Then $\tau \in K \cap M$, so in case (i), $\zeta \in S_\tau$. 
Also $\tau \in K \cap N$, so in case (ii), $\zeta \in S_\tau$. 
Suppose that $P \in (M \cap N) \cap \mathcal Y$ is $\vec S$-strong and 
$\sup(K \cap \zeta) < P \cap \kappa < \zeta$. 
Let $\tau \in K \cap P \cap \kappa^+$, and we will show that $\zeta \in S_\tau$. 
Then $P \in M \cap \mathcal Y$, so in case (i), $\zeta \in S_\tau$. 
And $P \in N \cap \mathcal Y$, so in case (ii), $\zeta \in S_\tau$. 
This completes the proof that $A \cup \{ M \cap N \}$ is $\vec S$-adequate.

Let $Q \in B$, and suppose that 
$\xi = \min((M \cap N \cap \kappa) \setminus (Q \cap \kappa))$. 
Since $M < N$, $M \cap N \cap \kappa = M \cap \beta_{M,N}$, which is 
an initial segment of $M \cap \kappa$. 
Hence $\xi = \min((M \cap \kappa) \setminus (Q \cap \kappa))$.  
Let $\tau \in (M \cap N) \cap Q \cap \kappa^+$, 
and we will show that $\xi \in S_\tau$. 
Then $\tau \in M \cap Q$, so $\xi \in S_\tau$ since $(A,B)$ is $\vec S$-obedient.

\bigskip

(2) Since $P$ is $\vec S$-strong, $\cf(P \cap \kappa) > \omega$. 
So clearly $\sup(M \cap P \cap \kappa) < P \cap \kappa$. 
It follows that $A \cup \{ M \cap P \}$ is adequate by Proposition 1.33.

To show that $A \cup \{ M \cap P \}$ is $\vec S$-adequate, let 
$K \in A$ be given, and we will show that $\{ K, M \cap P \}$ 
is $\vec S$-adequate.

Let $\zeta \in R_K(M \cap P)$. 
Then $\zeta \in R_K(M)$ by Lemma 2.8. 
Let $\tau \in K \cap (M \cap P) \cap \kappa^+$, 
and we will show that $\zeta \in S_\tau$. 
Then $\tau \in K \cap M$ implies that $\zeta \in S_\tau$.

Suppose that $Q \in K \cap \mathcal Y$ is $\vec S$-strong and 
$\sup(M \cap P \cap \zeta) < Q \cap \kappa < \zeta$. 
Let $\tau \in Q \cap (M \cap P) \cap \kappa^+$, 
and we will show that $\zeta \in S_\tau$. 
Since $\zeta \in P \cap \kappa$ and $P \cap \kappa$ is an ordinal, 
$\sup(M \cap P \cap \zeta) = \sup(M \cap \zeta)$. 
So $\sup(M \cap \zeta) < Q \cap \kappa < \zeta$. 
Since $\zeta \in R_K(M)$, $Q \in K \cap \mathcal Y$, and 
$\tau \in M \cap Q$, it follows that $\zeta \in S_\tau$.

Let $\zeta \in R_{M \cap P}(K)$. 
Then by Lemma 2.8, either (i) $\zeta \in R_M(K)$ or 
(ii) $\zeta = \min((K \cap \kappa) \setminus (P \cap \kappa))$.

Let $\tau \in (M \cap P) \cap K \cap \kappa^+$, 
and we will show that $\zeta \in S_\tau$. 
In case (i), $\tau \in M \cap K$ implies that $\zeta \in S_\tau$ since 
$A$ is $\vec S$-adequate. 
In case (ii), $\tau \in K \cap P$ implies that $\zeta \in S_\tau$ since 
$(A,B)$ is $\vec S$-obedient.

Suppose that $Q \in (M \cap P) \cap \mathcal Y$ is $\vec S$-strong and 
$\sup(K \cap \zeta) < Q \cap \kappa < \zeta$. 
Let $\tau \in K \cap Q \cap \kappa^+$, and we will show that $\zeta \in S_\tau$. 
In case (i), $Q \in M \cap \mathcal Y$ implies that $\zeta \in S_\tau$ 
since $A$ is $\vec S$-adequate. 
In case (ii), since $\tau \in Q$ and $Q \in P$, $\tau \in P$. 
So $\zeta \in S_\tau$ since $(A,B)$ is $\vec S$-obedient. 
This completes the proof that $A \cup \{ M \cap P \}$ is $\vec S$-adequate.

Let $Q \in B$, and suppose that 
$\zeta = \min((M \cap P \cap \kappa) \setminus (Q \cap \kappa))$. 
Let $\tau \in (M \cap P) \cap Q \cap \kappa^+$, 
and we will show that $\zeta \in S_\tau$. 
Since $P \cap \kappa \in \kappa$, $M \cap P \cap \kappa$ is an 
initial segment of $M \cap \kappa$. 
Hence $\zeta = \min((M \cap \kappa) \setminus (Q \cap \kappa))$. 
Since $\tau \in M \cap Q$, it follows that $\zeta \in S_\tau$ since $(A,B)$ is $\vec S$-obedient.

\bigskip

(3) Note that since $P \cap \kappa < Q \cap \kappa$, 
$P \cap Q \cap \kappa = P \cap \kappa$.

To show that $P \cap Q$ is $\vec S$-strong, let $\tau \in P \cap Q \cap \kappa^+$. 
Then $\tau \in P$. 
Since $P$ is $\vec S$-strong, 
$P \cap Q \cap \kappa = P \cap \kappa \in S_\tau$. 

Let $M \in A$, and suppose that 
$\zeta = \min((M \cap \kappa) \setminus (P \cap Q \cap \kappa))$. 
Then $\zeta = \min((M \cap \kappa) \setminus (P \cap \kappa))$. 
Let $\tau \in M \cap (P \cap Q) \cap \kappa^+$. 
Then $\tau \in M \cap P$, so $\zeta \in S_\tau$.
\end{proof}

We conclude the section with an easy lemma which will be used frequently 
for checking that certain models are $\vec S$-strong.

\begin{lemma}
Suppose that $N \in \mathcal X_0 \cup \mathcal Y_0$, 
$Q \in N \cap \mathcal Y_0$, and $P \in \mathcal Y_0$. 
Suppose that $P$ is $\vec S$-strong and 
$N \prec (H(\lambda),\in,\mathcal Y_0,\vec S)$. 
Assume that $Q \cap N \cap \kappa^+ \subseteq P$ and 
$P \cap \kappa = Q \cap \kappa$. 
Then $Q$ is $\vec S$-strong.
\end{lemma}

\begin{proof}
Since $Q \in N$, it suffices to show that $N$ models that $Q$ 
is $\vec S$-strong. 
So let $\tau \in Q \cap N \cap \kappa^+$. 
Since $Q \cap N \cap \kappa^+ \subseteq P$, $\tau \in P$. 
As $P$ is $\vec S$-strong, 
$Q \cap \kappa = P \cap \kappa \in S_\tau$.
\end{proof}

\bigskip

\addcontentsline{toc}{section}{6. The approximation property and factorization}

\textbf{\S 6. The approximation property and factorization}

\stepcounter{section}

\bigskip

We briefly discuss the approximation property, and state the theorem 
on factoring a generic extension which we will use in the proof of 
Mitchell's theorem in Part III. 

Let $(W_1,W_2)$ be a pair of transitive class models of ZFC such that 
$W_1 \subseteq W_2$. 
We say that the pair $(W_1,W_2)$ 
satisfies the \emph{$\omega_1$-approximation property} 
if, whenever $X \in W_2$ is a set of ordinals such that 
$a \cap X \in W_1$ whenever $a \in W_1$ is countable in $W_1$, then 
the set $X$ itself is in $W_1$.

The approximation property is due to Hamkins \cite{hamkins}, and is similar 
to properties studied in Mitchell's construction of a model with no 
Aronszajn trees on $\omega_2$ \cite{mitchellold}. 
It plays a crucial role in the original proof of Mitchell's theorem 
on the approachability ideal, as well 
as in the proof presented in Part III.

We will use the following easy consequence of the approximation property.

\begin{lemma}
Suppose that $(W_1,W_2)$ satisfies the $\omega_1$-approximation property. 
Assume that $c$ is a set of ordinals of order type $\omega_1$ in $W_2$ 
such that for all $\beta < \sup(c)$, $c \cap \beta \in W_1$. 
Then $c \in W_1$.
\end{lemma}

\begin{proof}
To show that $c \in W_1$, it suffices to show that for any set 
$a \in W_1$ which is countable in $W_1$, $a \cap c \in W_1$. 
So let $a \in W_1$ be countable in $W_1$. 
Then $a$ is countable in $W_2$. 
Since $c$ has order type $\omega_1$, 
there is $\beta < \sup(c)$ such that 
$a \cap c \subseteq c \cap \beta$. 
By the assumption on $c$, $c \cap \beta \in W_1$.  
Since $c \cap \beta$ and $a$ are in $W_1$, 
$a \cap c = a \cap (c \cap \beta)$ is in $W_1$.
\end{proof}

\bigskip

In the original proof of Mitchell's theorem, being able to factor a generic 
extension in a way which satisfies the approximation property 
relied on what was called \emph{tidy} 
strongly generic conditions (see Lemma 2.22 of \cite{mitchell}). 
However, the strongly generic conditions used in the present paper 
are not tidy. 
Therefore we need a different factorization theorem which is applicable 
in the present context; such a theorem was provided by Cox-Krueger \cite{jk26}.

\bigskip

Let us recall the property $*$ introduced in \cite{jk26}.

\begin{definition}
Let $\p_1$ be a suborder of a forcing poset $\p_2$, where 
$\p_2$ has greatest lower bounds. 
We say that $\p_1$ satisfies property $*(\p_1,\p_2)$ if for all $p \in \p_1$ 
and $q, r \in \p_2$, if $p$, $q$, and $r$ are pairwise compatible 
in $\p_2$, then $p$ is compatible in $\p_2$ with $q \land r$.
\end{definition}

Note that if a forcing poset $\q$ satisfies property $*(\q,\q)$, then for any 
suborder $\p$ of $\q$, $*(\p,\q)$.

\bigskip

\begin{notation}
Let $\q$ be a forcing poset. 
If $q \in \q$ and $K$ is a subset of $\q$, 
let $(\q / q) / K$ denote the forcing poset consisting of conditions 
$s \in \q$ such that $s \le q$, and $s$ is compatible in $\q$ with 
all conditions in $K$.
\end{notation}

The following result appears as Theorem 4.3 in \cite{jk26}.

\begin{thm}[Factorization theorem]
Let $\q$ be a forcing poset with greatest lower bounds satisfying 
$*(\q,\q)$, $\chi$ a regular 
cardinal with $\lambda_\q \le \chi$, and 
$N \prec (H(\chi),\in,\q)$. 
Suppose that there are stationarily many models in $P_{\omega_1}(H(\chi))$ 
which have universal strongly generic conditions. 
Assume that $q$ is a universal strongly $N$-generic condition.

Then for any $V$-generic filter $G$ on $\q$ which contains $q$, 
$V[G] = V[G \cap N][H]$, where 
$G \cap N$ is a $V$-generic filter on $\q \cap N$, $H$ is a 
$V[G \cap N]$-generic filter on $(\q / q) / (G \cap N)$, and the pair 
$(V[G \cap N],V[G])$ satisfies the $\omega_1$-approximation property.
\end{thm}

This theorem will be used in the final argument of the proof of Mitchell's theorem 
in Section 16. 
It is interesting to note that not all intermediate extensions of a strongly 
proper forcing extension satisfy the 
$\omega_1$-approximation property; see Section 5 of \cite{jk26} for a 
counterexample.

\bigskip

\part{Advanced side condition methods}

\bigskip

\addcontentsline{toc}{section}{7. Mitchell's use of $\Box_\kappa$}

\textbf{\S 7. Mitchell's use of $\Box_\kappa$\footnote{Almost all of the 
arguments in this section and the next are due to Mitchell, but adopted 
to the present context.}}

\stepcounter{section}

\bigskip

For the remainder of the paper we will assume $\Box_\kappa$ and $2^\kappa = \kappa^+$. 
Also we let the cardinal $\lambda$ from Part I equal $\kappa^+$. 
Since $2^\kappa = \kappa^+$, $H(\kappa^+)$ has size $\kappa^+$.

\begin{notation}
Let $f^*$ denote a bijection from $\kappa^+$ to $H(\kappa^+)$.
\end{notation}

\begin{notation}
Fix a sequence $\vec C = \langle C_\alpha : \alpha < \kappa^+, \ 
\alpha \ \textrm{limit} \rangle$ satisfying that for all limit 
$\alpha < \kappa^+$:
\begin{enumerate}
\item $C_\alpha$ is a club subset of $\alpha$ with 
$\ot(C_\alpha) \le \kappa$; in particular, if $\cf(\alpha) < \kappa$ then 
$\ot(C_\alpha) < \kappa$;
\item if $\beta \in \lim(C_\alpha)$, then 
$C_\beta = C_\alpha \cap \beta$;
\item if $\alpha$ is a limit of limit ordinals, then every ordinal in $C_\alpha$ 
is a limit ordinal;
\item if $\alpha = \alpha_0 + \omega$ for a limit ordinal $\alpha_0$, 
then $\alpha_0 \in \lim(C_\alpha)$, and hence 
$C_{\alpha_0} = C_{\alpha} \cap \alpha_0$.
\end{enumerate}
\end{notation}

Properties (1) and (2) embody the standard definition of a square sequence. 
It is easy to modify a square sequence to also satisfy properties (3) and (4). 
For example, start by replacing each ordinal in $C_\alpha$ with 
the greatest limit ordinal less than or equal to it. 
The details are left to the reader.

\begin{notation}
For each limit ordinal $\alpha < \kappa^+$ and $\beta < \ot(C_\alpha)$, 
let $c_{\alpha,\beta}$ denote the $\beta$-th member of $C_\alpha$, that is, 
the unique $\gamma$ in $C_\alpha$ such that 
$\ot(C_\alpha \cap \gamma) = \beta$.
\end{notation}

\begin{notation}
Fix a sequence 
$\vec A = \langle A_{\eta,\beta} : \eta < \kappa^+, \ \beta < \kappa \rangle$ 
satisfying the following properties:
\begin{enumerate}

\item for each $\eta < \kappa^+$, $\{ A_{\eta,\beta} : \beta < \kappa \}$ 
is an increasing and continuous sequence of sets with union equal to $\eta$;

\item $A_{\eta+1,\beta} = A_{\eta,\beta} \cup \{ \eta \}$;

\item for all $\eta < \kappa^+$ and $\beta < \kappa$, 
$|A_{\eta,\beta}| \le |\beta| \cdot \omega$;

\item if $\xi \in \lim(C_\eta) \cup A_{\eta,\beta} \cup \lim(A_{\eta,\beta})$, 
then $A_{\xi,\beta} = A_{\eta,\beta} \cap \xi$;

\item there exists a function $c^* : \kappa \to \kappa$ such that for all 
$\eta < \kappa^+$ and $\beta < \kappa$, $\ot(A_{\eta,\beta}) < c^*(\beta)$;

\item if $\beta < \ot(C_\eta)$, then 
$A_{\eta,\beta} \subseteq c_{\eta,\beta}$ 
and $\lim(C_\eta) \cap c_{\eta,\beta} \subseteq A_{\eta,\beta}$;

\item if $\gamma \in \eta \setminus C_\eta$, then 
$\gamma \in A_{\eta,\beta}$ iff 
$$
\gamma \in A_{\min(C_\eta \setminus \gamma),\beta} \ \textrm{and} \ 
\min(C_\eta \setminus \gamma) \in A_{\eta,\beta};
$$

\item if $\xi \in \lim(A_{\eta,\beta}) \cap \eta$ and $\ot(C_\xi) < \beta$, 
then $\xi \in A_{\eta,\beta}$.
\end{enumerate}
\end{notation}

Properties (1), (2), and (3) describe a typical kind of 
filtration of each ordinal 
$\eta < \kappa^+$. 
The coherence property (4) is one of the most often used facts in the paper. 
It gives sufficient conditions for coherence to hold between $A_{\xi,\beta}$ and 
$A_{\eta,\beta}$, where $\xi < \eta$. 
If $\xi$ is a limit point of $C_\eta$, then 
$$
\forall \beta < \kappa \ A_{\xi,\beta} = A_{\eta,\beta} \cap \xi.
$$
And if $\beta < \kappa$ and $\xi$ is either in $A_{\eta,\beta}$, or a limit point 
of $A_{\eta,\beta}$, then 
$$
A_{\xi,\beta} = A_{\eta,\beta} \cap \xi.
$$
We recommend that the reader memorize this important fact before proceeding.

Property (5) follows immediately from property (3) in the case when 
$\kappa$ is weakly inaccessible, by letting $c^*(\beta) = \beta^+$. 
This property is only used in one lemma in the paper, namely Lemma 8.6. 
Likewise, properties (6), (7), and (8) are technical facts about $\vec A$ which 
are only used in Lemmas 8.10 and 8.11, and in several places in Section 12. 
There is no harm in the reader forgetting about properties (5)--(8) for now, 
and just looking back at them later in the rare places that they are used.

\begin{thm}[Mitchell \cite{mitchell}]
Assume that $\kappa$ is weakly inaccessible and 
$\vec C$ is a sequence as in Notation 7.2. 
Then there exists a sequence $\vec A$ as described in Notation 7.4.
\end{thm}

Mitchell constructs the sequence $\vec A$ using the square sequence $\vec C$ 
in a careful way. 
The only place where the weak inaccessibility of $\kappa$ 
is used is to derive property (5) 
from property (3), as mentioned above. 
If $\kappa$ is weakly inaccessible in an inner model $W$ which satisfies 
$\Box_\kappa$, and $(\kappa^+)^W = (\kappa^+)^V$, then the sequence 
$\vec A$ constructed in $W$ still satisfies properties (1)--(8) in $V$ 
by upwards absoluteness. 
For example, if $V$ is obtained from $W$ by collapsing $\kappa$ to become 
$\omega_2$ while preserving $\kappa^+$, then there is a 
sequence $\vec A$ as above in $V$.

The construction of $\vec A$ appears in 
\cite[Section 3.1]{mitchell}. 
We do not repeat it here because it is technical and not helpful for 
understanding the other material in our paper.

\begin{notation}
Let $\mathcal A$ denote some expansion of the structure 
$$
(H(\kappa^+),\in,\unlhd,\kappa,T^*,\pi^*,C^*,\Lambda,\mathcal Y_0,
f^*,\vec C,\vec A,c^*).
$$
\end{notation}

Note that $\mathcal A$ is an expansion of the structure described in 
Notations 1.9 and 1.10. 
Thus any elementary substructure of $\mathcal A$ is also an elementary 
substructure of that structure.

\begin{notation}
Let $\mathcal X$ denote the set of $M$ in $P_{\omega_1}(H(\kappa^+))$ such 
that $M \cap \kappa \in T^*$, $M \prec \mathcal A$, 
and $\lim(C_{\sup(M)}) \cap M$ is cofinal in $\sup(M)$.
\end{notation}

\begin{notation}
Let $\mathcal Y$ denote the set of $P$ in $P_{\kappa}(H(\kappa^+))$ such that 
$P \cap \kappa \in \kappa$, $P \prec \mathcal A$, 
and $\lim(C_{\sup(P)}) \cap P$ is cofinal in $\sup(P)$.
\end{notation}

Note that $\mathcal X \subseteq \mathcal X_0$ and 
$\mathcal Y \subseteq \mathcal Y_0$, where $\mathcal X_0$ and $\mathcal Y_0$ 
were defined in Notations 1.9 and 
1.10.\footnote{The requirement that $\lim(C_{\sup(X)}) \cap X$ 
is cofinal in $\sup(X)$ 
appears in Mitchell's definition of a \emph{model} (\cite[Definition 3.14]{mitchell}). 
We do not, however, assume that $\ot(C_{\sup(X)}) \notin X$, as in his definition.}

Observe that by elementarity, for any $M \in \mathcal X$ and $P \in \mathcal Y$, 
$M = f^*[M \cap \kappa^+]$ and $P = f^*[P \cap \kappa^+]$. 
In particular, if $M$ and $N$ are in $\mathcal X \cup \mathcal Y$ and 
$M \cap \kappa^+ \in N$, then by elementarity, $M \in N$.

As a result of the presence of the well-ordering $\unlhd$, 
the structure $\mathcal A$ 
described in Notation 7.6 has definable Skolem functions. 
Let $\langle \tau_n : n < \omega \rangle$ be a complete list of 
definable Skolem terms for $\mathcal A$. 
For any set $a \subseteq H(\kappa^+)$, let $Sk(a)$ denote the 
closure of $a$ under the Skolem terms.

For $n < \omega$ and $m$ equal to the arity of $\tau_n$, 
we define a partial function $\tau_n' : (\kappa^+)^m \to \kappa^+$ by letting 
$\tau_n'(\alpha_0,\ldots,\alpha_{m-1}) = 
\tau_n(\alpha_0,\ldots,\alpha_{m-1})$, provided that this is an ordinal, 
and otherwise is undefined. 
Note that $\tau_n'$ is also definable in $\mathcal A$.

\begin{notation}
Let $H^* : (\kappa^+)^{<\omega} \to \kappa^+$ be a function such that 
any elementary substructure of $\mathcal A$ is closed under $H^*$, and 
whenever $a \subseteq \kappa^+$ is closed under $H^*$, 
then $Sk(a) \cap \kappa^+ = a$. 
In addition, $a$ is closed under $H^*$ iff 
$a$ is closed under $\tau_n'$ for all 
$n < \omega$.
\end{notation}

The existence of such a function $H^*$ is proved by standard arguments. 
Note that if $a$ is a set of ordinals 
closed under $H^*$, then $Sk(a) = f^*[a]$. 
In particular, if $M \in \mathcal X \cup \mathcal Y$ and $a \in M$ is a 
set of ordinals which is closed 
under $H^*$, then by elementarity, $Sk(a) \in M$.

The next simple lemma will prove very useful throughout the paper.

\begin{lemma}
Suppose that $N \in \mathcal X \cup \mathcal Y$, 
$a$ is a set of ordinals in $N$, 
and for some set $b$ which is closed under $H^*$, 
$N \cap a = N \cap b$. 
Then $a$ is closed under $H^*$.
\end{lemma}

\begin{proof}
It suffices to show that $a$ is closed under $\tau_n'$ for all $n < \omega$. 
Fix $n < \omega$, and let $k$ be the arity of $\tau_n'$. 
Since $a \in N$ and $\tau_n'$ is definable in $\mathcal A$, 
it suffices to show that $N$ models that $a$ is 
closed under $\tau_n'$. 
Let $\alpha_0,\ldots,\alpha_{k-1} \in N \cap a$. 
Then $\alpha_0,\ldots,\alpha_{k-1} \in N \cap b$. 
Since $N$ and $b$ are both closed under $H^*$, they are closed 
under $\tau_n'$. 
So $\tau_n'(\alpha_0,\ldots,\alpha_{k-1}) \in N \cap b$. 
Since $N \cap b \subseteq a$, 
$\tau_n'(\alpha_0,\ldots,\alpha_{k-1}) \in a$.
\end{proof}

In the remainder of this section, we will provide a thorough analysis of 
the models in $\mathcal X$ and $\mathcal Y$. 

The following notation will be useful.

\begin{notation}
Let $N \subseteq H(\kappa^+)$ be a set and $\gamma \in \kappa^+ \cap \sup(N)$. 
Let $\gamma_N$ denote the ordinal $\min((N \cap \kappa^+) \setminus \gamma)$.
\end{notation}

Recall that if $a$ is a set of ordinals, then 
$\cl(a)$ denotes the set $a \cup \lim(a)$.

\begin{lemma}
Let $N \in \mathcal X \cup \mathcal Y$ and $\eta \in \cl(N)$, and 
suppose that $\eta < \sup(N)$. 
Then either $\eta \in N$ or $\eta \in \lim(C_{\eta_N})$. 
Hence:
\begin{enumerate}
\item $C_\eta = C_{\eta_N} \cap \eta$;
\item $A_{\eta,\xi} = A_{\eta_N,\xi} \cap \eta$ for all 
$\xi < \kappa$;
\item $N \cap A_{\eta_N,\xi} = N \cap A_{\eta,\xi}$ for all 
$\xi < \kappa$.
\end{enumerate}
\end{lemma}

\begin{proof}
If $\eta \in N$, then $\eta_N = \eta$, and (1), (2), and (3) are trivial.

Suppose that $\eta < \eta_N$, 
and we will show that $\eta \in \lim(C_{\eta_N})$. 
Let $\gamma < \eta$. 
Since $\eta \in \cl(N) \setminus N$, $\eta \in \lim(N)$. 
So we can fix $\sigma \in (N \cap \eta) \setminus \gamma$. 
Now $\eta_N \in N$ and $\sigma \in N \cap \eta_N$, so by 
elementarity there is $\delta \in C_{\eta_N} \cap N$ larger than $\sigma$. 
Then $\delta \in N \cap \eta_N \subseteq \eta$. 
So $\gamma < \delta < \eta$ and $\delta \in C_{\eta_N}$. 
This proves that $\eta \in \lim(C_{\eta_N})$.

(1) follows from the definition of a square sequence, and (2) follows 
from Notation 7.4(4). 
For (3), since $N \cap \eta_N = N \cap \eta$, it follows that 
for all $\xi < \kappa$,
$$
N \cap A_{\eta_N,\xi} = N \cap A_{\eta_N,\xi} \cap \eta = 
N \cap A_{\eta,\xi}.
$$
\end{proof}

\begin{lemma}
Let $N \in \mathcal X \cup \mathcal Y$, 
and suppose that $\eta \in \cl(N) \setminus N$. 
Then $\lim(C_\eta) \cap N$ is cofinal in $\eta$.
\end{lemma}

\begin{proof}
Note that $\eta \in \lim(N)$. 
If $\eta = \sup(N)$, 
then the statement of the lemma 
follows from the definitions of $\mathcal X$ and $\mathcal Y$. 
Otherwise by Lemma 7.12, 
$C_{\eta} = C_{\eta_N} \cap \eta$. 
Let $\gamma < \eta$. 
Since $\eta \in \lim(N)$, we can fix $\sigma \in N \cap \eta$ 
larger than $\gamma$. 
As $\eta < \eta_N$, $\eta_N$ has uncountable cofinality. 
So certainly $\lim(C_{\eta_N})$ is cofinal in $\eta_N$. 
By elementarity, we can find 
$\delta \in \lim(C_{\eta_N}) \cap \eta_N \cap N$ which is 
larger than $\sigma$. 
Then $\delta < \eta$. 
Since $C_\eta = C_{\eta_N} \cap \eta$, it follows that 
$\delta \in \lim(C_\eta)$. 
Thus $\gamma \le \delta$ and $\delta \in \lim(C_\eta) \cap N$.
\end{proof}

The next lemma is standard.

\begin{lemma}
Suppose that $P \in P_{\kappa}(H(\kappa^+))$, 
$P \prec (H(\kappa^+),\in)$, $P \cap \kappa \in \kappa$, and 
$\cf(P \cap \kappa) > \omega$. 
Assume that $\gamma$ is a limit point of 
$P \cap \kappa^+$ below $\sup(P)$, and 
$\cf(\gamma) < \cf(P \cap \kappa)$. 
Then $\gamma \in P$.
\end{lemma}

\begin{proof}
Suppose for a contradiction that $\gamma \notin P$. 
Then $\gamma_P$ is in $P$ and $\gamma < \gamma_P$. 
By elementarity, we can fix an increasing and cofinal function 
$f : \cf(\gamma_P) \to \gamma_P$ which is in $P$. 
Since $\gamma_P < \kappa^+$, either $\cf(\gamma_P) < \kappa$ 
or $\cf(\gamma_P) = \kappa$. 
In the first case, $\cf(\gamma_P) \in P \cap \kappa \in \kappa$ implies that 
$\cf(\gamma_P) \subseteq P$. 
By elementarity, $f[\cf(\gamma_P)] \subseteq P \cap \gamma_P \subseteq \gamma$, 
which is impossible 
since $f[\cf(\gamma_P)]$ is cofinal in $\gamma_P$ and $\gamma < \gamma_P$. 
Therefore $\cf(\gamma_P) = \kappa$. 
By elementarity, $f \restriction P \cap \kappa$ is cofinal in $P \cap \gamma_P$, 
and hence is cofinal in $\gamma$. 
But then $\gamma$ has cofinality equal to $\cf(P \cap \kappa)$, which 
contradicts our assumption on $\gamma$.
\end{proof}

\begin{lemma}
Let $P \in P_{\kappa}(H(\kappa^+))$ with $P \cap \kappa \in \kappa$ 
and $P \prec \mathcal A$. 
If $\cf(P \cap \kappa) > \omega$ and 
$\cf(\sup(P)) > \omega$, then $P \in \mathcal Y$.
\end{lemma}

\begin{proof}
Let $\sigma := \sup(P)$. 
By the definition of $\mathcal Y$, 
it suffices to show that $\lim(C_\sigma) \cap P$ is cofinal in $\sigma$. 
So let $\xi < \sigma$. 
Since $\sup(P) = \sigma$ has uncountable cofinality, 
there exists 
a sequence $\langle \gamma_n : n < \omega \rangle$ increasing 
and bounded below $\sigma$ such that $\xi < \gamma_0$, 
$\gamma_n \in P$ if 
$n$ is even, and $\gamma_n \in C_\sigma$ if $n$ is odd. 
Let $\gamma$ be the supremum of this sequence. 
Then $\gamma$ is a limit point of $P$ which is strictly below $\sup(P)$ with 
cofinality $\omega$. 
Since $\cf(P \cap \kappa) > \omega$, $\cf(\gamma) < \cf(P \cap \kappa)$. 
So $\gamma \in P$ by Lemma 7.14. 
On the other hand, $\gamma$ is a limit point of $C_\sigma$. 
So $\xi < \gamma$ and $\gamma \in \lim(C_\sigma) \cap P$.
\end{proof}

\begin{lemma}
Let $M$ and $N$ be in $\mathcal X \cup \mathcal Y$, and assume that 
$\{ M, N \}$ is adequate in the case that $M$ and $N$ are in $\mathcal X$. 
Then $M \cap N \in \mathcal X \cup \mathcal Y$. 
Specifically:
\begin{enumerate}
\item if $M \in \mathcal X$ and $N \in \mathcal X \cup \mathcal Y$, then 
$M \cap N \in \mathcal X$;
\item if $M \in \mathcal Y$ and $N \in \mathcal Y$, then 
$M \cap N \in \mathcal Y$. 
\end{enumerate}
\end{lemma}

\begin{proof}
Obviously $M \cap N$ is an elementary substructure of $\mathcal A$. 
If $M \in \mathcal X$, then $M \cap N \in \mathcal X_0$ by Lemma 1.23 and 
the comment after Notation 1.10. 
Hence $M \cap N \cap \kappa \in T^*$. 
And if $M$ and $N$ are in $\mathcal Y$, then 
$M \cap N \cap \kappa = \min \{ M \cap \kappa,  N \cap \kappa \} \in \kappa$. 
Let $\alpha := \sup(M \cap N)$. 
It remains to show that $\lim(C_\alpha) \cap (M \cap N)$ is cofinal in $\alpha$. 

First we claim that $\lim(C_\alpha) \cap M$ and $\lim(C_\alpha) \cap N$ 
are cofinal in $\alpha$. 
Since $M \cap N \cap \kappa^+$ is closed under successors, it does not 
have a maximum element, and therefore $\alpha$ is a limit point of 
$M \cap N \cap \kappa^+$. 
As $\alpha \in \lim(M)$, if $\alpha \notin M$ then 
$\lim(C_{\alpha}) \cap M$ is cofinal in $\alpha$ by Lemma 7.13, and similarly 
with $N$. 
So if $\alpha$ is neither in $M$ nor $N$, then the claim is proved. 
Assume that $\alpha$ is in one of them. 
Since $\alpha = \sup(M \cap N)$, $\alpha$ cannot be in both in $M$ and $N$. 
Without loss of generality, assume that $\alpha \in N \setminus M$. 
Then $\lim(C_\alpha) \cap M$ is cofinal in $\alpha$ as just observed, 
and so in particular, 
$\lim(C_\alpha)$ is cofinal in $\alpha$. 
By the elementarity of $N$, and since $\alpha \in N$ and 
also $\alpha$ is a limit point of $N$, 
easily $\lim(C_\alpha) \cap N$ is cofinal in $\alpha$.
This completes the proof of the claim.

To show that $\lim(C_\alpha) \cap (M \cap N)$ is cofinal in $\alpha$, 
let $\gamma < \alpha$. 
Fix $\gamma' \in M \cap N \cap \kappa^+$ with $\gamma < \gamma'$. 
Let $\sigma = \min(\lim(C_\alpha) \setminus \gamma')$. 
We claim that $\sigma \in M \cap N$, which completes the proof. 
Since $\lim(C_\alpha) \cap M$ is cofinal in $\alpha$, we can fix 
$\eta \in \lim(C_\alpha) \cap M$ with $\sigma < \eta$. 
As $\eta \in \lim(C_\alpha)$, $C_\eta = C_\alpha \cap \eta$. 
Therefore $\sigma = \min(\lim(C_\eta) \setminus \gamma')$. 
Since $\eta$ and $\gamma'$ are in $M$, $\sigma \in M$ by elementarity. 
The same argument shows that $\sigma \in N$.
\end{proof}

We now introduce the idea of a simple model.\footnote{This is the same idea as 
described in \cite[Section 3.3]{mitchell}.} 
These are the models for which there exist strongly generic conditions. 
To motivate the definition, we prove a bound on $\ot(C_{\sup(N)})$.

\begin{lemma}
Let $N \in \mathcal X \cup \mathcal Y$. 
If $\eta \in \lim(N)$, then $\ot(C_\eta) \in \cl(N \cap \kappa)$. 
In particular, $\ot(C_{\sup(N)}) \le \sup(N \cap \kappa)$.
\end{lemma}

\begin{proof}
Since $\eta \in \lim(N)$ and $|N| < \kappa$, it follows that 
$\cf(\eta) < \kappa$. 
If $\eta \in N$, then $\ot(C_\eta) \in N \cap \kappa$ by elementarity. 
Assume that $\eta$ is not in $N$. 
Then $\eta \in \cl(N) \setminus N$. 
By Lemma 7.13, $\lim(C_\eta) \cap N$ is cofinal in $\eta$. 
We claim that $\ot(C_{\eta})$ is a limit point of $N \cap \kappa$. 
Let $\gamma < \ot(C_\eta)$. 
Then we can find $\delta \in \lim(C_\eta) \cap N$ such that 
$$
\gamma < \ot(C_\eta \cap \delta) = \ot(C_\delta) < \ot(C_\eta).
$$
Since $\delta \in N$, $\ot(C_\delta) \in N \cap \ot(C_\eta)$. 
So $\gamma < \ot(C_\delta) < \ot(C_\eta)$ and 
$\ot(C_\delta) \in N \cap \kappa$.
\end{proof}

\begin{definition}
Let $N \in \mathcal X \cup \mathcal Y$. 
We say that $N$ is \emph{simple} if $\ot(C_{\sup(N)}) = \sup(N \cap \kappa)$.
\end{definition}

We prove next that there exist stationarily many simple models 
in $\mathcal X$.

\begin{lemma}
Let $N \in \mathcal X \cup \mathcal Y$ and $\delta := \sup(N)$. 
Then for all $\xi \in N \cap \ot(C_\delta)$, 
$c_{\delta,\xi} \in N$. 
\end{lemma}

\begin{proof}
Let $\xi \in N \cap \ot(C_\delta)$. 
Then $\xi < \ot(C_\delta)$. 
As $\lim(C_{\delta}) \cap N$ is cofinal $\delta$, 
we can fix $\eta \in \lim(C_{\delta}) \cap N$ such that 
$\xi < \ot(C_\delta \cap \eta) = \ot(C_\eta)$. 
Hence $c_{\delta,\xi} = c_{\eta,\xi}$. 
Since $\eta$ and $\xi$ are 
in $N$, by elementarity, $c_{\eta,\xi}$ is in $N$.
\end{proof}

\begin{proposition}
The collection of models in $\mathcal X$ which are simple 
is stationary in $P_{\omega_1}(H(\kappa^+))$.
\end{proposition}

\begin{proof}
Let $F : H(\kappa^+)^{<\omega} \to H(\kappa^+)$, and we will find a simple model 
in $\mathcal X$ which is closed under $F$. 
Fix $X$ of size $\kappa$ such that $X \prec \mathcal A$, $X$ is closed under $F$, 
and $\theta := X \cap \kappa^+$ has cofinality $\kappa$. 
Let $c : \kappa \to \theta$ be the function 
$c(\xi) = c_{\theta,\xi}$ for all $\xi < \kappa$. 
Since $T^*$ is stationary in $P_{\omega_1}(\kappa)$ and $X \prec \mathcal A$, 
we can find 
$M \in P_{\omega_1}(X)$ which is closed under $F$ such that 
$$
M \cap \kappa \in T^*, \ M \prec \mathcal A, \ \textrm{and} \ 
M \prec (X,\in,C_\theta,c).
$$
Let $\delta := \sup(M)$.

We claim that $M$ is in $\mathcal X$ and is simple. 
To show that $M \in \mathcal X$, it suffices to show that 
$\lim(C_\delta) \cap M$ is cofinal in $\delta$. 
By elementarity, clearly $\delta \in \lim(C_\theta)$. 
Hence $C_\delta = C_\theta \cap \delta$. 
As $M$ is closed under $c$, for all $\xi \in M \cap \kappa$, 
$$
c(\xi) = c_{\theta,\xi} \in M \cap \kappa^+ \subseteq \delta.
$$ 
Thus $c(\xi) \in C_{\theta} \cap \delta = C_{\delta}$. 
So $c(\xi) = c_{\theta,\xi} = c_{\delta,\xi}$. 
It follows by elementarity that $\{ c(\xi) : \xi \in M \cap \kappa \}$ 
is increasing and cofinal in $M \cap \delta$. 
In particular, the set $\{ c(\xi) : \xi \in M \cap \kappa, \ \xi \ \textrm{limit} \}$ 
witnesses that $\lim(C_\delta) \cap M$ is cofinal $\delta$.

It remains to show that $M$ is simple, which means that 
$\ot(C_\delta) = \sup(M \cap \kappa)$. 
We know that $\ot(C_{\delta}) \le \sup(M \cap \kappa)$ by Lemma 7.17. 
Suppose for a contradiction that $\ot(C_\delta) < \sup(M \cap \kappa)$. 
Fix $\beta \in (M \cap \kappa) \setminus \ot(C_{\delta})$. 
Then $c(\beta) \in M$ by elementarity. 
But $c(\beta) = c_{\theta,\beta} = c_{\delta,\beta}$, as previously observed, 
which is absurd since $\ot(C_\delta) \le \beta$.
\end{proof}

Regarding the stationarity of simple models in $\mathcal Y$, see 
Lemma 8.3 and Proposition 14.2.

\bigskip

Next we will show that a model $M$ in $\mathcal X \cup \mathcal Y$ 
is determined by $\sup(M \cap \kappa)$ 
and $\sup(M)$.

\begin{notation}
Consider $\eta < \kappa^+$ and $\beta < \kappa$. 
For $\gamma < \ot(A_{\eta,\beta})$, let 
$a_{\eta,\beta,\gamma}$ be equal to the $\gamma$-th element of 
$A_{\eta,\beta}$.  
Define $\pi_\eta : \kappa \times \kappa \to \eta$ by letting 
$\pi_\eta(\gamma,\beta) = a_{\eta,\beta,\gamma}$ if 
$\gamma < \ot(A_{\eta,\beta})$, and $0$ otherwise.
\end{notation}

Note that $\pi_\eta$ is a surjection of $\kappa \times \kappa$ onto $\eta$. 
Also if $\xi \in A_{\eta,\beta}$, then 
$\xi = \pi_{\eta}(\ot(A_{\eta,\beta} \cap \xi),\beta)$.

Observe that $\pi_\eta$ is definable in the structure $\mathcal A$.

\begin{lemma}
Let $\eta < \kappa^+$ and $\beta < \kappa$. 
Suppose that 
$$
\delta \in \lim(C_\eta) \cup A_{\eta,\beta} \cup \lim(A_{\eta,\beta})
$$
and $\gamma < \ot(A_{\delta,\beta})$. 
Then $a_{\eta,\beta,\gamma} = a_{\delta,\beta,\gamma}$. 
So $\pi_\eta(\gamma,\beta) = \pi_\delta(\gamma,\beta)$.
\end{lemma}

\begin{proof}
By Notation 7.4(4), $A_{\delta,\beta} = A_{\eta,\beta} \cap \delta$, so clearly 
$a_{\eta,\beta,\gamma} = a_{\delta,\beta,\gamma}$.
\end{proof}

\begin{lemma}
Let $N \in \mathcal X \cup \mathcal Y$. 
Then 
$$
N \cap \kappa^+ = \{ \pi_{\sup(N)}(\gamma,\beta) : 
\gamma, \beta \in N \cap \kappa \}.
$$
\end{lemma}

\begin{proof}
Let $\eta := \sup(N)$. 
Suppose that $\gamma$ and $\beta$ are in $N \cap \kappa$, and we will 
show that $\pi_{\eta}(\gamma,\beta) \in N$. 
This is obvious if $\pi_{\eta}(\gamma,\beta) = 0$. 
So assume that $\gamma < \ot(A_{\eta,\beta})$ and 
$\pi_\eta(\gamma,\beta) = a_{\eta,\beta,\gamma}$. 
Since $N \in \mathcal X \cup \mathcal Y$, 
$\lim(C_{\eta}) \cap N$ is cofinal in $\eta$. 
As $a_{\eta,\beta,\gamma} < \eta$, we can fix 
$\delta \in \lim(C_\eta) \cap N$ such that 
$a_{\eta,\beta,\gamma} < \delta$. 
Since $A_{\delta,\beta} = A_{\eta,\beta} \cap \delta$, 
clearly $\gamma < \ot(A_{\delta,\beta})$. 
By Lemma 7.22, $\pi_{\eta}(\gamma,\beta) = \pi_{\delta}(\gamma,\beta)$. 
As $\delta$, $\gamma$, and $\beta$ are in $N$, 
$\pi_{\delta}(\gamma,\beta) \in N$ by elementarity.

Conversely, let $\xi \in N \cap \kappa^+$ be given, and we will 
find $\gamma$ and $\beta$ in $N \cap \kappa$ 
such that $\pi_\eta(\gamma,\beta) = \xi$. 
Since $\lim(C_\eta) \cap N$ is cofinal in $\eta$, we can fix 
$\delta \in \lim(C_\eta) \cap N$ such that $\xi < \delta$. 
Then $\xi$ and $\delta$ are in $N$. 
By elementarity, there is $\beta \in N \cap \kappa$ such that 
$\xi \in A_{\delta,\beta}$. 
Let $\gamma := \ot(A_{\delta,\beta} \cap \xi)$. 
Since $\delta$, $\beta$, and $\xi$ are in $N$, $\gamma \in N$. 
As noted after Notation 7.21, 
$a_{\delta,\beta,\gamma} = \pi_\delta(\gamma,\beta) = \xi$. 
Since $\delta \in \lim(C_\eta)$, 
by Lemma 7.22, $\pi_{\eta}(\gamma,\beta) = 
\pi_{\delta}(\gamma,\beta) = \xi$.
\end{proof}

\begin{lemma}
Let $M$ and $N$ be in $\mathcal X \cup \mathcal Y$, and suppose that 
$M \cap \kappa$ and $\sup(M)$ are in $N$. 
Then $M \in N$.
\end{lemma}

\begin{proof}
Since $M = f^*[M \cap \kappa^+]$, by elementarity 
it suffices to show that $M \cap \kappa^+ \in N$. 
Let $\eta := \sup(M)$. 
Then 
$$
M \cap \kappa^+ = \{ \pi_\eta(\gamma,\beta) : \gamma, \beta \in 
M \cap \kappa \}
$$
by Lemma 7.23. 
Since $\eta$ and $M \cap \kappa$ are in $N$, 
$M \cap \kappa^+ \in N$ by elementarity.
\end{proof}

\bigskip

The next topic we consider is the set 
$A_{\sup(M),\sup(M \cap \kappa)}$, where $M \in \mathcal X \cup \mathcal Y$.

\begin{lemma}
Let $M \in \mathcal X \cup \mathcal Y$. 
Then $M \cap \kappa^+ \subseteq A_{\sup(M),\sup(M \cap \kappa)}$.
\end{lemma}

\begin{proof}
Let $\xi \in M \cap \kappa^+$. 
Since $\lim(C_{\sup(M)}) \cap M$ is cofinal in $\sup(M)$, we can fix 
$\sigma \in \lim(C_{\sup(M)}) \cap M$ which is strictly greater than $\xi$. 
By elementarity, we can fix $\beta \in M \cap \kappa$ such that 
$\xi \in A_{\sigma,\beta}$. 
Since $\sigma \in \lim(C_{\sup(M)})$, we have that 
$A_{\sigma,\beta} = A_{\sup(M),\beta} \cap \sigma$. 
Hence $\xi \in A_{\sup(M),\beta}$. 
As $\beta \in M \cap \kappa$, $\beta < \sup(M \cap \kappa)$. 
Therefore 
$A_{\sup(M),\beta} \subseteq A_{\sup(M),\sup(M \cap \kappa)}$. 
Hence $\xi \in A_{\sup(M),\sup(M \cap \kappa)}$.
\end{proof}

\begin{lemma}
Let $Q \in \mathcal Y$. 
Then $Q \cap \kappa^+ = A_{\sup(Q),Q \cap \kappa}$.
\end{lemma}

\begin{proof}
By Lemma 7.25, we have that $Q \cap \kappa^+ \subseteq 
A_{\sup(Q),Q \cap \kappa}$. 
Conversely, let $\xi \in A_{\sup(Q),Q \cap \kappa}$, and we will 
show that $\xi \in Q$. 
Since $\lim(C_{\sup(Q)}) \cap Q$ is cofinal in $\sup(Q)$, 
we can fix $\sigma \in \lim(C_{\sup(Q)}) \cap Q$ which is strictly greater than $\xi$. 
As $\sigma \in \lim(C_{\sup(Q)})$, 
$A_{\sigma,Q \cap \kappa} = A_{\sup(Q),Q \cap \kappa} \cap \sigma$. 
Therefore $\xi \in A_{\sigma,Q \cap \kappa}$. 
Since $Q \cap \kappa$ is a limit ordinal, 
we can fix $\beta < Q \cap \kappa$ 
such that $\xi \in A_{\sigma,\beta}$. 
Then $\sigma$ and $\beta$ are in $Q$, and hence 
$A_{\sigma,\beta} \in Q$. 
Since $|A_{\sigma,\beta}| < \kappa$ by Notation 7.4(3), 
$A_{\sigma,\beta} \subseteq Q$. 
Therefore $\xi \in Q$.
\end{proof}

\begin{lemma}
Let $Q \in \mathcal Y$ and $\eta \in \cl(Q)$. 
Then $Q \cap \eta = A_{\eta,Q \cap \kappa}$.
\end{lemma}

\begin{proof}
By Lemma 7.26, $Q \cap \kappa^+ = A_{\sup(Q),Q \cap \kappa}$. 
Since $\eta \in \cl(Q)$, 
$$
\eta \in A_{\sup(Q),Q \cap \kappa} \cup \lim(A_{\sup(Q),Q \cap \kappa}).
$$
By Notation 7.4(4),
$$
A_{\eta,Q \cap \kappa} = A_{\sup(Q),Q \cap \kappa} \cap \eta = 
Q \cap \eta.
$$
\end{proof}

\begin{lemma}
Suppose that $P_1$ and $P_2$ are in $\mathcal Y$.
\begin{enumerate}
\item If $P_1 \cap \kappa \le P_2 \cap \kappa$ 
and $\eta \in \cl(P_1) \cap \cl(P_2)$, then 
$P_1 \cap \eta \subseteq P_2 \cap \eta$.
\item If $P_1 \cap \kappa < P_2 \cap \kappa$ 
and $\eta \in P_1 \cap P_2 \cap \cof(\kappa)$, 
then $P_1 \cap \eta \in P_2$. 
In particular, $\sup(P_1 \cap \eta) \in P_2 \cap \eta$.
\end{enumerate}
\end{lemma}

\begin{proof}
By Lemma 7.27, under the assumptions of either (1) or (2), we have that 
$$
P_1 \cap \eta = A_{\eta,P_1 \cap \kappa}, \ \textrm{and} 
\ P_2 \cap \eta = A_{\eta,P_2 \cap \kappa}.
$$

(1) If $P_1 \cap \kappa \le P_2 \cap \kappa$, then clearly 
$A_{\eta,P_1 \cap \kappa} \subseteq A_{\eta,P_2 \cap \kappa}$. 
Therefore $P_1 \cap \eta \subseteq P_2 \cap \eta$.

(2) If $P_1 \cap \kappa < P_2 \cap \kappa$, then 
$P_1 \cap \kappa \in P_2$. 
So $P_1 \cap \kappa$ and $\eta$ are in $P_2$, and hence 
$A_{\eta,P_1 \cap \kappa} = P_1 \cap \eta$ is in $P_2$ by elementarity. 
Since $\eta$ has cofinality $\kappa$, 
$\sup(P_1 \cap \eta) < \eta$. 
So $\sup(P_1 \cap \eta) \in P_2 \cap \eta$.
\end{proof}

\begin{lemma}
Let $M \in \mathcal X$. 
Let $A := A_{\sup(M),\sup(M \cap \kappa)}$. 
Then $A$ is closed under $H^*$, $A \cap \kappa = \sup(M \cap \kappa)$, 
and $\sup(A) = \sup(M)$.
\end{lemma}

\begin{proof}
To show that $A$ is closed under $H^*$, it suffices to show 
that for each $n < \omega$, $A$ is closed under $\tau_n'$. 
At the same time, we will show that $\sup(M \cap \kappa) \subseteq A$. 
So fix $n < \omega$, and let $k$ be the arity of $\tau_n'$. 
Let $\alpha_0,\ldots,\alpha_{k-1} \in A$ and $\beta < \sup(M \cap \kappa)$, 
and we will show that 
$\tau_n'(\alpha_0,\ldots,\alpha_{k-1})$ and $\beta$ are in $A$. 
Fix $\eta_0 \in \lim(C_{\sup(M)}) \cap M$ such that 
$\alpha_0,\ldots,\alpha_{k-1}$ and $\beta_M$ are strictly less than $\eta_0$. 
Then $A \cap \eta_0 = A_{\eta_0,\sup(M \cap \kappa)}$. 
So $\alpha_0,\ldots,\alpha_{k-1}$ are in 
$A_{\eta_0,\sup(M \cap \kappa)}$. 
Also $\beta_M \in M \cap \kappa \subseteq A$ by Lemma 7.25, so also 
$\beta_M \in A_{\eta_0,\sup(M \cap \kappa)}$. 
As $\sup(M \cap \kappa)$ is a limit ordinal, we can fix an infinite 
$\gamma \in M \cap \kappa$ such that 
$\alpha_0,\ldots,\alpha_{k-1}$ and $\beta_M$ are in $A_{\eta_0,\gamma}$.

By the elementarity of $M$, 
we can fix $\eta_1 \in M$ strictly greater than $\eta_0$ such that 
$\eta_1$ is closed under $\tau_n'$. 
Fix $\eta_2 \in \lim(C_{\sup(M)}) \cap M$ with $\eta_1 < \eta_2$. 
Since $\eta_0 < \eta_1$ and 
$\eta_1$ is closed under $\tau_n'$, 
for all $\gamma_0,\ldots,\gamma_{k-1}$ 
in $A_{\eta_0,\gamma}$, 
$\tau_n'(\gamma_0,\ldots,\gamma_{k-1}) < \eta_1 < \eta_2$. 
Define a function $h : A_{\eta_0,\gamma}^{k} \to \kappa$ by letting 
$h(\gamma_0,\ldots,\gamma_{k-1})$ be the least ordinal 
$\xi < \kappa$ such that 
$\tau_n'(\gamma_0,\ldots,\gamma_{k-1})$ and all ordinals below $\beta_M$ 
are in $A_{\eta_2,\xi}$. 
Since $\eta_0$, $\gamma$, $\eta_2$, and $\beta_M$ are in $M$, 
by elementarity $h$ is in $M$.

Now the domain of $h$ has size $|A_{\eta_0,\gamma}^k| \le |\gamma| < \kappa$. 
So there exists a minimal $\xi < \kappa$ such that 
$h[A_{\eta_0,\gamma}^{k}] \subseteq \xi$. 
By elementarity, $\xi \in M \cap \kappa$. 
In particular, $\delta := h(\alpha_0,\ldots,\alpha_{k-1})$ is less than $\xi$. 
That means $\tau_n'(\alpha_0,\ldots,\alpha_{k-1})$ and $\beta$ are in 
$A_{\eta_2,\delta} \subseteq A_{\eta_2,\xi}$. 
Since $\eta_2 \in \lim(C_{\sup(M)})$, 
$A_{\eta_2,\xi} = A_{\sup(M),\xi} \cap \eta_2$. 
So $\tau_n'(\alpha_0,\ldots,\alpha_{k-1})$ and $\beta$ are in 
$A_{\sup(M),\xi} \subseteq 
A_{\sup(M),\sup(M \cap \kappa)} = A$. 
This completes the proof that $A$ is closed under $\tau_n'$ for all $n < \omega$ 
and $\sup(M \cap \kappa) \subseteq A$. 
It follows that $A$ is closed under $H^*$.

Now $A$ is the union of sets of the form 
$A_{\delta,\beta}$, where $\delta \in \lim(C_{\sup(M)}) \cap M$ and 
$\beta \in M \cap \kappa$, and each such set is in $M$. 
Thus each such set $A_{\delta,\beta}$ satisfies that 
$\sup(A_{\delta,\beta} \cap \kappa) < \sup(M \cap \kappa)$. 
It follows that $\sup(A \cap \kappa) \le \sup(M \cap \kappa)$. 
But we just proved that $\sup(M \cap \kappa) \subseteq A$, and therefore 
$\sup(M \cap \kappa) = A \cap \kappa$. 
By the definition of $A$, obviously $A \subseteq \sup(M)$. 
And since $M \cap \kappa^+ \subseteq A$ by Lemma 7.25, 
$\sup(A) = \sup(M)$.
\end{proof}

\bigskip

We conclude this section with two technical lemmas which will 
be useful later.

\begin{lemma}
Let $N \in \mathcal X$, 
$a \in N$, and $\tau \in N \cap \kappa^+$. 
Suppose that $\cf(\tau) > \omega$. 
If $a \cap [\sup(N \cap \tau),\tau) \ne \emptyset$, 
then $\tau$ is a limit point of $a$.
\end{lemma}

\begin{proof}
If $\tau$ is not a limit point of $a$, then $\sup(a \cap \tau) < \tau$. 
Since $a$ and $\tau$ are in $N$, 
$\sup(a \cap \tau) \in N \cap \tau$ by elementarity. 
Hence $\sup(a \cap \tau) < \sup(N \cap \tau)$, 
which contradicts the assumption that 
$a \cap [\sup(N \cap \tau),\tau) \ne \emptyset$.
\end{proof}

\begin{lemma}
Let $N \in \mathcal X$. 
Suppose that $\eta \in N \cap \kappa^+$ and $\beta \in N \cap \kappa$. 
Let $\xi \in A_{\eta,\beta} \setminus N$. 
Then $\xi \in A_{\xi_N,\beta}$.
\end{lemma}

\begin{proof}
Note that since $\eta \in N$ and $\xi < \eta$, $\xi_N$ exists. 
Since $\xi \notin N$, $\xi < \xi_N$. 
It follows that $\cf(\xi_N) > \omega$, since otherwise 
$N \cap \xi_N$ would be cofinal in $\xi_N$ by elementarity. 
Also $\sup(N \cap \xi_N) \le \xi$. 
Since $A_{\eta,\beta} \in N$ and 
$\xi \in A_{\eta,\beta} \cap [\sup(N \cap \xi_N),\xi_N)$, 
it follows that 
$\xi_N$ is a limit point of $A_{\eta,\beta}$ by Lemma 7.30. 
So $A_{\xi_N,\beta} = A_{\eta,\beta} \cap \xi_N$. 
Since $\xi \in A_{\eta,\beta} \cap \xi_N$, $\xi \in A_{\xi_N,\beta}$.
\end{proof}

\bigskip

\addcontentsline{toc}{section}{8. Interaction of models past $\kappa$}

\textbf{\S 8. Interaction of models past $\kappa$}

\stepcounter{section}

\bigskip

The method of adequate sets, which we dealt with in Part I, 
handles the interaction of countable elementary 
substructures below $\kappa$. 
In this section we will show how the coherent filtration system $\vec A$ 
from Section 7 can be used to control the interaction of models 
between $\kappa$ and $\kappa^+$. 

\begin{notation}
For $M$ and $N$ in $\mathcal X \cup \mathcal Y$, let 
$\alpha_{M,N}$ denote the ordinal $\sup(M \cap N)$.
\end{notation}

As we discussed in Section 1 in the paragraph after Propostion 1.29, 
if $M < N$, in general it does not necessarily follow that 
$M \cap N \in N$. 
The next lemma describes a situation in which this implication does hold.

\begin{lemma}
Let $M$ and $N$ be in $\mathcal X \cup \mathcal Y$, where $N$ is simple.
Suppose that:
\begin{enumerate}
\item $M$ and $N$ are in $\mathcal X$ and $M < N$, or
\item $M$ and $N$ are in $\mathcal Y$ and $M \cap \kappa < N \cap \kappa$, or 
\item $M \in \mathcal X$, $N \in \mathcal Y$, and 
$\sup(M \cap N \cap \kappa) < N \cap \kappa$.
\end{enumerate}
Then $M \cap N \in N$. 
In particular, $\alpha_{M,N} \in N$.
\end{lemma}

\begin{proof}
By Lemma 1.30, $M \cap N \cap \kappa \in N$. 
Therefore by elementarity, 
$\cl(M \cap N \cap \kappa) \in N$. 
We claim that 
$$
\cl(M \cap N \cap \kappa) \subseteq N \cap \kappa.
$$
If $M \cap N$ is countable, then so is $\cl(M \cap N \cap \kappa)$. 
Since $\cl(M \cap N \cap \kappa) \in N$, it follows that 
$\cl(M \cap N \cap \kappa) \subseteq N \cap \kappa$. 
If $M \cap N$ is uncountable, then $M$ and $N$ are 
both in $\mathcal Y$. 
So $M \cap \kappa < N \cap \kappa$ by (2), and hence 
$M \cap N \cap \kappa = M \cap \kappa$. 
Therefore 
$$
\cl(M \cap N \cap \kappa) = 
(M \cap \kappa) \cup \{ M \cap \kappa \},
$$
which is a subset of $N \cap \kappa$. 

Next we claim that 
$$
(N \cap \kappa^+) \setminus \alpha_{M,N} \ne \emptyset.
$$
Suppose for a contradiction that 
$(N \cap \kappa^+) \setminus \alpha_{M,N} = \emptyset$, 
which means that $\sup(N) = \alpha_{M,N}$. 
Since $N$ is simple, it follows that 
$$
\ot(C_{\alpha_{M,N}}) = \sup(N \cap \kappa).
$$
But $\ot(C_{\alpha_{M,N}}) \in \cl(M \cap N \cap \kappa)$ 
by Lemma 7.17. 
By the first claim, it follows that 
$\ot(C_{\alpha_{M,N}}) \in N \cap \kappa$, which contradicts that 
$\ot(C_{\alpha_{M,N}}) = \sup(N \cap \kappa)$.

Let $\alpha := \min((N \cap \kappa^+) \setminus \alpha_{M,N})$. 
By Lemma 7.23, 
$$
M \cap N \cap \kappa^+ = 
\{ \pi_{\alpha_{M,N}}(\gamma,\beta) : 
\gamma, \beta \in M \cap N \cap \kappa \}.
$$
We claim that for all $\gamma, \beta \in M \cap N \cap \kappa$, 
$\pi_{\alpha_{M,N}}(\gamma,\beta) = \pi_\alpha(\gamma,\beta)$. 
This is immediate if $\alpha = \alpha_{M,N}$, so assume that 
$\alpha_{M,N} \notin N$. 
Then by Lemma 7.12, $\alpha_{M,N} \in \lim(C_\alpha)$.  
Fix $\gamma$ and $\beta$ in $M \cap N \cap \kappa$. 

First, assume that $\ot(A_{\alpha,\beta}) \le \gamma$. 
Then $\pi_{\alpha}(\gamma,\beta) = 0$. 
Since $A_{\alpha_{M,N},\beta} = A_{\alpha,\beta} \cap \alpha_{M,N}$, 
clearly $\ot(A_{\alpha_{M,N},\beta}) \le \ot(A_{\alpha,\beta}) \le \gamma$. 
So $\pi_{\alpha_{M,N}}(\gamma,\beta) = 0$. 
Secondly, assume that $\gamma < \ot(A_{\alpha,\beta})$, so that 
$\pi_{\alpha}(\gamma,\beta) = a_{\alpha,\beta,\gamma}$. 
Since $\alpha$, $\gamma$, and $\beta$ are in $N$, 
$a_{\alpha,\beta,\gamma} \in N \cap \alpha \subseteq \alpha_{M,N}$. 
As $A_{\alpha_{M,N},\beta} = A_{\alpha,\beta} \cap \alpha_{M,N}$, 
clearly $\gamma < \ot(A_{\alpha_{M,N},\beta})$. 
By Lemma 7.22, $\pi_{\alpha}(\gamma,\beta) = 
\pi_{\alpha_{M,N}}(\gamma,\beta)$.

It follows that 
$$
M \cap N \cap \kappa^+ = 
\{ \pi_\alpha(\gamma,\beta) : \gamma, \beta \in M \cap N \cap \kappa \}.
$$
Since $\alpha$ and $M \cap N \cap \kappa$ are in $N$, 
so is $M \cap N \cap \kappa^+$ by elementarity. 
Hence by elementarity, 
$M \cap N = f^*[M \cap N \cap \kappa^+] \in N$, and 
$\sup(M \cap N) = \alpha_{M,N} \in N$.
\end{proof}

\begin{lemma}
Let $P \in \mathcal Y$, and assume that 
$\cf(\sup(P)) = P \cap \kappa$. 
Then $P$ is simple.
\end{lemma}

\begin{proof}
Since $C_{\sup(P)}$ is cofinal in $\sup(P)$, 
$$
P \cap \kappa = \cf(\sup(P)) \le \ot(C_{\sup(P)}).
$$
On the other hand, 
as $\sup(P) \in \lim(P)$, Lemma 7.17 implies that 
$$
\ot(C_{\sup(P)}) \in \cl(P \cap \kappa) = 
(P \cap \kappa) \cup \{ P \cap \kappa \}.
$$
Hence $\ot(C_{\sup(P)}) \le P \cap \kappa$. 
Therefore $P \cap \kappa = \ot(C_{\sup(P)})$, and $P$ is simple.
\end{proof}

\begin{lemma}
Let $P \in \mathcal Y$, and assume that $\cf(\sup(P)) = P \cap \kappa$. 
If $M \in \mathcal X$, then $M \cap P \in P$. 
If $Q \in \mathcal Y$ and $Q \cap \kappa < P \cap \kappa$, 
then $Q \cap P \in P$. 
\end{lemma}

\begin{proof}
By Lemma 8.3, $P$ is simple. 
Since $\cf(\sup(P)) = P \cap \kappa$, $P \cap \kappa$ is 
a regular cardinal. 
Obviously $\omega < P \cap \kappa$, so $P \cap \kappa$ is a regular 
uncountable cardinal. 
If $M \in \mathcal X$, then $\sup(M \cap P \cap \kappa) < P \cap \kappa$ 
since $\sup(M \cap P \cap \kappa)$ has countable cofinality. 
By Lemma 8.2, we are done.
\end{proof}

\begin{lemma}
Let $M \in \mathcal X$ and $N \in \mathcal X \cup \mathcal Y$, where 
$\{ M, N \}$ is adequate if $N \in \mathcal X$, and 
$\sup(M \cap N \cap \kappa) < N \cap \kappa$ if $N \in \mathcal Y$. 
Then 
$$
\lim(M) \cap \lim(N) \subseteq \alpha_{M,N} + 1.
$$
\end{lemma}

\begin{proof}
Let $\eta \in \lim(M) \cap \lim(N)$, and we will show that $\eta \le \alpha_{M,N}$. 
By Lemma 7.17, 
$$
\ot(C_\eta) \in \cl(M \cap \kappa) \cap \cl(N \cap \kappa).
$$
Since $\eta$ is a limit point of $M$ and $M$ is countable, 
$\eta$ has cofinality $\omega$. 
Therefore $\ot(C_\eta)$ has cofinality $\omega$. 
We claim that $\ot(C_\eta)$ is a limit point of $M \cap \kappa$. 
If $\ot(C_\eta) \in M$, then since $\ot(C_\eta)$ has countable cofinality, 
easily $M \cap \ot(C_\eta)$ is cofinal in $\ot(C_\eta)$ by elementarity. 
Hence $\ot(C_\eta)$ is a limit point of $M \cap \kappa$. 
Otherwise if $\ot(C_\eta) \notin M$, then since 
$\ot(C_\eta) \in \cl(M \cap \kappa)$, it follows immediately that 
$\ot(C_\eta)$ is in $\lim(M \cap \kappa)$.

Next we claim that 
$$
\ot(C_\eta) \in \cl(M \cap N  \cap \kappa).
$$
First, assume that $N \in \mathcal X$. 
Then by Lemma 1.20, 
$$
\ot(C_\eta) \in \cl(M \cap \kappa) \cap \cl(N \cap \kappa) = 
\cl(M \cap N \cap \kappa).
$$

Secondly, assume that $N \in \mathcal Y$. 
Now 
$$
\ot(C_\eta) \in \cl(N \cap \kappa) = 
(N \cap \kappa) \cup \{ N \cap \kappa \}.
$$
So $\ot(C_\eta) \le N \cap \kappa$. 
Since $\sup(M \cap N \cap \kappa) < N \cap \kappa$ 
and $\ot(C_\eta)$ is a limit point of $M \cap \kappa$, 
it cannot be 
the case that $\ot(C_\eta) = N \cap \kappa$. 
Therefore $\ot(C_\eta) < N \cap \kappa$. 
By Lemma 1.31, 
$$
\cl(M \cap N \cap \kappa) = 
\cl(M \cap \kappa) \cap \cl(N \cap \kappa) \cap (N \cap \kappa).
$$ 
Since $\ot(C_\eta)$ is in the set on the right, it is in 
$\cl(M \cap N \cap \kappa)$.

Now we claim that 
$$
\ot(C_\eta) \in \lim(M \cap N \cap \kappa).
$$
As $\ot(C_\eta) \in \cl(M \cap N \cap \kappa)$, either 
$\ot(C_\eta) \in M \cap N \cap \kappa$, or 
$\ot(C_\eta) \in \lim(M \cap N \cap \kappa)$. 
In the latter case, we are done. 
In the former case, since $\ot(C_\eta)$ has cofinality $\omega$, by the 
elementarity of $M \cap N$, clearly 
$M \cap N \cap \kappa$ is cofinal in $\ot(C_\eta)$, so again 
$\ot(C_\eta) \in \lim(M \cap N \cap \kappa)$.

Finally, we are ready to prove that $\eta \le \alpha_{M,N}$. 
Suppose for a contradiction that $\alpha_{M,N} < \eta$. 
Since $\ot(C_\eta)$ is a limit point of 
$M \cap N \cap \kappa$, we can fix 
$\gamma \in M \cap N \cap \ot(C_\eta)$ 
such that $\alpha_{M,N} < c_{\eta,\gamma}$. 
We claim that $c_{\eta,\gamma} \in M \cap N$, 
which is a contradiction since 
$M \cap N \cap \kappa^+ \subseteq \alpha_{M,N}$. 
If $\eta \in M$, then obviously $c_{\eta,\gamma} \in M$ by elementarity. 
Otherwise $\eta \in \cl(M) \setminus M$. 
By Lemma 7.13, $\lim(C_\eta) \cap M$ is cofinal in $\eta$. 
So we can fix $\delta \in \lim(C_\eta) \cap M$ 
such that $c_{\eta,\gamma} < \delta$. 
Then clearly $c_{\eta,\gamma} = c_{\delta,\gamma}$, which is in $M$ 
by elementarity. 
The proof that $c_{\eta,\gamma} \in N$ is the same.
\end{proof}

\bigskip

We now turn to address the following general issue. 
Suppose that $M$ and $N$ are in $\mathcal X \cup \mathcal Y$, and 
$P \in N \cap \mathcal Y$. 
Under what circumstances can we conclude that an ordinal 
in $M \cap P$ is in $N$, or is in some canonically described member of $N$?

The next lemma is the most frequently used result on this topic.

\begin{lemma}
Let $M$ and $N$ be in $\mathcal X$, where $M \le N$. 
Suppose that 
$$
\eta \in N \cap \kappa^+ \ \textrm{and} \ \beta < \sup(M \cap N \cap \kappa).
$$
Then $A_{\eta,\beta} \cap M \subseteq N$. 
Therefore 
$$
A_{\eta,\sup(M \cap N \cap \kappa)} \cap M \subseteq N.
$$
\end{lemma}

\begin{proof}
Let $\xi \in A_{\eta,\beta} \cap M$, and we will show that $\xi \in N$. 
Since $\beta < \sup(M \cap N \cap \kappa)$, we can fix 
$\gamma \in M \cap N \cap \kappa$ greater than $\beta$. 
Then $\xi \in A_{\eta,\gamma}$. 
So $\xi = \pi_{\eta}(\ot(A_{\eta,\gamma} \cap \xi),\gamma)$, as noted in 
the comments after Notation 7.21. 
Since $\xi \in A_{\eta,\gamma}$, 
$A_{\xi,\gamma} = A_{\eta,\gamma} \cap \xi$. 
Therefore $\ot(A_{\eta,\gamma} \cap \xi) = \ot(A_{\xi,\gamma})$. 
Hence $\xi = \pi_{\eta}(\ot(A_{\xi,\gamma}),\gamma)$. 
Since $\xi$ and $\gamma$ are in $M$, so is $\ot(A_{\xi,\gamma})$.

Since $\gamma \in M \cap N \cap \kappa$, by elementarity $c^*(\gamma) \in 
M \cap N \cap \kappa$. 
By Notation 7.4(5), since $M \le N$, we have that 
$$
\ot(A_{\xi,\gamma}) \in M \cap c^*(\gamma) 
\subseteq M \cap N \cap \kappa \subseteq N.
$$
So $\ot(A_{\xi,\gamma}) \in N \cap \kappa$. 
Hence $\eta$, $\gamma$, and $\ot(A_{\xi,\gamma})$ are in $N$, which implies 
that $\pi_{\eta}(\ot(A_{\xi,\gamma}),\gamma) = \xi$ is in $N$.

To show that $A_{\eta,\sup(M \cap N \cap \kappa)} \cap M \subseteq N$, 
let $\tau \in A_{\eta,\sup(M \cap N \cap \kappa)} \cap M$. 
Since $\sup(M \cap N \cap \kappa)$ is a limit ordinal, 
there is $\beta < \sup(M \cap N \cap \kappa)$ such that 
$\tau \in A_{\eta,\beta}$. 
By what we just proved, $A_{\eta,\beta} \cap M \subseteq N$. 
So $\tau \in N$.
\end{proof}

\begin{lemma}
Let $M$ and $N$ be in $\mathcal X$, where $M \le N$. 
Let $Q \in N \cap \mathcal Y$ with 
$Q \cap \kappa < \sup(M \cap N \cap \kappa)$. 
Then $Q \cap M \cap \kappa^+ \subseteq N$.
\end{lemma}

\begin{proof}
By Lemma 7.26, $Q \cap \kappa^+ = A_{\sup(Q),Q \cap \kappa}$. 
By elementarity, $\sup(Q) \in N \cap \kappa^+$, and 
by assumption, $Q \cap \kappa < \sup(M \cap N \cap \kappa)$. 
By Lemma 8.6, 
$$
Q \cap M \cap \kappa^+ = 
A_{\sup(Q),Q \cap \kappa} \cap M \subseteq N.
$$
\end{proof}

\begin{lemma}
Let $M$ and $N$ be in $\mathcal X$.
\begin{enumerate}
\item If $M \le N$, then 
$$
A_{\alpha_{M,N},\sup(M \cap N \cap \kappa)} \cap M \subseteq N.
$$
\item If $M \sim N$, then 
$$
A_{\alpha_{M,N},\sup(M \cap N \cap \kappa)} \cap M = 
A_{\alpha_{M,N},\sup(M \cap N \cap \kappa)} \cap N.
$$
\end{enumerate}
\end{lemma}

\begin{proof}
Note that (2) follows from (1). 
To prove (1), 
assume that $M \le N$, and let 
$\xi \in A_{\alpha_{M,N},\sup(M \cap N \cap \kappa)} \cap M$. 
We will show that $\xi \in N$. 
Fix $\beta \in M \cap N \cap \kappa$ 
such that $\xi \in A_{\alpha_{M,N},\beta}$. 
As $M \cap N \in \mathcal X$ and $\sup(M \cap N) = \alpha_{M,N}$, 
it follows that 
$\lim(C_{\alpha_{M,N}}) \cap (M \cap N)$ is cofinal in $\alpha_{M,N}$. 
So we can fix $\delta \in \lim(C_{\alpha_{M,N}}) \cap (M \cap N)$ 
which is strictly larger than $\xi$. 
Then $A_{\delta,\beta} = A_{\alpha_{M,N},\beta} \cap \delta$, and 
hence $\xi \in A_{\delta,\beta}$. 
Since $\delta \in N$, $\beta < \sup(M \cap N \cap \kappa)$, 
and $M \le N$, it follows that 
$A_{\delta,\beta} \cap M \subseteq N$ by Lemma 8.6. 
So $\xi \in N$.
\end{proof}

\begin{lemma}
Let $M \in \mathcal X$ and $N \in \mathcal X \cup \mathcal Y$. 
Then 
$$
M \cap N \cap \kappa^+ \subseteq 
A_{\alpha_{M,N},\sup(M \cap N \cap \kappa)}.
$$
\end{lemma}

\begin{proof}
Since $M \cap N \in \mathcal X$ and $\sup(M \cap N) = \alpha_{M,N}$, 
the statement follows immediately from Lemma 7.25.
\end{proof}

\begin{lemma}
Let $M \in \mathcal X$ and $N \in \mathcal X \cup \mathcal Y$.
Let $Q \in M \cap \mathcal Y$, and suppose that 
$\sup(M \cap N \cap \kappa) \le Q \cap \kappa$. 
Then 
$$
Q \cap N \cap \alpha_{M,N} \subseteq 
A_{\alpha_{M,N},Q \cap \kappa}.
$$
\end{lemma}

\begin{proof}
Let $\xi \in Q \cap N \cap \alpha_{M,N}$, and we will show 
that $\xi \in A_{\alpha_{M,N},Q \cap \kappa}$. 
First assume that $\xi \in M$. 
Then $\xi \in M \cap N \cap \kappa^+$, so by Lemma 8.9, 
$\xi \in A_{\alpha_{M,N},\sup(M \cap N \cap \kappa)}$. 
Since $\sup(M \cap N \cap \kappa) \le Q \cap \kappa$, 
it follows that $\xi \in A_{\alpha_{M,N},Q \cap \kappa}$.

Assume that $\xi$ is not in $M$. 
Then $\xi \in Q \cap \kappa^+ = A_{\sup(Q),Q \cap \kappa}$, 
where $\sup(Q)$ and $Q \cap \kappa$ are in $M$, and 
$\xi \notin M$. 
By Lemma 7.31, 
$$
\xi \in A_{\xi_M,Q \cap \kappa}.
$$

We claim that 
$$
\forall \nu \in M \cap N \cap \kappa^+ \ 
(\xi_M < \nu \implies \xi \in A_{\nu,Q \cap \kappa}).
$$
We will prove the claim by induction. 
So let $\nu \in M \cap N \cap \kappa^+$ be strictly greater than $\xi_M$, 
and assume that the claim holds for all 
$\nu' \in M \cap N \cap \nu$.

\bigskip

\emph{Case 1:} $\nu = \nu_0 + 1$ is a successor ordinal. 
Since $\nu \in M \cap N$, $\nu_0 \in M \cap N$ by elementarity. 
If $\xi_M < \nu_0$, then 
by the inductive hypothesis, $\xi \in A_{\nu_0,Q \cap \kappa}$. 
So 
$$
\xi \in A_{\nu_0,Q \cap \kappa} \cup \{ \nu_0 \} = 
A_{\nu,Q \cap \kappa}.
$$
If $\nu_0 = \xi_M$, then 
$$
\xi \in A_{\xi_M,Q \cap \kappa} = A_{\nu_0,Q \cap \kappa} 
\subseteq A_{\nu,Q \cap \kappa}.
$$

\bigskip

\emph{Case 2:} $\nu$ is a limit ordinal and 
$\xi_M \in \lim(C_{\nu})$. 
Then 
$$
A_{\xi_M,Q \cap \kappa} = A_{\nu,Q \cap \kappa} \cap \xi_M.
$$
Since $\xi \in A_{\xi_M,Q \cap \kappa}$, it follows that 
$\xi \in A_{\nu,Q \cap \kappa}$, and we are done. 

\bigskip

\emph{Case 3:} $\nu$ is a limit ordinal and 
$\xi_M$ is not in $\lim(C_{\nu})$. 
Let $\nu' := \min(C_\nu \setminus \xi_M)$, and let 
$\nu'' := \sup(C_\nu \cap \nu')$. 
Since $\nu' = \min(C_\nu \setminus \xi_M)$, clearly 
$\nu'' = \sup(C_\nu \cap \xi_M)$. 
As $\xi_M$ is not a limit point of $C_\nu$, 
$\nu'' < \xi_M$.

We claim that $\nu' \in M \cap N$. 
Since $\nu$ and $\xi_M$ are in $M$, $\nu' = \min(C_\nu \setminus \xi_M)$ 
is in $M$ by elementarity. 
And as $\nu$ and $\nu'$ are in $M$, 
$\nu'' = \sup(C_\nu \cap \nu')$ is in $M$ by elementarity. 
But $\nu'' < \xi_M$ and $\xi_M$ is the least ordinal in $M$ with 
$\xi \le \xi_M$. 
It follows that $\nu'' < \xi$. 
So $\nu' = \min(C_\nu \setminus (\nu''+1)) = 
\min(C_\nu \setminus \xi)$. 
Since $\nu$ and $\xi$ are in $N$, so is $\nu'$ by elementarity.

Next we claim that $\xi \in A_{\nu',Q \cap \kappa}$. 
This is immediate if $\xi_M = \nu'$, so assume that $\xi_M < \nu'$. 
Then $\xi_M < \nu' < \nu$ and $\nu' \in M \cap N \cap \kappa^+$, 
which imply by the inductive hypothesis that $\xi \in A_{\nu',Q \cap \kappa}$.

Let $\beta$ be the least ordinal in $\kappa$ such that 
$\nu' \in A_{\nu,\beta}$. 
Since $\nu$ and $\nu'$ are in $M \cap N$, it follows that 
$\beta \in M \cap N \cap \kappa$ 
by elementarity. 
As 
$$
\beta < \sup(M \cap N \cap \kappa) \le Q \cap \kappa,
$$
we have that $\nu' \in A_{\nu,Q \cap \kappa}$. 
And since $\nu'' = \sup(C_\nu \cap \nu')$ and 
$$
\nu'' < \xi < \xi_M \le \nu',
$$
$\xi \notin C_\nu$. 
So 
$$
\xi \in \nu \setminus C_\nu, \ 
\nu' = \min(C_\nu \setminus \xi) \in A_{\nu,Q \cap \kappa}, \ 
\textrm{and} \ \xi \in A_{\min(C_\nu \setminus \xi),Q \cap \kappa}.
$$
By Notation 7.4(7), $\xi \in A_{\nu,Q \cap \kappa}$, which completes 
the proof of the claim.

Since $M \cap N \in \mathcal X$ and $\sup(M \cap N) = \alpha_{M,N}$, 
it follows that 
$\lim(C_{\alpha_{M,N}}) \cap (M \cap N)$ is cofinal in 
$\alpha_{M,N}$. 
Since $\xi < \alpha_{M,N}$ by assumption and $\alpha_{M,N}$ is a limit 
point of $M$, we have that $\xi_M < \alpha_{M,N}$. 
So we can fix $\nu \in \lim(C_{\alpha_{M,N}}) \cap (M \cap N)$ which is strictly 
greater than $\xi_M$. 
By the claim, $\xi \in A_{\nu,Q \cap \kappa}$. 
Since $\nu \in \lim(C_{\alpha_{M,N}})$, 
$A_{\nu,Q \cap \kappa} = A_{\alpha_{M,N},Q \cap \kappa} \cap \nu$. 
Therefore $\xi \in A_{\alpha_{M,N},Q \cap \kappa}$. 
\end{proof}

\begin{lemma}
Let $M \in \mathcal X$ and $N \in \mathcal X \cup \mathcal Y$. 
Suppose that $M < N$ in the case that $N \in \mathcal X$. 
Let $P \in M \cap \mathcal Y$, and suppose that 
$P \cap \kappa \in M \cap N \cap \kappa$ and 
$P \cap \alpha_{M,N}$ is bounded below $\alpha_{M,N}$. 

Define 
$$
\sigma := \sup(P \cap A_{\alpha_{M,N},\sup(M \cap N \cap \kappa)}).
$$
Then $\sigma$ satisfies:
\begin{enumerate}
\item $\sigma \in M \cap N \cap \kappa^+$;
\item $P \cap \sigma = A_{\sigma,P \cap \kappa}$;
\item $P \cap (M \cap N) \cap \kappa^+ = A_{\sigma,P \cap \kappa} \cap 
(M \cap N)$;
\item $N \cap P \cap \alpha_{M,N} \subseteq A_{\sigma,P \cap \kappa}$.
\end{enumerate}
\end{lemma}

\begin{proof}
Let $\alpha := \alpha_{M,N}$ and $\delta := \sup(M \cap N \cap \kappa)$. 
Note that by Lemma 7.29, 
$$
Sk(A_{\alpha,\delta}) \cap \kappa^+ = 
A_{\alpha,\delta}, \ 
A_{\alpha,\delta} \cap \kappa = \delta, \ 
\textrm{and} \ \sup(A_{\alpha,\delta}) = \alpha.
$$
Since $P$ and $A_{\alpha,\delta}$ are closed under successors, 
$P \cap A_{\alpha,\delta}$ has no maximal element. 
Note that since $P \cap \alpha_{M,N}$ is bounded below $\alpha_{M,N}$, we have 
that $\sigma < \alpha$. 
We claim that $\sigma$ satisfies 
(1)--(4). 
Observe that since $\sigma$ is a limit point of $P$, 
it follows that $P \cap \sigma = A_{\sigma,P \cap \kappa}$ 
by Lemma 7.27, which proves (2).

(3) We prove that 
$$
P \cap (M \cap N) \cap \kappa^+ = A_{\sigma,P \cap \kappa} \cap 
(M \cap N).
$$
Let $\gamma \in P \cap (M \cap N) \cap \kappa^+$, and we will show that 
$\gamma \in A_{\sigma,P \cap \kappa}$. 
Since $\gamma \in M \cap N$, by Lemma 8.9 we have that 
 $\gamma \in A_{\alpha,\delta}$. 
So $\gamma \in P \cap A_{\alpha,\delta} \subseteq \sigma$. 
Hence $\gamma \in P \cap \sigma = A_{\sigma,P \cap \kappa}$. 
Conversely, 
$$
A_{\sigma,P \cap \kappa} \cap (M \cap N) \subseteq 
A_{\sigma,P \cap \kappa} = P \cap \sigma \subseteq P.
$$

(1,4) It remains to show that $\sigma \in M \cap N$ and 
$N \cap P \cap \alpha_{M,N} \subseteq A_{\sigma,P \cap \kappa}$. 
We claim that 
$$
\sigma = \sup(A_{\sup(P \cap \alpha),P \cap \kappa} \cap 
A_{\alpha,\delta}).
$$
As $P$ and $\alpha$ are closed under successors, $P \cap \alpha$ has 
no maximal element. 
So $\sup(P \cap \alpha)$ is a limit point of $P$, which 
by Lemma 7.27 implies that 
$$
P \cap \alpha = P \cap \sup(P \cap \alpha) = A_{\sup(P \cap \alpha),P \cap \kappa}.
$$ 
Therefore 
$$
A_{\sup(P \cap \alpha),P \cap \kappa} \cap 
A_{\alpha,\delta} = P \cap \alpha \cap A_{\alpha,\delta} = 
P \cap A_{\alpha,\delta}.
$$
Taking supremums of both sides yields the claim.

Next, we claim that $\sigma \in A_{\alpha,\delta}$. 
As $P \cap \alpha$ and $A_{\alpha,\delta}$ are closed 
under successors and $P \cap \alpha = 
A_{\sup(P \cap \alpha),P \cap \kappa}$, it follows that 
$\sigma$ is a limit point of $A_{\sup(P \cap \alpha),P \cap \kappa}$ 
and a limit point of $A_{\alpha,\delta}$. 
By Notation 7.4(8) and the fact that $\sigma \in 
\lim(A_{\alpha,\delta}) \cap \alpha$, to show that 
$\sigma \in A_{\alpha,\delta}$ it suffices to show that 
$\ot(C_\sigma) < \delta$.

Since $P \cap \sigma = A_{\sigma,P \cap \kappa}$ and $\sigma$ is a limit 
point of $P$, it follows that $\sigma = \sup(A_{\sigma,P \cap \kappa})$. 
If $P \cap \kappa < \ot(C_{\sigma})$, then by Notation 7.4(6), 
it follows that $A_{\sigma,P \cap \kappa} \subseteq c_{\sigma,P \cap \kappa} 
< \sigma$. 
But this contradicts that $\sigma = \sup(A_{\sigma,P \cap \kappa})$. 
Hence $\ot(C_{\sigma}) \le P \cap \kappa < \delta$, which completes the 
proof of the claim that $\sigma \in A_{\alpha,\delta}$.

Fix $\eta \in \lim(C_\alpha) \cap (M \cap N)$ such that 
$\sigma < \eta$. 
Then $A_{\eta,\delta} = A_{\alpha,\delta} \cap \eta$. 
Therefore $\sigma \in A_{\eta,\delta}$. 
Since $\delta = \sup(M \cap N \cap \kappa)$, 
we can fix $\gamma \in M \cap N \cap \kappa$ such that 
$\sigma \in A_{\eta,\gamma}$. 
Then $\eta$ and $\gamma$ are in $M \cap N$. 

Let us show that 
$$
\sigma = \max(A_{\eta,\gamma} \cap \lim(P)).
$$
Suppose for a contradiction that 
$\sigma' \in A_{\eta,\gamma} \cap \lim(P)$ and 
$\sigma < \sigma'$. 
Since $\eta \in \lim(C_\alpha)$ and $\gamma < \delta$, 
$\sigma' \in A_{\alpha,\delta} \cap \lim(P)$. 
But then by Lemma 7.27, it follows that 
$$
P \cap \sigma' = A_{\sigma',P \cap \kappa} \subseteq 
A_{\sigma',\delta} = A_{\alpha,\delta} \cap \sigma'.
$$
Since $\sigma'$ is a limit point of $P$, there is 
$\tau \in P \cap \sigma'$ strictly greater than $\sigma$. 
Then $\tau \in P \cap A_{\alpha,\delta}$, which contradicts that 
$\sigma = \sup(P \cap A_{\alpha,\delta})$. 

Now we prove that $\sigma \in M \cap N$. 
Since $\eta$, $\gamma$, and $P$ are in $M$, and 
$\sigma = \max(A_{\eta,\gamma} \cap \lim(P))$, it follows that 
$\sigma \in M$ by elementarity. 
On the other hand, $\sigma \in M \cap A_{\eta,\gamma}$, where 
$\eta \in N$ and $\gamma \in M \cap N \cap \kappa$. 
So $\sigma \in N$ by Lemma 8.6, in the case when 
$M$ and $N$ are in $\mathcal X$. 
If $N \in \mathcal Y$, then $\sigma \in A_{\eta,\gamma} \in N$ implies that 
$\sigma \in N$, since $|A_{\eta,\gamma}| < \kappa$. 
This proves that $\sigma \in M \cap N$.

Now we claim that $N \cap P \cap \alpha_{M,N} \subseteq \sigma$. 
This completes the proof, for then 
$$
N \cap P \cap \alpha_{M,N} \subseteq P \cap \sigma = A_{\sigma,P \cap \kappa}.
$$ 
Suppose for a contradiction that $\pi \in N \cap P \cap \alpha_{M,N}$ 
and $\sigma \le \pi$. 
Let $\pi_0 := \min((P \cap \kappa^+) \setminus \sigma)$, and note 
that $\pi_0 \le \pi$. 
Since $P$ and $\sigma$ are in $M$, $\pi_0$ is in $M$ by elementarity. 
We claim that $\pi_0$ is in $N$. 
This is immediate if $\pi_0 = \pi$, so assume that $\pi_0 < \pi$. 
Then $\pi \in P$ implies that 
$P \cap \pi = A_{\pi,P \cap \kappa}$ by Lemma 7.27. 
Since $\pi$ and $P \cap \kappa$ are in $N$, 
so is $A_{\pi,P \cap \kappa} = P \cap \pi$. 
But $\pi_0 = \min((P \cap \pi) \setminus \sigma)$. 
Hence $\pi_0 \in N$ by elementarity. 
So $\pi_0 \in M \cap N \cap \alpha_{M,N}$, and therefore 
$\pi_0 \in A_{\alpha,\delta}$ by Lemma 8.9. 
So $\pi_0 \in P \cap A_{\alpha,\delta}$. 
Since $\sigma = \sup(P \cap A_{\alpha,\delta})$ and 
$P \cap A_{\alpha,\delta}$ has no maximal element as previously 
observed, it follows that $\pi_0 < \sigma$. 
But this contradicts that fact that $\sigma \le \pi_0$.
\end{proof}

\bigskip

So far in this section we have been mostly concerned about the interaction of 
models $M$ and $N$ below $\alpha_{M,N} = \sup(M \cap N)$. 
We now turn to analyze what happens above $\alpha_{M,N}$.

The next two lemmas state that for a simple model $N$, if a model 
does not bound $N$ below $\kappa$, then it does not bound $N$ 
above $\kappa$.

\begin{lemma}
Let $M$ and $N$ be in $\mathcal X$, where $N$ is simple and 
$\{ M, N \}$ is adequate. 
If $R_M(N) \ne \emptyset$, then 
$(N \cap \kappa^+) \setminus \alpha_{M,N} \ne \emptyset$.
\end{lemma}

\begin{proof}
Since $R_M(N)$ is nonempty, 
$\beta_{M,N} \le \sup(N \cap \kappa)$. 
As $\alpha_{M,N}$ is a limit point of $M$ and a limit point of $N$, 
Lemma 7.17 implies that $\ot(C_{\alpha_{M,N}})$ is in 
$\cl(M \cap \kappa) \cap \cl(N \cap \kappa)$. 
Hence by Lemma 1.15, $\ot(C_{\alpha_{M,N}}) < \beta_{M,N}$. 
If $(N \cap \kappa^+) \setminus \alpha_{M,N}$ is empty, then 
$\sup(N) = \alpha_{M,N}$. 
Since $N$ is simple, it follows that 
$\ot(C_{\alpha_{M,N}}) = \sup(N \cap \kappa)$. 
But then $\sup(N \cap \kappa) < \beta_{M,N}$, which contradicts the 
first line above.
\end{proof}

\begin{lemma}
Let $N \in \mathcal X$ be simple and $Q \in \mathcal Y$. 
If $Q \cap \kappa < \sup(N \cap \kappa)$, then
$\sup(N \cap Q) < \sup(N)$.
\end{lemma}

\begin{proof}
Let $\beta := Q \cap \kappa$ and 
$\eta := \sup(N \cap Q)$, and assume that 
$\beta < \sup(N \cap \kappa)$. 
Since $N$ and $Q$ are closed under successors, 
$\eta$ is a limit point of $N \cap Q$. 
In particular, $\eta$ is a limit point of $N$. 
Suppose for a contradiction that $\sup(N) = \eta$. 
Then since $N$ is simple, $\ot(C_\eta) = \sup(N \cap \kappa)$. 
So $\beta < \ot(C_\eta)$. 
As $N \cap Q$ is in $\mathcal X$ by Lemma 7.16 and $\sup(N \cap Q) = \eta$, 
it follows that $\lim(C_\eta) \cap (N \cap Q)$ is cofinal in $\eta$. 
So we can fix $\delta \in \lim(C_\eta) \cap (N \cap Q)$ such that 
$\beta < \ot(C_\eta \cap \delta) = \ot(C_\delta)$. 
But $\delta \in Q$, and therefore by elementarity, 
$\ot(C_\delta) \in Q \cap \kappa = \beta$. 
So $\ot(C_\delta) < \beta$, which is a contradiction.
\end{proof}

\bigskip

We now introduce an analogue of remainder points for ordinals between 
$\alpha_{M,N}$ and $\kappa^+$.

\begin{definition}
Let $M$ and $N$ be in $\mathcal X \cup \mathcal Y$. 
Define $R^+_N(M)$ as the set of ordinals $\eta$ such that either:
\begin{enumerate}
\item $\eta = \min((M \cap \kappa^+) \setminus \alpha_{M,N})$ 
and $\alpha_{M,N} < \eta$, or 
\item $\eta = \min((M \cap \kappa^+) \setminus \xi)$, 
for some $\xi \in (N \cap \kappa^+) \setminus \alpha_{M,N}$.
\end{enumerate}
\end{definition}

\begin{lemma}
Let $M \in \mathcal X$ and $N \in \mathcal X \cup \mathcal Y$, where 
$\{ M, N \}$ is adequate if $N \in \mathcal X$, and 
$\sup(M \cap N \cap \kappa) < N \cap \kappa$ if $N \in \mathcal Y$. 
Then:
\begin{enumerate}
\item $R^+_N(M)$ is finite;
\item if $\eta \in R^+_N(M)$, then 
$\cf(\eta) > \omega$;
\item suppose that $\eta \in R^+_N(M)$, $\eta$ is not 
equal to $\min((M \cap \kappa^+) \setminus \alpha_{M,N})$, and 
$\sigma := \min((N \cap \kappa^+) \setminus \sup(M \cap \eta))$; 
then $\sigma \in R^+_M(N)$ and $\eta = \min((M \cap \kappa^+) \setminus \sigma)$.
\end{enumerate}
\end{lemma}

\begin{proof}
(1) If $R^+_N(M)$ is not finite, then the supremum of the first 
$\omega$ many members of $R^+_N(M)$ is a limit point of $M$ and a 
limit point of $N$. 
Hence this supremum is less than or equal to $\alpha_{M,N}$ by Lemma 8.5, 
which contradicts the definition of $R^+_N(M)$.

(2) is easy. 
(3) Note that $\sigma$ exists, since otherwise $\eta$ would not be 
in $R^+_N(M)$. 
Clearly $\eta = \min((M \cap \kappa^+) \setminus \sigma)$. 
We will show that $\sigma \in R^+_M(N)$. 
Since $\eta$ is not equal to 
$\min((M \cap \kappa^+) \setminus \alpha_{M,N})$, fix 
$\theta \in (M \cap \eta) \setminus \alpha_{M,N}$. 
Then $\alpha_{M,N} \le \theta < \sigma$, 
and therefore $\alpha_{M,N} < \sigma$. 
If $\sigma = \min((N \cap \kappa^+) \setminus \alpha_{M,N})$, 
then $\sigma \in R^+_M(N)$ by definition. 
So assume not. 
Then we can fix $\xi \in (N \cap \sigma) \setminus \alpha_{M,N}$. 
By definition, $\zeta_0 := \min((M \cap \kappa^+) \setminus \xi)$ 
is in $R^+_N(M)$. 
Since $\xi < \sigma = \min((N \cap \kappa^+) \setminus 
\sup(M \cap \eta))$ and $\xi \in N$, it follows that 
$\xi < \sup(M \cap \eta)$. 
Therefore $\zeta_0 < \sup(M \cap \eta) \le \sigma$. 

Let $\zeta_1$ be the largest member of $R^+_N(M)$ which is below $\sigma$. 
Then $\zeta_1$ exists since $R^+_N(M) \cap \sigma$ is finite and 
nonempty, as witnessed by $\zeta_0$. 
We claim that $\sigma = \min((N \cap \kappa^+) \setminus \zeta_1)$, 
which proves that $\sigma \in R^+_M(N)$. 
Otherwise $\sigma_0 := \min((N \cap \kappa^+) \setminus \zeta_1)$ 
is strictly below $\sigma$. 
So $\sigma_0 < \sigma = \min((N \cap \kappa^+) 
\setminus \sup(M \cap \eta))$, 
which implies that $\sigma_0 < \sup(M \cap \eta)$. 
But then $\min((M \cap \eta) \setminus \sigma_0)$ is in $R^+_N(M) \cap \sigma$, 
and is strictly larger than $\zeta_1$, which contradicts the maximality 
of $\zeta_1$.
\end{proof}

\begin{lemma}
Let $M$ and $N$ be in $\mathcal X \cup \mathcal Y$. 
Then for all $\eta \in R^+_N(M) \cup R^+_M(N)$, 
$\eta$ is closed under $H^*$.
\end{lemma}

\begin{proof}
First consider $\eta = \min((N \cap \kappa^+) \setminus \alpha_{M,N})$. 
Let $n < \omega$ and let $k$ be the arity of $\tau_n'$, and we will 
show that $\eta$ is closed under $\tau_n'$. 
Since $\eta \in N$, by elementarity it suffices to show that $N$ models 
that $\eta$ is closed under $\tau_n'$. 
Let $\alpha_0,\ldots,\alpha_{k-1} \in N \cap \eta$. 
By the minimality of $\eta$, 
$N \cap \eta \subseteq \alpha_{M,N}$, so 
$\alpha_0,\ldots,\alpha_{k-1} < \alpha_{M,N}$. 
By the elementarity of $M \cap N$ and since $\sup(M \cap N) = \alpha_{M,N}$, 
there is some $\gamma \in M \cap N \cap \kappa^+$ such that 
$\alpha_0,\ldots,\alpha_{k-1}$ are below $\gamma$ and 
$\gamma$ is closed under $\tau_n'$. 
Then 
$$
\tau_n'(\alpha_0,\ldots,\alpha_{k-1}) < \gamma < \alpha_{M,N} 
\le \eta.
$$
The same proof works for 
$\min((M \cap \kappa^+) \setminus \alpha_{M,N})$.

Now we prove the general statement by induction on ordinals in 
$R^+_N(M) \cup R^+_M(N)$. 
Suppose that $\eta \in R^+_N(M)$, 
and for all $\sigma \in (R^+_N(M) \cup R^+_M(N)) \cap \eta$, 
$\sigma$ is closed under $H^*$. 
If $\eta = \min((M \cap \kappa^+) \setminus \alpha_{M,N})$, then 
we are done by the previous paragraph. 
Otherwise by Lemma 8.15(3), the ordinal 
$$
\sigma := \min((N \cap \kappa^+) \setminus \sup(M \cap \eta))
$$
is in 
$R^+_M(N)$, 
and $\eta = \min((M \cap \kappa^+) \setminus \sigma)$. 
By the inductive hypothesis, $\sigma$ is closed under $H^*$.

Let $n < \omega$, and we will show that $\eta$ is closed under $\tau_n'$. 
Let $k$ be the arity of $\tau_n'$. 
Since $\eta$ is in $M$, by elementarity it suffices to show that $M$ models 
that $\eta$ is closed under $\tau_n'$. 
Let $\alpha_0,\ldots,\alpha_{k-1} \in M \cap \eta$. 
Then $\alpha_0,\ldots,\alpha_{k-1}$ are strictly less than $\sup(M \cap \eta) \le \sigma$. 
Since $\sigma$ is closed under $H^*$, 
$$
\tau_n'(\alpha_0,\ldots,\alpha_{k-1}) < \sigma < \eta.
$$
The same argument works for ordinals in $R^+_M(N)$.
\end{proof}

\bigskip

\addcontentsline{toc}{section}{9. Canonical models}

\textbf{\S 9. Canonical models}

\stepcounter{section}

\bigskip

In this section we will introduce some models which are 
determined by canonical parameters which arise 
in the comparison of two models. 
Specifically, we consider a simple model $N \in \mathcal X$ and 
a model $M$ in $\mathcal X \cup \mathcal Y$ which is not 
necessarily a member of $N$. 
The canonical models associated with $M$ and 
$N$ are models in $N \cap \mathcal Y$ 
which reflect some information about $M$ inside $N$. 
Canonical models will be used when amalgamating side conditions or 
forcing conditions over a simple model $N$; see 
Sections 13 and 15.\footnote{The idea of a canonical model 
is new to this paper, and does not appear in Mitchell's original proof \cite{mitchell}.}

The three types of canonical models are described in 
Notations 9.1, 9.3, and 9.13.

\begin{notation}
Let $N \in \mathcal X$ be simple and $P \in \mathcal Y$, where 
$P \cap \kappa < \sup(N \cap \kappa)$. 
Let $\beta := P \cap \kappa$ and $\eta := \sup(N \cap P)$. 
We let $Q(N,P)$ denote the set $Sk(A_{\eta_N,\beta_N})$.
\end{notation}

Note that $\eta_N$ exists by Lemma 8.13. 
It is easy to check that if $P \in N \cap \mathcal Y$, then 
$Q(N,P) = P$.

\begin{lemma}
Let $N \in \mathcal X$ be simple 
and $P \in \mathcal Y$, where 
$P \cap \kappa < \sup(N \cap \kappa)$.  
Let $\beta := P \cap \kappa$ and $\eta := \sup(N \cap P)$. 
Then $Q := Q(N,P)$ satisfies the following properties:
\begin{enumerate}
\item $Q \in N \cap \mathcal Y$;
\item $Q \cap \kappa = \beta_N$, 
$Q \cap \kappa^+ = A_{\eta_N,\beta_N}$, and $\sup(Q) = \eta_N$;
\item $N \cap Q \cap \kappa^+ = N \cap P \cap \kappa^+$.
\end{enumerate}
\end{lemma}

\begin{proof}
Since $\eta_N$ and $\beta_N$ are in $N$, 
$A_{\eta_N,\beta_N}$ and $Q$ are in $N$ by elementarity. 
As $N \cap P$ is closed under ordinal successors, 
$\eta$ is a limit point of $N \cap P$. 
Therefore $P \cap \eta = A_{\eta,\beta}$ by Lemma 7.27. 
And since $\eta$ is a limit point of $N \cap \kappa^+$, by Lemma 7.12, 
$C_\eta = C_{\eta_N} \cap \eta$, and for all $\xi < \kappa$, 
$$
A_{\eta,\xi} = A_{\eta_N,\xi} \cap \eta \ \textrm{and} \ 
N \cap A_{\eta_N,\xi} = N \cap A_{\eta,\xi}.
$$

We claim that 
$$
N \cap P \cap \kappa^+ = N \cap A_{\eta_N,\beta_N}.
$$
Let $\alpha \in N \cap P \cap \kappa^+$, and we will show 
that $\alpha \in A_{\eta_N,\beta_N}$. 
Then $\alpha < \eta$ by the definition of $\eta$. 
So 
$$
\alpha \in P \cap \eta = A_{\eta,\beta} = A_{\eta_N,\beta} \cap \eta.
$$ 
Hence $\alpha \in A_{\eta_N,\beta}$. 
Since $\beta \le \beta_N$, $\alpha \in A_{\eta_N,\beta_N}$. 
Conversely, let $\alpha \in N \cap A_{\eta_N,\beta_N}$, 
and we will show that $\alpha \in P$. 
Since $\eta_N$, $\beta_N$, and $\alpha$ are in $N$ 
and $\beta_N$ is a limit ordinal, 
by elementarity we can fix $\xi \in N \cap \beta_N$ such that 
$\alpha \in A_{\eta_N,\xi}$. 
Then $\alpha \in N \cap A_{\eta_N,\xi} = N \cap A_{\eta,\xi}$. 
Since $\xi \in N \cap \beta_N$, $\xi < \beta$. 
So $A_{\eta,\xi} \subseteq A_{\eta,\beta} = P \cap \eta$. 
Hence $\alpha \in P$.

We have proven that 
$N \cap P \cap \kappa^+ = N \cap A_{\eta_N,\beta_N}$. 
By Lemma 7.10, 
it follows that $A_{\eta_N,\beta_N}$ is closed under $H^*$. 
In particular, $Q \cap \kappa^+ = A_{\eta_N,\beta_N}$. 
Therefore 
$$
N \cap Q \cap \kappa^+ = 
N \cap A_{\eta_N,\beta_N} = N \cap P \cap \kappa^+,
$$
which proves (3).

To show that $Q \cap \kappa = \beta_N$ and $\sup(Q) = \eta_N$, 
it suffices to prove that $N$ models these statements. 
Let $\alpha \in N \cap Q \cap \kappa$, and we will show 
that $\alpha < \beta_N$. 
Then 
$$
\alpha \in N \cap Q \cap \kappa = N \cap P \cap \kappa 
\subseteq P \cap \kappa = \beta \le \beta_N.
$$
So $\alpha < \beta_N$. 
Conversely, let $\alpha \in N \cap \beta_N$, and we will 
show that $\alpha \in Q$. 
Then $\alpha \in N \cap \beta_N \subseteq \beta$. 
So 
$$
\alpha \in N \cap \beta = N \cap P \cap \kappa = N \cap Q \cap \kappa.
$$ 
So indeed $\alpha \in Q$.

Since $Q \cap \kappa^+ = A_{\eta_N,\beta_N}$, 
clearly $\sup(Q) \le \eta_N$. 
To show that $N$ models that 
$\sup(Q) = \eta_N$, let $\xi \in N \cap \eta_N$. 
Then $\xi < \eta = \sup(N \cap P)$. 
So we can fix $\sigma \in N \cap P \cap \kappa^+$ which is larger 
than $\xi$. 
Then $\sigma \in N \cap P \cap \kappa^+ = N \cap Q \cap \kappa^+$. 
So $\sigma \in Q$ and $\xi \le \sigma$. 
Thus $\sup(Q) = \eta_N$. 
This completes the proof of (2).

To show that $Q \in \mathcal Y$, it suffices to prove that 
$\lim(C_{\eta_N}) \cap Q$ is cofinal in $\eta_N$. 
Again it will be enough to show that $N$ models this statement. 
So let $\xi \in N \cap \eta_N$. 
Then $\xi < \eta = \sup(N \cap P)$. 
Since $N \cap P \in \mathcal X$ by Lemma 7.16 and $\sup(N \cap P) = \eta$, 
there is $\sigma \in \lim(C_\eta) \cap (N \cap P)$ with $\xi \le \sigma$. 
Since $C_{\eta} = C_{\eta_N} \cap \eta$, 
$\sigma \in \lim(C_{\eta_N})$. 
Also $\sigma \in N \cap P \cap \kappa^+ = N \cap Q \cap \kappa^+$. 
So $\sigma \in \lim(C_{\eta_N}) \cap Q$ and $\xi \le \sigma$.
\end{proof}

\begin{notation}
Let $M$ and $N$ be in $\mathcal X$, where $N$ is simple. 
Let $\zeta \in R_M(N)$ and 
$\eta := \min((N \cap \kappa^+) \setminus \alpha_{M,N})$. 
We let $Q(N,M,\zeta)$ denote the set $Sk(A_{\eta,\zeta})$.
\end{notation}

Note that $\eta$ exists by Lemma 8.12.

\begin{lemma}
Let $M$ and $N$ be in $\mathcal X$, where $N$ is simple. 
Suppose that $N \le M$ and 
$\zeta = \min((N \cap \kappa) \setminus \beta_{M,N})$. 
Let $\eta := \min((N \cap \kappa^+) \setminus \alpha_{M,N})$. 
Then $Q := Q(N,M,\zeta)$ satisfies the following properties:
\begin{enumerate}
\item $Q \in N \cap \mathcal Y$;
\item $Q \cap \kappa = \zeta$, $Q \cap \kappa^+ = A_{\eta,\zeta}$, 
and $\sup(Q) = \eta$;
\item $N \cap Q \cap \kappa^+ = M \cap N \cap \kappa^+$.
\end{enumerate}
\end{lemma}

\begin{proof}
Since $\eta$ and $\zeta$ are in $N$, 
$A_{\eta,\zeta}$ and $Q$ are in $N$ by elementarity. 
As $\alpha_{M,N}$ is a limit point of $N$, by Lemma 7.12, 
$C_{\alpha_{M,N}} = C_{\eta} \cap \alpha_{M,N}$, and 
for all $\xi < \kappa$, 
$$
A_{\alpha_{M,N},\xi} = A_{\eta,\xi} \cap \alpha_{M,N} \ \textrm{and} \ 
N \cap A_{\alpha_{M,N},\xi} = N \cap A_{\eta,\xi}.
$$

We claim that 
$$
N \cap A_{\eta,\zeta} = M \cap N \cap \kappa^+.
$$
Let $\alpha \in M \cap N \cap \kappa^+$, and we will 
show that $\alpha \in A_{\eta,\zeta}$. 
By Lemma 8.9, $M \cap N \cap \kappa^+ \subseteq 
A_{\alpha_{M,N},\sup(M \cap N \cap \kappa)}$. 
Since $\zeta \in R_M(N)$, 
$\sup(M \cap N \cap \kappa) < \beta_{M,N} \le \zeta$. 
So 
$$
\alpha \in M \cap N \cap \kappa^+ \subseteq 
A_{\alpha_{M,N},\sup(M \cap N \cap \kappa)} 
\subseteq A_{\alpha_{M,N},\zeta} \subseteq A_{\eta,\zeta}.
$$
Hence $\alpha \in A_{\eta,\zeta}$.

Conversely, let $\alpha \in N \cap A_{\eta,\zeta}$, and we will 
show that $\alpha \in M$. 
Since $\alpha$, $\eta,$ and $\zeta$ are in $N$ and 
$\zeta$ is a limit ordinal, 
by elementarity we can fix 
$\gamma \in N \cap \zeta$ such that $\alpha \in A_{\eta,\gamma}$. 
Since $\zeta = \min((N \cap \kappa) \setminus \beta_{M,N})$ and $N \le M$, 
$\gamma \in N \cap \beta_{M,N} = M \cap N \cap \kappa$. 
So 
$$
\alpha \in N \cap A_{\eta,\gamma} = N \cap A_{\alpha_{M,N},\gamma} 
\subseteq N \cap A_{\alpha_{M,N},\sup(M \cap N \cap \kappa)}.
$$
By Lemma 8.8(1), $\alpha \in M$.

We have proven that $N \cap A_{\eta,\zeta} = M \cap N \cap \kappa^+$. 
By Lemma 7.10, it follows that $A_{\eta,\zeta}$ is closed under $H^*$. 
In particular, $Q \cap \kappa^+ = A_{\eta,\zeta}$. 
Therefore 
$$
N \cap Q \cap \kappa^+ = N \cap A_{\eta,\zeta} 
= M \cap N \cap \kappa^+,
$$
which proves (3).

To show that $Q \cap \kappa = \zeta$ and $\sup(Q) = \eta$, it 
suffices to show that $N$ models these statements. 
Let $\gamma \in N \cap Q \cap \kappa$, and we will show that 
$\gamma < \zeta$.  
Then $\gamma \in N \cap Q \cap \kappa = M \cap N \cap \kappa$, 
so 
$$
\gamma < \sup(M \cap N \cap \kappa) < \beta_{M,N} \le \zeta.
$$ 
Conversely, let $\gamma \in N \cap \zeta$, 
and we will show that $\gamma \in Q$. 
Since $\zeta = \min((N \cap \kappa) \setminus \beta_{M,N})$, 
$\gamma \in N \cap \beta_{M,N}$. 
As $N \le M$, $N \cap \beta_{M,N} \subseteq M$, so $\gamma \in M$. 
Hence $\gamma \in M \cap N \cap \kappa^+ \subseteq Q$.

Since $Q \cap \kappa^+ = A_{\eta,\zeta}$, obviously $\sup(Q) \le \eta$. 
To show that $N$ models that $\sup(Q) = \eta$, let 
$\xi \in N \cap \eta$ be given. 
Since $\eta = \min((N \cap \kappa^+) \setminus \alpha_{M,N})$, $\xi < \alpha_{M,N}$. 
As $\alpha_{M,N} = \sup(M \cap N \cap \kappa^+)$, we can fix 
$\sigma \in M \cap N \cap \kappa^+$ with $\xi \le \sigma$. 
Then $\sigma \in M \cap N \cap \kappa^+ \subseteq Q$. 
So $\xi \le \sigma$ and $\sigma \in Q$. 
This completes the proof of (2).

To show that $Q \in \mathcal Y$, it suffices to show that 
$N$ models that $\lim(C_\eta) \cap Q$ is cofinal in $\eta$. 
Let $\xi \in N \cap \eta$. 
Then $\xi < \alpha_{M,N}$. 
Since $M \cap N$ is in $\mathcal X$ and $\sup(M \cap N) = \alpha_{M,N}$, 
we can fix 
$\sigma \in \lim(C_{\alpha_{M,N}}) \cap (M \cap N)$ with $\xi \le \sigma$. 
But $C_{\alpha_{M,N}} = C_{\eta} \cap \alpha_{M,N}$, 
so $\sigma \in \lim(C_{\eta})$. 
Also $\sigma \in M \cap N \cap \kappa^+ \subseteq Q$. 
So $\sigma \in \lim(C_{\eta}) \cap Q$ and $\xi \le \sigma$.
\end{proof}

\begin{notation}
Let $M$ and $N$ be in $\mathcal X$ such that $\{ M, N \}$ is adequate. 
Let $\zeta \in R_N(M)$. 
We let $Q_0(M,N,\zeta)$ denote the set $Sk(A_{\alpha_{M,N},\zeta})$.
\end{notation}

\begin{lemma}
Let $M$ and $N$ be in $\mathcal X$ such that $\{ M, N \}$ is adequate. 
Let $\eta \in M \cap N \cap \kappa^+$ with $\kappa \le \eta$. 
Fix $m < \omega$, and let $k$ be the arity of $\tau_m'$. 
Define $f_{m,\eta} : \kappa \to \kappa$ by letting 
$f_{m,\eta}(\beta)$ be the least $\beta' < \kappa$ such that 
$\beta \in A_{\eta,\beta'}$ and 
$$
\eta \cap \tau_m'[A_{\eta,\beta}^k] \subseteq A_{\eta,\beta'}.
$$
Then for all $\sigma \in R_N(M) \cup R_M(N)$, 
$\sigma$ is closed under $f_{m,\eta}$.
\end{lemma}

Note that since $A_{\eta,\beta}$ has size less than $\kappa$ by 
Notation 7.4(3), the set $\tau_m'[A_{\eta,\beta}^k]$ also 
has size less than $\kappa$. 
So the definition of $f_{m,\eta}$ makes sense. 
Also note that $f_{m,\eta}$ is definable from $\eta$ in $\mathcal A$.

\begin{proof}
The proof is by induction on remainder points in $R_M(N) \cup R_N(M)$. 
For the base case, let $\sigma$ be the first ordinal in 
$R_M(N) \cup R_N(M)$. 
Without loss of generality, assume that $\sigma \in R_N(M)$. 
Since $\eta \in M$ and $f_{m,\eta}$ is definable from $\eta$ in $\mathcal A$, 
it suffices to show that $M$ models that $\sigma$ 
is closed under $f_{m,\eta}$. 
So let $\beta \in M \cap \sigma$, and we will show that 
$f_{m,\eta}(\beta) < \sigma$.

Since $\sigma \in R_N(M)$ and $\sigma$ 
is the first remainder point, we have that $M \le N$ and 
$\sigma = \min((M \cap \kappa) \setminus \beta_{M,N})$. 
As $\beta \in M \cap \sigma = M \cap \beta_{M,N}$ and $M \le N$, 
$\beta \in M \cap \beta_{M,N} \subseteq N$. 
So $\beta \in N$. 
Therefore $\eta$ and $\beta$ are both in $M \cap N$. 
By elementarity, $f_{m,\eta}(\beta) \in M \cap N \cap \kappa$. 
But $M \cap N \cap \kappa = M \cap \beta_{M,N} \subseteq \sigma$. 
So $f_{m,\eta}(\beta) < \sigma$.

Now suppose that $\zeta$ is a remainder point which is greater than the 
least remainder point, and assume that the lemma holds 
for all remainder points in $R_M(N) \cup R_N(M)$ which are below $\zeta$. 
Without loss of generality, assume that $\zeta \in R_N(M)$. 
Then by Lemma 2.2(3), $\sigma := \min((N \cap \kappa) \setminus 
\sup(M \cap \zeta))$ is in $R_M(N)$, and 
$\zeta = \min((M \cap \kappa) \setminus \sigma)$. 
To show that $M$ models that $\zeta$ is closed under $f_{m,\eta}$, 
let $\beta \in M \cap \zeta$. 
Then $\beta < \sup(M \cap \zeta) < \sigma$. 
By the inductive hypothesis, $\sigma$ is closed under $f_{m,\eta}$. 
So $f_{m,\eta}(\beta) < \sigma$. 
Since $\sigma < \zeta$, $f_{m,\eta}(\beta) < \zeta$.
\end{proof}

\begin{lemma}
Let $M$ and $N$ be in $\mathcal X$ such that $\{ M, N \}$ is adequate. 
Let $\sigma \in R_N(M)$. 
Then $Q_0 := Q_0(M,N,\sigma)$ satisfies the following properties:
\begin{enumerate}
\item $Q_0 \in \mathcal Y$;
\item $Q_0 \cap \kappa = \sigma$, 
$Q_0 \cap \kappa^+ = A_{\alpha_{M,N},\sigma}$, 
and $\sup(Q_0) = \alpha_{M,N}$;
\item $M \cap N \cap \kappa^+ \subseteq Q_0$.
\end{enumerate}
\end{lemma}

\begin{proof}
Recall that $Q_0 = Q_0(M,N,\sigma) = Sk(A_{\alpha_{M,N},\sigma})$. 
We begin by proving that 
$A_{\alpha_{M,N},\sigma}$ is closed under $H^*$. 
Let $m < \omega$, and let $k$ be the arity of $\tau_m'$. 
Let $\alpha_0,\ldots,\alpha_{k-1} \in A_{\alpha_{M,N},\sigma}$, and 
we will show that $\tau_m'(\alpha_0,\ldots,\alpha_{k-1}) 
\in A_{\alpha_{M,N},\sigma}$. 
Since $\sigma$ is a limit ordinal, we can fix $\beta < \sigma$ 
such that $\alpha_0,\ldots,\alpha_{k-1} \in A_{\alpha_{M,N},\beta}$.

By the elementarity of $M \cap N$, 
fix $\delta \in M \cap N \cap \kappa^+$ strictly greater than 
$\alpha_0,\ldots,\alpha_{k-1}$ such that 
$\delta$ is closed under $\tau_m'$. 
Now fix $\eta \in \lim(C_{\alpha_{M,N}}) \cap (M \cap N)$ strictly greater 
than $\delta$ and $\kappa$. 
Since $\delta$ is closed under $\tau_m'$, 
$$
\tau_m'(\alpha_0,\ldots,\alpha_{k-1}) < \delta < \eta.
$$
As $\eta \in \lim(C_{\alpha_{M,N}})$, 
$$
\alpha_0,\ldots,\alpha_{k-1} \in A_{\alpha_{M,N},\beta} \cap \eta = 
A_{\eta,\beta}.
$$
So 
$$
\tau_m'(\alpha_0,\ldots,\alpha_{k-1}) \in \eta \cap \tau_m'[A_{\eta,\beta}^k].
$$
As $\beta < \sigma$ and $\sigma \in R_N(M)$, by Lemma 9.6 there is 
$\beta' < \sigma$ such that 
$$
\eta \cap \tau_m'[A_{\eta,\beta}^k] \subseteq A_{\eta,\beta'}.
$$
Then 
$$
\tau_m'(\alpha_0,\ldots,\alpha_{k-1}) \in A_{\eta,\beta'} 
\subseteq A_{\eta,\sigma} = 
A_{\alpha_{M,N},\sigma} \cap \eta \subseteq A_{\alpha_{M,N},\sigma}.
$$

This proves that $A_{\alpha_{M,N},\sigma}$ is closed under $H^*$. 
In particular, $Q_0 \cap \kappa^+ = A_{\alpha_{M,N},\sigma}$. 
Since $M \cap N \cap \kappa^+ 
\subseteq A_{\alpha_{M,N},\sup(M \cap N \cap \kappa)}$ 
by Lemma 8.9, and $\sup(M \cap N \cap \kappa) < \sigma$, 
it follows that $M \cap N \cap \kappa^+ \subseteq A_{\alpha_{M,N},\sigma} \subseteq Q_0$. 
In particular, since $\sup(M \cap N) = \alpha_{M,N}$, 
it follows that $\sup(Q_0) = \alpha_{M,N}$.

It remains to show that $Q_0 \in \mathcal Y$ and 
$Q_0 \cap \kappa = \sigma$. 
For the first statement, once we know that $Q_0 \cap \kappa = \sigma$, 
it will suffice to show that 
$\lim(C_{\alpha_{M,N}}) \cap Q_0$ is cofinal in $\alpha_{M,N}$. 
But since $M \cap N \in \mathcal X$, 
$\lim(C_{\alpha_{M,N}}) \cap (M \cap N)$ is cofinal in $\sup(M \cap N) = \alpha_{M,N}$. 
And as $M \cap N \cap \kappa^+ \subseteq Q_0$, it follows that 
$\lim(C_{\alpha_{M,N}}) \cap Q_0$ is cofinal in $\alpha_{M,N}$.

Now we prove that $Q_0 \cap \kappa = \sigma$. 
First we will show that $Q_0 \cap \kappa \subseteq \sigma$. 
More generally, we will prove by induction on remainder points that 
$$
\forall \zeta \in R_M(N) \cup R_N(M), \ A_{\alpha_{M,N},\zeta} \cap \kappa 
\subseteq \zeta.
$$

Consider the first remainder point $\zeta$. 
Without loss of generality, assume that $\zeta \in R_N(M)$. 
Then $M \le N$ and 
$\zeta = \min((M \cap \kappa) \setminus \beta_{M,N})$. 
Let $\beta \in A_{\alpha_{M,N},\zeta} \cap \kappa$, and we will show 
that $\beta < \zeta$. 
Fix $\eta \in \lim(C_{\alpha_{M,N}}) \cap (M \cap N)$ with 
$\beta < \eta$. 
Then $\beta \in A_{\alpha_{M,N},\zeta} \cap \eta = A_{\eta,\zeta}$. 
To show that $\beta < \zeta$, it suffices to show that 
$A_{\eta,\zeta} \cap \kappa \subseteq \zeta$. 
Since $\eta$ and $\zeta$ are in $M$, it is enough to show that $M$ 
models that 
$A_{\eta,\zeta} \cap \kappa \subseteq \zeta$.

Let $\beta' \in M \cap A_{\eta,\zeta} \cap \kappa$, and we will 
show that $\beta' < \zeta$. 
Since $\zeta$ is a limit ordinal, by elementarity we can fix 
$\gamma \in M \cap \zeta$ with $\beta' \in A_{\eta,\gamma}$. 
Since $\zeta = \min((M \cap \kappa) \setminus \beta_{M,N})$ and $M \le N$, 
$\gamma \in M \cap \beta_{M,N} \subseteq N$. 
As $M \le N$, $\beta' \in M \cap A_{\eta,\gamma}$, $\eta \in N$, and 
$\gamma \in M \cap N \cap \kappa$, it follows that 
$\beta' \in N$ by Lemma 8.6. 
Hence 
$$
\beta' \in M \cap N \cap \kappa \subseteq \beta_{M,N} \le \zeta.
$$
So $\beta' < \zeta$.

For the inductive step, let $\zeta$ be a remainder point which is 
not the first remainder point. 
Without loss of generality, assume that $\zeta \in R_N(M)$. 
Then by Lemma 2.2(3), 
there is $\pi \in R_M(N)$ such that 
$\pi = \min((N \cap \kappa) \setminus \sup(M \cap \zeta))$ and 
$\zeta = \min((M \cap \kappa) \setminus \pi)$. 
Let $\beta \in A_{\alpha_{M,N},\zeta} \cap \kappa$, and we will show 
that $\beta < \zeta$. 
Fix $\eta \in \lim(C_{\alpha_{M,N}}) \cap (M \cap N)$ with $\beta < \eta$. 
Then 
$$
\beta \in A_{\alpha_{M,N},\zeta} \cap \eta = A_{\eta,\zeta}.
$$

To show that $\beta < \zeta$, it suffices to show that 
$A_{\eta,\zeta} \cap \kappa \subseteq \zeta$. 
Since $\eta$ and $\zeta$ are in $M$, by elementarity it suffices 
to show that $M$ models that 
$A_{\eta,\zeta} \cap \kappa \subseteq \zeta$. 
So let $\gamma \in M \cap A_{\eta,\zeta} \cap \kappa$, and we 
will show that $\gamma < \zeta$. 
Since $\zeta$ is a limit ordinal, by elementarity we can fix 
$\alpha \in M \cap \zeta$ such that $\gamma \in A_{\eta,\alpha}$. 
Then $\alpha < \sup(M \cap \zeta) < \pi$. 
So $\gamma \in A_{\eta,\alpha} \subseteq A_{\eta,\pi}$. 
Since $\eta \in M \cap N$ and $\pi \in R_M(N)$, 
the inductive hypothesis implies that 
$$
A_{\eta,\pi} \cap \kappa = A_{\alpha_{M,N},\pi} \cap \kappa \subseteq \pi.
$$
So $\gamma < \pi < \zeta$.

This completes the induction. 
In particular, 
$$
Q_0 \cap \kappa = A_{\alpha_{M,N},\sigma} \cap \kappa 
\subseteq \sigma.
$$ 
Conversely, let $\beta < \sigma$, and we will show that $\beta \in Q_0$. 
Fix $\eta \in \lim(C_{\alpha_{M,N}}) \cap (M \cap N)$ with $\kappa \le \eta$. 
Then by Lemma 9.6, there is $\beta' < \sigma$ such that 
$\beta \in A_{\eta,\beta'}$. 
So 
$$
\beta \in A_{\eta,\beta'} = A_{\alpha_{M,N},\beta'} \cap \eta 
\subseteq A_{\alpha_{M,N},\beta'} \subseteq A_{\alpha_{M,N},\sigma} \subseteq Q_0.
$$
\end{proof}

\begin{lemma}
Let $M$ and $N$ be in $\mathcal X$, where $\{ M, N \}$ is adequate and 
$N$ is simple. 
Suppose that $\sigma \in R_N(M)$, $\zeta \in R_M(N)$, and 
$\zeta = \min((N \cap \kappa) \setminus \sigma)$. 
Let $\eta := \min((N \cap \kappa^+) \setminus \alpha_{M,N})$. 
Let $Q_0 := Q_0(M,N,\sigma)$ and $Q := Q(N,M,\zeta)$. 
Then:
\begin{enumerate}
\item $Q \in N \cap \mathcal Y$;

\item $Q \cap \kappa = \zeta$, 
$Q \cap \kappa^+ = A_{\eta,\zeta}$, and 
$\sup(Q) = \eta$;

\item $N \cap Q \cap \kappa^+ = N \cap Q_0 \cap \kappa^+$;

\item $M \cap N \cap \kappa^+ \subseteq Q$.
\end{enumerate}
\end{lemma}

\begin{proof}
We will apply Lemma 9.2 to the models $N$ and $Q_0$. 
Let us check that the assumptions of this lemma hold, using Lemma 9.7. 
We know that $N \in \mathcal X$ is simple, 
$Q_0 \in \mathcal Y$, 
and 
$$
Q_0 \cap \kappa = \sigma < \zeta < \sup(N \cap \kappa).
$$ 
Also, $\sup(N \cap Q_0) = \alpha_{M,N}$, since 
$\sup(Q_0) = \alpha_{M,N}$, $\sup(M \cap N) = \alpha_{M,N}$, 
and $M \cap N \cap \kappa^+ \subseteq N \cap Q_0$. 
Moreover, 
$$
\min((N \cap \kappa) \setminus (Q_0 \cap \kappa)) = 
\min((N \cap \kappa) \setminus \sigma) = \zeta,
$$
and 
$$
\min((N \cap \kappa^+) \setminus \sup(N \cap Q_0)) = 
\min((N \cap \kappa^+) \setminus \alpha_{M,N}) = \eta.
$$
By Notations 9.1 and 9.3, 
$$
Q(N,Q_0) = Sk(A_{\eta,\zeta}) = Q(N,M,\zeta) = Q.
$$

By Lemma 9.2, we have that:
\begin{enumerate}
\item[(a)] $Q \in N \cap \mathcal Y$;
\item[(b)] $Q \cap \kappa = \zeta$, $Q \cap \kappa^+ = A_{\eta,\zeta}$, 
and $\sup(Q) = \eta$;
\item[(c)] $N \cap Q \cap \kappa^+ = N \cap Q_0 \cap \kappa^+$.
\end{enumerate}
This proves (1), (2), and (3). 
By Lemma 9.7(3), $M \cap N \cap \kappa^+ \subseteq Q_0$. 
So $M \cap N \cap \kappa^+ \subseteq N \cap Q_0 \cap \kappa^+ = 
N \cap Q \cap \kappa^+ \subseteq Q$, which proves (4).
\end{proof}

The next lemma summarizes Lemmas 9.4 and 9.8.

\begin{lemma}
Let $M$ and $N$ be in $\mathcal X$ such that $\{ M, N \}$ is adequate 
and $N$ is simple. 
Let $\zeta \in R_M(N)$, 
$\eta := \min((N \cap \kappa^+) \setminus \alpha_{M,N})$, 
and $Q := Q(N,M,\zeta)$. 
Then:
\begin{enumerate}
\item $Q \in N \cap \mathcal Y$;
\item $Q \cap \kappa = \zeta$, $Q \cap \kappa^+ = A_{\eta,\zeta}$, 
and $\sup(Q) = \eta$;
\item $M \cap N \cap \kappa^+ \subseteq Q$;
\item if $\zeta = \min((N \cap \kappa) \setminus \beta_{M,N})$, then 
$N \cap Q \cap \kappa^+ = M \cap N \cap \kappa^+$;
\item if $\zeta = \min((N \cap \kappa) \setminus \sigma)$, where 
$\sigma \in R_N(M)$, then $N \cap Q \cap \kappa^+ = 
N \cap Q_0(M,N,\sigma) \cap \kappa^+$.
\end{enumerate}
\end{lemma}

\begin{proof}
Immediate from Lemmas 9.4 and 9.8.
\end{proof}

Let us derive some additional information about the model 
$Q(N,M,\zeta)$.

\begin{lemma}
Let $M$ and $N$ be in $\mathcal X$ such that $\{ M, N \}$ is adequate 
and $N$ is simple. 
Let $\zeta \in R_M(N)$, 
$\eta := \min((N \cap \kappa^+) \setminus \alpha_{M,N})$, 
and $Q := Q(N,M,\zeta)$. 
Then:
\begin{enumerate}
\item if $P \in M \cap \mathcal Y$ and 
$\sup(N \cap \zeta) < P \cap \kappa < \zeta$, then 
$N \cap P \cap \alpha_{M,N} \subseteq Q$;

\item if $N \le M$, $P \in M \cap \mathcal Y$, and 
$P \cap \kappa < \sup(M \cap N \cap \kappa)$, then 
$N \cap P \cap \kappa^+ \subseteq Q$;

\item if $M < N$ and $P \in M \cap N \cap \mathcal Y$, 
then $N \cap P \cap \kappa^+ \subseteq Q$.
\end{enumerate}
\end{lemma}

\begin{proof}
Note that since $\alpha_{M,N}$ is a limit point of $N$, 
for all $\xi < \kappa$, 
$A_{\alpha_{M,N},\xi} = A_{\eta,\xi} \cap \alpha_{M,N}$ by 
Lemma 7.12.

(1) Suppose that $P \in M \cap \mathcal Y$ and 
$\sup(N \cap \zeta) < P \cap \kappa < \zeta$. 
Since $\beta_{M,N} \le \zeta$, 
$$
M \cap N \cap \kappa = M \cap N \cap \beta_{M,N} 
\subseteq N \cap \zeta.
$$
So $\sup(M \cap N \cap \kappa) \le \sup(N \cap \zeta) < P \cap \kappa$. 
By Lemma 8.10, 
$$
P \cap N \cap \alpha_{M,N} \subseteq A_{\alpha_{M,N},P \cap \kappa} 
\subseteq A_{\alpha_{M,N},\zeta} = 
A_{\eta,\zeta} \cap \alpha_{M,N} \subseteq Q.
$$

(2) If $N \le M$, $P \in M \cap \mathcal Y$, and 
$P \cap \kappa < \sup(M \cap N \cap \kappa)$, then 
$N \cap P \cap \kappa^+ \subseteq M$ by Lemma 8.7. 
Hence 
$$
N \cap P \cap \kappa^+ \subseteq M \cap N \cap \kappa^+ 
\subseteq Q
$$
by Lemma 9.9(3).

(3) Suppose that $M < N$ and $P \in M \cap N \cap \mathcal Y$. 
Then $\sup(P)$ and $P \cap \kappa$ are in $M \cap N \cap \kappa^+$ 
by elementarity, 
and hence in $Q$ by Lemma 9.9(3). 
So $A_{\sup(P),P \cap \kappa} = P \cap \kappa^+ \in Q$ 
by elementarity. 
So $P \cap \kappa^+ \subseteq Q$. 
In particular, $N \cap P \cap \kappa^+ \subseteq Q$.
\end{proof}

\bigskip

Finally, we consider canonical models determined by ordinals in 
$R^+_N(M)$.

\begin{notation}
Let $M$ and $N$ be in $\mathcal X$, where $\{ M, N \}$ is adequate 
and $N$ is simple. 
Let $\zeta \in R_M(N)$ and $\sigma \in R^+_N(M)$. 
Let $X$ be any nonempty set of $P \in M \cap \mathcal Y$ such that 
$\sup(N \cap \zeta) < P \cap \kappa < \zeta$ and 
$P \cap N \cap [\sup(M \cap \sigma),\sigma) \ne \emptyset$. 
We let $P_X$ denote the set 
$Sk(\bigcup \{ P \cap \sigma : P \in X \})$ and 
$\beta_X$ denote the ordinal $P_X \cap \kappa$.
\end{notation}

\begin{lemma}
Under the assumptions of Notation 9.11, the following statements hold:
\begin{enumerate}
\item $P_X \in \mathcal Y$;
\item $\beta_X = \sup\{ P \cap \kappa : P \in X \} < \zeta$;
\item $P_X \cap \kappa^+ = \bigcup \{ P \cap \sigma : P \in X \}$;
\item $\sup(P_X) = \sigma$.
\end{enumerate}
\end{lemma}

\begin{proof}
Let $P \in X$. 
Since $\sigma \in R^+_N(M)$, $\sigma \in M$ and $\sigma$ 
has uncountable cofinality. 
Also $P \in M$ and 
$P \cap [\sup(M \cap \sigma),\sigma) \ne \emptyset$, which imply that 
$\sigma$ is a limit point of $P$ by Lemma 7.30. 
It follows that if $P_1$ and $P_2$ are in $X$ and 
$P_1 \cap \kappa \le P_2 \cap \kappa$, then 
$P_1 \cap \sigma \subseteq P_2 \cap \sigma$ by Lemma 7.28. 
Thus $\{ P \cap \sigma : P \in X \}$ is a subset increasing sequence. 
Since each $P \cap \sigma$ is closed under $H^*$ by Lemma 8.16, the set 
$\bigcup \{ P \cap \sigma : P \in X \}$ is closed under $H^*$. 
Hence 
$$
P_X \cap \kappa^+ = 
Sk(\bigcup \{ P \cap \sigma : P \in X \}) \cap \kappa^+ = 
\bigcup \{ P \cap \sigma : P \in X \},
$$
which proves (3).

Since $\sigma$ is a limit point of $P \cap \sigma$ for each $P \in X$, 
obviously $\sigma$ is a limit point of $P_X$. 
But $P_X \cap \kappa^+ \subseteq \sigma$, so $\sup(P_X) = \sigma$, 
which proves (4). 
Clearly 
$$
\beta_X = 
P_X \cap \kappa = \sup \{ P \cap \kappa : P \in X \},
$$
which is in $\kappa$. 
Since $P \cap \kappa < \zeta$ for all $P \in X$, 
it follows that $P_X \cap \kappa \le \zeta$. 
But $P \cap \kappa \in M \cap \kappa$ for all $P \in X$. 
Therefore 
$P_X \cap \kappa \in \cl(M \cap \kappa)$, which implies that 
$\beta_X = P_X \cap \kappa < \zeta$ by Lemma 2.2(1), 
which proves (2).

To show that $P_X \in \mathcal Y$, it suffices to show that 
$\lim(C_\sigma) \cap P_X$ is cofinal in $\sigma$.  
Fix $P \in X$. 
Then it will suffice to show that $\lim(C_\sigma) \cap P$ is cofinal in $\sigma$, 
since this set is a subset of $P_X$. 
First, assume that $\sigma \notin P$. 
Then $\sigma \in \cl(P) \setminus P$, which implies by Lemma 7.13 that 
$\lim(C_\sigma) \cap P$ is cofinal in $\sigma$. 
Secondly, assume that $\sigma \in P$. 
Since $\sigma$ is a limit point of $P$ and $|P| < \kappa$, 
$\cf(\sigma) < \kappa$. 
So $\ot(C_\sigma) < \kappa$. 
Hence $\ot(C_\sigma) \in P \cap \kappa$ and $P \cap \kappa \in \kappa$, 
which implies that $C_\sigma \subseteq P$. 
As $\sigma$ has uncountable cofinality, clearly $\lim(C_\sigma)$ is cofinal 
in $\sigma$. 
So $\lim(C_\sigma) \cap P$ is cofinal in $\sigma$.
\end{proof}

\begin{notation}
Under the assumptions of Notation 9.11, we let 
$$
Q(N,M,\zeta,\sigma,X) := Sk(A_{\eta_N,\zeta}),
$$
where $\eta := \sup(N \cap P_X)$.
\end{notation}

Note that $P_X \in \mathcal Y$ and 
$\beta_X = P_X \cap \kappa < \zeta < \sup(N \cap \kappa)$ imply by 
Lemma 8.13 that $\eta_N$ exists. 
Also since $\zeta = \min((N \cap \kappa) \setminus \beta_X)$, 
$Q(N,M,\zeta,\sigma,X)$ is equal to $Q(N,P_X)$ from Notation 9.1. 

\begin{lemma}
Let $M$ and $N$ be in $\mathcal X$, where $\{ M, N \}$ is adequate and 
$N$ is simple. 
Let $\zeta \in R_M(N)$ and $\sigma \in R^+_N(M)$.

Let $X$ be any nonempty set of $P \in M \cap \mathcal Y$ such that 
$\sup(N \cap \zeta) < P \cap \kappa < \zeta$ and 
$P \cap N \cap [\sup(M \cap \sigma),\sigma) \ne \emptyset$. 
Let $\eta := \sup(N \cap P_X)$ and $Q := Q(N,M,\zeta,\sigma,X)$. 
Then:
\begin{enumerate}
\item $Q \in N \cap \mathcal Y$;

\item $Q \cap \kappa = \zeta$, 
$Q \cap \kappa^+ = A_{\eta_N,\zeta}$, and $\sup(Q) = \eta_N$;

\item $N \cap Q \cap \kappa^+ = N \cap P_X \cap \kappa^+$;

\item for all $P \in X$, $N \cap P \cap \sigma \subseteq Q$.
\end{enumerate}
\end{lemma}

\begin{proof}
As noted above, $Q = Q(N,P_X)$. 
Also $\eta = \sup(N \cap P_X)$ and 
$\zeta = \min((N \cap \kappa) \setminus (P_X \cap \kappa))$. 
By Lemma 9.2:
\begin{enumerate}
\item[(a)] $Q \in N \cap \mathcal Y$;
\item[(b)] $Q \cap \kappa = \zeta$, $Q \cap \kappa^+ = A_{\eta_N,\zeta}$, 
and $\sup(Q) = \eta_N$;
\item[(c)] $N \cap P_X \cap \kappa^+ = N \cap Q \cap \kappa^+$.
\end{enumerate}
This proves (1), (2), and (3). 
In particular, if $P \in X$, then 
$$
N \cap P \cap \sigma \subseteq 
N \cap P_X \cap \kappa^+ \subseteq Q,
$$
which proves (4).
\end{proof}

\bigskip

\addcontentsline{toc}{section}{10. Closure under canonical models}

\textbf{\S 10. Closure under canonical models}

\stepcounter{section}

\bigskip

Fix a sequence $\langle S_\eta : \eta < \kappa^+ \rangle$, where each 
$S_\eta$ is a subset of $\kappa \cap \cof(> \! \omega)$. 
Let us assume that the structure $\mathcal A$ from Notation 7.6 
includes $\vec S$ as a predicate. 
In this section we will show that we can add canonical models to 
an $\vec S$-obedient side condition and preserve 
$\vec S$-obediency.

As stated in the comments prior to Definition 5.2, 
the definitions of $\vec S$-adequate and $\vec S$-obedient are 
made relative to a subclass of $\mathcal Y_0$. 
For the remainder of the paper, this subclass will be the set $\mathcal Y$ from 
Notation 7.8.

\begin{lemma}
Let $(A,B)$ be an $\vec S$-obedient side condition. 
Suppose that $N \in A$ is simple. 
Let $P \in B$ be such that $P \cap \kappa < \sup(N \cap \kappa)$. 
Let $Q := Q(N,P)$. 
Then $(A,B \cup \{ Q \})$ is an $\vec S$-obedient side condition.
\end{lemma}

See Notation 9.1 for the definition of $Q(N,P)$. 

\begin{proof}
Let $\beta := P \cap \kappa$. 
By Lemma 9.2, 
$$
Q \in N \cap \mathcal Y, \ 
Q \cap \kappa = \beta_N, \ \textrm{and} \ 
N \cap Q \cap \kappa^+ = N \cap P \cap \kappa^+.
$$
Let us show that $Q$ is $\vec S$-strong. 
Since $Q \in N$, it suffices to show that $N$ models that 
$Q$ is $\vec S$-strong. 
Let $\tau \in N \cap Q \cap \kappa^+$, and we will show that 
$Q \cap \kappa = \beta_N \in S_\tau$. 
But $\tau \in N \cap Q \cap \kappa^+ = N \cap P \cap \kappa^+$. 
Since $N \in A$, $P \in B$, and $(A,B)$ is $\vec S$-obedient, 
it follows that $\beta_N \in S_\tau$.

Let $M \in A$, and suppose that 
$\zeta = \min((M \cap \kappa) \setminus \beta_N)$. 
Fix $\tau \in M \cap Q \cap \kappa^+$, and we will show that $\zeta \in S_\tau$. 
If $\zeta = \beta_N$, then $\zeta \in S_\tau$ because 
$Q$ is $\vec S$-strong. 
Assume that $\beta_N < \zeta$, which means that $\beta_N \notin M$. 

First assume that $\zeta \in R_N(M)$. 
Then since $Q \in N \cap \mathcal Y$ is $\vec S$-strong and 
$\sup(M \cap \zeta) < Q \cap \kappa = \beta_N < \zeta$, 
it follows that 
$\zeta \in S_\tau$ as $A$ is $\vec S$-adequate. 
In particular, if $\beta_{M,N} \le \beta_N$, then 
$\zeta \in R_N(M)$. 
Suppose that $\beta_N < \beta_{M,N} \le \zeta$. 
Then $\zeta = \min((M \cap \kappa) \setminus \beta_{M,N})$. 
Since $\beta_N \in (N \cap \beta_{M,N}) \setminus M$, 
we have that $M < N$. 
So $\zeta = \min((M \cap \kappa) \setminus \beta_{M,N})$ is in $R_N(M)$.

The remaining case is that $\zeta < \beta_{M,N}$. 
Then since $\beta_N \in (N \cap \beta_{M,N}) \setminus M$, 
it follows that $M < N$. 
So 
$$
M \cap \zeta \subseteq M \cap \beta_{M,N} \subseteq N.
$$
As $\tau \in M \cap Q \cap \kappa^+$, $Q \in N \cap \mathcal Y$, and 
$$
Q \cap \kappa < \zeta < \sup(M \cap \beta_{M,N}) = 
\sup(M \cap N \cap \kappa),
$$
it follows that $\tau \in N$ by Lemma 8.7. 
So $\tau \in N \cap Q \cap \kappa^+ = N \cap P \cap \kappa^+$. 
Since $M \cap \zeta \subseteq N$ and 
$\zeta = \min((M \cap \kappa) \setminus \beta_N)$, 
we have that 
$$
\zeta = \min((M \cap \kappa) \setminus \beta) = 
\min((M \cap \kappa) \setminus (P \cap \kappa)).
$$
Since $M \in A$, $P \in B$, and 
$\tau \in M \cap P \cap \kappa^+$, 
it follows that $\zeta \in S_\tau$ as $(A,B)$ is $\vec S$-obedient.
\end{proof}

\begin{lemma}
Let $(A,B)$ be an $\vec S$-obedient side condition. 
Let $N \in A$ be simple and $M \in A$. 
Suppose that $N \le M$ and 
$\zeta = \min((N \cap \kappa) \setminus \beta_{M,N})$. 
Let $Q := Q(N,M,\zeta)$. 
Then $(A,B \cup \{ Q \})$ is an $\vec S$-obedient side condition.
\end{lemma}

See Notation 9.3 for the definition of $Q(N,M,\zeta)$.

\begin{proof}
By Lemma 9.4, 
$$
Q \in N \cap \mathcal Y, \ 
Q \cap \kappa = \zeta, \ \textrm{and} \ 
N \cap Q \cap \kappa^+ = M \cap N \cap \kappa^+.
$$

First we show that $Q$ is $\vec S$-strong. 
Since $Q \in N$, it suffices to show that $N$ models that $Q$ is 
$\vec S$-strong. 
Let $\tau \in N \cap Q \cap \kappa^+$, and we will show 
that $Q \cap \kappa = \zeta$ is in $S_\tau$. 
Then $\tau \in N \cap Q \cap \kappa^+ = M \cap N \cap \kappa^+$. 
So $\tau \in M \cap N$. 
Since $\zeta \in R_M(N)$, $\zeta \in S_\tau$ 
as $A$ is $\vec S$-adequate.

Now let $K \in A$, and suppose that 
$\theta = \min((K \cap \kappa) \setminus \zeta)$. 
Fix $\tau \in K \cap Q \cap \kappa^+$, 
and we will show that $\theta \in S_\tau$. 
If $\zeta = \theta$, then $\theta \in S_\tau$ 
since $Q$ is $\vec S$-strong. 
So assume that $\zeta < \theta$, which means that 
$\zeta \notin K$.

Suppose first that $\theta \in R_N(K)$. 
Then since $Q \in N \cap \mathcal Y$ is $\vec S$-strong and 
$\sup(K \cap \theta) < Q \cap \kappa = \zeta < \theta$, 
it follows that $\theta \in S_\tau$ as $A$ is $\vec S$-adequate. 
In particular, if $\beta_{K,N} \le \zeta$, then $\theta \in R_N(K)$. 
Suppose that $\zeta < \beta_{K,N} \le \theta$. 
Then $\theta = \min((K \cap \kappa) \setminus \beta_{K,N})$. 
Since $\zeta \in (N \cap \beta_{K,N}) \setminus K$, 
we have that $K < N$. 
So $\theta = \min((K \cap \kappa) \setminus \beta_{K,N})$ is in $R_N(K)$.

The remaining case is that $\theta < \beta_{K,N}$. 
We apply Lemma 2.7. 
We have that $\{ K, M, N \}$ is adequate, $\zeta \in R_M(N)$, 
$\zeta \notin K$, $\theta = \min((K \cap \kappa) \setminus \zeta)$, 
and $\theta < \beta_{K,N}$. 
By Lemma 2.7, $\theta \in R_M(K)$. 
Since $\zeta \in (N \cap \beta_{K,N}) \setminus K$, it follows that $K < N$. 
As $Q \in N \cap \mathcal Y$, $K < N$, $\tau \in K \cap Q$, and 
$$
Q \cap \kappa = \zeta < \theta < \sup(K \cap \beta_{K,N}) = 
\sup(K \cap N \cap \kappa),
$$
it follows that $\tau \in N$ by Lemma 8.7. 
So $\tau \in N \cap Q \cap \kappa^+ = M \cap N \cap \kappa^+$. 
Hence $\tau \in K \cap M$. 
Since $\theta \in R_M(K)$, 
it follows that $\theta \in S_\tau$ 
as $A$ is $\vec S$-adequate.
\end{proof}

\begin{lemma}
Let $M$ and $N$ be in $\mathcal X$ such that $\{ M, N \}$ is adequate and 
$N$ is simple. 
Assume that $M \prec (\mathcal A,\mathcal Y)$. 
Let $\sigma \in R_N(M)$ and $\zeta \in R_M(N)$. 
Then $Q_0(M,N,\sigma)$ and $Q(N,M,\zeta)$ are $\vec S$-strong.
\end{lemma}

Recall that $(\mathcal A,\mathcal Y)$ is the structure $\mathcal A$ augmented 
with the additional predicate $\mathcal Y$. 

See Notations 9.3 and 9.5 for the definitions of $Q(N,M,\zeta)$ 
and $Q_0(M,N,\sigma)$. 

\begin{proof}
The proof is by induction on remainder points in $R_M(N) \cup R_N(M)$. 
First consider $\zeta \in R_M(N)$. 
If $\zeta = \min((N \cap \kappa) \setminus \beta_{M,N})$, 
then $Q(N,M,\zeta)$ is $\vec S$-strong by Lemma 10.2. 
So assume that 
$\zeta = \min((N \cap \kappa) \setminus \sigma)$, for some 
$\sigma \in R_N(M)$.

Let $Q := Q(N,M,\zeta)$ and $Q_0 := Q_0(M,N,\sigma)$. 
By the inductive hypothesis, $Q_0$ is $\vec S$-strong. 
And by Lemma 9.7, 
$$
Q_0 \cap \kappa = \sigma \ \textrm{and} \ 
Q_0 \cap \kappa^+ = A_{\alpha_{M,N},\sigma}.
$$
To show that $Q$ is $\vec S$-strong, 
it suffices to prove that $N$ models that $Q$ is $\vec S$-strong. 
Let $\tau \in N \cap Q \cap \kappa^+$, and we will show that 
$Q \cap \kappa \in S_\tau$.  
By Lemma 9.8, 
$$
Q \cap \kappa = \zeta \ \textrm{and} \ Q \cap \kappa^+ = 
A_{\eta,\zeta},
$$
where $\eta := \min((N \cap \kappa^+) \setminus \alpha_{M,N})$, 
and 
$$
N \cap Q \cap \kappa^+ = N \cap Q_0 \cap \kappa^+.
$$
In particular, $\tau \in N \cap Q_0$. 
Also 
$$
\tau \in N \cap A_{\eta,\zeta} \subseteq \alpha_{M,N},
$$
so $\tau < \alpha_{M,N}$.

Fix $\theta \in \lim(C_{\alpha_{M,N}}) \cap (M \cap N)$ 
greater then $\tau$. 
Then 
$$
\tau \in Q_0 \cap \theta = 
A_{\alpha_{M,N},\sigma} \cap \theta = A_{\theta,\sigma}.
$$
Since $\theta$ and $\sigma$ are in $M$, $Q_0$ is $\vec S$-strong, 
$Q_0 \cap \kappa = \sigma$, and $A_{\theta,\sigma} \subseteq Q_0$, 
by the elementarity of $M$ 
we can fix an $\vec S$-strong model $P \in M \cap \mathcal Y$ 
such that $P \cap \kappa = \sigma$ and $A_{\theta,\sigma} \subseteq P$. 
Then $\tau \in N \cap P$. 
Since $\zeta \in R_M(N)$ and 
$\sup(N \cap \zeta) < \sigma = P \cap \kappa < \zeta$, 
it follows that 
$\zeta \in S_\tau$ as $A$ is $\vec S$-adequate.

Now consider $\sigma \in R_N(M)$, and we will show that 
$Q_0 := Q_0(M,N,\sigma)$ is $\vec S$-strong. 
We first claim that for all $\theta \in \lim(C_{\alpha_{M,N}}) \cap (M \cap N)$ 
and for all $\tau \in A_{\theta,\sigma}$, $\sigma \in S_\tau$. 
So fix $\theta \in \lim(C_{\alpha_{M,N}}) \cap (M \cap N)$. 
Since $\theta$ and $\sigma$ are in $M$, it suffices to prove that $M$ models 
that for all $\tau \in A_{\theta,\sigma}$, $\sigma \in S_\tau$. 
Let $\tau \in M \cap A_{\theta,\sigma}$. 
Since $\sigma$ is a limit ordinal, 
by elementarity we can fix $\gamma \in M \cap \sigma$ 
such that $\tau \in A_{\theta,\gamma}$.

If $\sigma = \min((M \cap \kappa) \setminus \beta_{M,N})$, 
then $M \le N$ and 
$\gamma \in M \cap \beta_{M,N} \subseteq N$. 
So $\theta$ is in $N$, 
$\gamma < \sup(M \cap N \cap \kappa)$, and 
$\tau \in M \cap A_{\theta,\gamma}$, which by Lemma 8.6 implies that 
$\tau \in N$. 
So $\tau \in M \cap N \cap \kappa^+$. 
As $\sigma \in R_N(M)$, it follows that 
$\sigma \in S_\tau$ as $A$ is $\vec S$-adequate.

Otherwise there is $\zeta \in R_M(N)$ such that 
$\sigma = \min((M \cap \kappa) \setminus \zeta)$. 
By the inductive hypothesis, $Q := Q(N,M,\zeta)$ is $\vec S$-strong. 
Since $\alpha_{M,N}$ is a limit point of $M \cap N$ and 
$M \cap N \cap \kappa^+ \subseteq Q$ by Lemma 9.9(3), 
it follows that $\alpha_{M,N}$ is a limit point of $Q$. 
So 
$$
Q \cap \alpha_{M,N} = A_{\alpha_{M,N},Q \cap \kappa}
$$
by Lemma 7.27. 
By Lemma 9.9(2), $Q \cap \kappa = \zeta$. 
So 
$$
Q \cap \alpha_{M,N} = A_{\alpha_{M,N},\zeta}.
$$
Now $\gamma \in M \cap \sigma \subseteq \zeta$. 
Hence 
$$
\tau \in M \cap A_{\theta,\gamma} \subseteq M \cap A_{\theta,\zeta} 
= M \cap A_{\alpha_{M,N},\zeta} \cap \theta \subseteq M \cap Q.
$$
So we have that $Q \in N \cap \mathcal Y$ is $\vec S$-strong, 
$\sup(M \cap \sigma) < \zeta = Q \cap \kappa < \sigma$, 
and $\tau \in M \cap Q$. 
Since $\sigma \in R_N(M)$, it follows that 
$\sigma \in S_\tau$ as $A$ is $\vec S$-adequate.

This completes the proof of the claim that 
for all $\theta \in \lim(C_{\alpha_{M,N}}) \cap (M \cap N)$, 
for all $\tau \in A_{\theta,\sigma}$, $\sigma \in S_\tau$. 
Now we show that $Q_0$ is $\vec S$-strong. 
By Lemma 9.7, $Q_0 \cap \kappa^+ = A_{\alpha_{M,N},\sigma}$. 
Let $\tau \in Q_0 \cap \kappa^+$. 
Then $\tau < \alpha_{M,N}$. 
Fix $\theta \in \lim(C_{\alpha_{M,N}}) \cap (M \cap N)$ which is 
greater than $\tau$. 
Then 
$$
\tau \in Q_0 \cap \theta = A_{\alpha_{M,N},\sigma} \cap \theta = 
A_{\theta,\sigma}.
$$
By the claim, $\sigma \in S_\tau$.
\end{proof}

\begin{lemma}
Let $(A,B)$ be an $\vec S$-obedient side condition. 
Suppose that $N \in A$ is simple, $M \in A$, and 
$M \prec (\mathcal A,\mathcal Y)$. 
Let $\zeta \in R_M(N)$. 
Let $Q := Q(N,M,\zeta)$. 
Then $(A,B \cup \{ Q \})$ is an $\vec S$-obedient side condition.
\end{lemma}

\begin{proof}
If $\zeta = \min((N \cap \kappa) \setminus \beta_{M,N})$, then we are done by Lemma 10.2. 
So assume that $\zeta = \min((N \cap \kappa) \setminus \sigma)$, 
where $\sigma \in R_N(M)$. 
Let $Q_0 := Q_0(M,N,\sigma)$. 
By Lemma 9.8, 
$$
Q \cap \kappa = \zeta \ \textrm{and} \ 
N \cap Q \cap \kappa^+ = N \cap Q_0 \cap \kappa^+.
$$
By Lemma 9.7, 
$$
Q_0 \cap \kappa = \sigma \ \textrm{and} \ 
Q_0 \cap \kappa^+ = A_{\alpha_{M,N},\sigma}.
$$
Also $Q_0$ and $Q$ are $\vec S$-strong by Lemma 10.3.

Suppose that $K \in A$ and $\theta = \min((K \cap \kappa) \setminus \zeta)$. 
Let $\tau \in K \cap Q \cap \kappa^+$, and we will show that $\theta \in S_\tau$. 
If $\zeta = \theta$, then $\theta \in S_\tau$ since 
$Q$ is $\vec S$-strong. 
So assume that $\zeta < \theta$, which means that 
$\zeta \notin K$.

First consider the case that $\theta \in R_N(K)$. 
Then since 
$$
\sup(K \cap \theta) < \zeta = Q \cap \kappa < \theta
$$
and $Q \in N \cap \mathcal Y$ is $\vec S$-strong, it follows that 
$\theta \in S_\tau$ as $A$ is $\vec S$-adequate. 
In particular, if $\beta_{K,N} \le \zeta$, then $\theta \in R_N(K)$. 
Suppose that $\zeta < \beta_{K,N} \le \theta$. 
Then $\theta = \min((K \cap \kappa) \setminus \beta_{K,N})$. 
Since $\zeta \in (N \cap \beta_{K,N}) \setminus K$, we have that $K < N$. 
So $\theta = \min((K \cap \kappa) \setminus \beta_{K,N})$ 
is in $R_N(K)$.

The remaining case is that $\theta < \beta_{K,N}$. 
We apply Lemma 2.7. 
We have that $\{ K, M, N \}$ is adequate, 
$\zeta \in R_M(N)$, 
$\zeta \notin K$, 
$\theta = \min((K \cap \kappa) \setminus \zeta)$, 
and $\theta < \beta_{K,N}$. 
By Lemma 2.7, 
$\theta \in R_M(K)$. 
Since $\zeta \in (N \cap \beta_{K,N}) \setminus K$, we have that $K < N$. 
As $Q \in N \cap \mathcal Y$, 
$$
Q \cap \kappa = \zeta < \theta < 
\sup(K \cap \beta_{K,N}) = \sup(K \cap N \cap \kappa),
$$
and $\tau \in K \cap Q$, it follows that $\tau \in N$ by Lemma 8.7. 
So 
$$
\tau \in N \cap Q \cap \kappa^+ = N \cap Q_0 \cap \kappa^+.
$$
Hence $\tau \in K \cap Q_0$. 
Since $Q_0 \cap \kappa^+ = A_{\alpha_{M,N},\sigma}$, it follows 
that $\tau < \alpha_{M,N}$.

Fix $\pi \in \lim(C_{\alpha_{M,N}}) \cap (M \cap N)$ with $\tau < \pi$. 
Then 
$$
\tau \in Q_0 \cap \pi = A_{\alpha_{M,N},\sigma} \cap \pi = 
A_{\pi,\sigma}.
$$
Since $\pi$ and $\sigma$ are in $M$, 
$Q_0$ is $\vec S$-strong, $Q_0 \cap \kappa = \sigma$, and 
$A_{\pi,\sigma} \subseteq Q_0$, 
by the elementarity of $M$ 
we can fix $P \in M \cap \mathcal Y$ which is $\vec S$-strong 
such that $P \cap \kappa = \sigma$ and $A_{\pi,\sigma} \subseteq P$. 
In particular, $\tau \in P$. 
Since $K \cap \theta \subseteq N$ and 
$\zeta = \min((N \cap \kappa) \setminus \sigma)$, 
clearly $\theta = \min((K \cap \kappa) \setminus \sigma)$. 
So $\tau \in K \cap P$, $P \in M \cap \mathcal Y$ is $\vec S$-strong, 
and $\sup(K \cap \theta) < \sigma = P \cap \kappa < \theta$. 
Since $\theta \in R_M(K)$, it follows that 
$\theta \in S_\tau$ as $A$ is $\vec S$-adequate.
\end{proof}

\begin{notation}
Let $M$ and $N$ be in $\mathcal X$, where $\{ M, N \}$ is adequate and $N$ is simple. 
Let $\zeta \in R_M(N)$ and $\sigma \in R^+_N(M)$. 
Let $X$ be the set of $P \in M \cap \mathcal Y$ such that 
$P$ is $\vec S$-strong, $\sup(N \cap \zeta) < P \cap \kappa < \zeta$, 
and $P \cap N \cap [\sup(M \cap \sigma),\sigma) \ne \emptyset$. 
Assume that $X$ is nonempty. 
We let $Q(N,M,\zeta,\sigma,\vec S)$ denote the set 
$Q(N,M,\zeta,\sigma,X)$.
\end{notation}

See Notation 9.13 for the definition of $Q(N,M,\zeta,\sigma,X)$.

\begin{lemma}
Let $(A,B)$ be an $\vec S$-obedient side condition. 
Let $M$ and $N$ be in $A$, where $N$ is simple. 
Let $\zeta \in R_M(N)$ and $\sigma \in R^+_N(M)$. 
Let $Q := Q(N,M,\zeta,\sigma,\vec S)$. 
Then $(A,B \cup \{ Q \})$ is an $\vec S$-obedient side condition.
\end{lemma}

\begin{proof}
Let $X$ be as in Notation 10.5, and let 
$P_X$ be as in Notation 9.11. 
Then by Lemma 9.14, 
$$
Q \in N \cap \mathcal Y, \ 
Q \cap \kappa = \zeta, \ \textrm{and} \ 
N \cap Q \cap \kappa^+ = 
N \cap P_X \cap \kappa^+.
$$

Let us prove that $Q$ is $\vec S$-strong. 
Since $Q \in N$, it suffices to show that $N$ models that $Q$ 
is $\vec S$-strong. 
Fix $\tau \in N \cap Q \cap \kappa^+$, 
and we will show that $Q \cap \kappa = \zeta \in S_\tau$. 
Since $N \cap Q \cap \kappa^+ = N \cap P_X \cap \kappa^+$, 
we have that $\tau \in P_X$. 
By the definition of $P_X$, for some $P \in X$, $\tau \in P$. 
But then $\sup(N \cap \zeta) < P \cap \kappa < \zeta$, 
$P \in M \cap \mathcal Y$ is $\vec S$-strong, and $\tau \in N \cap P$. 
Since $\zeta \in R_M(N)$, this implies that $\zeta \in S_\tau$ as $A$ 
is $\vec S$-adequate.

Let $K \in A_p$, and suppose 
that $\theta = \min((K \cap \kappa) \setminus \zeta)$. 
Fix $\tau \in K \cap Q \cap \kappa^+$, and we will show that $\theta \in S_\tau$. 
If $\theta = \zeta$, then $\theta \in S_\tau$ since $Q$ is $\vec S$-strong.
So assume that $\zeta < \theta$, which means that 
$\zeta \notin K$.

If $\theta \in R_N(K)$, then since 
$Q \in N \cap \mathcal Y$ is $\vec S$-strong, 
$\sup(K \cap \theta) < Q \cap \kappa < \theta$, and $\tau \in K \cap Q$, 
it follows that 
$\theta \in S_\tau$ as $A$ is $\vec S$-adequate. 
In particular, if $\beta_{K,N} \le \zeta$, then $\theta \in R_N(K)$. 
Suppose that $\zeta < \beta_{K,N} \le \theta$. 
Then $\theta = \min((K \cap \kappa) \setminus \beta_{K,N})$. 
Since $\zeta \in (N \cap \beta_{K,N}) \setminus K$, we 
have that $K < N$, which implies that 
$\theta \in R_N(K)$.

The remaining case is that $\theta < \beta_{K,N}$. 
We apply Lemma 2.7. 
We have that $\{ K, M, N \}$ is adequate, $\zeta \in R_M(N)$, 
$\zeta \notin K$, $\theta = \min((K \cap \kappa) \setminus \zeta)$, 
and $\theta < \beta_{K,N}$. 
By Lemma 2.7, $\theta \in R_M(K)$. 
Since $\zeta \in (N \cap \beta_{K,N}) \setminus K$, 
we have that $K < N$. 
As $Q \in N \cap \mathcal Y$, 
$$
Q \cap \kappa = \zeta < \theta < \sup(K \cap \beta_{K,N}) = 
\sup(K \cap N \cap \kappa),
$$
and $\tau \in K \cap Q$, 
it follows that $\tau \in N$ by Lemma 8.7. 
So 
$$
\tau \in N \cap Q \cap \kappa^+ = N \cap P_X \cap \kappa^+.
$$
By the definition of $P_X$, there is $P \in X$ such that $\tau \in P$. 
Since $\sup(N \cap \zeta) < P \cap \kappa < \zeta$ and 
$K \cap \theta \subseteq N$, clearly 
$\sup(K \cap \theta) < P \cap \kappa < \theta$. 
As $P \in M \cap \mathcal Y$ is $\vec S$-strong, $\tau \in K \cap P$, and 
$\theta \in R_M(K)$, 
it follows that $\theta \in S_\tau$ since $A$ is $\vec S$-adequate.
\end{proof}

\begin{definition}
Let $(A,B)$ be an $\vec S$-obedient side condition. 
Suppose that $N \in A$ is simple. 
We say that $(A,B)$ is \emph{closed under canonical models with respect 
to $N$} if:
\begin{enumerate}
\item for all $P \in B$ with $P \cap \kappa < \sup(N \cap \kappa)$, 
$Q(N,P) \in B$;
\item for all $M \in A$ and $\zeta \in R_M(N)$, 
$Q(N,M,\zeta) \in B$;
\item for all $M \in A$, $\zeta \in R_M(N)$, and $\sigma \in R^+_N(M)$, 
$Q(N,M,\zeta,\sigma,\vec S) \in B$.
\end{enumerate}
\end{definition}

\begin{proposition}
Let $(A,B)$ be an $\vec S$-obedient side condition such that 
for all $M \in A$, $M \prec (\mathcal A,\mathcal Y)$. 
Suppose that $N \in A$ is simple. 
Then there exists $(A,C)$ such that $B \subseteq C$, 
$(A,C)$ is an $\vec S$-obedient 
side condition, and $(A,C)$ is closed under canonical models with 
respect to $N$.
\end{proposition}

\begin{proof}
First apply Lemma 10.1 finitely many times 
to obtain $C_0$ such that $B \subseteq C_0$, $(A,C_0)$ is an 
$\vec S$-obedient side condition, and $(A,C_0)$ satisfies property (1) 
of Definition 10.7. 
Then apply Lemmas 10.4 and 10.6 finitely many times to obtain $C$ 
such that $C_0 \subseteq C$, $(A,C)$ is an $\vec S$-obedient side condition, 
and $(A,C)$ satisfies properties (2) and (3) of Definition 10.7. 
Since all of the models which are added are in $N$, and for all 
$P \in N \cap \mathcal Y$, $Q(N,P) = P$, it follows that 
$(A,C)$ also satisfies property (1) of Definition 10.7.
\end{proof}

\begin{lemma}
Suppose that $(A,B)$ is an $\vec S$-obedient side condition, 
and $N \in A$ is simple. 
Assume that $(A,B)$ is closed under canonical models with respect 
to $N$. 
Then:
\begin{enumerate}
\item Suppose that $P \in B$, 
$P \cap \kappa < \sup(N \cap \kappa)$, 
and $\tau \in N \cap P \cap \kappa^+$. 
Then there is $Q \in B \cap N$ such that 
$Q \cap \kappa = 
\min((N \cap \kappa) \setminus (P \cap \kappa))$ and $\tau \in Q$.

\item Suppose that $M \in A$ and $\zeta \in R_M(N)$. 
Then there is $Q \in B \cap N$ 
such that $Q \cap \kappa = \zeta$ 
and $M \cap N \cap \kappa^+ \subseteq Q$.

\item Suppose that $M \in A$, $M < N$, and $\zeta \in R_M(N)$. 
Then there is 
$Q \in B \cap N$ 
such that $Q \cap \kappa = \zeta$, and 
for all $P \in M \cap N \cap \mathcal Y$ which is $\vec S$-strong, 
$N \cap P \cap \kappa^+ \subseteq Q$.

\item Suppose that $M \in A$, $\zeta \in R_M(N)$, 
$P \in M \cap \mathcal Y$ is $\vec S$-strong, 
$\sup(N \cap \zeta) < P \cap \kappa < \zeta$, and 
$\tau \in N \cap P \cap \kappa^+$. 
Then there is $Q \in B \cap N$ such that $Q \cap \kappa = \zeta$ 
and $\tau \in Q$.
\end{enumerate}
\end{lemma}

\begin{proof}
(1) Suppose that 
$$
P \in B, \ 
P \cap \kappa < \sup(N \cap \kappa), \ \textrm{and} \ 
\tau \in N \cap P \cap \kappa^+.
$$
Then $Q(N,P) \in B \cap N$. 
By Lemma 9.2, 
$$
Q(N,P) \cap \kappa = 
\min((N \cap \kappa) \setminus (P \cap \kappa)) \ \textrm{and} \ 
N \cap P \cap \kappa^+ \subseteq Q(N,P).
$$ 
In particular, $\tau \in Q(N,P)$.

\bigskip

(2,3) Let $M \in A$ and $\zeta \in R_M(N)$. 
Let $Q := Q(N,M,\zeta)$. 
Then $Q \in B \cap N$. 
By Lemma 9.9, 
$$
Q \cap \kappa = \zeta \ \textrm{and} \ M \cap N \cap \kappa^+ 
\subseteq Q,
$$
which proves (2). 
If in addition $M < N$, then by Lemma 9.10(3), 
for all $P \in M \cap N \cap \mathcal Y$, 
$N \cap P \cap \kappa^+ \subseteq Q$, which proves (3).

\bigskip

(4) Suppose that $M \in A$, $\zeta \in R_M(N)$, 
$P \in M \cap \mathcal Y$ is $\vec S$-strong, 
$\sup(N \cap \zeta) < P \cap \kappa < \zeta$, and 
$\tau \in N \cap P \cap \kappa^+$. 
First assume that $\tau < \alpha_{M,N}$. 
Then $Q(N,M,\zeta) \in B \cap N$, and 
$Q(N,M,\zeta) \cap \kappa = \zeta$ by Lemma 9.9. 
Also by Lemma 9.10(1), 
$$
N \cap P \cap \alpha_{M,N} \subseteq Q(N,M,\zeta).
$$
Hence $\tau \in Q(N,M,\zeta)$.

Assume that $\alpha_{M,N} \le \tau$. 
Note that $\sigma := \tau_M$ exists since $\tau < \sup(P) \in M$. 
As $\tau \in N$, $\sigma$ is in $R^+_N(M)$. 
So $\tau \in N \cap P \cap \sigma$. 
Let $Q := Q(N,M,\zeta,\sigma,\vec S)$, which is in $B \cap N$. 
Then 
$$
Q \cap \kappa = \zeta \ \textrm{and} \ 
N \cap P \cap \sigma \subseteq Q
$$
by Lemma 9.14. 
In particular, $\tau \in Q$.
\end{proof}

\bigskip

\addcontentsline{toc}{section}{11. The main proxy lemma}

\textbf{\S 11. The main proxy lemma}

\stepcounter{section}

\bigskip

Let $M \in \mathcal X$ and $N \in \mathcal X \cup \mathcal Y$, 
where $N$ is simple. 
Suppose that $M < N$ in the case that $N \in \mathcal X$, and 
$\sup(M \cap N \cap \kappa) < N \cap \kappa$ in the case that  
$N \in \mathcal Y$. 
Consider $P \in M \cap \mathcal Y$ such that 
$P \cap \kappa < \sup(M \cap N \cap \kappa)$, and assume that 
we are building an object in $N$ which needs to be 
compatible in some sense with the model $P$. 
By Lemma 8.2, we know that $M \cap N$ is a member of $N$. 
However, when we intersect $M$ with $N$, the model $P$ will disappear 
if it is not in $N$. 
Thus although $N$ sees a fragment of $M$, it does not necessarily see 
$P$ even though $P \cap \kappa$ is in $N$.

Proxies are designed to handle this situation. 
We will define an object $p(M,N)$, called the 
\emph{canonical proxy of $M$ and $N$}, which is a member of $N$. 
The canonical proxy codes enough information about $M$ that we can 
rebuild fragments of $P$ inside $N$ 
which can be used to avoid 
incompatibilities between $P$ and the object we are 
constructing.\footnote{The idea of a canonical 
proxy which we use in this paper is a variation of 
a technical device used by Mitchell for a similar purpose. 
In the proof of Mitchell's theorem from \cite{mitchell}, a side condition is 
a pair $(M,a)$, where $M$ is a countable model and $a$ is a proxy. 
In this paper we separate the idea of a side condition and a proxy. 
In contrast to \cite{mitchell}, where proxies are present in 
many different parts of the proof, 
all applications of proxies which we give reduce to a single 
lemma, which is the main proxy lemma, Lemma 11.5. 
The idea of a canonical proxy 
and the main proxy lemma are new to this paper and do not appear in 
\cite{mitchell}.}

Although the description and the proof of the existence of proxies is quite complicated, in practice when we use proxies we only need to appeal to 
a single result, called the \emph{main proxy lemma}, which is 
Lemma 11.5 below. 
In applications of proxies, it is not necessary to understand anything else 
about proxies except what is contained in that lemma.

The next lemma asserts the existence of proxies. 
We will postpone the proof until the next section.

\begin{lemma}[Proxy existence lemma]
Let $M \in \mathcal X$ and $N \in \mathcal X \cup \mathcal Y$, where $N$ 
is simple. 
Assume that $M < N$ in the case that $N \in \mathcal X$, and 
$\sup(M \cap N \cap \kappa) < N \cap \kappa$ in the case that 
$N \in \mathcal Y$. 
Let $\eta^* \in R^+_N(M)$. 
Then there exist finite sets $a$ and $a'$ satisfying the 
following statements:
\begin{enumerate}

\item $a$ is a finite set of pairs of ordinals, and 
$a' = \{ \sigma : \exists \beta \ (\beta,\sigma) \in a \}$.

\item For all $(\beta,\sigma)$ in $a$,
\begin{enumerate}
\item $\beta \in M \cap N \cap \kappa$;
\item $\sigma \in N \cap \kappa^+$ is a limit ordinal; 
\item $\sup(N \cap \sigma) \le \eta^*$;
\item if $a \ne \emptyset$, then 
$\min(a') = \min((N \cap \kappa^+) \setminus \sup(M \cap \eta^*))$.
\end{enumerate}

\item If $(\beta,\sigma) \in a$, where $\min(a') < \sigma$, 
and $\beta \le \gamma < \kappa$, 
then:
\begin{enumerate}
\item $A_{\eta^*,\gamma} \cap \sup(N \cap \sigma) = 
A_{\sigma,\gamma} \cap \sup(N \cap \sigma)$;
\item $A_{\eta^*,\gamma} \cap N \cap \sigma = A_{\sigma,\gamma} \cap N$.
\end{enumerate}

\item If $P \in M \cap \mathcal Y$, $P \cap \kappa \in M \cap N \cap \kappa$, 
and $P \cap N \cap [\sup(M \cap \eta^*),\eta^*) \ne \emptyset$, 
then there exists $\sigma \in a'$ such that: 
\begin{enumerate}
\item $P \cap N \cap \eta^* \subseteq \sigma$;
\item the least such $\sigma$ 
is equal to the largest $\sigma$ in $a'$ such that for some $\beta$, 
$\beta \le P \cap \kappa$ and $(\beta,\sigma) \in a$.
\end{enumerate}

\item Let $P$ and $\sigma$ be as in (4), and assume that 
$(\beta,\sigma) \in a$; then:
\begin{enumerate}
\item $\beta \le P \cap \kappa$;

\item $P \cap \sup(N \cap \sigma) = A_{\sigma,P \cap \kappa} \cap 
\sup(N \cap \sigma)$;

\item $P \cap N \cap \eta^* = A_{\sigma,P \cap \kappa} \cap N$.
\end{enumerate}
\end{enumerate}
\end{lemma}

For the remainder of this section, we will assume that the proxy 
existence lemma holds. 
We now define the canonical proxy $p(M,N)$.

A lexicographical ordering on sets of pairs of ordinals is described as follows. 
We identify a finite set of pairs of 
ordinals as a finite set of ordinals using the G\"odel pairing function, 
and then compare any two finite sets of pairs using the 
lexicographical 
ordering on their corresponding sets of ordinals. 

\begin{definition}
Let $M \in \mathcal X$ and $N \in \mathcal X \cup \mathcal Y$, where $N$ 
is simple. 
Assume that $M < N$ in the case that $N \in \mathcal X$, and 
$\sup(M \cap N \cap \kappa) < N \cap \kappa$ in the case that 
$N \in \mathcal Y$. 

Let $\eta_0,\ldots,\eta_{k-1}$ enumerate the ordinals 
in $R^+_N(M)$ in increasing order. 
Define $p(M,N)$ as the function with domain $k$ such that for all 
$i < k$, $p(M,N)(i)$ is the lexicographically least set $a$ 
satisfying (1)--(5) of Lemma 11.1 for $\eta^* = \eta_i$.
\end{definition}

Note that $p(M,N)$ is a member of $N$.

The proof of the main proxy lemma will use the next two technical lemmas.

\begin{lemma}
Let $M \in \mathcal X$ and $N \in \mathcal X \cup \mathcal Y$, where $N$ 
is simple. 
Assume that $M < N$ in the case that $N \in \mathcal X$, and 
$\sup(M \cap N \cap \kappa) < N \cap \kappa$ in the case that 
$N \in \mathcal Y$.

Let $k$ be the size of $R^+_N(M)$, and assume that 
$\eta^*$ is the $i$-th member of $R^+_N(M)$, 
where $i < k$. 
Let $a := p(M,N)(i)$ and 
$a' := \{ \sigma : \exists \beta \ (\beta,\sigma) \in a \}$.

Suppose that $P \in M \cap \mathcal Y$, $P \cap \kappa \in M \cap N \cap \kappa$, $\cf(P \cap \kappa) > \omega$, 
and 
$P \cap N \cap [\sup(M \cap \eta^*),\eta^*) \ne \emptyset$. 
Let $\sigma$ be the least ordinal in $a'$ such that 
$P \cap N \cap \eta^* \subseteq \sigma$, which exists by 
Lemma 11.1(4). 
Let $Q := Sk(A_{\sigma,P \cap \kappa})$. Then:
\begin{enumerate}
\item $Q \in N \cap \mathcal Y$;
\item $Q \cap \kappa = P \cap \kappa$ and 
$Q \cap \kappa^+ = A_{\sigma,P \cap \kappa}$;
\item $Q \cap N \cap \kappa^+ = P \cap N \cap \eta^*$;
\item $Q \cap \sup(N \cap \sigma) = P \cap \sup(N \cap \sigma)$.
\end{enumerate}
In particular, $P \cap N \cap \eta^* \subseteq Q$.
\end{lemma}

\begin{proof}
Let $\theta := \sup(N \cap \sigma)$. 
By Lemma 11.1(5(b,c)),
$$
P \cap \theta = 
A_{\sigma,P \cap \kappa} \cap \theta \ \textrm{and} \ 
P \cap N \cap \eta^* = 
A_{\sigma,P \cap \kappa} \cap N.
$$
In particular, 
as $A_{\sigma,P \cap \kappa} \in N$, 
Lemmas 7.10 and 8.16 and the second equality imply that 
$A_{\sigma,P \cap \kappa}$ is closed under $H^*$. 
So $Q \cap \kappa^+ = A_{\sigma,P \cap \kappa}$. 
Hence 
$$
Q \cap N \cap \kappa^+ = A_{\sigma,P \cap \kappa} \cap N 
= P \cap N \cap \eta^*,
$$
which proves (3). 
Also by the first equality, 
$$
Q \cap \sup(N \cap \sigma) = Q \cap \theta = 
A_{\sigma,P \cap \kappa} \cap \theta = P \cap \theta = 
P \cap \sup(N \cap \sigma),
$$
which proves (4).

We claim that $Q \cap \kappa = P \cap \kappa$, which proves (2). 
Since $Q$ and $P \cap \kappa$ are in $N$, 
it suffices to show that $N$ models that $Q \cap \kappa = P \cap \kappa$. 
So let $\alpha \in Q \cap N \cap \kappa$, and we will show 
that $\alpha < P \cap \kappa$. 
Then $\alpha \in Q \cap N \cap \kappa^+ = P \cap N \cap \eta^*$. 
So $\alpha \in P \cap \kappa$. 
Conversely, let $\alpha \in N \cap P \cap \kappa$, and we will show 
that $\alpha \in Q$. 
Then $\alpha \in P \cap N \cap \eta^* = Q \cap N \cap \kappa^+$, 
so $\alpha \in Q$. 

To prove (1), it suffices to show that 
$\lim(C_{\sup(Q)}) \cap Q$ is cofinal in $\sup(Q)$. 
Since $Q \cap \kappa^+ = A_{\sigma,P \cap \kappa}$, 
$\sup(Q) \le \sigma$. 
Also note that since $P \cap [\sup(M \cap \eta^*),\eta^*) \ne \emptyset$ 
and $P$ and $\eta^*$ are in $M$, 
$\eta^*$ is a limit point of $P$ by Lemma 7.30.

\bigskip

\emph{Case 1:} $\theta < \sigma$. 
Since $\cf(Q \cap \kappa) = \cf(P \cap \kappa) > \omega$, 
it suffices by Lemma 7.15 to show that $\cf(\sup(Q)) > \omega$. 
Since $\sigma = \min((N \cap \kappa^+) \setminus \theta)$, 
$\sigma$ has uncountable cofinality. 
So if $\sup(Q) = \sigma$, then we are done. 

Otherwise by elementarity, $\sup(Q) \in N \cap \sigma \subseteq \theta$. 
By (4), $Q \cap \kappa^+ = Q \cap \theta = P \cap \theta$. 
It follows that $\sup(Q) = \sup(P \cap \theta)$, which is a limit point 
of $P$ below $\theta$. 
Since $\eta^*$ is a limit point of $P$ and 
$\sup(Q) < \theta \le \eta^*$ by Lemma 11.1(2(c)), 
if $\sup(Q)$ has countable 
cofinality then $\sup(Q) \in P$ by Lemma 7.14. 
But then $\sup(Q) \in P \cap \theta = Q \cap \theta$, which is impossible. 
Therefore $\sup(Q)$ has uncountable cofinality.

\bigskip

\emph{Case 2:} $\theta = \sigma$. 
Then $\sigma$ is a limit point of $N$, and in particular, 
$\sigma$ has cofinality $\omega$. 
By Lemma 11.1(2(c)), $\sup(N \cap \sigma) = \sigma \le \eta^*$. 
Since $\sigma$ has cofinality $\omega$ and $\eta^*$ has uncountable 
cofinality, it follows that $\sigma < \eta^*$. 

We claim that $\sup(P \cap N \cap \sigma) < \sigma$. 
Suppose for a contradiction that $\sup(P \cap N \cap \sigma) = \sigma$. 
Then $\sigma$ is a limit point of $P$. 
As $\eta^*$ is a limit point of $P$, $\sigma < \sup(P)$. 
Since $\sigma$ has cofinality $\omega$ and 
$\cf(P \cap \kappa) > \omega$, Lemma 7.14 implies that 
$\sigma \in P$. 
So $\sigma \in N \cap P \cap \eta^*$, which contradicts 
that $N \cap P \cap \eta^* \subseteq \sigma$.

To show that $Q \in \mathcal Y$, by Lemma 7.15 
it suffices to show that 
$\cf(\sup(Q)) > \omega$. 
Since $\theta = \sigma$, by (4) we have that 
$Q \cap \sigma = P \cap \sigma$. 
Therefore $Q \cap N \cap \sigma = P \cap N \cap \sigma$. 
By the claim, 
$$
\sup(Q \cap N \cap \sigma) = 
\sup(P \cap N \cap \sigma) < \sigma.
$$
If $\sup(Q) = \sigma$, then since $Q \in N$ and $\sigma$ is a limit 
point of $N$, it is easy to argue by elementarity that 
$Q \cap N \cap \kappa^+$ is cofinal in $\sigma$, which is false. 
Therefore $\sup(Q) < \sigma = \theta$. 
Since $Q \cap \theta = P \cap \theta$, 
it follows that $\sup(Q) = \sup(Q \cap \theta) = 
\sup(P \cap \theta)$, which is a limit point 
of $P$ below $\theta$. 
Since $\eta^*$ is a limit point of $P$ and 
$\sup(Q) < \theta \le \eta^*$ by Lemma 11.1(2(c)), 
if $\sup(Q)$ has countable 
cofinality then $\sup(Q) \in P$ by Lemma 7.14. 
But then $\sup(Q) \in P \cap \theta = Q \cap \theta$, which is impossible. 
Therefore $\sup(Q)$ has uncountable cofinality.
\end{proof}

\begin{lemma}
Let $M \in \mathcal X$ and $N \in \mathcal X \cup \mathcal Y$, where $N$ 
is simple. 
Assume that $M < N$ in the case that $N \in \mathcal X$, and 
$\sup(M \cap N \cap \kappa) < N \cap \kappa$ in the case that 
$N \in \mathcal Y$.

Let $k$ be the size of $R^+_N(M)$, and assume that 
$\eta^*$ is the $i$-th member of $R^+_N(M)$, where $i < k$. 
Let $a := p(M,N)(i)$ and 
$a' := \{ \sigma : \exists \beta \ (\beta,\sigma) \in a \}$.

Suppose that $(\beta,\sigma) \in a$, where $\min(a') < \sigma$. 
Assume that $Q \in N \cap \mathcal Y$ is such that 
$\beta \le Q \cap \kappa$, $Q \cap \kappa \in M \cap N \cap \kappa$, 
$\cf(Q \cap \kappa) > \omega$, 
$Q \cap \kappa^+ = A_{\sigma,Q \cap \kappa}$, and 
$Q \cap N \cap [\sup(M \cap \eta^*),\sigma) \ne \emptyset$.

Let $P := Sk(A_{\eta^*,Q \cap \kappa})$. 
Then:
\begin{enumerate}
\item $P \in M \cap \mathcal Y$;
\item $P \cap \kappa = Q \cap \kappa$, 
$P \cap \kappa^+ = A_{\eta^*,Q \cap \kappa}$, and 
$\sup(P) = \eta^*$;
\item $P \cap N \cap [\sup(M \cap \eta^*),\eta^*) \ne \emptyset$;
\item $P \cap M \cap \kappa^+ = Q \cap M \cap \kappa^+$.
\end{enumerate}
\end{lemma}

\begin{proof}
Let $\gamma := Q \cap \kappa$ and $\theta := \sup(N \cap \sigma)$. 
By Lemma 11.1(3), 
$$
A_{\eta^*,\gamma} \cap \theta = A_{\sigma,\gamma} \cap \theta 
\ \textrm{and} \ 
A_{\eta^*,\gamma} \cap N \cap \sigma = A_{\sigma,\gamma} \cap N.
$$
We claim that 
$$
A_{\eta^*,\gamma} \cap M = A_{\sigma,\gamma} \cap M.
$$
Let $\alpha \in A_{\eta^*,\gamma} \cap M$, and we will show 
that $\alpha \in A_{\sigma,\gamma}$. 
Then $\alpha \in M \cap \eta^*$. 
By Lemma 11.1(2(d)), 
$\sup(M \cap \eta^*) \le \min(a') < \sigma$. 
Since $\min(a') \in N$, $\sup(M \cap \eta^*) < \sup(N \cap \sigma)$. 
So 
$$
\alpha < \sup(M \cap \eta^*) < \sup(N \cap \sigma) = \theta.
$$
Hence 
$$
\alpha \in A_{\eta^*,\gamma} \cap \theta = 
A_{\sigma,\gamma} \cap \theta,
$$
so $\alpha \in A_{\sigma,\gamma}$.

Conversely, let $\alpha \in A_{\sigma,\gamma} \cap M$, and we will show 
that $\alpha \in A_{\eta^*,\gamma}$. 
Since $Q \cap \kappa^+ = A_{\sigma,\gamma}$, 
$\alpha \in Q \cap M \cap \kappa^+$. 
We claim that $\alpha \in N$. 
If $N \in \mathcal Y$, then $\alpha \in Q \in N$ implies that $\alpha \in N$. 
Suppose that $N \in \mathcal X$. 
Then $M < N$, $Q \in N \cap \mathcal Y$, and 
$Q \cap \kappa = \gamma \in M \cap N \cap \kappa$, which implies by 
Lemma 8.7 that $Q \cap M \cap \kappa^+ \subseteq N$. 
In particular, $\alpha \in N$. 
Hence in either case,
$$
\alpha \in A_{\sigma,\gamma} \cap N = 
A_{\eta^*,\gamma} \cap N \cap \sigma.
$$
So $\alpha \in A_{\eta^*,\gamma}$.

We have proven that 
$A_{\eta^*,\gamma} \cap M = A_{\sigma,\gamma} \cap M = 
Q \cap M \cap \kappa^+$. 
Since $A_{\eta^*,\gamma}$ is in $M$, 
it follows that $A_{\eta^*,\gamma}$ is closed under $H^*$ 
by Lemma 7.10. 
In particular, $P \cap \kappa^+ = A_{\eta^*,\gamma}$. 
So 
$$
P \cap M \cap \kappa^+ = A_{\eta^*,\gamma} \cap M = 
A_{\sigma,\gamma} \cap M = Q \cap M \cap \kappa^+,
$$
which proves (4).

Next we claim that $P \cap \kappa = \gamma$ and 
$\sup(P) = \eta^*$, which proves (2). 
Since $P$, $\gamma$, and $\eta^*$ are in $M$, 
it suffices to show that $M$ models these statements. 
Let $\alpha \in P \cap M \cap \kappa$, and we will show that 
$\alpha < \gamma$. 
Then $\alpha \in P \cap M \cap \kappa^+ = Q \cap M \cap \kappa^+$. 
So $\alpha \in Q \cap \kappa = \gamma$. 
Conversely, let $\alpha \in M \cap \gamma$, and we will show 
that $\alpha \in P \cap \kappa$. 
Then 
$$
\alpha \in M \cap \gamma = M \cap Q \cap \kappa \subseteq 
M \cap Q \cap \kappa^+ = M \cap P \cap \kappa^+.
$$
So $\alpha \in P$.

Now $Q \cap N \cap [\sup(M \cap \eta^*),\sigma)$ is nonempty by 
assumption, so fix $\tau$ in this intersection. 
Then 
$$
\tau \in Q \cap \kappa^+ \cap N = 
A_{\sigma,\gamma} \cap N = A_{\eta^*,\gamma} \cap N \cap \sigma.
$$
So $\tau \in A_{\eta^*,\gamma} = P \cap \kappa^+$. 
Hence $P \cap N \cap [\sup(M \cap \eta^*),\eta^*) \ne \emptyset$, which proves (3). 
By Lemma 7.30, it follows that $\eta^*$ is a limit point of $P$. 
Since $\sup(P \cap \kappa^+) = \sup(A_{\eta^*,\gamma}) 
\le \eta^*$, we have that $\sup(P) \le \eta^*$. 
As $\eta^*$ is a limit point of $P$, it follows that $\sup(P) = \eta^*$, finishing 
the proof of (2). 
In particular, as the ordinals $P \cap \kappa = Q \cap \kappa$ and 
$\eta^*$ both have uncountable cofinality, 
it follows that $P \in \mathcal Y$ by Lemma 7.15, which proves (1).
\end{proof}

We are now ready to prove the main lemma on proxies. 
This lemma contains all the information about proxies that we will need 
for applications.

\begin{lemma}[Main proxy lemma]
Let $M \in \mathcal X$ and $N \in \mathcal X \cup \mathcal Y$, where $N$ 
is simple. 
Assume that $M < N$ in the case that $N \in \mathcal X$, and 
$\sup(M \cap N \cap \kappa) < N \cap \kappa$ in the case that 
$N \in \mathcal Y$. 
Let $\eta^* \in R^+_N(M)$.

Suppose that $M' \in N \cap \mathcal X$ and 
$N' \in N \cap (\mathcal X \cup \mathcal Y)$, where $N'$ is simple. 
Assume that $M' < N'$ in the case that $N' \in \mathcal X$, 
$\sup(M' \cap N' \cap \kappa) < N' \cap \kappa$ in the case that 
$N' \in \mathcal Y$, and $N \in \mathcal X$ iff $N' \in \mathcal X$. 
Suppose that $M \cap N = M' \cap N'$ and 
$p(M,N) = p(M',N')$.

Assume that $P \in M \cap \mathcal Y$, 
$P \cap \kappa \in M \cap N \cap \kappa$, 
$\cf(P \cap \kappa) > \omega$, and 
$\tau \in P \cap N \cap [\sup(M \cap \eta^*),\eta^*)$. 
Then:
\begin{enumerate}
\item There is $Q \in N' \cap \mathcal Y$ such that 
$Q \cap \kappa = P \cap \kappa$ and 
$Q \cap N \cap \kappa^+ = P \cap N \cap \eta^*$; 
in particular, $\tau \in Q$.
\item If $\tau \in N'$, then there is 
$P' \in M' \cap \mathcal Y$ 
such that $P' \cap \kappa = P \cap \kappa$, 
$P' \cap N' \cap \kappa^+ = Q \cap N' \cap \kappa^+$, and 
$P' \cap M' \cap \kappa^+ = Q \cap M' \cap \kappa^+$;
in particular, $\tau \in P'$.
\item If $N \in \mathcal Y$, then there is $P' \in M' \cap \mathcal Y$ 
such that $P' \cap \kappa = P \cap \kappa$ and $\tau \in P'$.
\end{enumerate}
Moreover, if $\vec S$ is given and $P$ is $\vec S$-strong, then 
the models $Q$ and $P'$ described in (1), (2), and (3) 
are also $\vec S$-strong.
\end{lemma}

\begin{proof}
Let $k$ be the size of $R^+_N(M)$, and fix $i < k$ such that 
$\eta^*$ is the $i$-th member of $R^+_N(M)$. 
Let $a := p(M,N)(i)$ and 
$a' := \{ \sigma : \exists \beta \ (\beta,\sigma) \in a \}$.

\bigskip

(1) Let $\sigma$ be the least ordinal in $a'$ such that 
$P \cap N \cap \eta^* \subseteq \sigma$, which exists 
by Lemma 11.1(4). 
By Lemma 11.1(2(d)), 
$\min(a') = \min((N \cap \kappa^+) \setminus \sup(M \cap \eta^*))$, 
which is strictly less than 
$\sigma$ since $P \cap N \cap [\sup(M \cap \eta^*),\eta^*) \ne \emptyset$ 
and $P \cap N \cap \eta^* \subseteq \sigma$. 
Fix $\beta$ such that $(\beta,\sigma) \in a$. 
By Lemma 11.1(5(a)), $\beta \le P \cap \kappa$.

We apply Lemma 11.3. 
Note that all of the assumptions of this lemma are satisfied. 
Let $Q := Sk(A_{\sigma,P \cap \kappa})$. 
Then by Lemma 11.3,
\begin{enumerate}
\item[(a)] $Q \in N \cap \mathcal Y$;
\item[(b)] $Q \cap \kappa = P \cap \kappa$ and 
$Q \cap \kappa^+ = A_{\sigma,P \cap \kappa}$;
\item[(c)] $Q \cap N \cap \kappa^+ = P \cap N \cap \eta^*$;
\item[(d)] $Q \cap \sup(N \cap \sigma) = P \cap \sup(N \cap \sigma)$.
\end{enumerate}

Since $p(M,N) = p(M',N')$, it follows that 
$p(M,N) \in N'$, and so in particular, 
$\sigma \in N'$. 
And $Q \cap \kappa = P \cap \kappa \in M \cap N \cap \kappa = 
M' \cap N' \cap \kappa \subseteq N'$. 
So $\sigma$ and $Q \cap \kappa$ are in $N'$. 
Therefore $A_{\sigma,Q \cap \kappa}$ and $Q$ are in $N'$ 
by elementarity. 
Properties (a), (b), and (c) above imply (1).

\bigskip

(2) Assume that $\tau \in N'$. 
We apply Lemma 11.4 to $M'$ and $N'$. 
Let $\eta_0^*$ be the $i$-th member of $R^+_{N'}(M')$. 
Let $a_0 := p(M',N')(i)$. 
Note that $a_0 = p(M,N)(i) = a$ and $a_0' = a'$.

Let us check that the assumptions of Lemma 11.4 are satisfied. 
Since $p(M,N) = p(M',N')$, $(\beta,\sigma) \in p(M',N')(i) = a_0$, and 
$\min(a_0') = \min(a') < \sigma$. 
We know that $Q \in N' \cap \mathcal Y$, 
$\beta \le P \cap \kappa = Q \cap \kappa$, 
$Q \cap \kappa = P \cap \kappa \in M \cap N \cap \kappa = 
M' \cap N' \cap \kappa$, $\cf(Q \cap \kappa) = \cf(P \cap \kappa) > \omega$, 
and $Q \cap \kappa^+ = A_{\sigma,Q \cap \kappa}$. 

It remains to show that 
$$
Q \cap N' \cap 
[\sup(M' \cap \eta_0^*),\sigma) \ne \emptyset.
$$
Since $\tau \in P \cap N \cap [\sup(M \cap \eta^*),\eta^*)$ and 
$$
P \cap N \cap \eta^* \subseteq N \cap \sigma \subseteq \sup(N \cap \sigma),
$$
it follows that $\tau < \sigma$, and 
$$
\tau \in P \cap \sup(N \cap \sigma) = Q \cap \sup(N \cap \sigma)
$$
by property (d) above. 
Also $\tau \in N'$ by assumption. 
So 
$$
\tau \in Q \cap N' \cap [\sup(M \cap \eta^*),\eta^*),
$$
and 
therefore $\min(a') = 
\min((N \cap \kappa^+) \setminus \sup(M \cap \eta^*)) \le \tau$. 
But then 
$$
\sup(M' \cap \eta_0^*) \le 
\min((N' \cap \kappa^+) \setminus \sup(M' \cap \eta_0^*)) = 
\min(a_0') = \min(a') \le \tau.
$$
So $\tau \in Q \cap N' \cap [\sup(M' \cap \eta_0^*),\sigma)$. 
This completes the verification of the assumptions of Lemma 11.4.

Let $P' := Sk(A_{\eta_0^*,Q \cap \kappa})$. 
Then by Lemma 11.4, 
\begin{enumerate}
\item[(i)] $P' \in M' \cap \mathcal Y$;
\item[(ii)] $P' \cap \kappa = Q \cap \kappa$, 
$P' \cap \kappa^+ = A_{\eta_0^*,Q \cap \kappa}$, and 
$\sup(P') = \eta_0^*$;
\item[(iii)] $P' \cap N' \cap [\sup(M' \cap \eta_0^*),\eta_0^*) \ne \emptyset$;
\item[(iv)] $P' \cap M' \cap \kappa^+ = Q \cap M' \cap \kappa^+$.
\end{enumerate}
In particular, $P' \cap \kappa = P \cap \kappa$.
It remains to prove that 
$$
P' \cap N' \cap \kappa^+ = Q \cap N' \cap \kappa^+.
$$

We apply Lemma 11.3 to $M'$, $N'$, and $P'$. 
Note that the assumptions of Lemma 11.3 are obviously satisfied, 
except for the claim that $\sigma$ is the least ordinal in $a_0'$ such 
that $P' \cap N' \cap \eta_0^* \subseteq \sigma$. 
So let $\sigma'$ be the least ordinal in $a_0'$ such that 
$P' \cap N' \cap \eta_0^* \subseteq \sigma'$. 
Then by Lemma 11.1(4(b)), $\sigma'$ is the largest ordinal in $a_0'$ 
such that for some $\gamma$, 
$\gamma \le P' \cap \kappa$ and $(\gamma,\sigma') \in a_0$. 
Now $\sigma$ is the least ordinal in $a' = a_0'$ such that 
$N \cap P \cap \eta^* \subseteq \sigma$, so again by Lemma 11.1(4(b)), 
$\sigma$ is the largest ordinal in $a' = a_0'$ such that 
for some $\gamma$, $\gamma \le P \cap \kappa = P' \cap \kappa$ 
and $(\gamma,\sigma) \in a = a_0$. 
So $\sigma$ and $\sigma'$ satisfy the same definition, and hence 
$\sigma = \sigma'$. 
So indeed $\sigma$ is the least ordinal in $a_0'$ such that 
$P' \cap N' \cap \eta_0^* \subseteq \sigma$.

Since $\sigma = \sigma'$ and $P \cap \kappa = P' \cap \kappa$, 
$Q = Sk(A_{\sigma,P \cap \kappa}) = 
Sk(A_{\sigma',P' \cap \kappa})$. 
By Lemma 11.3, 
$Q \cap N' \cap \kappa^+ = P' \cap N' \cap \eta_0^*$. 
But $P' = Sk(A_{\eta_0^*,Q \cap \kappa})$, and therefore 
$P' \cap \eta_0^* = P' \cap \kappa^+$. 
So $Q \cap N' \cap \kappa^+ = P' \cap N' \cap \kappa^+$.

\bigskip

(3) If $N \in \mathcal Y$, then $N' \in \mathcal Y$. 
So $Q \in N'$  implies that $Q \subseteq N'$. 
Hence $\tau \in N'$. 
So we are done by (2).

\bigskip

Finally, the last statement follows from the properties of $Q$ and $P'$ 
described in (1) and (2) together with Lemma 5.6.
\end{proof}

\bigskip

The main proxy lemma was concerned with the case that 
$P \in M \cap \mathcal Y$, 
$P \cap \kappa \in M \cap N \cap \kappa$, $\eta^* \in R^+_N(M)$, 
and $\tau \in P \cap N \cap [\sup(M \cap \eta^*),\eta^*)$. 
Another case which often occurs in the same contexts is that 
$P \in M \cap \mathcal Y$, $P \cap \kappa \in M \cap N \cap \kappa$, 
and $\tau \in P \cap N \cap \alpha_{M,N}$. 
This situation is handled by the next two lemmas.

\begin{lemma}
Let $M \in \mathcal X$ and $N \in \mathcal X \cup \mathcal Y$, where $N$ 
is simple. 
Assume that $M < N$ in the case that $N \in \mathcal X$, and 
$\sup(M \cap N \cap \kappa) < N \cap \kappa$ in the case that 
$N \in \mathcal Y$.

Suppose that $M' \in N \cap \mathcal X$ and 
$N' \in N \cap (\mathcal X \cup \mathcal Y)$, where $N'$ is simple. 
Assume that $M' < N'$ in the case that $N' \in \mathcal X$, 
$\sup(M' \cap N' \cap \kappa) < N' \cap \kappa$ in the case that 
$N' \in \mathcal Y$, and $N \in \mathcal X$ iff $N' \in \mathcal X$. 
Suppose that $M \cap N = M' \cap N'$ and 
$p(M,N) = p(M',N')$.

Assume that $P \in M \cap \mathcal Y$, 
$P \cap \kappa \in M \cap N \cap \kappa$, 
$\cf(P \cap \kappa) > \omega$, $P \cap \alpha_{M,N}$ is unbounded 
in $\alpha_{M,N}$, and 
$\tau_0 \in P \cap N \cap \alpha_{M,N}$. 
Let $\eta^* := \min((M \cap \kappa^+) \setminus \alpha_{M,N})$.
Then:
\begin{enumerate}
\item There is $Q \in N' \cap \mathcal Y$ such that 
$Q \cap \kappa = P \cap \kappa$ and 
$Q \cap N \cap \kappa^+ = P \cap N \cap \eta^*$; 
in particular, $\tau_0 \in Q$.
\item If $\tau_0 \in N'$, then there is 
$P' \in M' \cap \mathcal Y$ 
such that $P' \cap \kappa = P \cap \kappa$ and $\tau_0 \in P'$.
\end{enumerate}
Moreover, if $\vec S$ is given and $P$ is $\vec S$-strong, then 
the models $Q$ and $P'$ described in (1) and (2) 
are also $\vec S$-strong.
\end{lemma}

\begin{proof}
Since $\alpha_{M,N} \in N$ by Lemma 8.2, $\alpha_{M,N} \notin M$. 
But $\alpha_{M,N} \le \sup(P)$ and $\sup(P) \in M$. 
It follows that $\alpha_{M,N} < \sup(P)$. 
Consequently, the ordinal $\eta^* = \min((M \cap \kappa^+) \setminus \alpha_{M,N})$ 
exists and is greater than $\alpha_{M,N}$. 
Therefore $\eta^* \in R^+_N(M)$. 
Since $\alpha_{M,N}$ is a limit point of the countable set $M \cap N$, 
it follows that $\cf(\alpha_{M,N}) = \omega$. 
As $\cf(P \cap \kappa) > \omega$, we have that $\alpha_{M,N} \in P$ 
by Lemma 7.14. 
So $\alpha_{M,N} \in N \cap P \cap [\sup(M \cap \eta^*),\eta^*)$.

We apply the main proxy lemma, Lemma 11.5, letting $\tau = \alpha_{M,N}$. 
Then the first statement of (1) above follows from Lemma 11.5(1). 
Since $\tau_0 < \alpha_{M,N}$, $\tau_0 \in N \cap P \cap \eta^* \subseteq Q$. 
For (2), we have that 
$$
\alpha_{M,N} = \sup(M \cap N) = \sup(M' \cap N'),
$$
which is in $N'$ by Lemma 8.2. 
By Lemma 11.5(2), there is $P' \in M' \cap \mathcal Y$ such that 
$P' \cap \kappa = P \cap \kappa$ and 
$P' \cap N' \cap \kappa^+ = Q \cap N' \cap \kappa^+$. 
Assume that $\tau_0 \in N'$. 
Then 
$$
\tau_0 \in Q \cap N' \cap \kappa^+ = P' \cap N' \cap \kappa^+ \subseteq P'.
$$
Therefore $\tau_0 \in P'$.
\end{proof}

\begin{lemma}
Let $M \in \mathcal X$ and $N \in \mathcal X \cup \mathcal Y$, where $N$ 
is simple. 
Assume that $M < N$ in the case that $N \in \mathcal X$, and 
$\sup(M \cap N \cap \kappa) < N \cap \kappa$ in the case that 
$N \in \mathcal Y$. 

Suppose that $P \in M \cap \mathcal Y$, 
$P \cap \kappa \in M \cap N \cap \kappa$, 
$P \cap \alpha_{M,N}$ is bounded below $\alpha_{M,N}$, and 
$\tau \in P \cap N \cap \alpha_{M,N}$. 
Then there is $P' \in M \cap N \cap \mathcal Y$ 
such that $P' \cap \kappa = P \cap \kappa$ and $\tau \in P'$. 
Moreover, if $\vec S$ is given and $P$ is 
$\vec S$-strong, then $P'$ is $\vec S$-strong.
\end{lemma}

\begin{proof}
Let $\alpha := \alpha_{M,N}$ and $\delta := \sup(M \cap N \cap \kappa)$. 
Define 
$$
\sigma := \sup(P \cap A_{\alpha,\delta}).
$$
By Lemma 8.11, $\sigma$ satisfies:
\begin{enumerate}
\item[(a)] $\sigma \in M \cap N \cap \kappa^+$;
\item[(b)] $P \cap \sigma = A_{\sigma,P \cap \kappa}$;
\item[(c)] $P \cap (M \cap N) \cap \kappa^+ = A_{\sigma,P \cap \kappa} 
\cap (M \cap N)$;
\item[(d)] $N \cap P \cap \alpha_{M,N} \subseteq A_{\sigma,P \cap \kappa}$.
\end{enumerate}
Since $\sigma$ and $P \cap \kappa$ are in $M \cap N$, 
Lemma 7.10 and (c) imply that 
$A_{\sigma,P \cap \kappa}$ is closed under $H^*$. 
Let $P' := Sk(A_{\sigma,P \cap \kappa})$. 
Then $P'$ is in $M \cap N$. 

By (b), we have that 
$$
P' \cap \kappa = A_{\sigma,P \cap \kappa} \cap \kappa 
= P \cap \sigma \cap \kappa = P \cap \kappa.
$$
By (d), $\tau \in A_{\sigma,P \cap \kappa} \subseteq P'$. 
It remains to show that $P' \in \mathcal Y$. 
It suffices to show that 
$\lim(C_{\sup(P')}) \cap P'$ is cofinal in $\sup(P')$.

Since $P$ and $A_{\alpha,\delta}$ are closed under successors, 
$P \cap A_{\alpha,\delta}$ has no maximal element. 
As $\sigma$ is a limit point of $P$, 
$$
\sup(P') = \sup(A_{\sigma,P \cap \kappa}) = \sup(P \cap \sigma) = \sigma.
$$
Also $P' \cap \kappa^+ = A_{\sigma,P \cap \kappa} = P \cap \sigma$. 
So it is enough to show that $\lim(C_\sigma) \cap P$ is cofinal in $\sigma$.

Now $\sigma$ is a limit point of $P$, and therefore has cofinality less 
than $\kappa$. 
If $\sigma \notin P$, then $\sigma \in \cl(P) \setminus P$, so by Lemma 7.13, 
$\lim(C_\sigma) \cap P$ is cofinal in $\sigma$ and we are done. 
Otherwise $\sigma \in P$. 
By the definition of $\sigma$, $\sigma$ is not in $A_{\alpha,\delta}$. 
Now $A_{\alpha,\delta}$ is closed under $H^*$ by Lemma 7.29. 
So $Q := Sk(A_{\alpha,\delta})$ is an elementary substructure of 
$\mathcal A$ with $Q \cap \kappa^+ = A_{\alpha,\delta}$.
Let $\sigma' := \min((Q \cap \kappa^+) \setminus \sigma)$, which 
exists since $\sigma < \alpha$. 
Then $\sigma'$ has uncountable cofinality, which implies that 
$\lim(C_{\sigma'})$ is cofinal in $\sigma'$. 
Also by the elementarity of $Q$, $\sigma$ is a limit point of $C_{\sigma'}$, and 
therefore 
$$
C_{\sigma} = C_{\sigma'} \cap \sigma.
$$
Again by the elementarity of $Q$, $\lim(C_{\sigma'}) \cap Q$ is cofinal 
in $\sup(Q \cap \sigma') = \sigma$. 
In particular, $\lim(C_\sigma)$ is cofinal in $\sigma$.

Since $\sigma \in P$ and $\sigma$ has cofinality less than $\kappa$, 
$\ot(C_\sigma) \in P \cap \kappa$, and therefore 
$C_{\sigma} \subseteq P$. 
Hence $\lim(C_\sigma) \cap P = \lim(C_{\sigma})$, and this set 
is cofinal in $\sigma$ as observed above.

Finally, assume that $P$ is $\vec S$-strong. 
Then $P' \cap \kappa = P \cap \kappa$, $P' \in M \cap N$, and by (b), 
$$
P' \cap (M \cap N) \cap \kappa^+ = A_{\sigma,P \cap \kappa} \cap (M \cap N) 
\subseteq A_{\sigma,P \cap \kappa} \subseteq P.
$$
So $P'$ is $\vec S$-strong by Lemma 5.6.
\end{proof}

\bigskip

\addcontentsline{toc}{section}{12. The proxy construction}

\textbf{\S 12. The proxy construction}

\stepcounter{section}

\bigskip

Let $M \in \mathcal X$ and $N \in \mathcal X \cup \mathcal Y$, where $N$ 
is simple. 
Assume that $M < N$ in the case that $N \in \mathcal X$, and 
$\sup(M \cap N \cap \kappa) < N \cap \kappa$ in the case that 
$N \in \mathcal Y$. 
Let $\eta^* \in R^+_N(M)$. 
We will prove that there exist sets $a$ and $a'$ satisfying 
properties (1)--(5) of Lemma 11.1.\footnote{Our proof of the 
proxy existence lemma is based on the construction of 
Mitchell \cite[Lemma 3.46]{mitchell}. We point out that there is a mistake 
in Mitchell's construction. 
The problem arises in the case when the ordinal $\eta$ from that proof 
is defined as $\max(\lim(C_\alpha) \cap \overline{X})$, 
and $\eta$ happens to have dropped below $\sup(M')$. 
In this case, there appears to be no reason why recursion hypothesis (1c) 
can be maintained. 
This problem was discovered by Gilton, and later corrected by Krueger.}

We recall a well-ordering on finite sets of ordinals which was used in \cite{mitchell}. 
For finite sets of ordinals $x$ and $y$, 
define $x \prec y$ if $x \ne y$ and $\max(x \triangle y) \in y$.

\begin{lemma}
The relation $\prec$ is a well-ordering of $[On]^{<\omega}$.
\end{lemma}

\begin{proof}
It is obvious that $\prec$ is irreflexive and total. 
For transitivity, let $x \prec y \prec z$, and we will show that $x \prec z$. 
Let $\alpha := \max(x \triangle y)$, $\beta := \max(y \triangle z)$, and 
$\gamma := \max(x \triangle z)$. 
Then $\alpha \in y \setminus x$ and $\beta \in z \setminus y$. 
We will show that $\gamma \in z$. 
Suppose for a contradiction that $\gamma \notin z$, so that $\gamma \in x$.

The following statements can be easily proved: 
(1) $\alpha$, $\beta$, and $\gamma$ are distinct; 
(2) $\alpha \in z$ implies that $\alpha  < \gamma$; 
(3) $\alpha \notin z$ implies that $\alpha < \beta$; 
(4) $\beta \in x$ implies that $\beta < \alpha$; 
(5) $\beta \notin x$ implies that $\beta < \gamma$; 
(6) $\gamma \in y$ implies that $\gamma < \beta$; and 
(7) $\gamma \notin y$ implies that $\gamma < \alpha$. 
Now one can easily check by inspection that any Boolean combination 
of the statements $\alpha \in z$, $\beta \in x$, and $\gamma \in y$ 
yields a contradiction. 
For example, suppose that $\alpha \in z$, $\beta \in x$, and $\gamma \in y$. 
Then (2), (4), and (6) imply that 
$\alpha < \gamma$, $\beta < \alpha$, and $\gamma < \beta$, which in turn 
imply that $\alpha < \gamma < \beta < \alpha$, which is absurd. 
The other possibilities are ruled out in a similar manner. 
This completes the proof that $\prec$ is transitive.

To show that $\prec$ is a well-ordering, 
suppose for a contradiction that $\langle x_n : n < \omega \rangle$ 
is a $\prec$-decreasing sequence of finite sets of ordinals. 
We define by induction an increasing sequence 
$\langle k_n : n < \omega \rangle$ of integers and a $\subseteq$-decreasing 
sequence $\langle A_n : n < \omega \rangle$ of infinite subsets 
of $\omega$ as follows. 
Let $k_0 = 0$ and $A_0 = \omega$.

Assume that $k_n$ and $A_n$ are defined, where 
$A_n$ is an infinite subset of $\omega$. 
Let $k_{n+1}$ be the least integer in $A_n$ strictly greater than $k_n$. 
Now for all $r \in A_n$ with $r > k_{n+1}$, 
we have that $x_r \prec x_{k_{n+1}}$, and hence 
$\max(x_r \triangle x_{k_{n+1}}) \in x_{k_{n+1}}$. 
Since $x_{k_{n+1}}$ is finite and $A_n$ is infinite, 
we can find an infinite subset $A_{n+1}$ of 
$A_n \setminus (k_{n+1}+1)$ such that for all $r, s \in A_{n+1}$, 
$\max(x_r \triangle x_{k_{n+1}}) = \max(x_s \triangle x_{k_{n+1}})$. 

This completes the construction. 
For each $n$, let $\alpha_n := \max(x_{k_n} \triangle x_{k_{n+1}})$. 
We claim that $\langle \alpha_n : n < \omega \rangle$ 
is a descending sequence of ordinals, which gives a contradiction. 
Let $n < \omega$. 
Since $x_{k_{n+1}} \prec x_{k_n}$, 
$\alpha_n \in x_{k_n} \setminus x_{k_{n+1}}$. 
So clearly $\alpha_n \ne \alpha_{n+1}$. 
Suppose for a contradiction that $\alpha_{n} < \alpha_{n+1}$. 
Then by the maximality of $\alpha_n$, 
$\alpha_{n+1}$ cannot be in $x_{k_n} \triangle x_{k_{n+1}}$, 
and therefore must be in $x_{k_n} \cap x_{k_{n+1}}$. 
But by construction, 
$\max(x_{k_n} \triangle x_{k_{n+2}}) = \alpha_n$. 
Therefore $\alpha_{n+1}$ must be in $x_{k_{n+2}}$, since otherwise 
it is in $x_{k_n} \triangle x_{k_{n+2}}$ but larger than $\alpha_n$. 
This contradicts that 
$\alpha_{n+1} = \max(x_{k_{n+1}} \triangle x_{k_{n+2}})$ is in 
$x_{k_{n+1}} \setminus x_{k_{n+2}}$.
\end{proof}

We will define by induction two sequences of sets 
$a_0, \ldots, a_n$ and $b_0, \ldots, b_n$. 
The induction stops when $b_n = \emptyset$. 
Each $a_k$ and $b_k$ will be a finite set of pairs of ordinals. 
We let $a_k' := \{ \sigma : \exists \beta \ (\beta,\sigma) \in a_k \}$ and 
$b_k' := \{ \eta : \exists \beta \ (\beta,\eta) \in b_k \}$.

By construction, for each $k$, $b_{k+1}'$ will be equal to 
$(b_{k}' \setminus \{ \eta \} ) \cup x$, where $\eta = \min(b_k')$ 
and $x$ is a finite subset of $\eta$. 
In particular, $\max(b_k' \triangle b_{k+1}')$ will be equal to $\eta$, 
which is in $b_k'$, and hence $b_{k+1}' \prec b_k'$. 
Therefore the sequence of $b'_k$'s is $\prec$-descending, and so 
must terminate with the empty set after finitely many steps.

When defining these sequences, we 
will maintain the following inductive hypotheses:

\bigskip

\begin{enumerate}

\item[(A)] For all $(\beta,\sigma) \in a_k$, $\beta \in M \cap N \cap \kappa$, 
$\sigma \in N \cap \kappa^+$ is a limit ordinal, 
and $\sup(N \cap \sigma) \le \eta^*$. 
The least member of $a_k'$, if it exists, is equal to 
$\min((N \cap \kappa^+) \setminus \sup(M \cap \eta^*))$. 
For each $\sigma \in a_k'$, there is a unique $\beta$ 
with $(\beta,\sigma) \in a_k$.

\bigskip

\item[(B)] For all $(\beta,\eta) \in b_k$, $\beta \in M \cap N \cap \kappa$ and 
$\eta \le \eta^*$ is a limit ordinal. 
If $\eta_0 < \eta_1$ are successive elements of $b_k'$, then 
$N \cap [\eta_0,\eta_1) \ne \emptyset$. 
For each $\eta \in b_k'$, there is a unique $\beta$ with 
$(\beta,\eta) \in b_k$. 

\bigskip

\item[(C)] If $b_k \ne \emptyset$, then $a_k \ne \emptyset$ and 
$\max(a_k') < \min(b_k')$.

\bigskip

\item[(D)] If $(\beta,\sigma) \in a_k$ and 
$\min(a_k') < \sigma$, then for all 
$\gamma$ with $\beta \le \gamma < \kappa$, 
$$
A_{\eta^*,\gamma} \cap \sup(N \cap \sigma) = 
A_{\sigma,\gamma} \cap \sup(N \cap \sigma).
$$

\bigskip

\item[(E)] If $(\beta,\eta) \in b_k$, then for all 
$\gamma$ with $\beta \le \gamma < \kappa$, 
$$
A_{\eta,\gamma} = A_{\eta^*,\gamma} \cap \eta.
$$

\bigskip

\item[(F)] If $(\beta,\eta) \in b_k$, then whenever $P \in M \cap \mathcal Y$ 
is such that 
$$
P \cap \kappa \in M \cap N \cap \kappa \ \textrm{and} \ 
P \cap N \cap [\eta^{-},\eta) \ne \emptyset,
$$
where 
$\eta^{-}$ is the largest ordinal in $a_k' \cup b_k'$ less than $\eta$, 
then $\beta \le P \cap \kappa$.

\bigskip

\item[(G)] Whenever $P \in M \cap \mathcal Y$ is such that 
$$
P \cap \kappa \in M \cap N \cap \kappa \ \textrm{and} \ 
P \cap N \cap [\sup(M \cap \eta^*),\eta^*) \ne \emptyset,
$$
then $P \cap N \cap \eta^* \subseteq \max(a_k' \cup b_k')$.

\bigskip

\item[(H)] Suppose that $P \in M \cap \mathcal Y$ and $\tau$ satisfy that 
$$
P \cap \kappa \in M \cap N \cap \kappa \ \textrm{and} \ 
\tau \in P \cap N \cap [\sup(M \cap \eta^*),\eta^*).
$$
Let $\sigma := \min((a_k' \cup b_k') \setminus (\tau+1))$, which exists by (G), 
and assume that $\sigma \in a_k'$. 
Fix $\beta$ with $(\beta,\sigma) \in a_k$. 
Then:
\begin{enumerate}
\item[(i)] $\beta \le P \cap \kappa$;
\item[(ii)] $P \cap \sup(N \cap \sigma) = 
A_{\sigma,P \cap \kappa} \cap \sup(N \cap \sigma)$.
\end{enumerate}

\bigskip

\item[(I)] If $(\beta,\sigma) \in a_k \cup b_k$, where 
$\min(a_k') < \sigma$, then 
$P := Sk(A_{\eta^*,\beta})$ satisfies that 
$$
P \in M \cap \mathcal Y, \ 
P \cap \kappa = \beta, \ 
P \cap \kappa^+ = A_{\eta^*,\beta}, \ \textrm{and} \ 
P \cap N \cap [\sigma^{-},\eta^*) \ne \emptyset,
$$
where $\sigma^{-}$ is the largest member of $a_k' \cup b_k'$ 
less than $\sigma$.
\end{enumerate}

\bigskip

Note that since $\alpha_{M,N}$ is a limit point of $M$ and 
$\alpha_{M,N} < \eta^*$, it follows that $\alpha_{M,N} \le \sup(M \cap \eta^*)$. 

Suppose that $P$ is as in (G), and $\sigma$ is the least ordinal in 
$a_k' \cup b_k'$ such that $P \cap N \cap \eta^* \subseteq \sigma$. 
By the minimality of $\sigma$, we can fix $\tau \in P \cap N \cap \eta^*$ 
such that $\sigma^{-} \le \tau$, where $\sigma^{-}$ is 
the greatest member of $a_k' \cup b_k'$ less than $\sigma$. 
Since $N$ and $P$ are closed under successors, $\tau+1 \in P \cap N \cap \eta^*$. 
As $P \cap N \cap \eta^* \subseteq \sigma$, it follows that 
$\sigma = \min((a_k' \cup b_k') \setminus (\tau+1))$.

In the arguments which follow, 
we will frequently consider models $P \in M \cap \mathcal Y$ such that 
$P \cap [\sup(M \cap \eta^*),\eta^*) \ne \emptyset$, for example, 
in (G) and (H). 
Note that by Lemma 7.30, for any such $P$, 
$\eta^*$ is a limit point of $P$. 
Therefore by Lemma 7.27, $P \cap \eta^* = A_{\eta^*,P \cap \kappa}$.

\bigskip

Assume that 
$a_0,\ldots,a_n$ and $b_0,\ldots,b_n$ are sequences 
satisfying properties 
(A)--(I), where $n$ is the least integer such that $b_n = \emptyset$. 
Let us show that the sets $a := a_n$ and 
$a' := \{ \sigma : \exists \beta \ (\beta,\sigma) \in a \}$ 
satisfy properties (1)--(5) in the conclusion of Lemma 11.1.

(1) is immediate, and (2) follows from (A). 
(3(a)) follows from (D). 
For (3(b)), let us prove that the equation 
$$
A_{\eta^*,\gamma} \cap N \cap \sigma = 
A_{\sigma,\gamma} \cap N
$$
follows from the equation 
$$
A_{\eta^*,\gamma} \cap \sup(N \cap \sigma) = 
A_{\sigma,\gamma} \cap \sup(N \cap \sigma),
$$
which holds by (3(a)). 
Let $\xi \in A_{\eta^*,\gamma} \cap N \cap \sigma$, and we will show that 
$\xi \in A_{\sigma,\gamma}$. 
Then $\xi < \sup(N \cap \sigma)$, so by the last equation, 
$\xi \in A_{\sigma,\gamma}$. 
Conversely, let $\xi \in A_{\sigma,\gamma} \cap N$, and we will show that 
$\xi \in A_{\eta^*,\gamma}$. 
Then $\xi \in N \cap \sigma$, so $\xi < \sup(N \cap \sigma)$. 
By the last equation, $\xi \in A_{\eta^*,\gamma}$.

(4) Suppose that 
$$
P \in M \cap \mathcal Y, \ P \cap \kappa \in 
M \cap N \cap \kappa, \ \textrm{and} \ 
P \cap N \cap [\sup(M \cap \eta^*),\eta^*) 
\ne \emptyset.
$$
By (G) and the fact that $b_n' = \emptyset$, 
$$
P \cap N \cap \eta^* \subseteq \max(a').
$$
Let $\sigma \in a'$ be the least ordinal 
such that $P \cap N \cap \eta^* \subseteq \sigma$. 
Fix $\beta$ such that $(\beta,\sigma) \in a$. 
Define 
$$
X := \{ (\beta',\sigma') \in a : \beta' \le P \cap \kappa \} \ \textrm{and} \ 
X' := \{ \sigma' : \exists \beta' \ (\beta',\sigma') \in X \}.
$$ 
We will prove that $\sigma = \max(X')$, which completes the 
proof of (4).

By the minimality of $\sigma$, clearly there is 
$\tau \in N \cap P \cap [\sup(M \cap \eta^*),\eta^*)$ such that 
$\sigma = \min(a' \setminus (\tau+1))$. 
By (H), $\beta \le P \cap \kappa$. 
It follows that $(\beta,\sigma) \in X$, and so $\sigma \in X'$.

Suppose for a contradiction that there is $\sigma' \in X'$ which is 
larger than $\sigma$. 
Fix $\beta'$ with $(\beta',\sigma') \in X$. 
Then $\sigma \le (\sigma')^{-}$, where $(\sigma')^{-}$ is the largest 
member of $a'$ which is less than $\sigma'$. 
By (I), 
$$
A_{\eta^*,\beta'} \cap N \cap [(\sigma')^{-},\eta^*) \ne \emptyset.
$$
Since $\sigma \le (\sigma')^{-}$, it follows that  
$$
A_{\eta^*,\beta'} \cap N \cap [\sigma,\eta^*) \ne \emptyset.
$$
As $\beta' \le P \cap \kappa$ by the definition of $X$, 
$$
A_{\eta^*,\beta'} \subseteq A_{\eta^*,P \cap \kappa} = P \cap \eta^*.
$$
But then $P \cap N \cap [\sigma,\eta^*) \ne \emptyset$, which contradicts 
that $P \cap N \cap \eta^* \subseteq \sigma$.

(5) Suppose that $P \in M \cap \mathcal Y$ satisfies that 
$$
P \cap \kappa \in M \cap N \cap \kappa \ \textrm{and} \  
P \cap N \cap [\sup(M \cap \eta^*),\eta^*) \ne \emptyset,
$$
$\sigma$ is the least ordinal in $a'$ such that 
$P \cap N \cap \eta^* \subseteq \sigma$, and $(\beta,\sigma) \in a$. 
By the minimality of $\sigma$, we can fix 
$\tau \in P \cap N \cap [\sup(M \cap \eta^*),\eta^*)$ such that 
$\sigma = \min(a' \setminus (\tau+1))$. 
Then (5(a,b)) follow immediately from (H(i,ii)). 
For (5(c)), $P \cap N \cap \eta^* = P \cap N \cap \sigma$, and 
$$
P \cap N \cap \sigma = 
P \cap \sup(N \cap \sigma) \cap N = 
A_{\sigma,P \cap \kappa} \cap \sup(N \cap \sigma) \cap N = 
A_{\sigma,P \cap \kappa} \cap N.
$$

\bigskip

We now turn to proving that there exist sequences 
$a_0,\ldots,a_n$ and 
$b_0,\ldots,b_n$ satisfying properties (A)--(I), where $n$ is the least 
integer such that $b_n = \emptyset$.

\bigskip

First we consider the base case. 
Let $\beta$ be the least ordinal in $M \cap N \cap \kappa$ 
for which there exists 
$P \in M \cap \mathcal Y$ such that 
$$
P \cap \kappa = \beta \ \textrm{and} \ 
P \cap N \cap [\sup(M \cap \eta^*),\eta^*) \ne \emptyset.
$$
If there is no such $\beta$, then let $a_0 = \emptyset$ and 
$b_0 = \emptyset$, and we are done.

Suppose that $\beta$ exists. 
Then obviously $N \cap [\sup(M \cap \eta^*),\eta^*) \ne \emptyset$. 
Define 
$$
a_0 := 
\{ (0,\min((N \cap \kappa^+) \setminus \sup(M \cap \eta^*))) \} \ 
\textrm{and} \ b_0 := (\beta,\eta^*).
$$

In the case that $a_0 = b_0 = \emptyset$, the inductive hypotheses 
are all vacuously true. 
In the other case, the inductive hypotheses are all either 
vacuously true or trivial.

\bigskip

Now we handle the induction step. 
Assume that $k < \omega$ and $a_k$ and $b_k$ have been 
defined and satisfy the inductive hypotheses. 
If $b_k = \emptyset$, then we are done. 
Assume that $b_k$ is nonempty. 
Then by (C), $a_k \ne \emptyset$. 
Let $\eta$ be the least member of $b_k'$, and let $\beta$ be the unique 
ordinal such that $(\beta,\eta) \in b_k$. 
By (A) and (B), $\max(a_k')$ and $\eta$ are limit ordinals. 
By (C), $\max(a_k') < \eta$, and in particular, $\omega < \eta$.

\bigskip

First consider the easy case that 
$\eta = \eta_0 + \omega$ for some limit ordinal $\eta_0$.
Let $a_{k+1} := a_k$. 
If $\max(a_k') < \eta_0$, then let $b_{k+1} := 
(b_k \setminus \{ (\beta,\eta) \}) \cup \{ (\beta,\eta_0) \}$. 
Suppose that $\eta_0 \le \max(a_k')$. 
Since $\max(a_k')$ is a limit ordinal and $\max(a_k') < \eta$, 
clearly $\max(a_k') = \eta_0$. 
In this case, let $b_{k+1} := b_k \setminus \{ (\beta,\eta) \}$. 
All of the inductive hypotheses can be easily checked, using Notation 7.2(4) and 
the fact that if $P \in \mathcal Y$ and $P \cap N \cap [\eta_0,\eta) \ne \emptyset$, 
then by the elementarity of $P \cap N$, $\eta \in P \cap N$.

\bigskip

From now on we will assume that $\eta$ is a limit of limit ordinals. 
In particular, every ordinal in $C_\eta$ is a limit ordinal by Notation 7.2(3).

\bigskip

Define 
$$
\theta := \sup(\lim(C_\eta) \cap \cl(N)).
$$
We split the definition of $a_{k+1}$ and $b_{k+1}$ into two cases.

\bigskip

\textbf{\emph{Case 1:}} $\theta = \sup(N)$.

\bigskip

Note that since $\max(a_k') \in N$ and $\sup(N) = \theta$, 
it follows that $\max(a_k') < \theta$.

\bigskip

\emph{Claim 1:} $\eta = \max(b_k')$. 
Suppose for a contradiction that there is $\eta' \in b_k'$ 
greater than $\eta$. 
Fix $\beta'$ with $(\beta',\eta') \in b_k$. 
Then $\eta \le (\eta')^{-}$, where $(\eta')^{-}$ is the largest ordinal 
in $a_k' \cup b_k'$ less than $\eta'$. 
By (I), 
$A_{\eta^*,\beta'} \cap N \cap [(\eta')^{-},\eta^*) \ne \emptyset$. 
Fix $\tau$ in this intersection. 
Then $\theta \le \eta \le (\eta')^{-} \le \tau$ and $\tau \in N$, 
which contradicts that $\theta = \sup(N)$.

\bigskip

Since $N$ is simple, $\ot(C_\theta) = \sup(N \cap \kappa)$. 
Let $\xi := \sup(M \cap N \cap \kappa)$, which is in 
$N \cap \kappa$ by Lemma 1.30. 
Define $\sigma := c_{\theta,\xi}$. 
Since $\theta = \sup(N)$ and $\xi \in N$, $\sigma \in N$ by Lemma 7.19. 
As $\xi$ has countable cofinality, so does $\sigma$. 
In particular, $N \cap \sigma$ is cofinal in $\sigma$.

\bigskip

\emph{Claim 2:} $A_{\theta,\xi} = A_{\sigma,\xi}$ and 
$\sup(A_{\sigma,\xi}) = \sigma$. 
As $\xi$ is a limit ordinal, $\sigma = c_{\theta,\xi} \in \lim(C_\theta)$. 
Therefore $A_{\sigma,\xi} = A_{\theta,\xi} \cap \sigma$. 
In particular, $A_{\sigma,\xi} \subseteq A_{\theta,\xi}$. 
On the other hand, since $\xi < \sup(N \cap \kappa) = \ot(C_\theta)$, 
it follows that $A_{\theta,\xi} \subseteq c_{\theta,\xi} = \sigma$ by Notation 7.4(6). 
So $A_{\theta,\xi} \subseteq A_{\theta,\xi} \cap \sigma = A_{\sigma,\xi}$. 
This proves that $A_{\theta,\xi} = A_{\sigma,\xi}$.

Since $\xi = \sup(M \cap N \cap \kappa)$, by the elementarity of $M \cap N$, 
$\xi$ is a limit of limit ordinals. 
Therefore $\sigma = c_{\theta,\xi}$ is a limit of $\lim(C_\theta) \cap \sigma$. 
For any ordinal $\zeta \in \lim(C_\theta) \cap \sigma$, the fact that 
$\zeta \in \lim(C_\theta) \cap c_{\theta,\xi}$ and $\xi < \ot(C_\theta)$ 
implies by Notation 7.4(6) that $\zeta \in A_{\theta,\xi}$. 
So $\lim(C_\theta) \cap \sigma$ is cofinal in $\sigma$ and is 
a subset of $A_{\theta,\xi}$. 
Hence $\sup(A_{\sigma,\xi}) = \sup(A_{\theta,\xi}) = \sigma$.

\bigskip

We now define $a_{k+1}$ and $b_{k+1}$. 
Let $b_{k+1} := \emptyset$. 
If $\sigma \le \max(a_k')$, then let $a_{k+1} := a_k$. 
If $\max(a_k') < \sigma$, then let $a_{k+1} : = a_k \cup \{ (\beta,\sigma) \}$.

\bigskip

We prove that the inductive hypotheses are maintained. 
First, consider the case when $\sigma \le \max(a_k')$. 
Then $a_{k+1} = a_k$ and $b_{k+1} = \emptyset$. 
Inductive hypotheses (A), (B), (C), (D), (E), (F), and (I) are all either 
vacuously true, or follow immediately from the inductive hypotheses. 

For (G) and (H), suppose that $P \in M \cap \mathcal Y$ satisfies that 
$$
P \cap \kappa \in M \cap N \cap \kappa \ \textrm{and} \ 
P \cap N \cap [\sup(M \cap \eta^*),\eta^*) \ne \emptyset.
$$
By inductive hypothesis (G) and Claim 1, 
$$
P \cap N \cap \eta^* \subseteq 
\max(a_k' \cup b_k') = \eta.
$$
If $P \cap N \cap \eta^* \subseteq \max(a_k')$, then this proves 
(G) for $k+1$, and in that case (H) follows immediately from the inductive 
hypotheses.

Otherwise $P \cap N \cap [\max(a_k'),\eta) \ne \emptyset$. 
Let us show that this is impossible. 
Since $\max(a_k')$ is the predecessor of $\eta$ in $a_k' \cup b_k'$, 
inductive hypothesis (F) implies that $\beta \le P \cap \kappa$. 
Inductive hypothesis (E) then implies that 
$$
A_{\eta,P \cap \kappa} = A_{\eta^*,P \cap \kappa} \cap \eta.
$$
As $\eta^*$ is a limit point of $P$, 
$$
P \cap \eta = A_{\eta^*,P \cap \kappa} \cap \eta = 
A_{\eta,P \cap \kappa}.
$$
Since $\theta \in \lim(C_\eta)$, 
$$
A_{\theta,P \cap \kappa} = A_{\eta,P \cap \kappa} \cap \theta = 
P \cap \theta.
$$
As $\theta = \sup(N)$, 
$$
P \cap N \cap \eta = P \cap N \cap \theta = 
A_{\theta,P \cap \kappa} \cap N.
$$
And as $P \cap \kappa \in M \cap N \cap \kappa$ and 
$\xi = \sup(M \cap N \cap \kappa)$, it follows that $P \cap \kappa \le \xi$, so 
$$
A_{\theta,P \cap \kappa} \subseteq 
A_{\theta,\xi} = A_{\sigma,\xi}
$$
by Claim 2. 
So 
$$
P \cap N \cap \eta = A_{\theta,P \cap \kappa} \cap N 
\subseteq A_{\sigma,\xi} \subseteq \sigma \le \max(a_k').
$$
This contradicts the initial assumption that 
$P \cap N \cap [\max(a_k'),\eta) \ne \emptyset$.

\bigskip

Secondly, consider the case that $\max(a_k') < \sigma$. 
Then $a_{k+1} = a_k \cup \{ (\beta,\sigma) \}$ and $b_{k+1} = \emptyset$.
We prove that the inductive hypotheses are maintained. 
Inductive hypotheses (A), (B), (C), (E), and (F) are all either vacuously true, 
or follow immediately from the inductive hypotheses. 
It remains to show (D), (G), (H), and (I).

\bigskip

(D) By inductive hypothesis (D), we only need to check that (D) holds 
for $k+1$ in the case of $(\beta,\sigma)$. 
As noted in the paragraph before Claim 2, $N \cap \sigma$ 
is cofinal in $\sigma$. 
Also observe that since $\theta \in \lim(C_\eta)$, 
$\sigma = c_{\theta,\xi} = c_{\eta,\xi}$ is a limit point of $C_\eta$.

Let $\beta \le \gamma < \kappa$ be given. 
Since $(\beta,\eta) \in b_k$, inductive hypothesis (E) implies that 
$A_{\eta,\gamma} = A_{\eta^*,\gamma} \cap \eta$. 
Since $\sigma$ is a limit point of $C_\eta$ and 
$\sup(N \cap \sigma) = \sigma$, it follows that 
$A_{\sigma,\gamma} \cap \sup(N \cap \sigma) = 
A_{\sigma,\gamma} \cap \sigma = 
A_{\sigma,\gamma} = 
A_{\eta,\gamma} \cap \sigma = 
(A_{\eta^*,\gamma} \cap \eta) \cap \sigma = 
A_{\eta^*,\gamma} \cap \sigma = 
A_{\eta^*,\gamma} \cap \sup(N \cap \sigma)$, which proves (D).

\bigskip

(G) Suppose that $P \in M \cap \mathcal Y$ satisfies that 
$$
P \cap \kappa \in M \cap N \cap \kappa \ \textrm{and} \ 
P \cap N \cap [\sup(M \cap \eta^*),\eta^*) \ne \emptyset.
$$
By inductive hypothesis (G) and Claim 1, 
$$
P \cap N \cap \eta^* \subseteq \max(a_k' \cup b_k') = \eta.
$$
If $P \cap N \cap \eta^* \subseteq \max(a_k')$, then 
since $\max(a_k') \le \max(a_{k+1}')$, we are done.

So assume that there exists 
$\tau \in (P \cap N \cap \eta^*) \setminus \max(a_k')$. 
We will show that $P \cap N \cap \eta^* \subseteq \sigma$, which 
completes the proof since $\sigma = \max(a_{k+1}')$. 
As $P \cap N \cap \eta^* \subseteq \eta$, it follows that 
$\tau \in P \cap N \cap [\max(a_k'),\eta)$. 
Since $\max(a_k')$ is the largest member of $a_k' \cup b_k'$ 
less than $\eta$, 
inductive hypothesis (F) implies that 
$\beta \le P \cap \kappa$. 
Inductive hypothesis (E) then implies that 
$$
A_{\eta,P \cap \kappa} = A_{\eta^*,P \cap \kappa} \cap \eta.
$$
But $\theta \in \lim(C_\eta)$ and $\theta = \sup(N)$, so 
$$
P \cap N \cap \eta = A_{\eta^*,P \cap \kappa} \cap N \cap \eta = 
A_{\eta,P \cap \kappa} \cap N = 
A_{\eta,P \cap \kappa} \cap N \cap \theta = 
A_{\theta,P \cap \kappa} \cap N.
$$
Hence $P \cap N \cap \eta \subseteq A_{\theta,P \cap \kappa}$. 
Since $P \cap \kappa \in M \cap N \cap \kappa$, 
$P \cap \kappa < \sup(M \cap N \cap \kappa) = \xi$. 
So by Claim 2, 
$$
P \cap N \cap \eta \subseteq A_{\theta,P \cap \kappa} \subseteq 
A_{\theta,\xi} = A_{\sigma,\xi} \subseteq \sigma.
$$
Since $P \cap N \cap \eta^* \subseteq \eta$ as noted above, 
we have that 
$$
P \cap N \cap \eta^* = P \cap N \cap \eta \subseteq \sigma = 
\max(a_{k+1}').
$$

\bigskip

(H) Suppose that $P \in M \cap \mathcal Y$ and $\tau$ satisfy that 
$$
P \cap \kappa \in M \cap N \cap \kappa \ \textrm{and} \ 
\tau \in P \cap N \cap [\sup(M \cap \eta^*),\eta^*).
$$
Let $\sigma' = \min(a_{k+1}' \setminus (\tau+1))$. 
Fix $\beta'$ with $(\beta',\sigma') \in a_{k+1}$. 
If $\sigma' < \sigma$, then clearly 
$\sigma' = \min(a_k' \setminus (\tau+1))$, so (i) and (ii) 
follow from inductive hypothesis (H). 

Suppose that $\sigma' = \sigma$, which means that 
$\max(a_k') \le \tau+1$. 
Then $\beta' = \beta$. 
Since $\max(a_k')$ is the largest ordinal in $a_k' \cup b_k'$ 
less than $\eta$, inductive hypothesis (F) implies that $\beta \le P \cap \kappa$, 
which proves (i). 
By inductive hypothesis (E), 
$$
A_{\eta,P \cap \kappa} = 
A_{\eta^*,P \cap \kappa} \cap \eta.
$$
Since $\sigma \in \lim(C_\eta)$, 
$$
P \cap \sigma = 
A_{\eta^*,P \cap \kappa} \cap \sigma = A_{\eta,P \cap \kappa} \cap \sigma 
= A_{\sigma,P \cap \kappa} = A_{\sigma,P \cap \kappa} \cap \sigma.
$$
As $\sigma = \sup(N \cap \sigma)$, we have that 
$P \cap \sup(N \cap \sigma) = 
A_{\sigma,P \cap \kappa} \cap \sup(N \cap \sigma)$, which proves (ii).

\bigskip

(I) By inductive hypothesis (I), 
it suffices to consider $(\beta,\sigma)$. 
Let $P := Sk(A_{\eta^*,\beta})$. 
Since $(\beta,\eta) \in b_k$ and $\max(a_k')$ is the largest ordinal in 
$a_k' \cup b_k'$ less than $\eta$, 
inductive hypothesis (I) implies that 
$P \in M \cap \mathcal Y$, $P \cap \kappa = \beta$, 
$P \cap \kappa^+ = A_{\eta^*,\beta}$, and 
$P \cap N \cap [\max(a_k'),\eta^*) \ne \emptyset$. 
Since $\max(a_k')$ is also the largest ordinal in $a_{k+1}' \cup b_{k+1}'$ 
less than $\sigma$, we are done.

\bigskip

\textbf{\emph{Case 2:}} $\theta < \sup(N)$.

\bigskip

Let $\sigma' := \min((N \cap \kappa^+) \setminus \theta)$, which 
exists by Case 2. 
If $\sigma' \le \max(a_k')$, then let 
$\sigma := \max(a_k')$ and 
$a_{k+1} := a_k$. 
If $\max(a_k') < \sigma'$, then let $\sigma := \sigma'$ and 
$a_{k+1} := a_k \cup \{ (\beta,\sigma) \}$.

\bigskip

Define $A$ as the set of ordinals of the form 
$\min(C_\eta \setminus (\xi+1))$, where 
for some $P \in M \cap \mathcal Y$ with 
$P \cap \kappa \in M \cap N \cap \kappa$, 
$\xi \in P \cap N \cap [\sigma,\eta)$.\footnote{The set $A$ could be empty.  
In fact, it is possible for example that 
$\theta = \eta^*$ and $\sigma = \min((N \cap \kappa^+) \setminus \eta^*)$, 
so that $\eta < \sigma$. 
In this case we interpret $[\sigma,\eta)$ to be the empty set, so that $A$ 
is empty as well.}

Note that every ordinal in $A$ is a limit ordinal, since $C_\eta$ 
consists of limit ordinals, and is strictly greater than $\sigma$. 
Suppose that $\gamma_0 < \gamma_1$ are in $A$. 
Then for some $\xi \in P \cap N \cap [\sigma,\eta)$, 
$\gamma_1 = \min(C_\eta \setminus (\xi+1))$. 
So $\gamma_0 \le \xi < \xi+1 < \gamma_1$. 
In particular, $N \cap [\gamma_0,\gamma_1) \ne \emptyset$.

\bigskip

\emph{Claim:} $A$ is finite. 
Suppose for a contradiction that $A$ is infinite. 
Fix an increasing sequence $\langle \gamma_n : n < \omega \rangle$ 
from $A$. 
Then $\sigma < \gamma_0$, and for each $n$, $\gamma_n \in C_\eta$ 
and $N \cap [\gamma_n,\gamma_{n+1}) \ne \emptyset$. 
It follows that the ordinal 
$\sup \{ \gamma_n : n < \omega \}$ is in 
$\lim(C_\eta) \cap \cl(N)$, and yet is greater than $\sigma$ and hence $\theta$. 
This contradicts the definition of $\theta$.

\bigskip

For each $\delta \in A$, define $\beta_\delta$ as the least ordinal in 
$M \cap N \cap \kappa$ 
such that for some $P \in M \cap \mathcal Y$, 
$P \cap \kappa = \beta_\delta$, and there is 
$\xi \in P \cap N \cap [\sigma,\eta)$ 
such that $\delta = \min(C_{\eta} \setminus (\xi+1))$. 
Note that $\beta_\delta$ exists by the definition of $A$. 
Also as $\max(a_k') \le \sigma$, 
$P \cap N \cap [\max(a_k'),\eta) \ne \emptyset$. 
Therefore since 
$\max(a_k')$ is the largest ordinal in $a_k' \cup b_k'$ 
less than $\eta$, inductive hypothesis (F) implies that 
$\beta \le P \cap \kappa = \beta_\delta$.

Define $b_{k+1} := (b_k \setminus \{ (\beta,\eta) \}) 
\cup \{ (\beta_\delta,\delta) : \delta \in A \}$.

\bigskip

We verify the inductive hypotheses. 
Hypotheses (A) and (B) are straightforward to check.

\bigskip

(C) We know that $a_{k+1} \ne \emptyset$ and 
$\max(a_{k+1}') = \sigma$. 
If $A$ is nonempty, then 
$$
\max(a_{k+1}') = \sigma < 
\min(A) = \min(b_{k+1}').
$$
If $A$ is empty and $b_{k+1}$ is nonempty, then 
$\min(b_{k+1}')$ is the least member of $b_k'$ greater than $\eta$. 
So if $\max(a_{k+1}') = \max(a_k')$, then by inductive hypothesis (C), 
$$
\max(a_{k+1}') = \max(a_k') < \min(b_k') = \eta < \min(b_{k+1}').
$$
Suppose that $\max(a_k') < \sigma$. 
By inductive hypothesis (B), we have that 
$N \cap [\eta,\min(b_{k+1}')) \ne \emptyset$. 
Since $\theta \le \eta$, $N \cap [\theta,\min(b_{k+1}')) \ne \emptyset$. 
As $\sigma = \min((N \cap \kappa^+) \setminus \theta)$, this implies 
that $\max(a_{k+1}') = \sigma < \min(b_{k+1}')$.

\bigskip

(D) By inductive hypothesis (D), 
it suffices to consider $(\beta,\sigma)$ 
in the case where $\max(a_k') < \sigma$ and 
$a_{k+1} = a_k \cup \{ (\beta,\sigma) \}$. 
So $\sigma = \min((N \cap \kappa^+) \setminus \theta)$, and 
therefore $\theta = \sup(N \cap \sigma)$. 
Let $\beta \le \gamma < \kappa$. 
Then by Lemma 7.12(2), 
$$
A_{\theta,\gamma} = A_{\sigma,\gamma} \cap \theta.
$$
Since $(\beta,\eta) \in b_k$, inductive hypothesis (E) implies that 
$$
A_{\eta,\gamma} = A_{\eta^*,\gamma} \cap \eta.
$$ 
And since $\theta \in \lim(C_\eta)$, 
$$
A_{\theta,\gamma} = A_{\eta,\gamma} \cap \theta.
$$
Therefore 
$$
A_{\sigma,\gamma} \cap \theta = A_{\theta,\gamma} = 
A_{\eta,\gamma} \cap \theta = A_{\eta^*,\gamma} \cap \theta.
$$
Since $\sup(N \cap \sigma) = \theta$, this proves (D).

\bigskip

(E) Consider $(\beta_\delta,\delta) \in b_{k+1}$, where 
$\delta \in A$. 
Fix $P \in M \cap \mathcal Y$ and $\xi$ such that 
$P \cap \kappa = \beta_\delta$, 
$\xi \in P \cap N \cap [\sigma,\eta)$, and 
$\delta = \min(C_\eta \setminus (\xi+1))$. 
Let $\beta_\delta \le \gamma < \kappa$, and we will show that 
$$
A_{\delta,\gamma} = A_{\eta^*,\gamma} \cap \delta.
$$
As observed above, 
$\beta \le \beta_\delta \le \gamma$. 
Since $(\beta,\eta) \in b_k$, 
by inductive hypothesis (E), 
$$
A_{\eta,\beta_\delta} = A_{\eta^*,\beta_{\delta}} \cap \eta \ 
\textrm{and} \ 
A_{\eta,\gamma} = 
A_{\eta^*,\gamma} \cap \eta.
$$
As $\xi \in N \cap P$, by elementarity $\xi + 1 \in N \cap P$. 
Hence 
$$
\xi + 1 \in P \cap \eta = A_{\eta^*,P \cap \kappa} \cap \eta = 
A_{\eta^*,\beta_\delta} \cap \eta = A_{\eta,\beta_\delta}.
$$
So $\xi + 1 \in A_{\eta,\beta_\delta} \setminus C_\eta$. 
By Notation 7.4(7), $\min(C_\eta \setminus (\xi+1)) = \delta$ 
is in $A_{\eta,\beta_\delta}$. 
Since $\beta_\delta \le \gamma$, $\delta \in A_{\eta,\gamma}$. 
Therefore 
$$
A_{\delta,\gamma} = A_{\eta,\gamma} \cap \delta = 
(A_{\eta^*,\gamma} \cap \eta) \cap \delta = A_{\eta^*,\gamma} \cap \delta,
$$
proving (E).

\bigskip

(F) Let $(\gamma,\zeta) \in b_{k+1}$. 
Then either $(\gamma,\zeta) = (\beta_\delta,\delta)$ 
for some $\delta \in A$, or $(\gamma,\zeta) \in b_k$ and $\eta < \zeta$.

\bigskip

\emph{Case a:} $(\gamma,\zeta) = (\beta_\delta,\delta)$ for 
some $\delta \in A$. 
Suppose that $P \in M \cap \mathcal Y$ satisfies that 
$$
P \cap \kappa \in M \cap N \cap \kappa \ \textrm{and} \ 
P \cap N \cap [\delta^{-},\delta) \ne \emptyset,
$$
where 
$\delta^{-}$ is the greatest member of $a_{k+1}' \cup b_{k+1}'$ 
which is less than $\delta$. 
Then clearly $\sigma = \max(a_{k+1}') \le \delta^{-}$. 
So if we fix $\xi \in P \cap N \cap [\delta^{-},\delta)$, 
then $\sigma \le \xi$ and $\delta = \min(C_\eta \setminus (\xi+1))$. 
By the minimality of $\beta_\delta$, $\beta_\delta \le P \cap \kappa$.

\bigskip

\emph{Case b:} $(\gamma,\zeta) \in b_k$ and 
$\eta < \zeta$. 
If $\zeta$ is not the least element of $b_k'$ greater than $\eta$, then 
the greatest ordinal in $a_k' \cup b_k'$ less than $\zeta$ is equal to 
the greatest ordinal in $a_{k+1}' \cup b_{k+1}'$ less than $\zeta$. 
In that case, (F) follows easily from inductive hypothesis (F).

Suppose that $\zeta$ is the least member of $b_k'$ greater than $\eta$. 
Then the greatest member of $a_{k+1}' \cup b_{k+1}'$ less 
than $\zeta$, which we denote by $\zeta^{-}$, 
is equal to either $\max(A)$ if $A$ is nonempty, 
or $\sigma$ if $A$ is empty. 

Assume that $P \in M \cap \mathcal Y$ satisfies that 
$$
P \cap \kappa \in M \cap N \cap \kappa \ \textrm{and} \ 
P \cap N \cap [\zeta^{-},\zeta) \ne \emptyset.
$$
We will show that $\gamma \le P \cap \kappa$. 
If $P \cap N \cap [\eta,\zeta) \ne \emptyset$, then since $\eta$ is the 
greatest member of $a_k' \cup b_k'$ less than $\zeta$, 
$\gamma \le P \cap \kappa$ 
by inductive hypothesis (F).

Otherwise $P \cap N \cap [\zeta^{-},\eta) \ne \emptyset$. 
Fix $\xi$ in this intersection. 
Then $\xi+1$ is also in this intersection, by the elementarity of $P \cap N$ 
and because $\eta$ is a limit ordinal.  
By the definition of $A$, $\min(C_\eta \setminus (\xi+1))$ is in $A$. 
So $A$ is nonempty, and hence $\zeta^{-} = \max(A)$. 
Yet $\min(C_\eta \setminus (\xi+1))$ is in $A$ and is strictly 
greater than $\zeta^{-} = \max(A)$, which is a contradiction.

\bigskip

(G) Let $P \in M \cap \mathcal Y$ satisfy that 
$$
P \cap \kappa \in M \cap N \cap \kappa \ \textrm{and} \ 
P \cap N \cap [\sup(M \cap \eta^*),\eta^*) \ne \emptyset.
$$
By inductive hypothesis (G), 
$$
P \cap N \cap \eta^* \subseteq \max(a_k' \cup b_k') = \max(b_k').
$$
If $\max(b_k') = \max(b_{k+1}')$, then we are done. 
Otherwise $\eta$ is equal to $\max(b_k')$, so 
$P \cap N \cap \eta^* \subseteq \eta$.
If $P \cap N \cap \eta^*$ is not a subset of $\max(a_{k+1}' \cup b_{k+1}')$, 
then there is $\xi \in P \cap N \cap \eta^*$ such that 
$\max(a_{k+1}') = \sigma \le \xi$, and also 
$\max(A) = \max(b_{k+1}') \le \xi$ if $A$ is nonempty. 
So $\xi \in P \cap N \cap [\sigma,\eta)$, which implies 
that $\min(C_\eta \setminus (\xi+1))$ is in $A$. 
So $A$ is nonempty, and $\max(A) \le \xi < 
\min(C_\eta \setminus (\xi+1)) \in A$, which is impossible.

\bigskip

(H) Suppose that $P \in M \cap \mathcal Y$ and $\tau$ satisfy 
$$
P \cap \kappa \in M \cap N \cap \kappa \ \textrm{and} \ 
\tau \in P \cap N \cap [\sup(M \cap \eta^*),\eta^*).
$$
Let 
$$
\sigma^* := \min((a_{k+1}' \cup b_{k+1}') \setminus (\tau + 1)),
$$
and assume that $\sigma^* \in a_{k+1}'$. 
Fix $\beta^*$ with $(\beta^*,\sigma^*) \in a_{k+1}$.

If $\sigma^* \in a_k'$, then the conclusion of (H) follows 
immediately from inductive hypothesis (H) for $a_k$. 
Otherwise we are in the case that $\max(a_k') < \sigma$ and 
$\sigma^* = \sigma$. 
Hence also $\beta^* = \beta$. 
Clearly $\min((a_k' \cup b_k') \setminus (\tau+1))$ is 
equal to $\eta$. 
Since $(\beta,\eta) \in b_k$ and 
$\max(a_k')$ is the greatest member of $a_k' \cup b_k'$ 
less than $\eta$, inductive hypothesis (F) implies that 
$\beta \le P \cap \kappa$, proving (H(i)).

By inductive hypothesis (E), 
$$
A_{\eta,P \cap \kappa} = A_{\eta^*,P \cap \kappa} \cap \eta = 
P \cap \eta.
$$
Since $\theta \le \eta$, 
$$
A_{\eta,P \cap \kappa} \cap \theta = P \cap \theta.
$$
As $\theta \in \lim(C_\eta)$, 
$$
A_{\theta,P \cap \kappa} = A_{\eta,P \cap \kappa} \cap \theta.
$$
By Lemma 7.12(2), 
$$
A_{\theta,P \cap \kappa} = 
A_{\sigma,P \cap \kappa} \cap \theta.
$$
So 
$$
P \cap \theta = 
A_{\eta,P \cap \kappa} \cap \theta = A_{\theta,P \cap \kappa} = 
A_{\sigma,P \cap \kappa} \cap \theta.
$$
Since $\sup(N \cap \sigma) = \theta$, it follows that 
$$
P \cap \sup(N \cap \sigma) = A_{\sigma,P \cap \kappa} \cap 
\sup(N \cap \sigma),
$$
which proves (H(ii)).

\bigskip

(I) Let $(\gamma,\zeta) \in a_{k+1} \cup b_{k+1}$, where 
$\min(a_{k+1}') = \min(a_k') < \zeta$. 
If $(\gamma,\zeta) \in a_k$, then the conclusion of (I) follows 
from inductive hypothesis (I). 
If $(\gamma,\zeta) \in a_{k+1} \setminus a_k$, then 
$(\gamma,\zeta) = (\beta,\sigma)$ and 
$\max(a_k') < \sigma$. 
Let $P := Sk(A_{\eta^*,\beta})$. 
Since $(\beta,\eta) \in b_k$, by inductive hypothesis (I) we know that 
$$
P \in M \cap \mathcal Y, \ P \cap \kappa = \beta, \ \textrm{and} \ 
P \cap \kappa^+ = A_{\eta^*,\beta}.
$$
Also since $\max(a_k')$ is the greatest member of 
$a_k' \cup b_k'$ less than $\eta$, inductive hypothesis (I) implies that 
$$
P \cap N \cap [\max(a_k'),\eta^*) \ne \emptyset.
$$
But the greatest member of $a_{k+1}' \cup b_{k+1}'$ less than 
$\sigma$ is also equal to $\max(a_k')$, so we are done.

Suppose that $(\gamma,\zeta) \in b_{k}$ and $\eta < \zeta$. 
If $\zeta$ is not the second element of $b_k'$, then 
(I) follows immediately from inductive hypothesis (I). 
Suppose that $\zeta$ is the second element of $b_k'$. 
Then $\eta$ is the greatest member of $a_k' \cup b_k'$ less than $\zeta$. 
Let $P := Sk(A_{\eta^*,\gamma})$. 
By inductive hypothesis (I), 
$$
P \in M \cap \mathcal Y, \ 
P \cap \kappa = \gamma, \ 
P \cap \kappa^+ = A_{\eta^*,\gamma}, 
\ \textrm{and} \ P \cap N \cap [\eta,\eta^*) \ne \emptyset.
$$

Let $\zeta^{-}$ denote the largest member of $a_{k+1}' \cup b_{k+1}'$ 
less than $\zeta$, and we will show that 
$P \cap N \cap [\zeta^{-},\eta^*) \ne \emptyset$. 
If $\zeta^{-} \le \eta$, then this follows immediately 
from the fact that $P \cap N \cap [\eta,\eta^*) \ne \emptyset$. 
If $A$ is nonempty, then clearly $\zeta^{-} = \max(A) < \eta$, and 
we are done. 

Suppose that $A$ is empty. 
If $\min((N \cap \kappa^+) \setminus \theta) \le \max(a_k')$, 
then $\zeta^{-} = \max(a_k') < \eta$, and we are done. 
Suppose that $\max(a_k') < \sigma = 
\min((N \cap \kappa^+) \setminus \theta)$. 
By inductive hypothesis (B), $N \cap [\eta,\zeta) \ne \emptyset$. 
Since $\theta \le \eta$, $N \cap [\theta,\zeta) \ne \emptyset$. 
As $\sigma = \min((N \cap \kappa^+) \setminus \theta)$, 
it follows that $\sigma < \zeta$, and so clearly $\zeta^{-} = \sigma$. 
If $\sigma \le \eta$, then we are done. 
Otherwise $\sigma = \min((N \cap \kappa^+) \setminus \eta)$. 
We know that $P \cap N \cap [\eta,\eta^*) \ne \emptyset$. 
But the first member of this intersection must be greater than or equal to 
$\min((N \cap \kappa^+) \setminus \eta) = \sigma$. 
Hence $P \cap N \cap [\sigma,\eta^*) \ne \emptyset$, and we are done 
since $\sigma = \zeta^{-}$.

In the final case, assume that $(\gamma,\zeta)$ is equal to 
$(\beta_\delta,\delta)$, for some $\delta \in A$. 
By the definition of $\beta_{\delta}$, there exists 
$Q \in M \cap \mathcal Y$ and $\xi$ such that 
$$
Q \cap \kappa = \beta_\delta, \ \xi \in Q \cap N \cap [\sigma,\eta), 
\ \textrm{and} \ \delta = \min(C_\eta \setminus (\xi+1)).
$$
So $Q \cap \eta^* = A_{\eta^*,\beta_{\delta}}$.
Since $Q$ and $\eta^*$ are closed under $H^*$, it follows that 
$A_{\eta^*,\beta_\delta}$ is closed under $H^*$.

We will show that $P := Sk(A_{\eta^*,\beta_\delta})$ satisfies 
the conclusions of (I). 
Since $A_{\eta^*,\beta_\delta}$ is closed under $H^*$, 
$$
P \cap \kappa^+ = A_{\eta^*,\beta_{\delta}} = Q \cap \eta^*.
$$
Hence also 
$$
P \cap \kappa = Q \cap \kappa = \beta_\delta.
$$
Since $\eta^*$ is a limit point of $Q$, 
$$
\sup(P) = \sup(Q \cap \eta^*) = \eta^*.
$$
As $\eta^*$ and $\beta_\delta$ are in $M$, so is $P$.

To show that $P \in \mathcal Y$, it suffices to show that 
$\lim(C_{\eta^*}) \cap P$ is cofinal in $\eta^*$. 
Since $Q \cap \eta^* = P \cap \eta^*$, $\eta^*$ is a limit point of $Q$. 
If $\eta^*$ is not in $Q$, then $\eta^* \in \cl(Q) \setminus Q$. 
By Lemma 7.13, 
$$
\lim(C_{\eta^*}) \cap Q = \lim(C_{\eta^*}) \cap P
$$
is 
cofinal in $\eta^*$. 
Otherwise $\eta^* \in Q$. 
Since $\eta^*$ is a limit point of $Q$, $\cf(\eta^*) < \kappa$. 
Therefore $\ot(C_{\eta^*}) \in Q \cap \kappa$ by elementarity. 
Hence $C_{\eta^*} \subseteq Q$ by elementarity. 
Since $\eta^*$ has uncountable cofinality, $\lim(C_{\eta^*})$ is cofinal 
in $\eta^*$. 
So 
$$
\lim(C_{\eta^*}) \cap P = 
\lim(C_{\eta^*}) \cap Q = \lim(C_{\eta^*})
$$
is cofinal in $\eta^*$.

Let $\delta^{-}$ denote the largest member of $a_{k+1}' \cup b_{k+1}'$ 
which is less than $\delta$. 
Then either $\delta^{-} = \sigma$ if $\delta = \min(A)$, or else 
$\delta^{-}$ is the largest member of $A$ which is less than $\delta$. 
In the first case, the ordinal $\xi$, which is in 
$Q \cap N \cap [\sigma,\eta)$, is a witness to the fact that 
$Q \cap N \cap [\delta^{-},\eta^*) \ne \emptyset$. 
In the second case, $\xi+1$ must be greater than $\delta^{-}$, 
since otherwise $\delta = \min(C_\eta \setminus (\xi+1)) \le \delta^{-}$. 
So $\xi+1$ is a witness to the fact that 
$Q \cap N \cap [\delta^{-},\eta^*) \ne \emptyset$. 
In either case, since $Q \cap \eta^* \subseteq P$, 
$P \cap N \cap [\delta^{-},\eta^*) \ne \emptyset$.

\bigskip

\addcontentsline{toc}{section}{13. Amalgamation of side conditions}

\textbf{\S 13. Amalgamation of side conditions}

\stepcounter{section}

\bigskip

We are now in a position to prove amalgamation results for 
$\vec S$-obedient side conditions over simple models in $\mathcal X$, 
strong models in $\mathcal Y$, and transitive models. 
The proofs of these results will use almost the entirety of the technology developed in 
the paper thus far. 
In Part III, the amalgamation results we present here will be 
used to prove the existence of 
strongly generic conditions.

\begin{proposition}
Let $(A,B)$ be an $\vec S$-obedient side condition, where 
$A \subseteq \mathcal X$ and $B \subseteq \mathcal Y$. 
Suppose that $N \in A$ is simple, and $(A,B)$ is closed under 
canonical models with respect to $N$. 
Assume that for all $M \in A$, if $M < N$ then 
$M \cap N \in A$.

Let $(C,D)$ be an $\vec S$-obedient side condition, where 
$C \subseteq \mathcal X$ and $D \subseteq \mathcal Y$, 
such that 
$$
A \cap N \subseteq C \subseteq N \ \textrm{and} \ 
B \cap N \subseteq D \subseteq N.
$$
Also assume that $N' \in C$ is simple, and 
for all $M \in A$, if $M < N$, then 
there is $M'$ in $C$ such that 
$$
M' < N', \ M \cap N = 
M' \cap N', \ \textrm{and} \ p(M,N) = p(M',N').
$$

Then $(A \cup C,B \cup D)$ is an $\vec S$-obedient side condition.
\end{proposition}

\begin{proof}
First note that for all $M \in A$, if $M < N$ then $M \cap N \in C$. 
For since $N$ is simple, $M \cap N \in N$ by Lemma 8.2. 
So $M \cap N \in A \cap N \subseteq C$.

Consider $M < N$ in $A$. 
Since $M \cap N = M' \cap N'$ and $M' < N'$, it follows 
by Lemma 1.19(2) that 
$$
M \cap \beta_{M,N} = M \cap N \cap \kappa = M' \cap N' \cap \kappa = 
M' \cap \beta_{M',N'}.
$$
So $M \cap \beta_{M,N} = M' \cap \beta_{M',N'}$. 
Also by Lemma 1.19(3), 
$$
\beta_{M,N} = \min(\Lambda \setminus \sup(M \cap N \cap \kappa)) = 
\min(\Lambda \setminus \sup(M' \cap N' \cap \kappa)) = \beta_{M',N'}.
$$
So $\beta_{M,N} = \beta_{M',N'}$.

To show that $(A \cup C,B \cup D)$ is $\vec S$-obedient, we verify 
properties (1), (2), and (3) of Definition 5.3. 
(2) is immediate. 

\bigskip

(3) Let $M \in C$ and $P \in B$. 
Let $\beta := P \cap \kappa$, and 
suppose that $\zeta = \min((M \cap \kappa) \setminus \beta)$. 
Fix $\tau \in M \cap P \cap \kappa^+$, and we will show that $\zeta \in S_\tau$. 
If $\beta = \zeta$, then $\zeta \in S_\tau$ since $P$ is $\vec S$-strong. 
So assume that $\beta < \zeta$, which means that $\beta \notin M$. 
If $P \in N$ then $P \in B \cap N \subseteq D$, 
so $\zeta \in S_\tau$ since $(C,D)$ is $\vec S$-obedient.

Assume that $P \notin N$. 
Since $M \in C$ and $C \subseteq N$, $M \in N$. 
Therefore $\zeta$ and $\tau$ are in $N$. 
Hence $P \cap \kappa = \beta  < \zeta < \sup(N \cap \kappa)$. 
Let $\xi := \min((N \cap \kappa) \setminus \beta)$. 
Since $M \subseteq N$, $\zeta = \min((M \cap \kappa) \setminus \xi)$. 
By Lemma 10.9(1), there is $Q \in B \cap N \subseteq D$ 
such that $Q \cap \kappa = \xi$ and $\tau \in Q$. 
Then $\zeta = \min((M \cap \kappa) \setminus (Q \cap \kappa))$ and 
$\tau \in M \cap Q \cap \kappa^+$. 
So $\zeta \in S_\tau$ since $(C,D)$ is $\vec S$-obedient.

\bigskip

Let $M \in A$ and $P \in D$. 
Let $\beta := P \cap \kappa$, and suppose that 
$\zeta = \min((M \cap \kappa) \setminus \beta)$. 
Fix $\tau \in M \cap P \cap \kappa^+$, and we will show that $\zeta \in S_\tau$. 
If $\beta = \zeta$, then $\zeta \in S_\tau$ since $P$ is $\vec S$-strong. 
So assume that $\beta < \zeta$, which means that $\beta \notin M$.

Since $P \in D$ and $D \subseteq N$, $P \in N$. 
So $\beta = P \cap \kappa$ is in $N$ by elementarity.

\bigskip

\emph{Case 1:} $\beta_{M,N} \le \beta$. 
Since $\beta \in N$, 
$\zeta = \min((M \cap \kappa) \setminus \beta)$ is in $R_N(M)$. 
As $P \in N \cap \mathcal Y$ is $\vec S$-strong, 
$\tau \in M \cap P \cap \kappa^+$, 
and 
$$
\sup(M \cap \zeta) < P \cap \kappa = \beta < \zeta,
$$
it follows that 
$\zeta \in S_\tau$ since $A$ is $\vec S$-adequate.

\bigskip

\emph{Case 2:} $\beta < \beta_{M,N} \le \zeta$. 
Then $\zeta = \min((M \cap \kappa) \setminus \beta_{M,N})$. 
Since $\beta \in (N \cap \beta_{M,N}) \setminus M$, 
it follows that $M < N$. 
Therefore $\zeta \in R_N(M)$. 
As $P \in N \cap \mathcal Y$ is $\vec S$-strong, 
$\tau \in M \cap P \cap \kappa^+$, and 
$$
\sup(M \cap \zeta) < P \cap \kappa = \beta < \zeta,
$$
we have that 
$\zeta \in S_\tau$ since $A$ is $\vec S$-adequate.

\bigskip

\emph{Case 3:} $\beta < \zeta < \beta_{M,N}$. 
Since $\beta \in (N \cap \beta_{M,N}) \setminus M$, 
it follows that $M < N$. 
As $M < N$, $P \in N \cap \mathcal Y$, and 
$$
P \cap \kappa < \zeta < 
\sup(M \cap \beta_{M,N}) = \sup(M \cap N \cap \kappa),
$$
it follows that $M \cap P \cap \kappa^+ \subseteq N$ by Lemma 8.7. 
In particular, $\tau \in M \cap N \cap \kappa^+$. 
As $\zeta < \beta_{M,N}$ and 
$M \cap \beta_{M,N} = M \cap N \cap \kappa$, 
$\zeta = \min((M \cap N \cap \kappa) \setminus \beta)$. 
But $M \cap N \in C$, $P \in D$, 
and $\tau \in (M \cap N) \cap P \cap \kappa^+$. 
So $\zeta \in S_\tau$ since $(C,D)$ is $\vec S$-obedient.

\bigskip

(1) Now we prove that $A \cup C$ is $\vec S$-adequate. 
By Proposition 1.29, $A \cup C$ is adequate. 
Let $M \in A$ and $L \in C$. 
Then $L \in N$. 
We will prove that the remainder points in $R_M(L)$ and $R_L(M)$ 
are as required.

\bigskip

First, consider $\zeta \in R_L(M)$. 
Then by Lemmas 2.4 and 2.5, either 
(1) $M < N$, $\zeta < \beta_{M,N}$, and 
$\zeta \in R_L(M \cap N)$, or (2) $\beta_{M,N} \le \zeta$ and 
$\zeta \in R_N(M)$.

\bigskip

\emph{Case 1:} $M < N$, $\zeta < \beta_{M,N}$, and 
$\zeta \in R_L(M \cap N)$. 
Recall that $L$ and $M \cap N$ are in $C$. 
Fix $\tau \in L \cap M \cap \kappa^+$, 
and we will show that $\zeta \in S_\tau$. 
Since $L \in N$, $\tau \in N$. 
So $\tau \in L \cap (M \cap N)$. 
Since $\zeta \in R_L(M \cap N)$, it follows that $\zeta \in S_\tau$ since 
$C$ is $\vec S$-adequate.

Suppose that $P \in L \cap \mathcal Y$ is $\vec S$-strong, 
$$
\sup(M \cap \zeta) < P \cap \kappa < \zeta, \ \textrm{and} \ 
\tau \in M \cap P \cap \kappa^+.
$$
We will show that $\zeta \in S_\tau$. 
Since $P \in L$ and $L \in N$, $P \in N$. 
So $P \in N \cap \mathcal Y$, $M < N$, and 
$$
P \cap \kappa <  \zeta < \sup(M \cap \beta_{M,N}) = 
\sup(M \cap N \cap \kappa).
$$
By Lemma 8.7, 
$M \cap P \cap \kappa^+ \subseteq N$. 
In particular, $\tau \in N$. 
So $\tau \in (M \cap N) \cap P \cap \kappa^+$. 
Since $\zeta < \beta_{M,N}$ and $M < N$, we have that 
$M \cap \zeta = M \cap N \cap \zeta$. 
Therefore 
$$
\sup((M \cap N) \cap \zeta) = \sup(M \cap \zeta) < 
P \cap \kappa < \zeta.
$$
So $\zeta \in R_L(M \cap N)$, $P \in L \cap \mathcal Y$ is $\vec S$-strong,  
$\sup((M \cap N) \cap \zeta) < P \cap \kappa < \zeta$, 
and $\tau \in (M \cap N) \cap P \cap \kappa^+$. 
It follows that $\zeta \in S_\tau$ since $C$ is $\vec S$-adequate.

\bigskip

\emph{Case 2:} $\beta_{M,N} \le \zeta$ and $\zeta \in R_N(M)$. 
Fix $\tau \in L \cap M \cap \kappa^+$, and we will show that $\zeta \in S_\tau$. 
Since $L \in N$, $\tau \in N$. 
So $\tau \in M \cap N \cap \kappa^+$. 
As $\zeta \in R_N(M)$, it follows that 
$\zeta \in S_\tau$ since $A$ is $\vec S$-adequate.

Suppose that $P \in L \cap \mathcal Y$ is $\vec S$-strong and 
$\sup(M \cap \zeta) < P \cap \kappa < \zeta$. 
Fix $\tau \in M \cap P \cap \kappa^+$, 
and we will show that $\zeta \in S_\tau$. 
Since $P \in L$ and $L \in N$, $P \in N$. 
So $P \in N \cap \mathcal Y$ is $\vec S$-strong. 
Since $\zeta \in R_N(M)$,  
$\sup(M \cap \zeta) < P \cap \kappa < \zeta$, and $\tau \in M \cap P$, 
it follows that 
$\zeta \in S_\tau$ since $A$ is $\vec S$-adequate.

\bigskip

This completes the proof that the ordinals in $R_L(M)$ are as required.

\bigskip
\bigskip

Now consider $\zeta \in R_M(L)$. 
Then by Lemmas 2.4 and 2.5, 
either $M < N$ and $\zeta \in R_{M \cap N}(L)$, 
or there is $\xi \in R_M(N)$ such that 
$\zeta = \min((L \cap \kappa) \setminus \xi)$.
Since $\zeta \in L$ and $L \in N$, $\zeta \in N$.

\bigskip

Let $\tau \in L \cap M \cap \kappa^+$, 
and we will show that $\zeta \in S_\tau$. 
Since $L \in N$, $\tau \in N$. 
So $\tau \in L \cap (M \cap N)$. 
First, assume that $M < N$ and $\zeta \in R_{M \cap N}(L)$. 
Since $\zeta \in R_{M \cap N}(L)$ and $\tau \in L \cap (M \cap N)$, 
it follows that $\zeta \in S_\tau$ since $C$ is $\vec S$-adequate.

Secondly, assume that there is $\xi \in R_M(N)$ such that 
$\zeta = \min((L \cap \kappa) \setminus \xi)$. 
Since $\xi \in R_M(N)$ and $\tau \in M \cap N \cap \kappa^+$, 
by Lemma 10.9(2) there is $Q \in B \cap N \subseteq D$ 
such that $Q \cap \kappa = \xi$ and $\tau \in Q$. 
Then $\zeta = \min((L \cap \kappa) \setminus (Q \cap \kappa))$ 
and $\tau \in L \cap Q \cap \kappa^+$, which implies that 
$\zeta \in S_\tau$ since $(C,D)$ is $\vec S$-obedient.

\bigskip

Suppose that $P \in M \cap \mathcal Y$ is $\vec S$-strong, 
$$
\sup(L \cap \zeta) < P \cap \kappa < \zeta, \ \textrm{and} \ 
\tau \in L \cap P \cap \kappa^+.
$$ 
We will prove that $\zeta \in S_\tau$. 
Note that since $\tau \in L$ and $L \in N$, $\tau \in N$.

\bigskip

\emph{Case 1:} $\beta_{M,N} \le P \cap \kappa$. 
Let $\theta := \min((N \cap \kappa) \setminus (P \cap \kappa))$. 
Note that $\theta$ exists since $\zeta \in N$. 
Also $\zeta = \min((L \cap \kappa) \setminus \theta)$. 
Since $P \cap \kappa \in (M \cap \kappa) \setminus \beta_{M,N}$, 
it follows that $\theta \in R_M(N)$ and 
$$
\sup(N \cap \theta) < P \cap \kappa < \theta.
$$
Also $\tau \in N \cap P \cap \kappa^+$. 
By Lemma 10.9(4), there is $Q \in B \cap N \subseteq D$ such that 
$Q \cap \kappa = \theta$ and $\tau \in Q$. 
But then $\zeta = \min((L \cap \kappa) \setminus (Q \cap \kappa))$ and 
$\tau \in L \cap Q \cap \kappa^+$, which implies that 
$\zeta \in S_\tau$ since $(C,D)$ is $\vec S$-obedient.

\bigskip

\emph{Case 2:} $P \cap \kappa < \beta_{M,N}$ and $N \le M$. 
Recall that either $M < N$ and $\zeta \in R_{M \cap N}(L)$, 
or there is $\xi \in R_M(N)$ such that 
$\zeta = \min((L \cap \kappa) \setminus \xi)$. 
Since $N \le M$, we are in the second case. 

\bigskip

\emph{Subcase 2(a):} $P \cap \kappa < \sup(N \cap \beta_{M,N})$. 
Since $N \le M$, $\sup(N \cap \beta_{M,N}) = \sup(M \cap N \cap \kappa)$. 
Therefore $P \cap \kappa < \sup(M \cap N \cap \kappa)$. 
Since $\tau \in N \cap P \cap \kappa^+$, it follows that 
$\tau \in M$ by Lemma 8.7. 
So $\tau \in M \cap N \cap \kappa^+$. 
Since $\xi \in R_M(N)$, 
by Lemma 10.9(2) there is $Q \in B \cap N \subseteq D$ such that 
$Q \cap \kappa = \xi$ and $\tau \in Q$. 
Since $\zeta = \min((L \cap \kappa) \setminus (Q \cap \kappa))$ 
and $\tau \in L \cap Q \cap \kappa^+$, it follows that 
$\zeta \in S_\tau$ since $(C,D)$ is $\vec S$-obedient.

\bigskip

\emph{Subcase 2(b):} $\sup(N \cap \beta_{M,N}) \le P \cap \kappa$. 
Since $\sup(N \cap \beta_{M,N})$ has countable cofinality and 
$P \cap \kappa$ has uncountable cofinality, we have that 
$\sup(N \cap \beta_{M,N}) < P \cap \kappa$. 
Let $\delta := \min((N \cap \kappa) \setminus \beta_{M,N})$, 
which exists since $\zeta \in N$. 
As $N \le M$, $\delta \in R_M(N)$. 
Also 
$$
\sup(N \cap \delta) = \sup(N \cap \beta_{M,N}) < P \cap \kappa < \delta
$$
and $\tau \in N \cap P \cap \kappa^+$.  
By Lemma 10.9(4), there is $Q \in B \cap N \subseteq D$ such that 
$Q \cap \kappa = \delta$ and $\tau \in Q$. 
Since $\sup(N \cap \delta) < P \cap \kappa < \delta$, we have that 
$\delta = \min((N \cap \kappa) \setminus (P \cap \kappa))$. 
As $L \subseteq N$ and $\sup(L \cap \zeta) < P \cap \kappa < \zeta$, 
clearly $\zeta = \min((L \cap \kappa) \setminus \delta)$. 
So $\zeta = \min((L \cap \kappa) \setminus (Q \cap \kappa))$ 
and $\tau \in L \cap Q \cap \kappa^+$. 
It follows that $\zeta \in S_\tau$ since $(C,D)$ is $\vec S$-obedient.

\bigskip

\emph{Case 3:} $P \cap \kappa < \beta_{M,N}$ and $M < N$. 
Then $P \cap \kappa \in M \cap \beta_{M,N} = M \cap N \cap \kappa$.

\bigskip

\emph{Subcase 3(a):} $\tau < \alpha_{M,N}$ and $P \cap \alpha_{M,N}$ 
is bounded below $\alpha_{M,N}$. 
Then $\tau \in P \cap N \cap \alpha_{M,N}$. 
By Lemma 11.7, there is $P' \in M \cap N \cap \mathcal Y$ which is 
$\vec S$-strong such that $P' \cap \kappa = P \cap \kappa$ and $\tau \in P'$.

Recall that either $\zeta \in R_{M \cap N}(L)$, or there is $\xi \in R_M(N)$ 
such that $\zeta = \min((L \cap \kappa) \setminus \xi)$. 
Suppose first that $\zeta \in R_{M \cap N}(L)$. 
Since $P' \in M \cap N \cap \mathcal Y$ is $\vec S$-strong, 
$\sup(L \cap \zeta) < P \cap \kappa = P' \cap \kappa < \zeta$, 
and $\tau \in L \cap P' \cap \kappa^+$, 
it follows that $\zeta \in S_\tau$ since $C$ is $\vec S$-adequate.

Now suppose that 
there is $\xi \in R_M(N)$ such that 
$\zeta = \min((L \cap \kappa) \setminus \xi)$. 
Then by Lemma 10.9(3), there is $Q \in B \cap N \subseteq D$ 
such that $Q \cap \kappa = \xi$ and 
$N \cap P' \cap \kappa^+ \subseteq Q$. 
So $\zeta = \min((L \cap \kappa) \setminus (Q \cap \kappa))$ and 
$\tau \in L \cap Q \cap \kappa^+$. 
It follows that $\zeta \in S_\tau$ since $(C,D)$ is $\vec S$-obedient.

\bigskip

\emph{Case 3(b):} Either $\tau < \alpha_{M,N}$ and $P \cap \alpha_{M,N}$ 
is unbounded in $\alpha_{M,N}$, or $\alpha_{M,N} \le \tau$. 
In the first case, we apply Lemma 11.6 letting $\tau_0 = \tau$ to get 
that there exists 
$Q \in N' \cap \mathcal Y$ which is 
$\vec S$-strong such that 
$Q \cap \kappa = P \cap \kappa$ and $\tau \in Q$. 
In the second case, we apply the main proxy lemma, Lemma 11.5. 
Since $\tau \in P$ and $P \in M$, 
$$
\alpha_{M,N} \le \tau < \sup(P) < \sup(M).
$$
Let $\eta := \min((M \cap \kappa^+) \setminus \tau)$, which is in $R^+_N(M)$. 
Note that the assumptions of Lemma 11.5 for $\eta^* = \eta$ are satisfied. 
By Lemma 11.5(1), there is $Q \in N' \cap \mathcal Y$ which is 
$\vec S$-strong such that 
$Q \cap \kappa = P \cap \kappa$ and $\tau \in Q$. 
In either case, we have that $Q \in N' \cap \mathcal Y$ is $\vec S$-strong, 
$Q \cap \kappa = P \cap \kappa$, and $\tau \in Q$.

Let us note that if $\zeta \in R_{N'}(L)$, 
then $\zeta \in S_\tau$ and we are done. 
For $Q \in N' \cap \mathcal Y$ is $\vec S$-strong, 
$$
\sup(L \cap \zeta) < P \cap \kappa = Q \cap \kappa < \zeta,
$$
and $\tau \in L \cap Q \cap \kappa^+$. 
It follows that $\zeta \in S_\tau$ since $C$ is $\vec S$-adequate.

\bigskip

\emph{Subcase 3(b(i)):} $\beta_{L,N'} \le \zeta$. 
We claim that $\zeta \in R_{N'}(L)$, which finishes the proof. 
If $\beta_{L,N'} \le P \cap \kappa$, 
then $P \cap \kappa = Q \cap \kappa \in 
(N' \cap \kappa) \setminus \beta_{L,N'}$. 
Therefore $\zeta = \min((L \cap \kappa) \setminus (P \cap \kappa))$ 
is in $R_{N'}(L)$. 
Suppose on the other hand that $P \cap \kappa < \beta_{L,N'}$. 
Then $\zeta = \min((L \cap \kappa) \setminus \beta_{L,N'})$. 
Since $P \cap \kappa \in (N' \cap \beta_{L,N'}) \setminus L$, 
we have that $L < N'$. 
So $\zeta \in R_{N'}(L)$.

\bigskip

\emph{Subcase 3(b(ii)):} $\zeta < \beta_{L,N'}$. 
In particular, since $Q \cap \kappa = P \cap \kappa < \zeta$ and 
$\zeta \in L \cap \beta_{L,N'}$, it follows that 
$$
Q \cap \kappa < \sup(L \cap \beta_{L,N'}) < \beta_{L,N'}.
$$
As $Q \cap \kappa \in (N' \cap \beta_{L,N'}) \setminus L$, 
we have that $L < N'$. 
So $L < N'$, $Q \in N' \cap \mathcal Y$, and 
$$
Q \cap \kappa < \sup(L \cap \beta_{L,N'}) = \sup(L \cap N' \cap \kappa).
$$
By Lemma 8.7, $Q \cap L \cap \kappa^+ \subseteq N'$. 
Since $\tau \in L \cap Q \cap \kappa^+$, it follows that $\tau \in N'$. 
By Lemma 11.5(2) in the case that $\alpha_{M,N} \le \tau$, and 
by Lemma 11.6(2) in the case that $\tau < \alpha_{M,N}$, 
there is $P' \in M' \cap \mathcal Y$ which is $\vec S$-strong 
such that $P' \cap \kappa = P \cap \kappa$ and $\tau \in P'$.

By Lemmas 1.27(1) and 2.6(1), 
either $\beta_{L,M} = \beta_{L,M \cap N} = \beta_{L,M'}$, 
or $\beta_{M,N} < \beta_{L,M'}$. 
Suppose first that $\beta_{L,M} = \beta_{L,M'}$. 
We claim that 
$$
\beta_{L,M'} \le P \cap \kappa.
$$
Suppose for a contradiction that $P \cap \kappa < \beta_{L,M'} = \beta_{L,M}$. 
Since $\Lambda$ is cofinal in $P \cap \kappa$ by the elementarity of $P$, 
and $\sup(L \cap \zeta) < P \cap \kappa$, we can find 
$\pi \in \Lambda \cap P \cap \kappa$ such that $\sup(L \cap \zeta) < \pi$. 
Then $\pi < \beta_{L,M}$. 
By Lemma 1.19(5),
$$
L \cap M \cap [\pi,\beta_{L,M}) \ne \emptyset.
$$
Fix $\xi$ in this intersection. 
As $\zeta \in R_M(L)$, $\beta_{L,M} \le \zeta$. 
So $\xi \in L$ and 
$$
\sup(L \cap \zeta) < \pi \le \xi < \beta_{L,M} \le \zeta.
$$ 
But $\sup(L \cap \zeta) < \xi < \zeta$ and $\xi \in L$ is obviously 
impossible. 
Hence indeed $\beta_{L,M'} \le P \cap \kappa$.

So $P \cap \kappa = P' \cap \kappa \in 
(M' \cap \kappa) \setminus \beta_{L,M'}$. 
Since $\zeta = \min((L \cap \kappa) \setminus (P' \cap \kappa))$, 
we have that $\zeta \in R_{M'}(L)$. 
As $P' \in M' \cap \mathcal Y$ is $\vec S$-strong, 
$\sup(L \cap \zeta) < P' \cap \kappa < \zeta$, 
and $\tau \in L \cap P' \cap \kappa^+$, 
it follows that $\zeta \in S_\tau$ since $C$ is $\vec S$-adequate.

The other alternative is that $\beta_{M,N} < \beta_{L,M'}$. 
We will show that this is impossible. 
So assume that $\beta_{M,N} < \beta_{L,M'}$. 
We claim that $L < M'$. 
For $P \cap \kappa = P' \cap \kappa \in M'$. 
Also 
$$
P \cap \kappa < \beta_{M,N} < \beta_{L,M'}.
$$
So 
$$
P \cap \kappa \in (M' \cap \beta_{L,M'}) \setminus L,
$$
which implies that $L < M'$.

Next we claim that $\beta_{M,N} \le \zeta$. 
Suppose for a contradiction that $\zeta < \beta_{M,N}$. 
Then $\zeta \in L \cap \beta_{M,N} \subseteq L \cap \beta_{L,M'}$, and 
$L \cap \beta_{L,M'} \subseteq M'$ since $L < M'$. 
So 
$$
\zeta \in M' \cap \beta_{M,N} = 
M' \cap \beta_{M',N'} = M \cap \beta_{M,N}.
$$
Hence $\zeta \in M$. 
But this is not true since $\zeta \in R_M(L)$.

Since $P \cap \kappa < \beta_{M,N} \le \zeta$ and 
$\sup(L \cap \zeta) < P \cap \kappa$, 
clearly $\zeta = \min((L \cap \kappa) \setminus \beta_{M,N})$. 
Since $\beta_{M,N} < \beta_{L,M'}$ and $\beta_{M,N} \in \Lambda$, 
by Lemma 1.19(5) we have that 
$L \cap M' \cap [\beta_{M,N},\beta_{L,M'})$ is nonempty. 
Since $\zeta = \min((L \cap \kappa) \setminus \beta_{M,N})$, it follows 
that $\zeta < \beta_{L,M'}$. 
As $L < M'$, $\zeta \in M'$.

Now we will get a contradiction. 
By Subcase 3(b(ii)), $\zeta < \beta_{L,N'}$. 
So $\zeta \in L \cap \beta_{L,N'}$. 
Since $L < M' < N'$, we have that $L < N'$. 
Hence $\zeta \in N'$. 
So $\zeta \in M' \cap N' \cap \kappa$, which implies 
that $\zeta < \beta_{M',N'} = \beta_{M,N}$. 
But $\beta_{M,N} \le \zeta$, and we have a contradiction.
\end{proof}

\begin{proposition}
Let $(A,B)$ be an $\vec S$-obedient side condition, 
where $A \subseteq \mathcal X$ and $B \subseteq \mathcal Y$. 
Suppose that $P \in B$ satisfies that $\cf(\sup(P)) = P \cap \kappa$. 
Assume that 
for all $M \in A$, $M \cap P \in A$, and 
for all $Q \in B$, if $Q \cap \kappa < P \cap \kappa$ then 
$Q \cap P \in B$.

Let $(C,D)$ be an $\vec  S$-obedient side condition, 
where $C \subseteq \mathcal X$ and $D \subseteq \mathcal Y$, 
such that 
$$
A \cap P \subseteq C \subseteq P \ \textrm{and} \ 
B \cap P \subseteq D \subseteq P.
$$
In addition, assume that there exists $P' \in D$ such that 
$\cf(\sup(P')) = P' \cap \kappa$, and 
for all $M \in A$, there exists $M' \in C$ such that 
$$
M \cap P = M' \cap P' \ \textrm{and} \ p(M,P) = p(M',P').
$$

Then $(A \cup C,B \cup D)$ is an $\vec S$-obedient side condition.
\end{proposition}

\begin{proof}
By Lemma 8.3, $P$ and $P'$ are simple. 
Note that for all $M \in A$, $M \cap P \in C$. 
For $M \cap P \in P$ by Lemma 8.4, and 
so $M \cap P \in A \cap P \subseteq C$. 
Similarly, for all $Q \in B$ with $Q \cap \kappa < P \cap \kappa$, 
$Q \cap P \in D$. 
For $Q \cap P \in P$ by Lemma 8.4, and 
hence $Q \cap P \in B \cap P \subseteq D$.

Let $\beta := P \cap \kappa$ and $\beta' := P' \cap \kappa$. 
Consider $M \in A$. 
Then by our assumptions, 
$$
M \cap \beta = M \cap P \cap \kappa = 
M' \cap P' \cap \kappa = M' \cap \beta'.
$$
So $M \cap \beta = M' \cap \beta'$. 
In particular, since $\beta'$ has uncountable cofinality, 
$\sup(M \cap \beta) < \beta'$.

To show that $(A \cup C,B \cup D)$ is $\vec S$-obedient, 
we verify properties 
(1), (2), and (3) of Definition 5.3. 
(2) is immediate.

\bigskip

(3) Let $M \in A$ and $Q \in D$. 
Let $\theta := Q \cap \kappa$, and suppose that 
$\zeta = \min((M \cap \kappa) \setminus \theta)$. 
Let $\tau \in Q \cap M \cap \kappa^+$, and we will show that 
$\zeta \in S_\tau$.

\bigskip

\emph{Case 1:} $P \cap \kappa \le \zeta$. 
Since $Q \in P$, $\theta < P \cap \kappa$, so 
$\zeta = \min((M \cap \kappa) \setminus (P \cap \kappa))$. 
As $\tau \in Q$ and $Q \in P$, $\tau \in P$. 
So $\tau \in P \cap M \cap \kappa^+$. 
Therefore $\zeta \in S_\tau$ since $(A,B)$ is $\vec S$-obedient.

\bigskip

\emph{Case 2:} $\zeta < P \cap \kappa$. 
Then $\zeta \in M \cap P \cap \kappa$. 
Since $M \cap P \cap \kappa = M \cap \beta$ 
is an initial segment of $M \cap \kappa$, 
$\zeta = \min((M \cap P \cap \kappa) \setminus \theta)$. 
As $\tau \in Q$ and $Q \in P$, $\tau \in P$. 
So $\tau \in Q \cap (M \cap P) \cap \kappa^+$. 
Since $M \cap P \in C$, it follows that 
$\zeta \in S_\tau$ as $(C,D)$ is $\vec S$-obedient.

\bigskip

Let $M \in C$ and $Q \in B$. 
Let $\theta := Q \cap \kappa$, and suppose that 
$\zeta = \min((M \cap \kappa) \setminus \theta)$. 
Fix $\tau \in Q \cap M \cap \kappa^+$, and we will show that 
$\zeta \in S_\tau$. 
Since $\zeta \in M$ and $M \in P$, $\zeta \in P \cap \kappa$. 
Hence $Q \cap \kappa < P \cap \kappa$. 
So by our assumptions, $Q \cap P \in D$. 
As $\tau \in M$ and $M \in P$, $\tau \in P$. 
So $\tau \in (Q \cap P) \cap M \cap \kappa^+$. 
Since $Q \cap P \cap \kappa = Q \cap \kappa = \theta$, 
$$
\zeta = \min((M \cap \kappa) \setminus (Q \cap P \cap \kappa)).
$$
It follows that $\zeta \in S_\tau$ since $(C,D)$ is $\vec S$-obedient.

\bigskip

(1) The set $A \cup C$ is adequate by Proposition 1.35. 
Let $M \in A$ and $L \in C$. 
Then $L \in P$. 
We will prove that the remainder points in $R_M(L)$ 
and $R_L(M)$ are as required.

\bigskip

Consider $\zeta \in R_{L}(M)$. 
Then by Lemma 2.9, either $\zeta \in R_L(M \cap P)$ or 
$\zeta = \min((M \cap \kappa) \setminus \beta)$. 

\bigskip

\emph{Case 1:} $\zeta \in R_L(M \cap P)$. 
Fix $\tau \in L \cap M \cap \kappa^+$, and we will show that 
$\zeta \in S_\tau$. 
Since $\tau \in L$ and $L \in P$, $\tau \in P$. 
So $\tau \in L \cap (M \cap P) \cap \kappa^+$. 
Since $\zeta \in R_L(M \cap P)$, it follows that 
$\zeta \in S_\tau$ since $C$ is $\vec S$-adequate.

Suppose that $Q \in L \cap \mathcal Y$ is $\vec S$-strong and 
$$
\sup(M \cap \zeta) < Q \cap \kappa < \zeta.
$$
Fix $\tau \in Q \cap M \cap \kappa^+$, and we will show 
that $\zeta \in S_\tau$. 
Since $M \cap P \cap \kappa = M \cap \beta$ 
is an initial segment of $M \cap \kappa$ 
and $\zeta \in M \cap P \cap \kappa$, 
$$
\sup(M \cap P \cap \zeta) = \sup(M \cap \zeta) < Q \cap \kappa < \zeta.
$$
As $Q \in L$ and $L \in P$, $Q \in P$. 
And since $\tau \in Q$ and $Q \in P$, $\tau \in P$. 
So $\tau \in Q \cap (M \cap P) \cap \kappa^+$. 
Since $\zeta \in R_L(M \cap P)$, it follows that $\zeta \in S_\tau$ 
since $C$ is $\vec S$-adequate.

\bigskip

\emph{Case 2:} $\zeta = \min((M \cap \kappa) \setminus \beta)$. 
Fix $\tau \in L \cap M \cap \kappa^+$, and we will show that 
$\zeta \in S_\tau$. 
Since $\tau \in L$ and $L \in P$, $\tau \in P$. 
So $\tau \in P \cap M \cap \kappa^+$. 
Since $M \in A$ and $P \in B$, 
it follows that $\zeta \in S_\tau$ since $(A,B)$ is $\vec S$-obedient.

Suppose that $Q \in L \cap \mathcal Y$ is $\vec S$-strong and 
$$
\sup(M \cap \zeta) < Q \cap \kappa < \zeta.
$$
Fix $\tau \in Q \cap M \cap \kappa^+$, and we will show that 
$\zeta \in S_\tau$. 
As $Q \in L$ and $L \in P$, $Q \in P$. 
And since $\tau \in Q$ and $Q \in P$, $\tau \in P$. 
So $\tau \in M \cap P \cap \kappa^+$. 
As $M \in A$ and $P \in B$, it follows that 
$\zeta \in S_\tau$ since $(A,B)$ is $\vec S$-obedient.

\bigskip

Consider $\zeta \in R_M(L)$. 
By Lemma 2.9, $\zeta \in R_{M \cap P}(L)$. 
Fix $\tau \in L \cap M \cap \kappa^+$, 
and we will show that $\zeta \in S_\tau$. 
Since $\tau \in L$ and $L \in P$, $\tau \in P$. 
So $\tau \in L \cap (M \cap P) \cap \kappa^+$. 
As $\zeta \in R_{M \cap P}(L)$, it follows that 
$\zeta \in S_\tau$ since $C$ is $\vec S$-adequate.

Suppose that $Q \in M \cap \mathcal Y$ is $\vec S$-strong and 
$$
\sup(L \cap \zeta) < Q \cap \kappa < \zeta.
$$
Fix $\tau \in Q \cap L \cap \kappa^+$, 
and we will show that $\zeta \in S_\tau$. 
Since $\zeta \in L$ and $L \in P$, $\zeta \in P \cap \kappa$. 
Therefore $Q \cap \kappa \in M \cap P \cap \kappa$. 
As $\tau \in L$ and $L \in P$, $\tau \in P$. 
So $\tau \in Q \cap P \cap \kappa^+$.

\bigskip

\emph{Case 1:} $\tau < \alpha_{M,P}$ and 
$Q \cap \alpha_{M,P}$ is bounded below $\alpha_{M,P}$. 
Then $\tau \in Q \cap P \cap \alpha_{M,P}$. 
By Lemma 11.7, there is 
$Q' \in M \cap P \cap \mathcal Y$ which is $\vec S$-strong such that 
$Q' \cap \kappa = Q \cap \kappa$ and $\tau \in Q'$. 
So 
$$
\sup(L \cap \zeta) < Q' \cap \kappa = Q \cap \kappa < \zeta
$$ 
and $\tau \in Q' \cap L \cap \kappa^+$. 
Since $\zeta \in R_{M \cap P}(L)$ and $Q' \in (M \cap P) \cap \mathcal Y$ 
is $\vec S$-strong, 
it follows that $\zeta \in S_\tau$ since $C$ is $\vec S$-adequate.

\bigskip

\emph{Case 2:} Either $\tau < \alpha_{M,P}$ and $Q \cap \alpha_{M,P}$ 
is unbounded in $\alpha_{M,P}$, or $\alpha_{M,P} \le \tau$. 
In the first case, we apply Lemma 11.6(1) to get 
$Q^* \in P' \cap \mathcal Y$ such that $\tau \in Q^*$. 
In the second case, we apply the main proxy lemma, Lemma 11.5. 
Assuming $\alpha_{M,P} \le \tau$, 
let $\eta := \min((M \cap \kappa^+) \setminus \tau)$. 
Note that $\eta$ exists since $\sup(Q) \in M$ and 
$\tau < \sup(Q)$. 
Also 
$$
\tau \in Q \cap P \cap [\sup(M \cap \eta),\eta).
$$
By Lemma 11.5(1) 
there is $Q^* \in P' \cap \mathcal Y$ such that 
$\tau \in Q^*$.

Thus in either case, there is $Q^* \in P' \cap \mathcal Y$ such that 
$\tau \in Q^*$. 
As $\tau \in Q^*$ and $Q^* \in P'$, $\tau \in P'$. 
Also by Lemma 11.6(2) in the first case, and Lemma 11.5(3) in the second case, 
there is $Q' \in M' \cap \mathcal Y$ such that $Q'$ is 
$\vec S$-strong, $Q' \cap \kappa = Q \cap \kappa$, and $\tau \in Q'$.

By Lemma 2.10(2), 
either $\beta_{L,M} = \beta_{L,M'}$ or $\beta' < \beta_{L,M'}$. 
First, suppose that $\beta_{L,M} = \beta_{L,M'}$. 
Since $\zeta \in R_{M \cap P}(L)$, Lemma 2.10(3) implies that 
$\zeta \in R_{M'}(L)$. 
But $Q' \in M' \cap \mathcal Y$ is $\vec S$-strong, 
$$
\sup(L \cap \zeta) < Q' \cap \kappa = Q \cap \kappa < \zeta,
$$
and $\tau \in Q' \cap L \cap \kappa^+$. 
It follows that $\zeta \in S_\tau$ since $C$ is $\vec S$-adequate.

Secondly, assume that $\beta' < \beta_{L,M'}$. 
Since 
$$
Q \cap \kappa \in M \cap P \cap \kappa = M \cap \beta = M' \cap \beta',
$$
it follows that 
$$
Q \cap \kappa < \beta' < \beta_{L,M'}.
$$
As $Q \cap \kappa = Q' \cap \kappa 
\in (M' \cap \beta_{L,M'}) \setminus L$, we have that $L < M'$. 

We claim that $\beta' \le \zeta$. 
Otherwise since $L < M'$,
$$
\zeta \in L \cap \beta' \subseteq L \cap \beta_{L,M'} \subseteq M'.
$$
So $\zeta \in M' \cap \beta' = M \cap \beta$. 
Hence $\zeta \in M$, which contradicts that $\zeta \in R_M(L)$.

Since 
$$
\sup(L \cap \zeta) < Q \cap \kappa < \beta' \le \zeta,
$$
we have that $\zeta = \min((L \cap \kappa) \setminus \beta')$. 
As noted above, $\tau \in P'$. 
So $\tau \in L \cap P' \cap \kappa^+$, and 
$$
\zeta = \min((L \cap \kappa) \setminus \beta') = 
\min((L \cap \kappa) \setminus (P' \cap \kappa)).
$$
It follows that $\zeta \in S_\tau$ since $(C,D)$ is $\vec S$-obedient.
\end{proof}

\begin{proposition}
Let $(A,B)$ be an $\vec S$-obedient side condition, 
where $A \subseteq \mathcal X$ and $B \subseteq \mathcal Y$. 
Suppose that $X \prec \mathcal A$ is such that 
$|X| = \kappa$, $X \cap \kappa^+ \in \kappa^+$, and 
$X^{<\kappa} \subseteq X$. 
Let $\theta := X \cap \kappa^+$. 

Let $M_0,\ldots,M_{k-1}$ and $P_0,\ldots,P_{m-1}$ enumerate the members of 
$A$ and $B$ respectively. 
For each $i < k$, let $\langle Q^i_n : n < \omega \rangle$ enumerate 
the $\vec S$-strong models in $M_i \cap \mathcal Y$.

Let $(C,D)$ be an $\vec S$-obedient side condition, 
where $C \subseteq \mathcal X$ and $D \subseteq \mathcal Y$, 
such that 
$$
A \cap X \subseteq C \subseteq X \ \textrm{and} \ 
B \cap X \subseteq D \subseteq X.
$$
Assume that $M_0',\ldots,M_{k-1}'$, $P_0',\ldots,P_{m-1}'$, $\theta'$, and 
$\langle R^i_n : n < \omega \rangle$ for $i < k$ 
satisfy the following properties:
\begin{enumerate}
\item $\theta' \in \theta \cap \cof(\kappa)$;
\item for all $i < k$, $M_i' \in C$ and 
$M_i \cap \theta = M_i' \cap \theta'$;
\item for all $j < m$, $P_j' \in D$ and 
$P_j \cap \theta = P_j' \cap \theta'$;
\item for all $i < k$, $\langle R^i_n : n < \omega \rangle$ enumerates 
the $\vec S$-strong models in $M_i' \cap \mathcal Y$, and 
for all $n < \omega$, $Q^i_n \cap \theta = R^i_n \cap \theta'$. 
\end{enumerate}

Then $(A \cup C,B \cup D)$ is an $\vec S$-obedient side condition.
\end{proposition}

\begin{proof}
To show that $(A \cup C,B \cup D)$ is $\vec S$-obedient, we verify properties 
(1), (2), and (3) of Definition 5.3. 
(2) is immediate.

\bigskip

(3) Let $M \in C$ and $P \in B$. 
Fix $j < m$ such that $P = P_j$, and let $P' := P_j'$. 
Let $\beta := P \cap \kappa$, and suppose that 
$\zeta := \min((M \cap \kappa) \setminus \beta)$. 
Fix $\tau \in M \cap P \cap \kappa^+$, and we will show that 
$\zeta \in S_\tau$. 
Since $\tau \in M$ and $M \in X$, 
$\tau \in X \cap \kappa^+ = \theta$. 
So $\tau \in P \cap \theta = P' \cap \theta'$. 
Also $P' \cap \kappa = P \cap \kappa = \beta$. 
Hence $\tau \in M \cap P' \cap \kappa^+$ and 
$\zeta = \min((M \cap \kappa) \setminus (P' \cap \kappa))$. 
It follows that $\zeta \in S_\tau$ since $(C,D)$ is $\vec S$-obedient.

Let $M \in A$ and $P \in D$. 
Fix $i < k$ such that $M = M_i$, and let $M' := M_i'$. 
Let $\beta := P \cap \kappa$, and suppose that 
$\zeta = \min((M \cap \kappa) \setminus \beta)$. 
Fix $\tau \in M \cap P \cap \kappa^+$, and we will show that 
$\zeta \in S_\tau$. 
Since $\tau \in P$ and $P \in X$, 
$\tau \in X \cap \kappa^+ = \theta$. 
Hence $\tau \in M \cap \theta = M' \cap \theta'$. 
Also $M \cap \kappa = M' \cap \kappa$, so 
$\zeta = \min((M' \cap \kappa) \setminus \beta)$. 
Since $\tau \in M' \cap P \cap \kappa^+$, it follows that 
$\zeta \in S_\tau$ 
since $(C,D)$ is $\vec S$-obedient.

\bigskip

(1) The set $A \cup C$ is adequate by Proposition 1.38. 
Let $M \in A$ and $L \in C$. 
We will prove that the remainder points in $R_M(L)$ and $R_L(M)$ 
are as required. 
Fix $i < k$ such that $M = M_i$, and let $M' := M_i'$.

Consider $\zeta \in R_L(M)$. 
Since $M \cap \kappa = M' \cap \kappa$, 
$\zeta \in R_L(M')$ by Lemma 2.11. 
Let $\tau \in L \cap M \cap \kappa^+$, and we will show that 
$\zeta \in S_\tau$. 
Since $\tau \in L$ and $L \in X$, 
$\tau \in X \cap \kappa^+ = \theta$. 
Hence $\tau \in M \cap \theta = M' \cap \theta'$. 
Therefore $\tau \in L \cap M'$. 
Since $\zeta \in R_L(M')$, 
it follows that 
$\zeta \in S_\tau$ as $C$ 
is $\vec S$-adequate.

Suppose that $P \in L \cap \mathcal Y$ is $\vec S$-strong and 
$$
\sup(M \cap \zeta) < P \cap \kappa < \zeta.
$$
Let $\tau \in P \cap M \cap \kappa^+$, 
and we will show that $\zeta \in S_\tau$. 
Since $P \in L$ and $L \in X$, $P \in X$. 
And as $\tau \in P$ and $P \in X$, $\tau \in X \cap \kappa^+ = \theta$. 
Hence $\tau \in M \cap \theta = M' \cap \theta'$. 
Therefore $\tau \in P \cap M' \cap \kappa^+$. 
Also since $M \cap \kappa = M' \cap \kappa$, 
$$
\sup(M' \cap \zeta) = \sup(M \cap \zeta) < P \cap \kappa < \zeta.
$$
As $\zeta \in R_L(M')$, it follows that 
$\zeta \in S_\tau$ since $C$ is $\vec S$-adequate.

\bigskip

Now consider $\zeta \in R_M(L)$. 
Since $M \cap \kappa = M' \cap \kappa$, $\zeta \in R_{M'}(L)$ 
by Lemma 2.11. 
Let $\tau \in L \cap M \cap \kappa^+$, and we will show that 
$\zeta \in S_\tau$. 
Since $\tau \in L$ and $L \in X$, $\tau \in X \cap \kappa^+ = \theta$. 
Hence $\tau \in M \cap \theta = M' \cap \theta'$. 
Therefore $\tau \in L \cap M'$. 
Since $\zeta \in R_{M'}(L)$, it follows that 
$\zeta \in S_\tau$ since 
$C$ is $\vec S$-adequate.

Let $P \in M \cap \mathcal Y$ be $\vec S$-strong, and assume that 
$$
\sup(L \cap \zeta) < P \cap \kappa < \zeta.
$$
Let $\tau \in L \cap P \cap \kappa^+$, and we will show that 
$\zeta \in S_\tau$. 
Since $\tau \in L$ and $L \in X$, 
$\tau \in X \cap \kappa^+ = \theta$. 
Fix $n < \omega$ such that $P = Q^i_n$. 
Then 
$$
\tau \in P \cap \theta = Q^i_n \cap \theta = 
R^i_n \cap \theta'.
$$
So $\tau \in R^i_n \cap L \cap \kappa^+$. 
Now $R^i_n \in M' \cap \mathcal Y$ is $\vec S$-strong, and 
$$
\sup(L \cap \zeta) < P \cap \kappa = Q^i_n \cap \kappa = 
R^i_n \cap \kappa < \zeta.
$$
Since $\zeta \in R_{M'}(L)$, it follows that 
$\zeta \in S_\tau$ since $C$ is $\vec S$-adequate.
\end{proof}

\bigskip

\part{Mitchell's Theorem}

\bigskip

\addcontentsline{toc}{section}{14. The ground model}

\textbf{\S 14. The ground model}

\stepcounter{section}

\bigskip

With the general development of side conditions from Parts I and II at our disposal, 
we now begin our proof of Mitchell's theorem. 
We start by describing the ground model over which we will force a generic 
extension satisfying that there is no stationary subset of 
$\omega_2 \cap \cof(\omega_1)$ in the approachability ideal $I[\omega_2]$.

\bigskip

We will use the same notation which was introduced at the beginning of Parts I and II, 
together with some additional assumptions. 
Recall that $\kappa \ge \omega_2$ is regular, 
$2^\kappa = \kappa^+$, and $\Box_\kappa$. 
Also the cardinal $\lambda$ from Part I is equal to $\kappa^+$. 
In addition, we will assume that $\kappa$ is a greatly Mahlo cardinal, and 
the thin stationary set $T^*$ from Notation 1.4 is equal to $P_{\omega_1}(\kappa)$.

Define a sequence of sets $\langle S_\xi : \xi < \kappa^+ \rangle$ inductively as follows. 
Let $S_0$ denote the set of inaccessible cardinals less than $\kappa$. 
Let $\delta < \kappa^+$, and suppose that $S_\xi$ has been 
defined for all $\xi < \delta$. 
If $\delta = \delta_0 + 1$, then let $\alpha \in S_\delta$ if 
$\alpha$ is inaccessible, 
$\alpha \in S_{\delta_0}$, 
and $S_{\delta_0} \cap \alpha$ is stationary in $\alpha$. 
If $\delta$ is a limit ordinal, then let $\alpha \in S_\delta$ if $\alpha$ is 
inaccessible, and for all 
$\xi \in A_{\delta,\alpha}$, $\alpha \in S_\xi$. 
Let $\vec S := \langle S_\xi : \xi < \kappa^+ \rangle$.

The fact that $\kappa$ is greatly Mahlo implies that 
for all $\delta < \kappa^+$, 
$S_\delta$ is stationary in $\kappa$. 
In fact, it is easily seen that this consequence is actually 
equivalent to $\kappa$ being 
greatly Mahlo. 
See \cite[Definition 4.2]{baumgartner} for more information about greatly Mahlo cardinals. 

\begin{notation}
For the remainder of Part III, the structure $\mathcal A$ from 
Notation 7.6 will be equal to 
$$
(H(\kappa^+),\in,\unlhd,\kappa,T^*,\pi^*,C^*,\Lambda,\mathcal Y_0,
f^*,\vec C,\vec A,c^*,\vec S).
$$
\end{notation}

\bigskip

In Proposition 7.20, we proved that the set of simple models in 
$\mathcal X$ is stationary in $P_{\omega_1}(H(\kappa^+))$. 
We will prove in Proposition 15.3 that most simple models in $\mathcal X$ 
have strongly generic conditions. 
The next proposition describes the kind of models in $\mathcal Y$ which 
will have strongly generic conditions.

\begin{proposition}
There are stationarily many $P \in P_{\kappa}(H(\kappa^+))$ 
such that $P \in \mathcal Y$, 
$P$ is $\vec S$-strong, and $\cf(\sup(P)) = P \cap \kappa$.
\end{proposition}

\begin{proof}
Let $F : H(\kappa^+)^{< \omega} \to H(\kappa^+)$. 
Fix $X$ which is an elementary substructure of $\mathcal A$ 
of size $\kappa$ such that $X$ is closed under $F$, and 
$\tau := X \cap \kappa^+$ has cofinality $\kappa$. 
Note that $X = Sk(X \cap \kappa^+) = Sk(\tau)$. 
Since $\tau$ is the union of the increasing and continuous sequence 
of sets $\{ A_{\tau,i} : i < \kappa \}$, 
it follows that $X$ is the union of the 
increasing and continuous sequence of sets 
$\{ Sk(A_{\tau,i}) : i < \kappa \}$.

For all infinite 
$\beta < \kappa$, $|A_{\tau,\beta}| \le |\beta| < \kappa$ by Notation 7.4(3). 
Since $\tau$ has cofinality $\kappa$, $\sup(A_{\tau,\beta}) < \tau$, and hence 
$\sup(A_{\tau,\beta}) \in X$. 
Fix a club $C \subseteq \kappa$ such that for all $\alpha \in C$, 
$A_{\tau,\alpha}$ is closed under $H^*$, 
$A_{\tau,\alpha} \cap \kappa = \alpha$, 
$Sk(A_{\tau,\alpha})$ is closed under $F$, 
and for all $\beta < \alpha$, $\sup(A_{\tau,\beta}) \in A_{\tau,\alpha}$. 
As $S_\tau$ is stationary in $\kappa$, we can fix 
$\alpha \in \lim(C) \cap S_\tau$. 
Let $P := Sk(A_{\tau,\alpha})$.

We claim that $P \in \mathcal Y$, $P$ is $\vec S$-strong, 
$\cf(\sup(P)) = P \cap \kappa$, and $P$ is closed under $F$. 
The last statement follows from the fact that $\alpha \in C$. 
Since $A_{\tau,\alpha}$ is closed under $H^*$, 
$P \cap \kappa^+ = A_{\tau,\alpha}$. 
In particular, since $\alpha \in C$, 
$P \cap \kappa = \alpha$. 
As $\alpha \in S_\tau$, 
$\alpha$ is inaccessible. 
After we show that $\cf(\sup(P)) = \alpha$, it will follow 
that $P \in \mathcal Y$ by Lemma 7.15.

To show that $P$ is $\vec S$-strong, let $\sigma \in P \cap \kappa^+$. 
Then $\sigma \in P \cap \kappa^+ = A_{\tau,\alpha}$. 
Since $\alpha \in S_\tau$ and $\tau$ is a limit ordinal, for all 
$\pi \in A_{\tau,\alpha}$, $\alpha \in S_\pi$. 
In particular, $\alpha \in S_\sigma$.

It remains to show that $\cf(\sup(P)) = \alpha$. 
For $\alpha_0 < \alpha_1$ in $C \cap \alpha$, 
$$
\sup(A_{\tau,\alpha_0}) \in A_{\tau,\alpha_1} \subseteq P
$$ 
by the definition of $C$. 
Since $P \cap \kappa$ is a limit point of $C$, 
$$
A_{\tau,\alpha} = 
\bigcup \{ A_{\tau,\beta} : \beta \in C \cap \alpha \}.
$$
Therefore 
$$
\sup(P) = \sup(A_{\tau,\alpha}) = 
\sup \{ \sup(A_{\tau,\beta}) : \beta \in C \cap \alpha \},
$$
which is the 
supremum of a strictly increasing sequence. 
Since $\alpha$ is in $S_\tau$, $\alpha$ is inaccessible, so $C \cap \alpha$ 
has order type $\alpha$. 
Hence $\sup(P)$ has cofinality equal to $P \cap \kappa = \alpha$.
\end{proof}

\bigskip

\addcontentsline{toc}{section}{15. The forcing poset}

\textbf{\S 15. The forcing poset}

\stepcounter{section}

\bigskip

We now define and analyze the forcing poset which will force 
that there is no stationary subset of $\omega_2 \cap \cof(\omega_1)$ 
in the approachability ideal $I[\omega_2]$.

\begin{definition}
Let $\p$ be the forcing poset consisting of pairs $p = (A_p,B_p)$ 
satisfying:
\begin{enumerate}
\item $A_p \subseteq \mathcal X$, and 
for all $M \in A_p$, $M \prec (\mathcal A,\mathcal Y)$;
\item $B_p \subseteq \mathcal Y$;
\item $(A_p,B_p)$ is an $\vec S$-obedient side condition. 
\end{enumerate}
Let $q \le p$ if $A_p \subseteq A_q$ and $B_p \subseteq B_q$.
\end{definition}

The rest of this section is devoted to proving amalgamation results 
for $\p$, which in turn yield the existence of strongly 
generic conditions. 

\begin{lemma}
Let $N \in \mathcal X$ with $N \prec (\mathcal A,\mathcal Y)$. 
Then $q_N := (\{ N \},\emptyset)$ is in $\p$, and for all 
$p \in N \cap \p$, $p$ and $q_N$ are compatible.
\end{lemma}

\begin{proof}
Immediate from Lemma 5.4(1).
\end{proof}

\begin{proposition}
Let $N \in \mathcal X$ be simple such that 
$N \prec (\mathcal A,\mathcal Y,\p)$. 
Let $q_N := (\{N\},\emptyset)$. 
Then $q_N$ is a universal strongly $N$-generic condition.
\end{proposition}

See Section 3 for a discussion of universal strongly generic conditions.

\begin{proof}
By Lemma 15.2, $q_N$ is compatible with all 
conditions in $N \cap \p$. 
So it suffices to show that $q_N$ is strongly $N$-generic. 
Let $r_0 \le q_N$ be given. 
We will find a condition $v$ in $N \cap \p$ such that for all 
$w \le v$ in $N \cap \p$, $r_0$ and $w$ are compatible.

\bigskip

Let $M_0, \ldots, M_{k-1}$ 
list the models $M$ in $A_{r_0} \setminus N$ such that $M < N$. 
Note that by Lemma 8.2, $M_i \cap N \in N$ for all $i < k$.

By finitely many applications of Lemmas 5.5(1) and 7.16, 
together with the fact that $N \prec (\mathcal A,\mathcal Y)$, 
the pair 
$$
r_1 := (A_{r_0} \cup 
\{ M_0 \cap N, \ldots, M_{k-1} \cap N \},B_{r_0})
$$
is a condition below $r_0$.

By Proposition 10.8, there is a condition $r \le r_1$ such that 
$A_r = A_{r_1}$, and 
$(A_r,B_r)$ is closed under canonical models with respect to $N$. 
Note that the assumptions of the first paragraph of 
Proposition 13.1 hold for $A = A_r$ and $B = B_r$.

The objects $r$, $N$, and 
$M_0,\ldots,M_{k-1}$ witness that the 
following statement holds in $(\mathcal A,\mathcal Y,\p)$:

\bigskip

There exist $v$, $N'$, and $M_0',\ldots,M_{k-1}'$ satisfying:
\begin{enumerate}
\item $v \in \p$;
\item $A_{r} \cap N \subseteq A_v$, $B_{r} \cap N \subseteq B_v$, and 
$M_0',\ldots,M_{k-1}'$ and $N'$ are in $A_v$;
\item $N'$ is simple;
\item for all $i < k$, $M_i' < N'$, $M_i \cap N = M_i' \cap N'$, 
and $p(M_i,N) = p(M_i',N')$.
\end{enumerate}

\bigskip

The parameters which appear in the above statement, namely 
$A_{r} \cap N$, $B_{r} \cap N$, 
and for $i < k$, $M_i \cap N$ and $p(M_i,N)$, are all members of $N$. 
By the elementarity of $N$, there are 
$v$, $N'$, and $M_0',\ldots,M_{k-1}'$ in $N$ 
which satisfy the same statement.

\bigskip

We will show that for all $w \le v$ in $N \cap \p$, 
$w$ is compatible with $r$, and hence is compatible 
with $r_0$ since $r \le r_0$. 
This will complete the proof.

So fix $w \le v$ in $N \cap \p$. 
We claim that the pair 
$$
(A_r \cup A_w, B_r \cup B_w)
$$
is in $\p$. 
Note that the assumptions of the second paragraph of 
Proposition 13.1 hold for $C = A_w$ and $D = B_w$. 
So by Proposition 13.1, 
$(A_r \cup A_w, B_r \cup B_w)$ is an $\vec S$-obedient 
side condition. 
So this pair is a condition in $\p$, and it is obviously below 
$r$ and $w$.
\end{proof}

\begin{corollary}
The forcing poset $\p$ satisfies the $\omega_1$-covering property. 
In particular, it preserves $\omega_1$.
\end{corollary}

\begin{proof}
By Proposition 7.20, the set of $N \in P_{\omega_1}(H(\kappa^+))$ 
such that $N \in \mathcal X$ and $N$ is simple is stationary. 
Hence there are stationarily many such $N$ with 
$N \prec (\mathcal A,\mathcal Y,\p)$. 
Any such $N$ has a universal strongly $N$-generic condition 
by Proposition 15.3. 
By Corollary 3.10, $\p$ has the $\omega_1$-covering property.
\end{proof}

\bigskip

Next we prove that many models in $\mathcal Y$ have strongly 
generic conditions.

\begin{lemma}
Let $P \in \mathcal Y$ be $\vec S$-strong. 
Let $q_P := (\emptyset,\{P\})$. 
Then $q_P$ is in $\p$, and 
for all $p \in P \cap \p$, $p$ and $q_P$ are compatible.
\end{lemma}

\begin{proof}
Immediate from Lemma 5.4(2).
\end{proof}

\begin{proposition}
Let $P \in \mathcal Y$ be $\vec S$-strong such that 
$\cf(\sup(P)) = P \cap \kappa$ and 
$P \prec (\mathcal A,\mathcal Y,\p)$. 
Let $q_P := (\emptyset,\{ P \})$. 
Then $q_P$ is a universal strongly $P$-generic condition.
\end{proposition}

\begin{proof}
By Lemma 15.5, $q_P$ is compatible with all members of $P \cap \p$. 
So it suffices to show that $q_P$ is strongly $P$-generic. 
Let $r_0 \le q_P$ be given. 
We will find a condition $v \in P \cap \p$ such that for all $w \le v$ 
in $P \cap \p$, $r_0$ and $w$ are compatible.

By finitely many applications of 
Lemmas 5.5(2), 5.5(3), and 7.16, together with the fact that 
$P \prec (\mathcal A,\mathcal Y)$, 
there is a condition $r \le r_0$ 
such that 
$$
A_r = A_{r_0} \cup \{ P \cap M : M \in A_{r_0} \},
$$
and 
$$
B_r = B_{r_0} \cup 
\{ P \cap Q : Q \in B_{r_0}, \ Q \cap \kappa < P \cap \kappa \}.
$$
Then the assumptions of the first paragraph of Proposition 13.2 
hold for $A = A_r$ and $B = B_r$.

Let $\beta := P \cap \kappa$. 
Let $M_0, \ldots, M_{k-1}$ list the members of $A_r$.

The objects $r$, $P$, $\beta$, and $M_0, \ldots, M_{k-1}$ 
witness that the following 
statement holds in $(\mathcal A,\mathcal Y,\p)$:

\bigskip

There exist $v$, $P'$, $\beta'$, and $M_0',\ldots,M_{k-1}'$ satisfying:
\begin{enumerate}
\item $v \in \p$;
\item $A_r \cap P \subseteq A_v$, 
$B_r \cap P \subseteq B_v$, 
$M_0',\ldots,M_{k-1}'$ are in $A_v$, and $P' \in B_v$;
\item $P' \cap \kappa = \beta'$ and 
$\cf(\sup(P')) = \beta'$;
\item for all $i < k$, $M_i \cap P = M_i' \cap P'$ and 
$p(M_i,P) = p(M_i',P')$.
\end{enumerate}

\bigskip

The parameters appearing in the statement above, namely, 
$A_r \cap P$, $B_r \cap P$, and for $i < k$, $M_i \cap P$ 
and $p(M_i,P)$, 
are all members of $P$. 
By the elementarity of $P$, we can fix $v$, $P'$, $\beta'$, and 
$M_0',\ldots,M_{k-1}'$ in $P$ 
which satisfy the same statement.

For each $M \in A_r$, let $M'$ denote $M_i'$, 
where $i < k$ and $M = M_i$. 

\bigskip

We will show that for all $w \le v$ in $P \cap \p$, 
$w$ is compatible with $r$, and hence is compatible with $r_0$ since 
$r \le r_0$. 
This will complete the proof.

So fix $w \le v$ in $P \cap \p$. 
We claim that the pair 
$$
(A_r \cup A_w,B_r \cup B_w)
$$
is in $\p$. 
Note that the assumptions of the 
second paragraph of Proposition 13.2 hold 
for $C = A_w$ and $D = B_w$. 
So by Proposition 13.2, 
$(A_r \cup A_w,B_r \cup B_w)$ is an $\vec S$-obedient 
side condition. 
So this pair is a condition in $\p$, and it is obviously 
below $r$ and $w$.
\end{proof}

\begin{corollary}
The forcing poset $\p$ has the $\kappa$-covering property. 
In particular, $\p$ forces that $\kappa$ is a regular cardinal.
\end{corollary}

\begin{proof}
By Proposition 14.2, there are stationarily many $P$ in $\mathcal Y$ 
such that $P$ is $\vec S$-strong and $\cf(\sup(P)) = P \cap \kappa$. 
Therefore there are stationarily many such $P$ with 
$P \prec (\mathcal A,\mathcal Y,\p)$. 
By Proposition 15.6, any such $P$ has a universal 
strongly $P$-generic condition. 
Hence by Corollary 3.10, $\p$ has the $\kappa$-covering property.
\end{proof}

\bigskip

Finally, we prove that for most transitive models, the empty 
condition is a strongly generic condition.

\begin{proposition}
Suppose that $X$ is an elementary substructure of 
$(\mathcal A,\mathcal Y,\p)$ 
of size $\kappa$ such that $X \cap \kappa^+ \in \kappa^+$ 
and $X^{<\kappa} \subseteq X$. 
Then the pair $(\emptyset,\emptyset)$ is a 
strongly $X$-generic condition.
\end{proposition}

\begin{proof}
Let $\theta := X \cap \kappa^+$. 
Since $X^{<\kappa} \subseteq X$, $\theta$ has cofinality $\kappa$.

Let $D$ be a dense subset of $\p \cap X$, and we will show that 
$D$ is predense in $\p$. 
Let $p$ be a condition. 

Let $M_0,\ldots,M_{k-1}$ and $P_0,\ldots,P_{m-1}$ 
enumerate the members of 
$A_p$ and $B_p$ respectively. 
Note that since $X^{<\kappa} \subseteq X$, 
for any model $K$ on either of these lists, 
$K \cap \theta \in X$. 
For each $i < k$, let $\langle Q^i_n : n < \omega \rangle$ enumerate 
the $\vec S$-strong models in $M_i \cap \mathcal Y$. 
Since $X^{<\kappa} \subseteq X$, 
for each $n < \omega$, $Q^i_n \cap \theta \in X$. 
Therefore the sequence 
$\langle Q^i_n \cap \theta : n < \omega \rangle$ is in $X$. 

Note that the assumptions of the first and second paragraphs of 
Proposition 13.3 hold for $A = A_p$ and $B = B_p$.

The objects $p$, $\theta$, 
$M_0,\ldots,M_{k-1}$, $P_0,\ldots,P_{m-1}$, and 
$\langle Q^i_n : n < \omega \rangle$ for $i < k$ 
witness that $(\mathcal A,\mathcal Y,\p)$ satisfies the following statement:

\bigskip

There exist $v$, $\theta'$, $M_0',\ldots,M_{k-1}'$, 
$P_0',\ldots,P_{m-1}'$, 
and $\langle R^i_n : n < \omega \rangle$ for $i < k$ 
such that:
\begin{enumerate}
\item $v \in \p$;
\item $M_0',\ldots,M_{k-1}'$ are in $A_v$ and 
$P_0',\ldots,P_{m-1}'$ are in $B_v$;
\item $\cf(\theta') = \kappa$;
\item $M_i \cap \theta = M_i' \cap \theta'$ and 
$P_j \cap \theta = P_j' \cap \theta'$ for $i < k$ and $j < m$;
\item for all $i < k$, $\langle R^i_n : n < \omega \rangle$ 
enumerates the $\vec S$-strong models in $M_i' \cap \mathcal Y$, 
and for all $n < \omega$, 
$Q^i_n \cap \theta = R^i_n \cap \theta'$.
\end{enumerate}

\bigskip

The parameters which appear in the above statement, namely 
$\kappa$, $M_i \cap \theta$ for $i < k$, 
$P_j \cap \theta$ for $j < m$, 
and $\langle Q^i_n \cap \theta : n < \omega \rangle$ for $i < k$ 
are all members of $X$. 
By the elementarity of $X$, we can fix 
$v$, $\theta'$, $M_0',\ldots,M_{k-1}'$, $P_0',\ldots,P_{m-1}'$, and 
$\langle R^i_n : n < \omega \rangle$ for $i < k$ in $X$ which 
satisfy the same statement.

For each $M$ in $A_p$, let $M'$ denote $M_i'$, 
where $i < k$ and 
$M = M_i$. 
For each $P$ in $B_p$, let $P'$ denote $P_j'$, where $j < m$ and 
$P = P_j$.

Since $D$ is a dense subset of $\p \cap X$, we can fix $w \le v$ in $D$. 
Let us show that $w$ and $p$ are compatible. 
This proves that $D$ is predense in $\p$, finishing the proof. 
It suffices to show that the pair 
$$
(A_p \cup A_w,B_p \cup B_w)
$$
is a condition. 
Note that the assumptions of the third paragraph of 
Proposition 13.3 hold for 
$C = A_w$ and $D = B_w$. 
So by Proposition 13.3, 
$(A_p \cup A_w,B_p \cup B_w)$ is an 
$\vec S$-obedient side condition. 
Therefore this pair is in $\p$, and it is obviously below $p$ and $w$.
\end{proof}

\begin{corollary}
The forcing poset $\p$ is $\kappa^+$-c.c.
\end{corollary}

\begin{proof}
Since $(\emptyset,\emptyset)$ is the maximum element of 
$\p$, by Proposition 3.11 it suffices to show that there are 
stationarily many $X$ in $P_{\kappa^+}(H(\kappa^+))$ 
for which 
$(\emptyset,\emptyset)$ is strongly $X$-generic. 
By Proposition 15.8, it suffices to show that there are 
stationarily many $X$ in 
$P_{\kappa^+}(H(\kappa^+))$ such that 
$X \cap \kappa^+ \in \kappa^+$ and 
$X^{<\kappa} \subseteq X$. 
But this follows easily from the fact that $\kappa^{<\kappa} = \kappa$.
\end{proof}

\bigskip

\addcontentsline{toc}{section}{16. The final argument}

\textbf{\S 16. The final argument}

\stepcounter{section}

\bigskip

We now complete the proof of Mitchell's theorem. 
We begin by noting that the forcing poset $\p$ has the desired 
effect on cardinal structure.

\begin{proposition}
The forcing poset $\p$ preserves $\omega_1$, collapses $\kappa$ to 
become $\omega_2$, and is $\kappa^+$-c.c.
\end{proposition}

\begin{proof}
Immediate from Proposition 3.12, Lemma 15.2, and 
Corollaries 15.4, 15.7, and 15.9.
\end{proof}

Next we will show that we can 
apply the factorization theorem, 
Theorem 6.4.

It is easy to see that $\p$ has greatest lower bounds. 
Namely, if $(A,B)$ and $(C,D)$ are in $\p$ and are compatible, 
then $(A \cup C,B \cup D)$ is the greatest lower 
bound of $(A,B)$ and $(C,D)$.

\begin{lemma}
The forcing poset $\p$ satisfies property $*(\p,\p)$.
\end{lemma}

See Definition 6.2 for the definition of $*$.

\begin{proof}
Let $p$, $q$, and $r$ be pairwise compatible conditions in $\p$. 
Then $q \land r = (A_q \cup A_r,B_q \cup B_r)$, 
$p \land q = (A_p \cup A_q,B_p \cup B_q)$, and 
$p \land r = (A_p \cup A_r,B_p \cup B_r)$. 
To see that $p$ is compatible with $q \land r$, it suffices to show that 
$$
(A_p \cup A_q \cup A_r,B_p \cup B_q \cup B_r)
$$
is an $\vec S$-obedient side condition. 
But looking over the requirements of being $\vec S$-obedient, 
any violation of these requirements involves an incompatibility between 
two objects appearing in the components of the pair, 
and hence would lead to a 
violation of the same requirement for one 
of the triples $p \land q$, $p \land r$, or $q \land r$.
\end{proof}

\begin{proposition}
Let $Q \in \mathcal Y$ be $\vec S$-strong such that 
$\cf(\sup(Q)) = Q \cap \kappa$ and 
$Q \prec (\mathcal A,\mathcal Y,\p)$. 
Let $q_Q := (\emptyset,\{Q\})$.  
Let $G$ be a generic filter on $\p$ which contains $q_Q$. 
Then $G \cap Q$ is a $V$-generic filter on $\p \cap Q$, and 
$V[G] = V[G \cap Q][H]$, where $H$ is a 
$V[G \cap Q]$-generic filter on $(\p / q_Q) / (G \cap Q)$. 
Moreover, the pair $(V[G \cap Q],V[G])$ satisfies the $\omega_1$-approximation 
property.
\end{proposition}

\begin{proof}
By Proposition 15.6, $q_Q$ is a universal strongly $Q$-generic condition. 
By Propositions 7.20 and 15.3, 
there are stationarily many models in $P_{\omega_1}(H(\kappa)^+)$ 
which have universal strongly generic conditions. 
By Lemma 16.2, $\p$ satisfies property $*(\p,\p)$. 
So the assumptions of Theorem 6.4 are satisfied, and we are done.
\end{proof}

We will need the following technical lemma about names.

\begin{lemma}
Suppose that $Q \in \mathcal Y$ with $Q \prec (H(\kappa^+),\in,\p)$, 
and $q$ is a strongly $Q$-generic condition. 
Let $G$ be a $V$-generic filter on $\p$ which contains $q$. 
Let $\dot a \in Q$ be a nice $\p$-name for a set of ordinals, and suppose that 
$\dot a^G$ is a subset of $Q \cap \kappa$. 
Then $\dot a^G \in V[G \cap Q]$.
\end{lemma}

\begin{proof}
Note that since $q$ is strongly $Q$-generic, 
$G \cap Q$ is a $V$-generic filter on $\p \cap Q$ by 
Lemma 3.3. 

Let $\alpha := Q \cap \kappa$. 
Since $\dot a$ is a nice name, 
for each $\gamma < \alpha$ there is a unique antichain 
$A_\gamma$ such that $(p,\check \gamma) \in \dot a$ iff $p \in A_\gamma$. 
Since $\dot a \in Q$, by elementarity each $A_\gamma$ is in $Q$.

We claim that for all $\gamma < \alpha$, 
$$
\gamma \in \dot a^G \ \textrm{iff} \ 
A_\gamma \cap G \cap Q \ne \emptyset.
$$
Since $\dot a^G \subseteq Q \cap \kappa = \alpha$, 
it follows that $\dot a^G$ is definable in 
$V[G \cap Q]$ from 
the sequence $\langle A_\gamma : \gamma < \alpha \rangle$ and 
the set $G \cap Q$. 
Therefore $\dot a^G \in V[G \cap Q]$, which finishes the proof.

If $p \in A_\gamma \cap G \cap Q$, then $(p,\check \gamma) \in \dot a$ 
by the choice of $A_\gamma$. 
Since $p \in G$, it follows that $(\check \gamma)^G = \gamma$ is in 
$\dot a^G$. 
This shows that $A_\gamma \cap G \cap Q \ne \emptyset$ implies that 
$\gamma \in \dot a^G$.

Conversely, assume that $\gamma \in \dot a^G$. 
Then by the choice of $A_\gamma$, we can fix 
$p \in G \cap A_\gamma$. 
So to show that $A_\gamma \cap G \cap Q$ is nonempty, it 
suffices to show that $p \in Q$.

Since $A_\gamma \in Q$ is an antichain, by elementarity 
there is a maximal antichain 
$A \in Q$ with $A_\gamma \subseteq A$. 
Let $D$ be the dense set of $u \in \p$ such that for some 
$s \in A$, $u \le s$. 
By elementarity, $D \in Q$, and therefore $D \cap Q$ is dense in $\p \cap Q$ 
by elementarity. 
Since $q$ is strongly $Q$-generic, $D \cap Q$ is predense below $q$. 

As $q \in G$ and $D \cap Q$ is predense below $q$, we can fix  
$u \in G \cap D \cap Q$. 
By elementarity and the definition of $D$, there is $s \in A \cap Q$ such that 
$u \le s$. 
Since $u \in G$, $s \in G$. 
Now $p \in A_\gamma$ and $A_\gamma \subseteq A$, so $p \in A$. 
Also $s \in A$. 
Since $s$ and $p$ are both in $G$, they are compatible. 
But $A$ is an antichain, so $s = p$. 
Since $s \in Q$, $p \in Q$.
\end{proof}

\begin{proposition}
Let $\tau < \kappa^+$ be an ordinal with cofinality $\kappa$ which is 
closed under $H^*$. 
Let $\dot Y$ and $\dot D_\tau$ be $\p$-names such that $\p$ forces 
$$
\dot Y = \{ P : \exists p \in \dot G \ (P \in B_p) \} \ 
\textrm{and} \ 
\dot D_\tau = \{ P \cap \kappa : P \in \dot Y, \ \tau \in P \}.
$$
Suppose that $\beta < \kappa$ is an ordinal with uncountable cofinality, 
and $p$ is a condition which forces that 
$\beta$ is a limit point of $\dot D_\tau$. 
Let $Q := Sk(A_{\tau,\beta})$. 
Then:
\begin{enumerate}
\item $Q \in \mathcal Y$;
\item $Q \cap \kappa = \beta$ and $Q \cap \kappa^+ = A_{\tau,\beta}$;
\item $\beta \in S_{\tau + 1}$;
\item $\cf(\sup(Q)) = \beta$; 
\item $Q$ is $\vec S$-strong;
\item $p$ forces that $Q$ is in $\dot Y$.
\end{enumerate}
\end{proposition}

\begin{proof}
Define 
$$
Z^* := \{ P \in \mathcal Y : \textrm{$P$ is $\vec S$-strong}, \ 
P \cap \kappa < \beta, \ \tau \in P \},
$$
and 
$$
Z := \{ P \cap \tau : P \in Z^* \}.
$$
Note that by Lemma 7.28, 
if $P_1$ and $P_2$ are in $Z^*$ and $P_1 \cap \kappa \le P_2 \cap \kappa$, 
then $P_1 \cap \tau \subseteq P_2 \cap \tau$.

We claim that for all $\gamma < \beta$, there is $P \in Z^*$ 
such that $\gamma < P \cap \kappa$. 
Namely, since $p$ forces that $\beta$ is a limit point of $\dot D_\tau$, 
there is $q \le p$ and $P \in B_q$ such that 
$\tau \in P$ and $\gamma < P \cap \kappa < \beta$. 
Since $P \in B_q$, it follows that $P \in \mathcal Y$ is $\vec S$-strong. 
So $P \in Z^*$ and $\gamma < P \cap \kappa$, proving the claim. 
Consequently, $(\bigcup Z) \cap \kappa = \beta$.

Next we claim that $\bigcup Z = A_{\tau,\beta}$. 
First, suppose that $P \in Z^*$, and we will show that 
$P \cap \tau \subseteq A_{\tau,\beta}$. 
Since $\tau \in P$, it follows that $P \cap \tau = A_{\tau,P \cap \kappa}$ 
by Lemma 7.27, 
which is a subset of $A_{\tau,\beta}$ since $P \cap \kappa < \beta$. 
This shows that $\bigcup Z \subseteq A_{\tau,\beta}$. 
Conversely, let $\xi \in A_{\tau,\beta}$, and we will show that 
$\xi \in \bigcup Z$. 
Since $\beta$ is a limit ordinal, we can fix $\gamma < \beta$ 
such that $\xi \in A_{\tau,\gamma}$. 
By the first claim, there is $P \in Z^*$ 
such that $\gamma < P \cap \kappa$. 
Since $\gamma < P \cap \kappa$ and 
$\xi \in A_{\tau,\gamma}$, 
it follows that $\xi \in A_{\tau,P \cap \kappa}$. 
But $\tau \in P$, so $A_{\tau,P \cap \kappa} = P \cap \tau$ by Lemma 7.27. 
Hence $\xi \in P \cap \tau$. 
As $P \cap \tau \in Z$, we have that $\xi \in \bigcup Z$.

Since $\tau$ is closed under $H^*$, every set in $Z$ is closed under $H^*$. 
As $Z$ is a $\subseteq$-chain, $\bigcup Z = A_{\tau,\beta}$ is also 
closed under $H^*$. 
In particular, $Q \cap \kappa^+ = A_{\tau,\beta}$. 
As noted above, 
$$
Q \cap \kappa = A_{\tau,\beta} \cap \kappa = 
(\bigcup Z) \cap \kappa = \beta.
$$
This proves (2).

By Lemma 7.28(2), if $P_1$ and $P_2$ are in $Z^*$ and 
$P_1 \cap \kappa < P_2 \cap \kappa$, then 
$\sup(P_1 \cap \tau) \in P_2 \cap \tau$, and hence 
$\sup(P_1 \cap \tau) < \sup(P_2 \cap \tau)$.  
In particular, since $A_{\tau,\beta} = \bigcup Z$, 
$$
\sup(A_{\tau,\beta}) = \sup \{ \sup(P \cap \tau) : P \in Z^* \}.
$$
Since $\beta$ has uncountable cofinality, 
$\sup(A_{\tau,\beta}) = \sup(Q)$ has uncountable cofinality. 
It follows that $Q \in \mathcal Y$ by Lemma 7.15.

\bigskip

Now we prove that $Q \cap \kappa = \beta$ is in $S_{\tau+1}$. 
Fix $M$ in $\mathcal X$ such that $p$, $\beta$, and $\tau$ are in $M$. 
Then $q := (A_p \cup \{ M \},B_p)$ is in $\p$ and $q \le p$. 
Since $\beta$ has uncountable cofinality, $\sup(M \cap \beta) < \beta$. 
As $q$ forces that $\beta$ is a limit point of $\dot D_\tau$, we can fix 
$r \le q$ and $P \in B_r$ such that 
$\sup(M \cap \beta) < P \cap \kappa < \beta$ 
and $\tau \in P$. 
It follows that $\beta = \min((M \cap \kappa) \setminus (P \cap \kappa))$. 
As $\tau \in M \cap P$, $\tau + 1 \in M \cap P$ by elementarity. 
So $M \in A_r$, $P \in B_r$, and $\tau + 1 \in M \cap P \cap \kappa^+$, 
which implies that $\beta = \min((M \cap \kappa) \setminus (P \cap \kappa))$ 
is in $S_{\tau+1}$ by the fact that $(A_r,B_r)$ is $\vec S$-obedient.

\bigskip

Since $\beta \in S_{\tau+1}$, $\beta$ is inaccessible, and in particular is regular. 
Therefore the ordinal 
$\sup(A_{\tau,\beta}) = \sup \{ \sup(P \cap \tau) : P \in Z^* \}$ is the supremum of 
a sequence of ordinals of order type $\beta$. 
It follows that $\sup(Q)$ has cofinality $\beta$. 
So $\cf(\sup(Q)) = Q \cap \kappa$. 

Now we show that $Q$ is $\vec S$-strong. 
Since $\beta \in S_{\tau+1}$, $\beta \in S_\tau$. 
As $\tau$ is a limit ordinal, for all $\xi \in A_{\tau,\beta}$, 
$\beta \in S_\xi$. 
So if $\xi \in Q \cap \kappa^+ = A_{\tau,\beta}$, then 
$Q \cap \kappa = \beta \in S_\xi$, which shows that $Q$ is $\vec S$-strong.

\bigskip

It remains to show that $p$ forces that $Q$ is in $\dot Y$. 
It suffices to prove that for all $q \le p$, there is $r \le q$ such that $Q \in B_r$. 
So let $q \le p$. 
We claim that $(A_q,B_q \cup \{ Q \})$ is a condition below $q$. 

Since $Q$ is $\vec S$-strong, to prove that $(A_q,B_q \cup \{ Q \})$ 
is an $\vec S$-obedient side condition, it suffices to show that 
if $M \in A_q$ and 
$\zeta = \min((M \cap \kappa) \setminus \beta)$, then for all 
$\sigma \in M \cap Q \cap \kappa^+$, $\zeta \in S_\sigma$. 

Let $\sigma \in M \cap Q \cap \kappa^+$. 
Since $\sigma \in Q \cap \kappa^+ = A_{\tau,\beta}$ 
and $\beta$ is a limit ordinal, we can fix $\gamma < \beta$ 
such that $\sigma \in A_{\tau,\gamma}$. 
By increasing $\gamma$ if necessary, also assume that 
$\sup(M \cap \beta) < \gamma$. 
As $q$ forces that $\beta$ is a limit point of $\dot D_\tau$, 
we can fix $s \le q$ and $P \in B_s$ such that 
$\gamma < P \cap \kappa < \beta$ and $\tau \in P$. 
So 
$$
A_{\tau,\gamma} \subseteq A_{\tau,P \cap \kappa} = P \cap \tau.
$$
In particular, $\sigma \in P$. 
Since $\zeta = \min((M \cap \kappa) \setminus \beta)$, we have that  
$$
\sup(M \cap \zeta) = \sup(M \cap \beta) < \gamma < P \cap \kappa < \beta \le \zeta.
$$
So $M \in A_s$, $P \in B_s$, $\sigma \in M \cap P \cap \kappa^+$, and 
$\zeta = \min((M \cap \kappa) \setminus (P \cap \kappa))$. 
Therefore $\zeta \in S_\sigma$ since $(A_s,B_s)$ is an 
$\vec S$-obedient side condition.
\end{proof}

\begin{lemma}
Let $\tau < \kappa^+$ be an ordinal of cofinality $\kappa$, 
and let $\beta \in S_{\tau+1}$.  
Suppose that $Q \in \mathcal Y$ is $\vec S$-strong, $Q \cap \kappa = \beta$, 
$Q \cap \kappa^+ = A_{\tau,\beta}$, and 
$\cf(\sup(Q)) = \beta$. 
Then the set 
$$
\{ P \in Q \cap \mathcal Y : 
\textrm{$P$ is $\vec S$-strong}, \ 
\cf(\sup(P)) = P \cap \kappa \}
$$
is stationary in $P_{\beta}(Q)$.
\end{lemma}

\begin{proof}
Let $F : Q^{<\omega} \to Q$, and we will find 
$P \in Q \cap \mathcal Y$ such that 
$\cf(\sup(P)) = P \cap \kappa$, $P$ is $\vec S$-strong, and 
$P$ is closed under $F$. 
Since $Q \in \mathcal Y$, $Q \cap \kappa^+ = A_{\tau,\beta}$ is 
closed under $H^*$. 
As $Q \cap \kappa^+$ is the union of the increasing and continuous chain 
$\{ A_{\tau,i} : i < \beta \}$, 
there exists a club $C \subseteq \beta$ 
such that for all $\alpha \in C$, 
$A_{\tau,\alpha}$ is closed under $H^*$ and 
$A_{\tau,\alpha} \cap \kappa = \alpha$. 
Then $Q$ is the union of the increasing and continuous chain 
$\{ Sk(A_{\tau,\alpha}) : \alpha \in C \}$. 
Fix a club $D \subseteq C$ such that for all $\alpha \in D$, 
$Sk(A_{\tau,\alpha})$ is closed under $F$.

For each $\alpha \in D$, let $Q_\alpha := Sk(A_{\tau,\alpha})$. 
We claim that for all $\alpha \in D$, $Q_\alpha \in Q$. 
Since $|Q_\alpha \cap \kappa^+| = 
|A_{\tau,\alpha}| < |\alpha|^+ < \beta$ and $\cf(\sup(Q)) = \beta$, 
it follows that 
$A_{\tau,\alpha}$ is a bounded subset of $Q \cap \kappa^+ = A_{\tau,\beta}$. 
Also 
$$
\cf(\sup(Q_\alpha)) \le \alpha < \beta = \cf(Q \cap \kappa).
$$
By Lemma 7.14, $\sup(Q_\alpha) \in Q$. 
But $A_{\tau,\alpha} = A_{\sup(A_{\tau,\alpha}),\alpha}$ by coherence, 
and since $\sup(A_{\tau,\alpha})$ and $\alpha$ are in $Q$, so is 
$A_{\tau,\alpha}$ by elementarity. 
Hence $Q_\alpha \in Q$ by elementarity.

Fix a club $E \subseteq \lim(D)$ such that for all $\alpha \in E$, 
for all $\gamma \in \alpha \cap D$, $Q_\gamma \in Q_\alpha$. 
In particular, for all $\alpha \in E$, since 
$Q_\alpha = \bigcup \{ Q_\gamma : \gamma \in \alpha \cap D \}$, 
it follows that $\cf(\sup(Q_\alpha)) = \cf(\ot(\alpha \cap D))$. 
So if $\alpha \in E$ is regular, then $\cf(\sup(Q_\alpha)) = \alpha$.

Since $\beta \in S_{\tau+1}$, $S_\tau \cap \beta$ is stationary in $\beta$. 
So we can fix $\alpha \in E \cap S_{\tau}$. 
To finish the proof, it suffices to show that 
$Q_\alpha= Sk(A_{\tau,\alpha})$ is in $Q \cap \mathcal Y$, 
$Q_\alpha$ is $\vec S$-strong, 
$\cf(\sup(Q_\alpha)) = Q_\alpha \cap \kappa = \alpha$,  and 
$Q_\alpha$ is closed under $F$.

We know that $Q_\alpha$ is closed under $F$ by the definition of $D$. 
We previously observed that $Q_\alpha \in Q$, 
and since $\alpha \in E$ is regular, $\cf(\sup(Q_\alpha)) = \alpha$. 
In particular, $Q_\alpha \in \mathcal Y$ by Lemma 7.15. 
To see that $Q_\alpha$ is $\vec S$-strong, 
let $\xi \in Q_\alpha \cap \kappa^+ = A_{\tau,\alpha}$. 
Since $\alpha \in S_\tau$ and $\tau$ is a limit ordinal, 
$\alpha \in S_\eta$ for all $\eta \in A_{\tau,\alpha}$. 
In particular, $Q_\alpha \cap \kappa = \alpha \in S_\xi$.
\end{proof}

\begin{lemma}
Suppose that $Q \in \mathcal Y$ is $\vec S$-strong and 
$Q \prec (\mathcal A,\mathcal Y,\p)$. 
Let $\beta := Q \cap \kappa$. 
Suppose that the set 
$$
\{ P \in Q \cap \mathcal Y : 
\textrm{$P$ is $\vec S$-strong}, \ 
\cf(\sup(P)) = P \cap \kappa \}
$$
is stationary in $P_{\beta}(Q)$. 
Then the forcing poset $\p \cap Q$ forces that $\beta$ is a 
regular cardinal.
\end{lemma}

\begin{proof}
Let $\gamma < \beta$, and 
let $\dot f$ be a $(\p \cap Q)$-name for a function 
from $\gamma$ to $\beta$. 
Fix a condition $p \in \p \cap Q$, and we will find $q \le p$ in 
$\p \cap Q$ which forces that $\dot f$ is bounded in $\beta$. 
Let $F$ be the set of triples $(u,i,\xi)$ such that 
$u \in \p \cap Q$ and $u \Vdash_{\p \cap Q} \dot f(i) = \xi$.

Let $\{ g_n : n < \omega \}$ be a set of definable Skolem 
functions for the structure $(\mathcal A,\mathcal Y,\p)$. 
Since $Q \prec (\mathcal A,\mathcal Y,\p)$, $Q$ is closed under 
$g_n$ for all $n < \omega$. 
By the assumption of the lemma, we can fix $P \in Q \cap \mathcal Y$ 
such that $P$ is $\vec S$-strong, 
$\cf(\sup(P)) = P \cap \kappa$, 
$P$ is closed under $g_n$ for all $n < \omega$, 
and $P \prec (Q,\in,\p \cap Q,p,\gamma,F)$. 
In particular, $P \prec (\mathcal A,\mathcal Y,\p)$. 
As $p \in P$, Proposition 15.6 implies that 
$q := (A_p,B_p \cup \{ P \})$ is a strongly $P$-generic 
condition below $p$.

We claim that 
$$
q \Vdash_{\p \cap Q} \ran(\dot f) \subseteq P \cap \kappa.
$$
Since $P \cap \kappa < Q \cap \kappa = \beta$, this completes the proof. 
Let $i < \gamma$, and we will show that 
$$
q \Vdash_{\p \cap Q} \dot f(i) \in P \cap \kappa.
$$

Let $D$ be the set of $s \in \p \cap P$ such that for some 
$\xi \in P \cap \kappa$, $(s,i,\xi) \in F$. 
We claim that $D$ is dense in $\p \cap P$. 
So let $u \in \p \cap P$ be given. 
Then $u \in \p \cap Q$. 
Since $\dot f$ is a $(\p \cap Q)$-name for a function from $\gamma$ to $\beta$, 
there is $v \le u$ and $\xi < Q \cap \kappa$ such that $v \Vdash_{\p \cap Q} \dot f(i) = \xi$, 
and hence $(v,i,\xi) \in F$. 
Since $P \prec (Q,\in,\p \cap Q,p,\gamma,F)$ and $u$ and $i$ are in $P$, 
by elementarity there is $v \in P$ and $\xi \in P \cap \kappa$ 
such that $v \le u$ and $(v,i,\xi) \in F$. 
Then $v \le u$ and $v \in D$.

Since $q$ is strongly $P$-generic, $D$ is predense in $\p$ 
below $q$. 
Let $r \le q$ in $Q \cap \p$ decide the value of $\dot f(i)$ to be $\xi$, 
and we will show that $\xi \in P$. 
Then $r \le q$ is in $\p$, and $(r,i,\xi) \in F$. 
Since $D$ is predense in $\p$ below $q$, 
for some $u \in D$, $r$ and $u$ are compatible in $\p$. 
By the elementarity of $Q$, $\p \cap Q$ is closed under greatest 
lower bounds, and therefore 
$r$ and $u$ are also compatible in $\p \cap Q$. 
Since $u \in D$, by the definition of $D$ there is $\xi' \in P \cap \kappa$ 
such that $(u,i,\xi') \in F$. 
But $(r,i,\xi) \in F$ and $(u,i,\xi') \in F$ imply, by the compatibility of 
$r$ and $u$ in $\p \cap Q$, that $\xi = \xi'$. 
Since $\xi' \in P$, it follows that $\xi \in P$. 
\end{proof}

Recall that for a sequence 
$\vec a = \langle a_i : i < \omega_2 \rangle$ of countable sets, 
$S_{\vec a}$ is the set of limit ordinals 
$\alpha < \omega_2$ for which there exists 
a club $c \subseteq \alpha$ with order type $\cf(\alpha)$ such that for all 
$\beta < \alpha$, there is $i < \alpha$ with $c \cap \beta = a_i$. 
A set $S$ is in the approachability ideal 
$I[\omega_2]$ iff there exists such a sequence $\vec a$ and a club $D$ with 
$S \cap D \subseteq S_{\vec a}$. 
In particular, if $I[\omega_2]$ contains a stationary subset of 
$\omega_2 \cap \cof(\omega_1)$, then for some sequence $\vec a$, 
$S_{\vec a} \cap \cof(\omega_1)$ is stationary. 
We will show that this last statement fails in any generic extension by $\p$.

\begin{thm}
The forcing poset $\p$ forces that there is no stationary subset of 
$\omega_2 \cap \cof(\omega_1)$ in the approachability ideal $I[\omega_2]$.
\end{thm}

\begin{proof}
Suppose for a contradiction that $p$ is a condition, 
$\vec{a} = \langle \dot a_i : i < \kappa \rangle$ is a sequence of 
$\p$-names for countable subsets of $\kappa$, and $p$ forces that 
$\dot S_{\vec a} \cap \cof(\omega_1)$ is stationary. 
Without loss of generality, assume that each $\dot a_i$ is a nice name, 
which means that for some sequence of antichains 
$\langle A^i_\alpha : \alpha < \kappa \rangle$ 
of $\p$, $\dot a_i$ is equal to the set of pairs 
$\{ (p,\check \alpha) : p \in A^i_\alpha, \ \alpha < \kappa \}$. 
As $\p$ is $\kappa^+$-c.c., each name $\dot a_i$ is a 
member of $H(\kappa^+)$. 
Fix $\gamma < \kappa^+$ such that for all $i < \kappa$, 
$\dot a_i$ is in $f^*[\gamma]$, 
where $f^* : \kappa^+ \to H(\kappa^+)$ 
is the bijection described in Notation 7.1.

Let $M$ be an elementary substructure of $(\mathcal A,\mathcal Y,\p)$ 
such that $|M| = \kappa$, $M \cap \kappa^+ \in \kappa^+ \cap \cof(\kappa)$, 
and $\gamma < M \cap \kappa^+$. 
Let $\tau := M \cap \kappa^+$. 
Since $\gamma < M \cap \kappa^+$ and $M$ is closed under $f^*$, 
it follows that for all $i < \kappa$, $\dot a_i$ is in $M$. 
Fix $\p$-names $\dot Y$ and $\dot D_\tau$ such that 
$\p$ forces 
$$
\dot Y = \{ P : \exists p \in \dot G \ (P \in B_p) \} \ \textrm{and} \ 
\dot D_\tau = \{ P \cap \kappa : P \in \dot Y, \ \tau \in P \}.
$$
An easy observation which follows from Lemma 15.5 is that $\dot D_\tau$ 
is forced to be cofinal in $\kappa$. 
Therefore $\lim(\dot D_\tau)$ is forced to be club in $\kappa$.

For each $\alpha < \kappa$, let $Q_\alpha := Sk(A_{\tau,\alpha})$. 
Since $\tau$ is the union of the increasing and continuous sequence 
$\{ A_{\tau,\alpha} : \alpha < \kappa \}$, clearly 
$M = Sk(\tau)$ is the union of the increasing and continuous sequence 
$\{ Q_\alpha : \alpha < \kappa \}$. 
Let $E$ be a club subset of $\kappa$ such that for all $\alpha \in E$, 
$Q_\alpha \cap \kappa = \alpha$, 
$Q_{\alpha} \cap \kappa^+ = A_{\tau,\alpha}$, 
$Q_\alpha \prec (\mathcal A,\mathcal Y,\p)$, and for all $i < \alpha$, 
$\dot a_i \in Q_\alpha$.

Clearly $p$ forces that $E \cap \lim(\dot D_\tau)$ is club in $\kappa$. 
Since $p$ forces that $\dot S_{\vec a} \cap \cof(\omega_1)$ is stationary, 
we can fix $q \le p$ and $\alpha < \kappa$ such that 
$q$ forces that $\alpha$ is in 
$E \cap \lim(\dot D_\tau) \cap \dot S_{\vec a} \cap \cof(\omega_1)$. 
Since $q$ forces that $\cf(\alpha) = \omega_1$, clearly 
$\alpha$ has uncountable cofinality. 
Let $Q := Q_\alpha$. 
Then by Proposition 16.5, $Q \in \mathcal Y$ is $\vec S$-strong, 
$Q \cap \kappa = \alpha \in S_{\tau+1}$, 
$Q \cap \kappa^+ = A_{\tau,\alpha}$, 
$\cf(\sup(Q)) = Q \cap \kappa = \alpha$, and $q$ forces that 
$Q$ is in $\dot Y$. 
By extending $q$ if necessary, we can assume without loss of 
generality that $Q \in B_q$.

By Lemma 16.6, the set 
$$
\{ P \in Q \cap \mathcal Y : 
\textrm{$P$ is $\vec S$-strong}, \ 
\cf(\sup(P)) = P \cap \kappa \}
$$
is stationary in $P_{\alpha}(Q)$. 
By Lemma 16.7, the forcing poset $\p \cap Q$ forces that $\alpha$ 
is a regular cardinal.
Since $Q \in B_q$, 
clearly $q \le q_Q := (\emptyset,\{ Q \})$. 

Let $G$ be a $V$-generic filter on $\p$ containing $q$, and we will 
get a contradiction by considering the generic extension $V[G]$. 
Since $q \le q_Q$, it follows that $q_Q \in G$. 
By Proposition 16.3, 
$V[G]$ can be factored as 
$$
V[G] = V[G \cap Q][H],
$$
where $G \cap Q$ is a $V$-generic filter on $\p \cap Q$, $H$ is a 
$V[G \cap Q]$-generic filter on 
$(\p / q_Q) / (G \cap Q)$, and the pair 
$(V[G \cap Q],V[G])$ has the $\omega_1$-approximation property.

As $\alpha$ is in $S_{\vec a} \cap \cof(\omega_1)$, 
in $V[G]$ there is a 
club $c \subseteq \alpha$ with order type $\omega_1$ 
such that for all 
$\beta < \alpha$, there is $i < \alpha$ such that 
$c \cap \beta = \dot a_i^G$. 
For any such $\beta$ and $i$, $\dot a_i \in Q$, and 
$\dot a_i^G = c \cap \beta$ is a subset of $Q \cap \kappa = \alpha$. 
By Lemma 16.4, it follows that $\dot a_i^G \in V[G \cap Q]$. 
So for all $\beta < \alpha$, $c \cap \beta \in V[G \cap Q]$.

By Lemma 6.1, $c \in V[G \cap Q]$. 
But since $c$ has order type $\omega_1$, it follows that 
$\alpha$ has cofinality $\omega_1$ 
in $V[G \cap Q]$. 
Now $\alpha \in S_{\tau+1}$, and in particular, 
$\alpha$ is inaccessible in $V$, but $\alpha$ is not regular in $V[G \cap Q]$. 
However, we previously observed that $\p \cap Q$ forces that 
$\alpha$ is a regular cardinal, so we have a contradiction.
\end{proof}

\bibliographystyle{plain}
\bibliography{paper27}

\begin{thebibliography}{10}

\bibitem{baumgartner}
J.~Baumgartner, A.~Taylor, and S.~Wagon.
\newblock On splitting stationary subsets of large cardinals.
\newblock {\em J. Symbolic Logic}, 42(2):203--214, 1977.

\bibitem{jk26}
S.~Cox and J.~Krueger.
\newblock Quotients of strongly proper forcings and guessing models.
\newblock {\em J. Symbolic Logic}, 81(1):264--283, 2016.

\bibitem{friedman}
S.D. Friedman.
\newblock Forcing with finite conditions.
\newblock In {\em Set Theory: Centre de Recerca Matem\`atica, Barcelona,
  2003-2004, Trends in Mathematics}, pages 285--295. Birkh\"auser Verlag, 2006.

\bibitem{kruegerthin}
S.D. Friedman and J.~Krueger.
\newblock Thin stationary sets and disjoint club sequences.
\newblock {\em Trans. Amer. Math. Soc.}, 359:2407--2420, 2007.

\bibitem{hamkins}
J.~Hamkins.
\newblock Extensions with the approximation and cover properties have no new
  large cardinals.
\newblock {\em Fund. Math.}, 180:257--277, 2003.

\bibitem{jk21}
J.~Krueger.
\newblock Forcing with adequate sets of models as side conditions.
\newblock To appear in \emph{Mathematical Logic Quarterly}.

\bibitem{jk23}
J.~Krueger.
\newblock Coherent adequate sets and forcing square.
\newblock {\em Fund. Math.}, 224:279--300, 2014.

\bibitem{jk22}
J.~Krueger.
\newblock Strongly adequate sets and adding a club with finite conditions.
\newblock {\em Arch. Math. Logic}, 53(1-2):119--136, 2014.

\bibitem{jk24}
J.~Krueger.
\newblock Adding a club with finite conditions, part {II}.
\newblock {\em Arch. Math. Logic}, 54(1-2):161--172, 2015.

\bibitem{jk25}
J.~Krueger and M.A. Mota.
\newblock Coherent adequate forcing and preserving {CH}.
\newblock {\em J. Math. Log.}, 15(2), 2015.

\bibitem{mitchellold}
W.~Mitchell.
\newblock Aronszajn trees and the independence of the transfer property.
\newblock {\em Ann. Math. Logic}, pages 21--46, 1972.

\bibitem{mitchell}
W.~Mitchell.
\newblock {$I[\omega_2]$} can be the nonstationary ideal on {${\rm
  Cof}(\omega_1)$}.
\newblock {\em Trans. Amer. Math. Soc.}, 361(2):561--601, 2009.

\bibitem{neeman}
I.~Neeman.
\newblock Forcing with sequences of models of two types.
\newblock {\em Notre Dame J. Form. Log.}, 55(2):265--298, 2014.

\bibitem{shelah}
S.~Shelah.
\newblock Reflecting stationary sets and successors of singular cardinals.
\newblock {\em Arch. Math. Logic}, 31(1):25--53, 1991.

\bibitem{todor}
S.~{Todor\v cevi\' c}.
\newblock A note on the proper forcing axiom.
\newblock In {\em Axiomatic set theory (Boulder, Colo., 1983)}, volume~31 of
  {\em Contemp. Math.}, pages 209--218. Amer. Math. Soc., Providence, RI, 1984.

\end{thebibliography}

\end{document}